\documentclass[11pt,a4paper,reqno]{amsart}
\usepackage{amsmath,amssymb,amsfonts,epsfig,mathrsfs,cite}
\usepackage[T1]{fontenc}
\usepackage{color}
\usepackage{array}
\usepackage{amsthm}
\usepackage{amstext}
\usepackage{graphicx}
\usepackage{setspace}
\usepackage{floatrow} 
\usepackage[title]{appendix}

\usepackage{hyperref}

\usepackage{subfigure}
\makeatletter
\@namedef{subjclassname@2020}{%
	\textup{2020} Mathematics Subject Classification}
\makeatother

\usepackage[margin=2.5cm]{geometry}
\usepackage{color}
\usepackage{enumitem}
\usepackage{amscd,psfrag}
\usepackage{yhmath}
\usepackage[mathscr]{eucal}
\numberwithin{equation}{section}
\usepackage{comment}

\allowdisplaybreaks[4]
\usepackage{appendix}
\usepackage{slashed}
\usepackage{bbm}
\makeatletter
\pdfpageheight\paperheight
\pdfpagewidth\paperwidth

\usepackage{epstopdf}
\usepackage{chngcntr}
\counterwithin{figure}{section}
\usepackage{mathrsfs}

\usepackage{indentfirst}	

\usepackage[normalem]{ulem}
\theoremstyle{plain}

\newtheorem{theorem}{Theorem}[section] 
\newtheorem{definition}[theorem]{Definition} 
\newtheorem{lemma}{Lemma}[section]
\newtheorem{corollary}[theorem]{Corollary}

\newtheorem{remark}[theorem]{Remark}

\def\XXint#1#2#3{{\setbox0=\hbox{$#1{#2#3}{\int}$ }
		\vcenter{\hbox{$#2#3$ }}\kern-.6\wd0}}

\newcommand{\R}{{\mathbb{R}}}

\title{Global well-posedness for one-dimensional compressible Navier--Stokes system in dynamic combustion with small $BV\cap L^1$ initial data}

\author{Siran Li}

\address{Siran Li: School of Mathematical Sciences $\&$ CMA-Shanghai, Shanghai Jiao Tong University, No.~6 Science Buildings,
	800 Dongchuan Road, Minhang District, Shanghai, China (200240)}

\email{\texttt{siran.li@sjtu.edu.cn}}

\author{Haitao Wang}
\address{Haitao Wang: School of Mathematical Sciences, Institute of Natural Sciences, MOE-LSC, CMA-Shanghai, Shanghai Jiao Tong University, Shanghai, China (200240)}

\email{\texttt{haitallica@sjtu.edu.cn}}

\author{Jianing Yang}

\address{Jianing Yang: School of Mathematical Sciences, Shanghai Jiao Tong University, No.~6 Science Buildings,
	800 Dongchuan Road, Minhang District, Shanghai, China (200240)}

\email{\texttt{jnyang22@sjtu.edu.cn}}


\subjclass[2020]{}
\date{\today}

\begin{document}
	
	\begin{abstract}
	We establish the global well-posedness theory of small BV weak solutions to a one-dimensional compressible Navier--Stokes model for reacting gas mixtures in dynamic combustion. The unknowns of the PDE system consist of the specific volume, velocity, temperature, and mass fraction of the reactant. For initial data that are small perturbations around the constant equilibrium state $(1, 0, 1, 0)$ in the $L^1(\R) \cap {\rm BV}(\R)$-norm, we  establish the local-in-time existence of weak solutions via an iterative scheme, show the stability and uniqueness of local weak solutions, and prove the global-in-time existence of solutions for initial data with small BV-norm via an analysis of the Green's function of the linearised system. The large-time behaviour of the global  BV weak solutions is also characterised. This work is motivated by and extends the recent global well-posedness theory for BV weak solutions to the one-dimensional isentropic Navier--Stokes and Navier--Stokes--Fourier systems developed in [T.-P. Liu, S.-H. Yu, Commun. Pure Appl. Math. 75 (2022), 223--348] and [H. Wang, S.-H. Yu, X. Zhang, Arch. Ration. Mech. Anal. 245 (2022), 375--477].
	\end{abstract}
	\maketitle
	
	\tableofcontents
	\section{Introduction}\label{sec: intro}

We are concerned with the global well-posedness theory of a one-dimensional (1D) compressible Navier--Stokes model for a reacting gas mixture in dynamic combustion. The system of partial differential equations (PDE)  in the Lagrangian coordinates reads as follows:
	\begin{equation}\label{PDE,1}
		\left\{\begin{array}{l}
			v_t-u_x=0, \\
			u_t+p_x=\left(\frac{\mu u_x}{v}\right)_x, \\
			E_t+\left(p u\right)_x=\left(\frac{\mu u u_x}{v}\right)_x+\left(\frac{\nu \theta_x}{v}\right)_x+\left(\frac{qD z_x}{v^2}\right)_x,\\
			z_t+K \phi(\theta) z=\left(\frac{D}{v^2} z_x\right)_x,
		\end{array}\right.
	\end{equation}
See, \emph{e.g.}, G.-Q. Chen \cite{Chen1992}. Throughout this work, the physical variables $v\equiv \frac{1}{\rho}$, $u$, $\theta$, and $z$ denote the specific volume (namely, the inverse of density), velocity, temperature, and the mass fraction of the reactant, respectively.

Let us first describe the physics of the PDE system~\eqref{PDE,1}. The reacting gas mixture in consideration has total specific energy 
	\begin{align}\label{E energy}
		E = e + \frac{u^2}{2} + qz,
	\end{align} 
where $e$ is the specific internal energy, and $q$ is the difference between the heat of the reactant and the product. For ideal gases, we have the constitutive relations:
\begin{equation}\label{constitutive relations}
    e = c_v\theta,\qquad p=\frac{a\theta}{v}
\end{equation}
where $c_v$ is the specific heat capacity at constant volume, and $a$ equals the Boltzmann's constant times the molecular weight. The components of the gas mixture in consideration is assumed to obey the same \emph{$\gamma$-law} with
	\begin{equation*}
		\gamma = \frac{c_p}{c_v} > 1,
	\end{equation*}
where $c_p$ is the specific heat capacity taken at constant pressure. The function $\phi=\phi(\theta):[0,\infty) \to [0,\infty)$ describes the rate of chemical reaction at temperature $\theta$. Throughout this paper, it is assumed to be Lipschitz continuous. The positive constants $\mu$, $\nu$, $D$, and $K$ are the bulk viscosity, heat conductivity, diffusivity, and reactant rate coefficient, respectively.

In view of the discussions in the previous paragraph, we may recast Eq.~\eqref{PDE,1} as follows: 
	\begin{equation}\label{PDE,2}
		\left\{\begin{array}{l}
			v_t-u_x=0, \\
			u_t+p_x=\left(\frac{\mu u_x}{v}\right)_x, \\
			\theta_t+\frac{p}{c_v}u_x-\frac{\mu}{c_v v}u_x^2=\left(\frac{\nu \theta_x}{c_v v}\right)_x+\frac{q}{c_v}K\phi(\theta)z,\\
			z_t+K \phi(\theta) z=\left(\frac{D}{v^2} z_x\right)_x .
		\end{array}\right.
	\end{equation}
The pressure $p$ is given by the constitutive relations in Eq.~\eqref{constitutive relations}. 

There is abundant literature on well-posedness and large-time behaviour of solutions to Eq.~\eqref{PDE,1}, or equivalently, Eq.~\eqref{PDE,2}. Our list of references here is by no means exhaustive. Chen \cite{Chen1992} established the global existence of weak solutions for \eqref{PDE,1} on a bounded spatial domain $x\in [0,1]$ with more general reaction rate functions $\phi$ (\emph{e.g.}, for $\phi$ satisfying the Arrenhius' law, with a jump discontinuity on $(0,\infty)$); large-time behaviour of weak solutions has also been obtained. Moreover, in the case of Lipschitz continuous reaction rate functions, Chen--Hoff--Trivisa \cite{Chen2002, Chen2003} proved the existence for large discontinuous initial data.  Li ~\cite{Li2017} extended the above results to unbounded domains $\R$ or $\R_+$ with suitable boundary conditions. See also Ducomet \cite{Ducomet1999}, Wang \cite{dehua}, Jiang--Zheng \cite{Jiang2014}, and many others on the well-posedness theory and asymptotic behaviour of classical solutions to PDEs for 1D reactive and/or radiative gas dynamics. In addition, for the 1D combustion dynamic PDE models, detailed asymptotic analysis for the large-time behaviour has been carried out for the interactions of specific types of elementary waves. See \cite{fzz, peng, pl, xf} among other references.

We also mention Jenssen--Lyng--Williams~\cite{Jenssen2005}, Feng--Hong--Zhu~\cite{Feng2021}, Wang--Wen \cite{ww22}, Wang--Wu \cite{ww23} and Gao--Huang--Kuang--Wang--Xiang for the study of dynamic combustion PDE in multi-dimensions, as well as Qin--Zhang--Su--Cao  \cite{QinYM2016}, Zhang \cite{ZhangJL2017}, Liao--Wang--Zhao \cite{lwz}, Wan--Zhang \cite{wz23}, and Zhu \cite{zhu}, etc. for refined results under assumptions of spherical and cylindrical symmetry of solutions. It is interesting to investigate whether the results and methods in our work can be extended to the above settings.

In the above, by \emph{weak solutions} we mean the functions $(v,u,\theta,z)$ satisfying 
\begin{equation}\label{weak sol, def, Nov25}
    \begin{cases}
v-1 \in L^\infty_t H^1_x,\\
(u, \theta-1, z) \in L^\infty_t H^1_x \cap L^2_t H^2_x,\\
(u_t,\theta_t,z_t) \in L^2_t L^2_x.
    \end{cases}
\end{equation}
The well-posedness theory for weak solutions has been established on bounded or unbounded 1D domains for initial data $(v_0, u_0, \theta_0, z_0)\equiv (u,v,\theta,z)\big|_{\{t=0\}}$ satisfying 
\begin{equation}\label{initial data, weak sol, Nov25}
\begin{cases}
    0 < m_0 \leq v_0(x), \theta_0(x) \leq M_0,\quad |u_0(x)| \leq M_1<\infty, \quad 0 \leq z_0(x) \leq 1\qquad\text{for a.e. } x;\\
    \left( v_0-1, u, \theta_0-1, z_0 \right) \in H^1_x,\qquad z_0 \in L^1_x,
    \end{cases} 
\end{equation}
where $m_0$, $M_0$, and $M_1$ are positive constants. See  \cite{Chen1992, Chen2002, Chen2003, Li2017}. In contrast, this paper aims to establish the weak solution theory in a weaker topology than that in Eq.~\eqref{weak sol, def, Nov25}. In particular, the specific density $v=\rho^{-1}$ is now taken to be a perturbation of the constant background state:
\begin{align*}
    v-1 \in C_t\, {\rm BV}_x\qquad \text{in } \R^2_+ = \R \times [0,\infty).
\end{align*}
(See Definition~\ref{new def, weak sol, Nov25} for precise formulation.) This shall be established for initial data satisfying $$\|v_0-1\|_{(L^1\cap {\rm BV})(\R)}\ll 1$$ together with similar conditions on $u_0, \theta_0$, and $z_0$.

Our work is primarily motivated by the recent seminal work \cite{LiuTP2022} by T.-P. Liu and S.-H. Yu, which establishes the global existence theory of 1D compressible Navier--Stokes equations for small BV initial data, through a delicate study on the Green's function for 1D heat equation in divergence-form with BV-coefficients. It echoes the classical theory for the system of conservation laws in one spatial dimension via Glimm's scheme of random choice (\cite{glimm}; also \emph{cf.} Dafermos \cite{daf} and  Bressan \cite{bres} for comprehensive treatment of hyperbolic conservation laws). More precisely, Liu--Yu \cite{LiuTP2022} considered the isentropic Navier--Stokes equations
	\begin{equation}
		\begin{cases}
			\label{isen NS}
			v_t-u_x=0,\\
			u_t+p(v)_x=\left(\frac{\mu u_x}{v}\right)_x,
		\end{cases}
	\end{equation}
	with initial data small in the $L^1\cap {\rm BV}$-norm, provided that $p^{\prime}(v)<0$. One of the key ingredients of the proof in \cite{LiuTP2022} is the construction of the fundamental solution $H(x,t;y;f)$ to the heat equation with BV coefficients ($(x,t)\in \R^2_+$ are the spacetime variables and $y,f$ are parameters):
	\begin{equation}\label{heat with BV coeff, Nov25}
		\begin{cases}
\partial_tH(x,t;y;f) -\partial_x \Big(f\partial_xH(x,t;y;f)\Big)=0,\\
			H(x,0;y;f)=\delta(x-y),
		\end{cases}
	\end{equation}
	where $\delta(z)$ is the Dirac delta measure on $\R$ supported at $z=0$, and $f$ is a $BV$ function with 
	\begin{equation*}
\inf_{x\in\mathbb{R}}f(x)>0,\quad \|f\|_{BV}\ll 1.
	\end{equation*}
The analytical properties of $H(x,t;y;f)$ together with an iteration scheme lead to the local  existence and continuous dependence on initial data of the weak solutions. In addition, by introducing an ``\emph{effective kernel function}'' that interpolates between the short-time heat kernel and long-time Green's function, Liu-Yu derived integral representations of weak solutions, and thus proved that weak solutions exist globally in time and decay to the constant equilibrium state at the optimal rate of $t^{-1/2}$ for polytropic gases, \emph{i.e.}, when $p(v) = A v^{-\gamma}$ with $1 \leq \gamma < e$.

The pointwise construction
of the Green’s function in 1D was initiated in Zeng \cite{z1}. It has been combined with the classical time-asymptotic analysis of the pointwise convergence of smooth solutions to nonlinear waves in Liu--Zeng \cite{z2}. The particular construction of the Green's function here and in  \cite{LiuTP2022, WangHT2022} is motivated by that for the Boltzmann equation, especially the ``particlelike'' and ``wavelike'' decomposition in Liu--Yu \cite{ly'}. See Deng--Yu \cite{dy} for related developments.

Wang--Yu--Zhang~\cite{WangHT2022} extended the well-posedness theory from the 1D isentropic Navier--Stokes Eq.~\eqref{isen NS} to the full Navier--Stokes--Fourier system:
	\begin{equation}
		\label{full NS}
		\begin{cases}
			v_t-u_x=0,\\
			u_t+p(v,\theta)_x=\left(\frac{\mu u_x}{v}\right)_x,\\
			\left(e+\frac{1}{2}u^2\right)_t+(pu)_x=\left(\frac{\kappa}{v}\theta_x+\frac{\mu}{v}u u_x\right)_x,
		\end{cases}
	\end{equation}
also under the small $BV\cap L_x^1$ assumption for the initial data. Based on the heat kernel analysis in \cite{LiuTP2022}, the authors developed new H\"{o}lder-in-time estimates for the heat kernel and Lipschitz in space estimates for the fluxes of $u$ and $\theta$.

In this work, we extend the global well-posedness theory for small BV-solutions laid down in \cite{LiuTP2022} and further developed in \cite{WangHT2022} to the 1D Navier--Stokes combustion model for $\gamma$-law gas mixtures, namely Eqs.~\eqref{PDE,1} or \eqref{PDE,2}. We first state our framework of weak solutions.

\begin{definition}\label{new def, weak sol, Nov25}
The quadruplet $(v,u,\theta,z):=[0,t_\sharp) \times \R \to \R^4$ is a weak solution to Eqs.~\eqref{PDE,1} or \eqref{PDE,2} if the following holds:
\begin{enumerate}
    \item 
The PDE holds in the distributional sense: for any test function $\varphi\in C_0^{\infty}([0,t_\sharp)\times \mathbb{R})$,
\begin{equation}
\label{def distri}
\left\{\begin{aligned}
&\int_0^{+\infty} \int_{\mathbb{R}}\left[\varphi_x u-\varphi_t v\right]\,\mathrm{d}x \,\mathrm{d}t=\int_{\mathbb{R}} \varphi(x, 0) v(x, 0)\,\mathrm{d}x, \\
				&\int_0^{+\infty} \int_{\mathbb{R}}\left[\varphi_x\left(\frac{\mu u_x}{v}-p\right)-\varphi_t u\right]\,\mathrm{d}x \,\mathrm{d} t=\int_{\mathbb{R}}\varphi(x,0)u(x, 0) \,\mathrm{d}x,\\
				&\int_0^{+\infty} \int_{\mathbb{R}}\left[\varphi_x\left(\frac{\nu}{c_v v}\theta_x\right)+\varphi\left(\frac{p}{c_v} u_x-\frac{\mu}{c_v v}\left(u_x\right)^2-\frac{q}{c_v}K\phi(\theta)z\right)-\varphi_t \theta\right]\,\mathrm{d}x \,\mathrm{d}t\\
				&\qquad\qquad =\int_{\mathbb{R}}\varphi(x,0) \theta(x,0)\,\mathrm{d}x,\\
				&\int_0^{+\infty} \int_{\mathbb{R}}\left[\varphi_x\left(\frac{D z_x}{v^2}\right)+\varphi K\phi(\theta)z-\varphi_t z\right]\,\mathrm{d}x \,\mathrm{d} t=\int_{\mathbb{R}}\varphi(x,0)z(x, 0) \,\mathrm{d}x;
			\end{aligned}\right.
		\end{equation}
\item 
The following regularity conditions are verified:
\begin{equation}
			\label{weak solu}
			\left\{\begin{array}{l}
				v(x, t)-1 \in C\left(\left[0, t_{\sharp}\right) ; L^1(\mathbb{R}) \cap {\rm BV}(\mathbb{R})\right); \\
				u(x, t) \in L^{\infty}\left(0, t_{\sharp} ; W^{1,1}(\mathbb{R}) \cap L^{\infty}(\mathbb{R})\right), \quad \sqrt{t} u_x(x, t) \in L^{\infty}\left(0, t_{\sharp} ; L^{\infty}(\mathbb{R})\right), \\
				\sqrt{t} u_t(x, t) \in L^{\infty}\left(0, t_{\sharp} ; L^1(\mathbb{R})\right), \quad t u_t(x, t) \in L^{\infty}\left(0, t_{\sharp} ; L^{\infty}(\mathbb{R})\right); \\
				\theta(x, t)-1 \in L^{\infty}\left(0, t_{\sharp} ; W^{1,1}(\mathbb{R}) \cap L^{\infty}(\mathbb{R})\right), \quad \sqrt{t} \theta_x(x, t) \in L^{\infty}\left(0, t_{\sharp} ; L^{\infty}(\mathbb{R})\right), \\
				\sqrt{t} \theta_t(x, t) \in L^{\infty}\left(0, t_{\sharp} ; L^1(\mathbb{R})\right), \quad t \theta_t(x, t) \in L^{\infty}\left(0, t_{\sharp} ; L^{\infty}(\mathbb{R})\right);\\
				z(x, t) \in L^{\infty}\left(0, t_{\sharp} ; W^{1,1}(\mathbb{R}) \cap L^{\infty}(\mathbb{R})\right), \quad \sqrt{t} z_x(x, t) \in L^{\infty}\left(0, t_{\sharp} ; L^{\infty}(\mathbb{R})\right), \\
				\sqrt{t} z_t(x, t) \in L^{\infty}\left(0, t_{\sharp} ; L^1(\mathbb{R})\right), \quad t z_t(x, t) \in L^{\infty}\left(0, t_{\sharp} ; L^{\infty}(\mathbb{R})\right).
			\end{array}\right.
		\end{equation}
\end{enumerate} 
\end{definition}

The main results of our current work are to establish the global existence, stability, uniqueness, and large-time behaviour of solutions to Eqs.~\eqref{PDE,1} or \eqref{PDE,2} in the sense of Definition~\ref{new def, weak sol, Nov25}, provided that the initial data is a small ${L^1}\cap {\rm BV}$-perturbation of the constant background state $[1, 0, 1, 0]^\top$. Recall that $$\|w\|_{\rm BV} := \|w\|_{L^\infty_x} + {\rm Tot.\,Var.} (w;\R),$$ where $ {\rm Tot.\,Var.} (w;\R)$ is the total variation of $w$ over $x$. Unless otherwise specified, all the norms with respect to the spatial variable $x$ are taken over $\R$. The key smallness condition for the initial data in this paper reads as follows: 

\begin{quote}
\emph{There exists a small positive constant $\delta \ll 1$ such that}
\begin{align}\label{ini}
	&\left\|v_0-1\right\|_{L_x^1}+\left\|v_0\right\|_{B V}+\left\|u_0\right\|_{L_x^1}+\left\|u_0\right\|_{B V}+\left\|\theta_0-1\right\|_{L_x^1}\nonumber\\
    &\qquad \qquad +\left\|\theta_0\right\|_{B V} +\left\|z_0\right\|_{L_x^1}+\left\|z_0\right\|_{B V}<\delta.
	\end{align}
\end{quote}

	Our main theorems are as follows:
	\begin{theorem}[Local existence and regularity]\label{main thm: local exi}
There exists a universal constant $\delta>0$ such that the following holds. Suppose that the initial data $(v_0, u_0, \theta_0, z_0)\equiv (u,v,\theta,z)\big|_{\{t=0\}}$ satisfy \eqref{ini}. Then there exists $t_{\sharp}>0$ such that Eq.~\eqref{PDE,2} (with Lipschitz continuous reacting rate function $\phi$) admits a weak solution $(v,u,\theta,z)$ in the sense of Definition~\ref{new def, weak sol, Nov25} over $\R\times [0,t_\sharp)$. Moreover, for some constant $C_\sharp>0$ we have the following:
		\begin{itemize}
			\item Regularity estimates:
			\begin{equation}
				\left\{\begin{array}{l}
					\max \left\{\begin{array}{l}
						\|u(\cdot, t)\|_{L_x^1},\|u(\cdot, t)\|_{L_x^{\infty}},\left\|u_x(\cdot, t)\right\|_{L_x^1}, \sqrt{t}\left\|u_x(\cdot, t)\right\|_{L_x^{\infty}}, \\
						\sqrt{t}\left\|u_t(\cdot, t)\right\|_{L_x^1}, t\left\|u_t(\cdot, t)\right\|_{L_x^{\infty}}
					\end{array}\right\} \leq 2 C_{\sharp} \delta, \\
					\max \left\{\begin{array}{l}
						\|\theta(\cdot, t)-1\|_{L_x^1},\|\theta(\cdot, t)-1\|_{L_x^{\infty}},\left\|\theta_x(\cdot, t)\right\|_{L_x^1}, \sqrt{t}\left\|\theta_x(\cdot, t)\right\|_{L_x^{\infty}}, \\
						\sqrt{t}\left\|\theta_t(\cdot, t)\right\|_{L_x^1}, t\left\|\theta_t(\cdot, t)\right\|_{L_x^{\infty}}
					\end{array}\right\} \leq 2 C_{\sharp} \delta, \\
					\max \left\{\begin{array}{l}
						\left.\|v(\cdot, t)\|_{B V},\|v(\cdot, t)-1\|_{L_x^1},\|v(\cdot, t)-1\|_{L_x^{\infty}}, \sqrt{t}\left\|v_t(\cdot, t)\right\|_{L_x^{\infty}}\right\} \leq 2 C_{\sharp} \delta, \\
						v-1=v_{\tilde{a}}^*+v_j^*, \quad v_j^*(x, t)=\sum\limits_{\omega<x, \omega \in \mathcal{D}}\left.v(y)\right|_{y=\omega^-}^{y=\omega^+}  h(x-\omega), \quad v_{\tilde{a}}^* \text { is continuous, }
					\end{array}\right. \\
					\max \left\{\begin{array}{l}
						\|z(\cdot, t)\|_{L_x^1}, \|z(\cdot, t)\|_{L_x^{\infty}},\left\|z_x(\cdot, t)\right\|_{L_x^1},\sqrt{t}\left\|z_x(\cdot, t)\right\|_{L_x^{\infty}},\\
						\sqrt{t}\left\|z_t(\cdot, t)\right\|_{L_x^1}, t\left\|z_t(\cdot, t)\right\|_{L_x^{\infty}}
					\end{array}\right\} \leq 2 C_{\sharp} \delta, \\
					\left|\left.v(\cdot, t)\right|_{x=\omega^{-}}^{x=\omega^{+}}\right|\leq 2\left|\left. v_0^*(\cdot)\right|_{x=\omega^{-}} ^{x=\omega^{+}} \right|, \quad \omega \in \mathcal{D},
				\end{array}\right.
			\end{equation}
			where $h(x)$ is the Heaviside step function, $\mathcal{D}$ is the set of discontinuous points of $v_0$.
			\item The fluxes of $u$, $\theta$ and $z$, defined respectively as 
			$$\frac{\mu u_x}{v}-p,\quad\frac{\nu}{c_v v}\theta_x-\int_{-\infty}^{x}(\frac{p}{c_v}u_y-\frac{\mu}{c_v v}u_y^2-\frac{q}{c_v}K\phi(\theta)z)\,\mathrm{d}y,\quad \text{and} \quad \frac{D}{v^2}z_x,$$
			are globally Lipschitz continuous with respect to $x$ for any $0<t<t_{\sharp}$.
			\item The specific volume $v$ has the following continuous properties: For $0\leq s<t<t_{\sharp}$,
			\begin{equation}\label{minus v}
				\left\{\begin{array}{l}
					\|v(\cdot, t)-v(\cdot, s)\|_{B V} \leq 2 C_{\sharp} \delta \, \frac{(t-s)|\log (t-s)|}{\sqrt{t}}, \\
					\|v(\cdot, t)-v(\cdot, s)\|_{L_x^{\infty}} \leq 2 C_{\sharp}\delta\,\frac{t-s}{\sqrt{t}}, \\
					\|v(\cdot, t)-v(\cdot, s)\|_{L_x^1} \leq 2 C_{\sharp} \delta\,(t-s).
				\end{array}\right.
			\end{equation}
		\end{itemize}
	\end{theorem}
	
	\begin{theorem}[Stability and uniqueness]\label{main thm: sta}
    There exists a universal constant $\delta>0$ (possibly smaller than that in Theorem~\ref{main thm: local exi} above) such that the following holds.    Let  $(v_0^{\epsilon},u_0^{\epsilon},\theta_0^{\epsilon},z_0^{\epsilon})$ and $(v_0^{\iota},u_0^{\iota},\theta_0^{\iota},z_0^{\iota})$  be two initial data, both verifying the smallness condition~\eqref{ini}. Then, there exists $t_*\ll 1$ such that the corresponding weak solutions $(v^{\epsilon},u^{\epsilon},\theta^{\epsilon},z^{\epsilon})$ and $(v^{\iota},u^{\iota},\theta^{\iota},z^{\iota})$  of Eq.~\eqref{PDE,2} both exist in $\R\times[0,t_*)$, and they satisfy the stability estimate:
		\begin{equation}
			\begin{aligned}
				&\left\|v^{\epsilon}-v^{\iota}\right\|_{L_x^1}+\left\| u^{\epsilon}-u^{\iota}\right\|_{L_x^1}+\left\| \theta^{\epsilon}-\theta^{\iota} \right\|_{L_x^1}+\left\| z^{\epsilon}-z^{\iota}\right\|_{L_x^1}\\
				&\qquad \leq C_2\left(\left\|v_0^{\epsilon}-v_0^{\iota}\right\|_{L_x^1}+\left\|v_0^{\epsilon}-v_0^{\iota}\right\|_{L_x^{\infty}}+\left\|v_0^{\epsilon}-v_0^{\iota}\right\|_{B V}+\left\|u_0^{\epsilon}-u_0^{\iota}\right\|_{L_x^{\infty}}\right.\\
				&\left.\qquad+\left\|u_0^{\epsilon}-u_0^{\iota}\right\|_{L_x^1}+\left\|\theta_0^{\epsilon}-\theta_0^{\iota}\right\|_{L_x^{\infty}}+\left\|\theta_0^{\epsilon}-\theta_0^{\iota}\right\|_{L_x^1}+\left\|z_0^{\epsilon}-z_0^{\iota}\right\|_{L_x^{\infty}}+\left\|z_0^{\epsilon}-z_0^{\iota}\right\|_{L_x^1}\right)
			\end{aligned}
		\end{equation}
for some $C_2>0$. In particular, for a fixed initial datum satisfying the condition~\eqref{ini}, weak solutions to Eq.~\eqref{PDE,2} in the sense of Definition~\ref{new def, weak sol, Nov25} are unique.
	\end{theorem}

	\begin{theorem}[Global existence and large-time behaviour]\label{main thm: global} 
    There exists a universal constant $\delta>0$ (possibly smaller than that in Theorem~\ref{main thm: sta} above) such that the following holds. Suppose that the initial datum $(v_0,u_0,\theta_0,z_0)$ satisfies the smallness condition \eqref{ini}. 		Then the unique local solution to Eq.~\eqref{PDE,2} constructed in Theorems \ref{main thm: local exi} and \ref{main thm: sta} extends globally in time. Moreover, there exists a positive constant $C_3$ such that
		\begin{equation*}
			\begin{aligned}
				& \|\sqrt{t+1}(v(\cdot, t)-1)\|_{L_x^{\infty}}+\|\sqrt{t+1} u(\cdot, t)\|_{L_x^{\infty}}+\|\sqrt{t+1}(\theta(\cdot, t)-1)\|_{L_x^{\infty}}+\|\sqrt{t+1}z(\cdot, t)\|_{L_x^{\infty}} \\
				& \quad+\left\|\sqrt{t} u_x(\cdot, t)\right\|_{L_x^{\infty}}+\left\|\sqrt{t} \theta_x(\cdot, t)\right\|_{L_x^{\infty}}+\left\|\sqrt{t} z_x(\cdot, t)\right\|_{L_x^{\infty}} \leq C_3 \delta\qquad\text{for any }  t \in(0,+\infty).
			\end{aligned}
		\end{equation*}
	\end{theorem}

    The $\mathcal{O}(t^{-1/2})$-decay rate  as $t \to \infty$ of the weak solution established in Theorem~\ref{main thm: global} is optimal, since it agrees with the decay rate for the heat kernel.

In this work, building on the  framework in \cite{LiuTP2022, WangHT2022}, we overcome additional analytical difficulties caused by the reactant mass fraction term $z$. On the one hand, for the local existence of weak solutions, we adopt several technical estimates for $z$ to ensure the convergence of the relevant  iteration scheme. See the proof of Lemmata~\ref{lemma: V U Theta Z k+1} and \ref{lemma: Z n diff}. On the other hand, to extend the local solution to the global one, one needs to further develop the Green's matrix analysis in \cite{LiuTP2022, WangHT2022}. In particular, the presence of $z$ makes Eq.~\eqref{new, balance law, Dec25}, the linearisation of Eq.~\eqref{PDE,2}, a $4 \times 4$ partially degenerate hyperbolic system. Hence, the Green's matrix $\mathbb{G}$ is  a $4 \times 4$ matrix, with a mode $\lambda_4$ arisen from the $z$-equation. The bound for $z$ requires several delicate novel estimates for $\mathbb{G}_{4j}$, $1\leq j \leq 4$, as well as $\partial_x\mathbb{G}_{4k}$ and 
$\partial_{xx}\mathbb{G}_{4k}$, $k=2,3$. See, \emph{e.g.},  Lemma~\ref{lemma: z rep} for details.

For future investigations, an important question is to determine the behaviour of the weak and/or strong solutions in the singular limit when all or some of the parameters $\mu, \nu, D$ tend to zero. We expect that the weak solutions are \emph{unstable} in the limit when all the parameters $\mu, \nu, D\to 0$, as it has been observed that in dimension one, steady planar detonation waves are unstable and may evolve into oscillating waves, known as pulsating denotation waves. See \cite[1. Introduction]{cw} by Chen--Wagner and the references therein. On the other hand, the analysis in D. Wang \cite{dehua} suggests that shock wave,
turbulence, vacuum, mass or heat concentration shall not develop in finite time for the partially dissipative system with $D=0$ while $\mu, \nu$ remain positive.

	\section{Preliminaries}\label{sec: pre}
	This section collects several analytical tools that will be used throughout the paper. 

    \subsection{Heat kernel with constant conductivity coefficient}

	\begin{definition}
	We denote by $\mathbf{K}(x,t)$ the standard heat kernel on $\R^2_+$, namely the solution to
		\begin{equation}
			\left\{\begin{array}{l}
				\partial_t \mathbf{K}(x,t)-\partial_{xx} \mathbf{K}(x,t)=0,\\
				\lim\limits_{t\rightarrow 0+}\mathbf{K}(x,t)=\delta(x),
			\end{array}\right.
		\end{equation}
		where $\delta(x)$ is the Dirac delta measure supported at $x =0$.
	\end{definition}
The following property of $\mathbf{K}(x,t)$ is relevant to our later developments.
    
	\begin{lemma}[Lemma 2.1 \cite{ChenK2024}]\label{o kernel}
		For any $m, j \in \{0, 1, 2\}$, $t\geq 0$, and $p\in[1,+\infty]$, one has that
		\begin{equation*}
        \begin{cases}
			\left|\partial_t^j \partial_x^m \mathbf{K}(t,x)\right|\leq  \frac{O(1)}{\left(t^{\frac{1}{2}}+|x|\right)^{1+2 j+m}},\\
            \left(\int_{\mathbb{R}}\left|\partial_x^m \mathbf{K}(t,x)\right|^p \,{\rm d} x\right)^{\frac{1}{p}}\leq O(1) t^{\frac{1}{2}\left(\frac{1}{p}-1-m\right)}.
            \end{cases}
		\end{equation*}
	\end{lemma}

	
    
\subsection{BV functions on $\R$}
We recall some properties of BV functions in dimension one; \emph{cf}. \cite{BV}. 

For any local Borel measure $\lambda$ on an open set $\Omega\subset \mathbb{R}$, the Radon--Nikod\'{y}m theorem guarantees a unique decomposition  $\lambda=\lambda^a +\lambda^s$, where $\lambda^a$ and $\lambda^s$ are the absolutely continuous and singular parts of $\lambda$ with respect to the Lebesgue measure on $\R$, respectively. The singular part $\lambda^s$ can be further decomposed into a purely atomic one $\lambda^j$ and a diffuse singular one $\lambda^c$.

    Thus, for  $f\in {\rm BV}(\mathbb{R})$ one has the decomposition $$Df=D^a f+D^c f+D^j f$$ for the distributional derivative $Df$, where $D^a f$, $D^c f$, and $D^j f$ are the absolutely continuous part, singularly continuous part, and purely atomic part of $Df$, respectively. Here $D^jf$ contributes to the \emph{jump part of $f$}, denoted as 
	\begin{equation}
		f_j(x):= \sum_{\omega\in\mathcal{D},\omega<x} \left.f(y)\right|_{y=\omega^-}^{y=\omega^+} h(x-\omega),
	\end{equation}
	where $\mathcal{D}$ is the set of discontinuous points of $f(x)$, and $h(\cdot)$ is the Heaviside function, namely $h(t)=1$ for $t>0$ and $h(t)=0$ for $t\leq 0$. The total variation of $f\in BV(\mathbb{R})$ is thus
	\begin{equation}
		\begin{aligned}
			\left\|f\right\|_{BV}=&|D^a f|(\mathbb{R})+|D^c f|(\mathbb{R})+|D^j f|(\mathbb{R})\\
			=&\int_{\mathbb{R}}|\partial_x f_a|\,\mathrm{d}x+\int_{\mathbb{R}}|\mathrm{d} f_c|+\sum_{\omega\in \mathcal{D}}\left|\left.f_j(x)\right|_{x=\omega^-}^{x=\omega^+}\right|.
		\end{aligned}
	\end{equation}

For our purpose of establishing ${\rm BV}$-estimates for $f$, it has been shown in \cite[Lemma 3.4]{WangHT2022} that the estimates for the absolutely continuous and singular parts are similar. Hence we shall focus only on the absolutely continuous part in what follows. Write $f_{\tilde{a}} \equiv f_a + f_c$.

Note that for $f\in BV\cap L_x^1(\mathbb{R})$, one has that $f(x) \rightarrow 0$ as $|x|\rightarrow \infty$.  Also,
	\begin{equation}
		\|f\|_{L_x^{\infty}(\mathbb{R})}=\sup_{x\in\mathbb{R}} \left|f(x)-\lim\limits_{|x|\rightarrow \mathbb{R}}f(x)\right|\leq \|f\|_{BV(\mathbb{R})}.
	\end{equation}

\subsection{Heat kernel with BV-conductivity coefficient}

Now, let us define the fundamental solution  $H(x,t;y,t_0;f)$ for the heat equation with time-dependent BV conductivity coefficients. Here $x,t$ are variables and $y, t_0, f$ are parameters for the function $H$.
	\begin{definition}[Fundamental solution]\label{def funda}
		Assume a function $f(x,t)$ satisfies the following conditions: there exist positive constants $\bar{f}$ and $\delta_{*}\ll 1$ such that
		\begin{equation}\label{f}
			\left\{\begin{array}{l}
				\|f(\cdot)-\bar{f}\|_{L_x^1} \leq \delta_*, \quad\|f(\cdot, t)\|_{B V} \leq \delta_*, \quad\left\|f_t(\cdot, t)\right\|_{L_x^{\infty}} \leq \delta_* \max \left(\frac{1}{\sqrt{t}}, 1\right),\\
				\mathcal{D} \equiv\{\omega \mid f(\omega, t) \text { is not continuous at } \omega\} \text { is invariant with respect to } t .
			\end{array}\right.
		\end{equation}
		The fundamental solution $H(x,t;y,t_0;f)$ is the weak solution to the initial value problem
		\begin{equation}\label{eq: H}
			\left\{\begin{array}{l}
			\partial_t H(x,t;y,t_0;f)-\partial_x \Big(f(x,t)\partial_x H(x,t;y,t_0;f)\Big) =0 \quad\text{for } t_0<t, \\
				H(x,t_0;y,t_0;f)=\delta(x-y).
			\end{array}\right.
		\end{equation}
		That is, for any smooth test function $\psi \in C_c^\infty(\mathbb{R} \times [t_0, \infty))$, 
		\begin{equation}\label{eq: weak}
			\int_{t_0}^{\infty} \int_{\mathbb{R}} H(x,t;y,t_0;f) \Big\{ -\partial_t \psi(x,t) - \partial_x \left( f(x,t) \partial_x \psi(x,t) \right) \Big\} \, \mathrm{d}x  \mathrm{d}t = \psi(y, t_0).
		\end{equation}
	\end{definition}

\begin{lemma}[Lemma~2.6 in \cite{LiuTP2022}]\label{int H}
The fundamental solution $H$ of the system~\eqref{eq: H} satisfies
		\begin{equation}
			\left\{\begin{array}{l}
				\int_{\mathbb{R}} H(x,t;y,\tau;f)\,\mathrm{d}x=\int_{\mathbb{R}} H(x,t;y,\tau;f)\,\mathrm{d}y=1, \\
				\int_{\mathbb{R}} H_x(x,t;y,\tau;f)\,\mathrm{d}x=\int_{\mathbb{R}} H_x(x,t;y,\tau;f)\,\mathrm{d}y=0, \\
				\int_{\mathbb{R}} H_y(x,t;y,\tau;f)\,\mathrm{d}x=\int_{\mathbb{R}} H_y(x,t;y,\tau;f)\,\mathrm{d}y=0, \\
				\int_{\mathbb{R}} H_t(x,t;y,\tau;f)\,\mathrm{d}x=\int_{\mathbb{R}} H_t(x,t;y,\tau;f)\,\mathrm{d}y=0, \\
				\int_{\mathbb{R}} H_\tau(x,t;y,\tau;f)\,\mathrm{d}x=\int_{\mathbb{R}} H_\tau(x,t;y,\tau;f)\,\mathrm{d}y=0 .
			\end{array}\right.
		\end{equation}
	\end{lemma}
The following lemmas provide pointwise, derivative, and integral bounds for the fundamental solution $H(x,t;y,t_0;f)$.
	\begin{lemma}[Liu--Yu \cite{LiuTP2022}]\label{lemma: Liu}
		Suppose $f$ satisfies the conditions in \eqref{f}. Then there exist positive constants $C_{*}$ and $t_{\sharp}\ll 1$ such that the weak solution to Eq.~\eqref{eq: H} exists
		and satisfies 
		\begin{align*}
			& \left|H(x,t;y,t_0;f)\right| \leq C_* \frac{e^{-\frac{(x-y)^2}{C_*(t-t_0)}}}{\sqrt{t-t_0}},\\
			& \left|H_x(x,t;y,t_0; f)\right|+\left|H_y(x,t;y,t_0; f)\right| \leq C_* \frac{e^{-\frac{(x-y)^2}{C_*(t-t_0)}}}{t-t_0},\\
			& \left|\int_{t_0}^t H_x(x,\tau;y,t_0; f) \,\mathrm{d}\tau\right|,\left|\int_{t_0}^t H_x(x,t;y,s;f)\,\mathrm{d} s\right| \leq C_* e^{-\frac{(x-y)^2}{C_*(t-t_0)}}\qquad\text{for $t\in (t_0,t_0+t_{\sharp})$}.
		\end{align*}
	\end{lemma}
	\begin{lemma}[Lemma~2.2 in Wang--Yu--Zhang \cite{WangHT2022}]\label{lemma: 2 derivative}
		Under the same assumptions in Lemma \ref{lemma: Liu}, there exists a constant $C_*>0$ such that for any $t\in (t_0,t_0+t_{\sharp})$, the weak solution $H$ of Eq.~\eqref{eq: H} satisfies the following:
		\begin{align}
			& \label{es: H_t infty}\left|H_{x y}(x,t;y,t_0; f)\right|+\left|H_t(x,t;y,t_0;f)\right| \leq C_* \frac{e^{-\frac{(x-y)^2}{C_*(t-t_0)}}}{(t-t_0)^{\frac{3}{2}}},\\
			&\left|H_{t y}(x,t;y,t_0 ;f)\right| \leq C_* \frac{e^{-\frac{(x-y)^2}{C_*(t-t_0)}}}{(t-t_0)^2},\nonumber\\
			& \left|\int_{t_0}^t H_{xy}(x,\tau;y,t_0;f)\mathrm{d}\tau-\frac{\delta(x-y)}{f(x, t_0)}\right. \nonumber\\
			& \left.\quad-\int_{t_0}^t \frac{f(x, t_0)-f(x,\tau)}{f(x, t_0)} H_{xy}(x,\tau;y,t_0; f)\,\mathrm{d}\tau \right|\leq C_* \frac{e^{-\frac{(x-y)^2}{C_*(t-t_0)}}}{\sqrt{t-t_0}},\nonumber\\
			& \left|\int_{t_0}^t H_{x y}(x,t;y,s;f)\,\mathrm{d}s+\frac{\delta(x-y)}{f(y, t)}\right. \nonumber\\
			& \left.\quad-\int_{t_0}^t \frac{f(y, t)-f(y, s)}{f(y, t)} H_{x y}(x,t;y,s;f)\,\mathrm{d}s \right| \leq C_* \frac{e^{-\frac{(x-y)^2}{C_*\left(t-t_0\right)}}}{\sqrt{t-t_0}},\nonumber\\
			& \int_{t_0}^t H_{xx}\left(x,\tau;y,t_0; f\right)\,\mathrm{d}\tau=-\frac{\delta(x-y)}{f\left(x, t_0\right)} \nonumber\\
			& \quad-\frac{1}{f\left(x, t_0\right)} \partial_x\left[\int_{t_0}^t\left(f(x, \tau)-f(x, t_0)\right) H_x(x,\tau;y,t_0;f)\,\mathrm{d} \tau\right]\nonumber \\
			& \quad+O(1)\left(\left|\partial_x f(x, t_0)\right| e^{-\frac{(x-y)^2}{C_*(t-t_0)}}+\frac{e^{-\frac{(x-y)^2}{C_*(t-t_0)}}}{\sqrt{t-t_0}}\right), \quad \text { for } x \notin \mathcal{D},\nonumber\\
			& \int_{t_0}^t H_{xxy}(x,\tau;y,t_0 ; f)\,\mathrm{d} \tau=\frac{1}{f(x, t_0)} \nonumber\\
			& \quad \times\left[\delta^{\prime}(x-y)-\int_{t_0}^t \partial_x\left[\left(f(x, \tau)-f(x, t_0)\right) H_{x y}(x,\tau;y,t_0; f)\right]\,\mathrm{d}\tau\right]\nonumber \\
			& \quad-\frac{\partial_x f(x, t_0)}{f^2(x, t_0)}\left[\delta(x-y)-\int_{t_0}^t\left(f(x, \tau)-f(x, t_0)\right) H_{x y}(x,\tau;y,t_0; f)\,\mathrm{d}\tau\right]\nonumber \\
			& \quad+O(1)\left(\left|\partial_x f\left(x, t_0\right)\right| \frac{e^{-\frac{(x-y)^2}{C_*(t-t_0)}}}{\sqrt{t-t_0}}+\frac{e^{-\frac{(x-y)^2}{C_*(t-t_0)}}}{t-t_0}\right), \quad \text { for } x \notin \mathcal{D}, \nonumber\\
			& \label{es: int H_t}\int_{t_0}^t H_t(x,t;y,s;f)\,\mathrm{d} s=H(x,t-t_0;y;f(\cdot,t))-\delta(x-y)+O(1) \delta_* e^{-\frac{(x-y)^2}{C_*(t-t_0)}}.
		\end{align}
	\end{lemma}
	Next, we set the following norms
	\begin{align}
		\label{norm sup t}
		\left\|\left|f\right|\right\|_{\infty} & \equiv \sup _{t \in\left(0, t_{\sharp}\right)}\left\|f(\cdot, t)\right\|_{L_x^{\infty}}, \nonumber\\
		\left\|\left|f\right|\right\|_1 & \equiv \sup _{t \in\left(0, t_{\sharp}\right)}\left\|f(\cdot, t)\right\|_{L_x^1}, \\
		\left\|\left|f\right|\right\|_{BV} & \equiv \sup _{t \in\left(0, t_{\sharp}\right)}\left\|f(\cdot, t)\right\|_{B V},\nonumber
	\end{align}
	and the like. We have the following comparison estimates for heat kernels with different conductivity coefficients.

	\begin{lemma}[Corollaries~4.4 and 4.5 in \cite{LiuTP2022}]\label{lemma: comparison} Suppose that $f^a$ and $f^b$ both satisfy the conditions in \eqref{f}. Then there are positive constants $t_{\sharp}\ll 1$ and $C_*$ such that for any $t\in(t_0,t_0+t_{\sharp})$, 
		\begin{align*}
			& \left|H(x,t;y,t_0;f^b)-H(x,t;y,t_0; f^a)\right| \leq C_* \frac{e^{-\frac{(x-y)^2}{C_*(t-t_0)}}}{\sqrt{t-t_0}}\left\|\left|f^a-f^b\right|\right\|_{\infty},\\
			& \left|H_x(x,t;y,t_0;f^a)-H_x(x,t;y,t_0; f^b)\right|\\
			& +\left|H_y(x,t;y,t_0;f^a)-H_y(x,t;y,t_0; f^b)\right| \\
			& \quad \leq C_* \frac{e^{-\frac{(x-y)^2}{C_*\left(t-t_0\right)}}}{t-t_0}\left[|\log(t-t_0)\left\|\left|f^a-f^b\right|\right\| _{\infty}+\left\|\left|f^a-f^b\right|\right\|_{BV}\right. \\
			& \left.\qquad+\sqrt{t-t_0}\left(\left\|\left|f^a-f^b\right|\right\|_1+|\log t|\left\|\left| \frac{\sqrt{\tau}}{|\log \tau|} \partial_\tau\left[f^a-f^b\right]\right|\right\|_{\infty}\right)\right],\\
			& \left|\int_{t_0}^t\left[H_x(x,\tau;y,t_0; f^a)-H_x(x,\tau;y,t_0;f^b)\right]\,\mathrm{d}\tau\right|\\
			& \quad \leq C_* e^{-\frac{(x-y)^2}{C_*(t-t_0)}}\left[\left\|\left|f^a-f^b\right|\right\|_{\infty}+\left\|\left|f^a-f^b\right|\right\|_{BV}+\left\|\left|f^a-f^b\right|\right\|_{1}\right. \\
			& \left.\qquad+\left\|\left|\frac{\sqrt{\tau}}{|\log \tau|} \partial_\tau\left(f^a-f^b\right)\right|\right\|_{\infty}\right].
		\end{align*}
	\end{lemma}

	\begin{lemma}[Lemma~2.5 in \cite{WangHT2022}]\label{lemma: comparison2}Under the assumptions in Lemma \ref{lemma: comparison}, there exist positive constants $t_{\sharp}\ll 1$ and $C_*$ such that for any $t\in (t_0,t_0+t_{\sharp})$, it holds that
		\begin{align}
			& \left|H_{x y}(x,t;y,t_0;f^a)-H_{x y}(x,t;y,t_0; f^b)\right| \nonumber\\
			& +\left|H_t(x,t;y,t_0;f^a)-H_t(x,t;y,t_0; f^b)\right|\nonumber \\
			& \quad \leq C_* \frac{e^{-\frac{(x-y)^2}{C_*(t-t_0)}}}{(t-t_0)^{3/2}}\left[|\log (t-t_0)|\left\|\left|f^a-f^b\right| \right\|_{\infty}+\left\|\left|f^a-f^b\right|\right\|_{B V}\right.\nonumber\\
			& \left.\qquad+\sqrt{t-t_0}\left(\left\|\left|f^a-f^b\right|\right\|_1+|\log t|\left\|\left| \frac{\sqrt{\tau}}{|\log \tau|} \partial_\tau\left[f^a-f^b\right]\right|\right\|_{\infty}\right)\right],\nonumber \\
			& \left|\int_{t_0}^t\left[H_y(x,t;y,s; f^a)-H_y(x,t;y,s; f^b)\right]\,\mathrm{d}s\right|\nonumber\\
			& \quad \leq C_* e^{-\frac{(x-y)^2}{C_*(t-t_0)}}\left[\left\|\left|f^a-f^b\right|\right\|_{\infty}+\left\|\left|f^a-f^b\right|\right\|_{B V}+\left\|\left|f^a-f^b\right|\right\|_1\right. \nonumber\\
			& \left.\qquad+\left\|\left|\frac{\sqrt{\tau}}{|\log \tau|} \partial_\tau(f^a-f^b)\right|\right\|_{\infty}\right],\nonumber \\
			& \left|\int_{t_0}^t\left[H_x(x,t;y,s;f^a)-H_x(x,t;y,s;f^b)\right] \,\mathrm{d}s\right|,\nonumber\\
			& \left|\int_{t_0}^t\left[H_y(x,\tau;y,t_0; f^a)-H_y(x, \tau;y,t_0;f^b)\right]\,\mathrm{d}\tau\right| \nonumber\\
			& \quad\leq C_* e^{-\frac{(x-y)^2}{C_*\left(t-t_0\right)}}\left[\left\|\left|f^a-f^b\right|\right\|_{\infty}+\left\|\left|f^a-f^b\right|\right\|_{B V}+\left|\left\|f^a-f^b\right|\right\|_1\right. \nonumber\\
			& \left.\qquad+\left|\left\|\frac{\sqrt{\tau}}{|\log \tau|} \partial_\tau(f^a-f^b)\right|\right\|_{\infty}\right], \nonumber\\
			& \int_{t_0}^t\left[H_{x y}(x,\tau;y,t_0;f^a)-H_{x y}(x, \tau;y,t_0;f^b)\right]\,\mathrm{d}\tau\nonumber \\
			=&\left[\frac{1}{f^a(x, t_0)}-\frac{1}{f^b(x, t_0)}\right] \delta(x-y)-\int_{t_0}^t\left[\frac{f^a(x, \tau)-f^a(x, t_0)}{f^a(x, t_0)} H_{x y}(x,\tau;y,t_0;f^a)\right. \nonumber\\
			& \left.-\frac{f^b(x, \tau)-f^b(x, t_0)}{f^b(x, t_0)} H_{x y}(x,\tau;y,t_0;f^b)\right]\,\mathrm{d} \tau+O(1) \frac{e^{-\frac{(x-y)^2}{C_*(t-t_0)}}}{\sqrt{t-t_0}}\left[|\log(t-t_0)\left\|\left|f^a-f^b \right|\right\|_{\infty}\right. \nonumber\\
			& \left.+\left\|\left|f^a-f^b\right|\right\|_{B V}+\left\|\left|f^a-f^b\right|\right\|_1+\left\|\left|\frac{\sqrt{\tau}}{|\log \tau|} \partial_\tau(f^a-f^b)\right|\right\|_{\infty}\right],\nonumber \\
			& \int_{t_0}^t\left[H_{x y}(x,t;y,s;f^a)-H_{x y}(x,t;y,s; f^b)\right]\,\mathrm{d}s \nonumber\\
			=&\left[\frac{1}{f^a(y, t)}-\frac{1}{f^b(y,t)}\right] \delta(x-y)+\int_{t_0}^t\left[\frac{f^a(y,t)-f^a(y,s)}{f^a(y, t)} H_{x y}(x,t;y,s;f^a)\right.\nonumber \\
			& \left.-\frac{f^b(y,t)-f^b(y,s)}{f^b(y,t)} H_{x y}(x,t;y,s;f^b)\right]\,\mathrm{d}s+O(1) \frac{e^{-\frac{(x-y)^2}{C_*(t-t_0)}}}{\sqrt{t-t_0}}\left[|\log(t-t_0)|\left\|\left|f^a-f^b\right|\right\|_{\infty}\right.\nonumber\\
			&\left.+\left\|\left|f^a-f^b\right|\right\|_{B V}+\left\|\left|f^a-f^b\right|\right\|_1+\left\|\left|\frac{\sqrt{\tau}}{|\log \tau|} \partial_\tau\left(f^a-f^b\right)\right|\right\|_{\infty}\right] \nonumber.
		\end{align}
	\end{lemma}

	\begin{remark}
		\label{re: conti}
		The following useful observations are stated in \cite{WangHT2022, LiuTP2022}.
		\begin{enumerate}
			\item For the backward heat equation
			\begin{equation}
				\label{eq: H backward}
				\left\{\begin{array}{l}
					\left(\partial_{\tau}+\partial_y f(y,\tau)\partial_y\right) H(x,t;y,\tau;f)=0, \quad \tau<t, \\
					H(x,t;y,t;f)=\delta(x-y),
				\end{array}\right.
			\end{equation}
			the product $f(y,\tau)H_y(x,t;y,\tau;f)$ is continuous with respect to $y$.
			\item The weak solution to the heat equation~\eqref{eq: H} can be defined similarly as per  Definition~\ref{def funda}. Consider the equation \eqref{eq: H} with a source term
			conservative form,
			\begin{equation*}
				u_t(x,t)=(f(t,x)u_x(x,t)+g(x,t))_x,
			\end{equation*}
			where $g(x,t)\in BV$ with respect to $x$. The mild solution constructed by heat kernel and Duhamel  principle is also a weak solution to the above equation in the distributional sense. Furthermore, the flux term $(f(t,x)u_x(x,t)+g(x,t))$ is continuous with respect to $x$ if one of the following two conditions holds:
			\begin{itemize}
				\item $g(x,t)$ is Lipschitz continuous in $x$, i.e.,
				\begin{equation*}
					\|g_x(x,t)\|_{L_x^{\infty}}<+\infty;
				\end{equation*}
				\item $g(x,t)$ is H\"{o}lder continuous in $t$, i.e., there exist $C>0$ and $0<\alpha<1$ such that
				\begin{equation*}
					|g(x,t)-g(x,s)|\leq C\frac{(t-s)^{\alpha}}{s^{\alpha}},\quad 0<s<t.
				\end{equation*}
			\end{itemize}
		\end{enumerate}
	\end{remark}

\subsection{A lemma from Fourier anlaysis}
The following lemma will be used in Section~\ref{sec: Green func}.
	\begin{lemma}[Proposition~2.2 in \cite{LiuTP2022}]\label{lemma f}
		Let $\hat{f}(\eta)$ be the Fourier transform of $f(x)$. Suppose that $\hat{f}(\eta)$ is analytic in the region $|\operatorname{Im}(\eta)|< \sigma_0$ and $\hat{f}(\eta)$ has the  asymptotic property:
		\begin{equation*}
			\hat{f}(\eta)=O(1)\cdot\frac{\eta}{(1+|\eta|)^{2m+2}}+O(1)\cdot\frac{1}{(1+|\eta|)^{2m+2}},\quad m\geq 0,\quad \eta\rightarrow\infty.
		\end{equation*}
		Then $f(x)\in H^{2m}(\mathbb{R})$ and
		\begin{equation*}
			\sum\limits_{j=0}^{2m}\left|\frac{\mathrm{d}^j}{\mathrm{d}x^j}f(x)\right|=O(1)e^{-\sigma_0|x|},\quad x\in\mathbb{R}.
		\end{equation*} 
	\end{lemma}

	\section{Local solution}\label{sec: local sol}
	This section is devoted to the construction of local-in-time weak solutions to Eq.~\eqref{PDE,2}. For convenience of the reader, we reproduce Eq.~\eqref{PDE,2} as follows.\begin{equation}
		\label{Cauchy pde}
		\left\{\begin{array}{l}
			v_t-u_x=0, \\
			u_t+p_x=\left(\frac{\mu u_x}{v}\right)_x, \\
			\theta_t+\frac{p}{c_v} u_x-\frac{\mu}{c_v v}\left(u_x\right)^2=\left(\frac{\nu}{c_v v} \theta_x\right)_x+\frac{q}{c_v}K\phi(\theta)z, \\
			z_t+K\phi(\theta)z=\left(\frac{D}{v^2}z_x\right)_x,\\
			\left.(v, u, \theta,z)\right|_{t=0}=\left(v_0, u_0, \theta_0,z_0\right),
		\end{array}\right.
	\end{equation}
	where the initial datum is a perturbation around the constant state $[1,0,1,0]^\top$. We set
	\begin{equation*}
		v_0=1+v_0^*,\quad u_0=u_0^*,\quad \theta_0=1+\theta_0^*,\quad z_0=z_0^*,
	\end{equation*}
and requires as in \eqref{ini} that \begin{equation}\label{smallness condition}
		\|v_0^*\|_{BV\cap L_x^1}+\|u_0^*\|_{BV\cap L_x^1}+\|\theta_0^*\|_{BV\cap L_x^1}+\|z_0^*\|_{BV\cap L_x^1}\leq \delta\ll 1.
	\end{equation}

The existence of local solutions shall be established in a standard way. One first constructs a sequence of approximate solutions via an iterative scheme, then derive uniform \emph{a priori} estimates for this sequence using Lemmas~\ref{lemma: Liu}--\ref{lemma: comparison2}, and finally prove the convergence of the approximate sequence in suitable topologies via compactness arguments. 
	
	\subsection{Iteration scheme}\label{subsec: iteration}
Following Itaya \cite{Itaya}, we construct a sequence of approximate solutions $(V^n,U^n,\Theta^n,Z^n)$ via an iteration scheme. At each step $n+1$, we solve the system
	\begin{equation}\label{eq: n+1}
		\left\{\begin{array}{l}
			V_t^{n+1}-U_x^{n+1}=0, \\
			U_t^{n+1}-\left(\frac{\mu U_x^{n+1}}{1+V^n}\right)_x=-p\left(1+V^n, 1+\Theta^n\right)_x, \\
			\Theta_t^{n+1}-\left(\frac{\nu \Theta_x^{n+1}}{c_v\left(1+V^n\right)}\right)_x=-\frac{p\left(1+V^n, 1+\Theta^n\right)}{c_v} U_x^n+\frac{\mu}{c_v\left(1+V^n\right)}\left(U_x^n\right)^2+\frac{q}{c_v}K\phi(1+\Theta^n)Z^n, \\
			Z^{n+1}-\left(\frac{DZ_x^{n+1}}{(1+V^n)^2}\right)_x=-K\phi(1+\Theta^{n})Z^n,\\
			\left.\left(V^{n+1}, U^{n+1}, \Theta^{n+1},Z^{n+1}\right)\right|_{t=0}=\left(v_0^*, u_0^*, \theta_0^*,z_0^*\right), \\
			\left(V^0, U^0, \Theta^0, Z^0\right)=(0,0,0,0),
		\end{array}\right.
	\end{equation}
The base step ($n=0$) for this iteration is the following linear problem:
	\begin{equation}\label{eq: n=0}
		\left\{\begin{array}{l}
			V_t^{1}-U_x^{1}=0, \\
			U_t^{1}-\left(\mu U_x^{1}\right)_x=0, \\
			\Theta_t^{1}-\left(\frac{\nu \Theta_x^{1}}{b}\right)_x=0, \\
			Z^{1}-\left(D Z_x^{1}\right)_x=0,\\
			\left.\left(V^{1}, U^{1}, \Theta^{1},Z^{1}\right)\right|_{t=0}=\left(v_0^*, u_0^*, \theta_0^*,z_0^*\right), \\
		\end{array}\right.
	\end{equation}
	By using Lemma \ref{o kernel}, we have the following estimates for $\left(V^{1}, U^{1}, \Theta^{1},Z^{1}\right)$.
	\begin{lemma}\label{lemma: n=0}
		Suppose that the initial datum $\left(v_0^*, u_0^*, \theta_0^*,z_0^*\right)$ satisfies the smallness condition~\eqref{smallness condition}. There exists a positive constant $C_{\sharp}$ such that for any $0<t<t_{\sharp}$, the solution to Eq.~\eqref{eq: n=0} satisfies
		\begin{equation*}
			\left\{\begin{array}{ll}
				\max \left\{\left\|U^1(\cdot, t)\right\|_{L_x^1},\left\|U^1(\cdot, t)\right\|_{L_x^{\infty}},\left\|U_x^1(\cdot, t)\right\|_{L_x^1},\right. & \\
				\left.\sqrt{t}\left\|U_x^1(\cdot, t)\right\|_{L_x^{\infty}}, t\left\|U_t^1(\cdot, t)\right\|_{L_x^{\infty}}\right\} \leq C_{\sharp} \delta, & 0<t<t_{\sharp}, \\
				\max \left\{\left\|\Theta^1(\cdot, t)\right\|_{L_x^1},\left\|\Theta^1(\cdot, t)\right\|_{L_x^{\infty}},\left\|\Theta_x^1(\cdot, t)\right\|_{L_x^1},\right. &  \\
				\left.\sqrt{t}\left\|\Theta_x^1(\cdot, t)\right\|_{L_x^{\infty}}, t\left\|\Theta_t^1(\cdot, t)\right\|_{L_x^{\infty}}\right\} \leq C_{\sharp} \delta, & 0<t<t_{\sharp}, \\
				\max \left\{\sqrt{t}\left\|V_t^1(\cdot, t)\right\|_{L_x^{\infty}},\left\|V^1(\cdot, t)\right\|_{B V},\right. & \\
				\left.\left\|V^1(x, t)\right\|_{L_x^1},\left\|V^1(x, t)\right\|_{L_x^{\infty}}\right\} \leq C_{\sharp} \delta, &  0<t<t_{\sharp},\\
				\max \left\{\left\|Z^1(\cdot, t)\right\|_{L_x^1},\left\|Z^1(\cdot, t)\right\|_{L_x^{\infty}},\left\|Z_x^1(\cdot, t)\right\|_{L_x^1},\right. &  \\
				\left.\sqrt{t}\left\|Z_x^1(\cdot, t)\right\|_{L_x^{\infty}}, t\left\|Z_t^1(\cdot, t)\right\|_{L_x^{\infty}}\right\} \leq C_{\sharp} \delta, &   0<t<t_{\sharp}, \\
				\|V^1(\cdot,t)-V^1(\cdot,s)\|_{BV}\leq C_{\sharp}\frac{t-s}{\sqrt{t}}\delta, & 0\leq s\leq t<t_{\sharp},\\
				\left|\left.V^1(\cdot,t)\right|_{x=z^-}^{x=z^+}\right|=\left|\left.v_0^*(\cdot)\right|_{x=z^-}^{x=z^+}\right|, & z\in \mathcal{D}, 0<t<t_{\sharp},
			\end{array}\right.
		\end{equation*}
		where $t_{\sharp}>0$ is sufficiently small constructed in Lemma \ref{lemma: Liu}, and $\mathcal{D}$ is the
		discontinuity set of $v_0^*$.
	\end{lemma}

With Lemma~\ref{lemma: n=0} at hand, we now proceed with  induction. Fix $k \in \mathbb{N}$. Assume for each $n\leq k$ that  Eq.~\eqref{eq: n+1} has a solution $\left(V^n, U^n, \Theta^n, Z^n\right)$ that satisfies the following:
\begin{equation}\label{estimates n}
		\left\{\begin{array}{l}
			0<\delta, t_{\sharp} \ll 1, \quad 1 \leq n \leq k, \quad 0<t<t_{\sharp}, \\
			\max \left\{\left\|U^n(\cdot, t)\right\|_{L_x^1},\left\|U^n(\cdot, t)\right\|_{L_x^{\infty}},\left\|U_x^n(\cdot, t)\right\|_{L_x^1}, \sqrt{t}\left\|U_x^n(\cdot, t)\right\|_{L_x^{\infty}}\right\} \leq 2 C_{\sharp} \delta, \\
			\max \left\{\left\|\Theta^n(\cdot, t)\right\|_{L_x^1},\left\|\Theta^n(\cdot, t)\right\|_{L_x^{\infty}},\left\|\Theta_x^n(\cdot, t)\right\|_{L_x^1}, \sqrt{t}\left\|\Theta_x^n(\cdot, t)\right\|_{L_x^{\infty}}\right\} \leq 2 C_{\sharp} \delta, \\
			\max \left\{\left\|Z^{n}(\cdot,t)\right\|_{L_x^1},\left\|Z^n(\cdot, t)\right\|_{L_x^{\infty}},\left\|Z_x^n(\cdot, t)\right\|_{L_x^1}, \sqrt{t}\left\|Z_x^n(\cdot, t)\right\|_{L_x^{\infty}}\right\} \leq 2 C_{\sharp} \delta, \\
			\max \left\{\left\|V^n(\cdot, t)\right\|_{B V},\left\|V^n(\cdot, t)\right\|_{L_x^1},\left\|V^n(\cdot, t)\right\|_{L_x^{\infty}}, \sqrt{t}\left\|V_t^n(\cdot, t)\right\|_{L_x^{\infty}}\right\} \leq 2 C_{\sharp} \delta, \\
			\left\|V^n(\cdot, t)-V^n(\cdot, s)\right\|_{B V} \leq 2 C_{\sharp} \delta \frac{(t-s)|\log (t-s)|}{\sqrt{t}}, \quad 0 \leq s<t, \\
			\left|\left.V^n(\cdot, t)\right|_{x=\omega^{-}}^{x=\omega^{+}}\right|\leq 2\left|\left. v_0^*(\cdot)\right|_{x=\omega^{-}} ^{x=\omega^{+}} \right|, \quad \omega \in \mathcal{D}. 
		\end{array}\right.
	\end{equation}
The case $n=1$ follows from Lemma~\ref{lemma: n=0}. We will show \eqref{estimates n} for $\left(V^{k+1},U^{k+1}, \Theta^{k+1}, Z^{k+1}\right)$.

To this end, let us first introduce several notations:
	\begin{align}
		\label{N1 N2}
			\mu^k & \equiv \frac{\mu}{1+V^k}, \quad \mathcal{N}_1^k(x, t) \equiv-\partial_x p\left(1+V^k, 1+\Theta^k\right), \nonumber\\
			\nu^k & \equiv \frac{\nu}{c_v\left(1+V^k\right)}, \quad D^k  \equiv \frac{D}{(1+V^k)^2},\\
			\mathcal{N}_2^k(x, t) & \equiv-\frac{p\left(1+V^k, 1+\Theta^k\right)}{c_v} U_x^k+\frac{\mu}{c_v\left(1+V^k\right)}\left(U_x^k\right)^2+\frac{q}{c_v}K\phi(1+\Theta^k)Z^k.\nonumber
	\end{align}
Then, in view of the Duhamel principle, we have the expressions:
	\begin{align}
		\label{eq: V k+1}
		V^{k+1}(t,x)=&v_0^*(x)+\int_{0}^t U_x^{k+1}(x,s)\,\mathrm{d}s,\\
		\label{eq: U k+1}
		U^{k+1}(x,t)=&\int_{\mathbb{R}} H(x,t;y,0;\mu^k) u_0^*(y)\,\mathrm{d}y\nonumber \\
		& +\int_0^t \int_{\mathbb{R} \setminus \mathcal{D}} H_y\left(x,t;y,s;\mu^k\right) p\left(1+V^k, 1+\Theta^k\right)\,\mathrm{d}y \,\mathrm{d}s\\
		=&:\,\mathcal{I}_1^{u}+\mathcal{I}_2^{u},\nonumber\\
		\label{eq: Theta k+1}
		\Theta^{k+1}(x,t)=& \int_{\mathbb{R}} H(x,t;y,0; \nu^k) \theta_0^*(y) \,\mathrm{d}y \nonumber\\
		& +\int_0^t \int_{\mathbb{R}} H(x,t;y,s; \nu^k) \mathcal{N}_2^k(y,s)\,\mathrm{d}y \,\mathrm{d}s\\
		=&:\,\mathcal{I}_1^{\theta}+\mathcal{I}_2^{\theta},\nonumber\\
		\label{eq: Z k+1}
		Z^{k+1}(x,t)= & \int_{\mathbb{R}} H(x,t;y,0; D^k) z_0^*(y) \,\mathrm{d}y \nonumber\\
		& -\int_0^t \int_{\mathbb{R}} H(x,t;y,s; D^k) K\phi(1+\Theta^k)Z^k \,\mathrm{d}y \,\mathrm{d}s\\
		=&:\,\mathcal{I}_1^{z}+\mathcal{I}_2^{z}.\nonumber
	\end{align}

By adapting the arguments in Wang--Yu--Zhang \cite{WangHT2022}, we readily obtain the following:

	\begin{lemma}\label{lemma: V U Theta Z k+1}
		For  $\delta, t_{\sharp}>0$ sufficiently small, Eq.~\eqref{estimates n} holds for $(V^{k+1}, U^{k+1}, \Theta^{k+1}, Z^{k+1})$ whenever $0<t<t_{\sharp}$.
	\end{lemma}

	\begin{proof}[Sektch of proof for Lemma~\ref{lemma: V U Theta Z k+1}]

Our arguments are an adaptation of  \cite{WangHT2022}, so only the differences are emphasised. Key ingredients are the fundamental solution estimates in Lemmas \ref{lemma: Liu}--\ref{lemma: comparison2}. Throughout the proof, constants $O(1)$ and $C_{\sharp}$ are independent of $\delta, k$, and  $t$ whenever $t < t_\sharp$.

To begin with, in view of Eq.~\eqref{estimates n}, we find that $D^k$ satisfies Eq.~\eqref{f}. Thus, by applying Lemma~\ref{lemma: Liu} one obtains the following:
		\begin{enumerate}
			\item \textbf{Estimate of $\|Z^{k+1}\|_{L_x^1}$}: 
			Integrating the fourth equation in \eqref{eq: n+1} over $[0,t] \times \mathbb{R}$, we obtain
			\begin{equation}
				\int_{\mathbb{R}}Z^{k+1}(x,t) \,\mathrm{d}x+\int_0^t\int_{\mathbb{R}}K \phi(1+\Theta^k) Z^k (y,\tau)\,\mathrm{d}y\,\mathrm{d}\tau\leq \int_{\mathbb{R}}z_0^{*}(x) \,\mathrm{d}x.
			\end{equation}
			Since $\phi$ is Lipschitz continuous and satisfies $\phi(\theta) \geq 0$, it follows from the smallness condition on the initial data \eqref{smallness condition} that
			\begin{equation}
				\label{eq: phi z L1}
				\int_0^t\int_{\mathbb{R}}K \phi(1+\Theta^k) Z^k (y,\tau)\,\mathrm{d}y\,\mathrm{d}\tau\leq \delta,
			\end{equation}
			and
			\begin{equation}
				\label{eq: Zk L1}
					\left\| Z^{k+1}\right\|_{L_x^1} \leq \|z_0^*\|_{L_x^1} \leq C_{\sharp} \delta.
			\end{equation}

			\item \textbf{Estimate of $\|Z^{k+1}\|_{L_x^{\infty}}$}: 
			From \eqref{eq: Z k+1} and the estimate $|H(x,t;y,s;D^k)| \leq O(1) \frac{e^{-\frac{(x-y)^2}{C_*(t-s)}}}{\sqrt{t-s}}$ in Lemma \ref{lemma: Liu}, for sufficiently small $\delta$
			\begin{align}
				\label{eq: Zk L infty}
				\|Z^{k+1}\|_{L_x^{\infty}} 
				&\leq \int_{\mathbb{R}} \left| H(x,t;y,0;D^k) \right| |z_0^*(y)| \, \mathrm{d}y \nonumber\\
				&\quad + \int_0^t \int_{\mathbb{R}} \left| H(x,t;y,s;D^k) \right| K \phi(1+\Theta^k) Z^k(y,s)\, \mathrm{d}y \, \mathrm{d}s \nonumber\\
				&\leq O(1) \int_{\mathbb{R}} \frac{e^{-\frac{(x-y)^2}{C_* t}}}{\sqrt{t}} \, \mathrm{d}y \cdot \|z_0^*\|_{L_x^{\infty}} 
				+ O(1) \int_0^t \int_{\mathbb{R}} \frac{e^{-\frac{(x-y)^2}{C_*(t-s)}}}{\sqrt{t-s}} \left\|Z^k\right\|_{L_x^{\infty}} \, \mathrm{d}y \, \mathrm{d}s \nonumber\\
				&\leq O(1) \delta + O(1) t\delta \leq2 C_{\sharp} \delta.
			\end{align}
			
			\item \textbf{Estimate of $\|Z_x^{k+1}\|_{L_x^1}$}: 
			Differentiating \eqref{eq: Z k+1} with respect to $x$ gives
			\begin{align}
				\label{eq: Z_x k+1}
				Z_x^{k+1}(x,t) 
				&= \int_{\mathbb{R}} H_x(x,t;y,0;D^k) z_0^*(y) \, \mathrm{d}y \nonumber\\
				&\quad - \int_0^t \int_{\mathbb{R}} H_x(x,t;y,s;D^k) K \phi(1+\Theta^k) Z^k \, \mathrm{d}y \, \mathrm{d}s.
			\end{align}
			Since $\int_{\mathbb{R}} H_x(x,t;\omega,0;D^k) \, \mathrm{d}\omega = 0$, we define the anti-derivative
			\begin{equation}
				\label{W anti_deri}
				W(x,t;y,0;D^k) = 
				\begin{cases}
					\int_{-\infty}^y H_x(x,t;w,0;D^k) \, \mathrm{d}w, & \text{for } y < x, \\
					-\int_y^{\infty} H_x(x,t;w,0;D^k) \, \mathrm{d}w, & \text{for } y \geq x.
				\end{cases}
			\end{equation}
			Since $z_0^*$ is a BV function, integration by parts for the Stieltjes integral yields
			\begin{align*}
				\int_{\mathbb{R}} \left| \int_{\mathbb{R}} \mathrm{d}W(x,t;y,0;D^k) z_0^*(y) \right| \mathrm{d}x 
				&= \int_{\mathbb{R}} \left| \int_{\mathbb{R}} W(x,t;y,0;D^k) \, \mathrm{d}z_0^*(y) \right| \mathrm{d}x \\
				&\leq \int_{\mathbb{R}} \int_{\mathbb{R}} \left| W(x,t;y,0;D^k) \right| |\mathrm{d}z_0^*(y)| \, \mathrm{d}x \\
				&\leq O(1) \int_{\mathbb{R}} \frac{e^{-\frac{(x-y)^2}{C_* t}}}{\sqrt{t}} \, \mathrm{d}x \cdot \|z_0^*\|_{BV} \leq C_{\sharp} \delta.
			\end{align*}
			Using the estimate $|H_x(x,t;y,s;D^k)| \leq O(1) \frac{e^{-\frac{(x-y)^2}{C_*(t-s)}}}{t-s}$ from Lemma \ref{lemma: Liu} and the ansatz \eqref{estimates n}, we obtain for sufficiently small $\delta$
			\begin{align}
				\label{eq: Zx k L1}
				\|Z_x^{k+1}\|_{L_x^1} 
				&\leq \int_{\mathbb{R}} \int_{\mathbb{R}} \left| H_x(x,t;y,0;D^k) \right| |z_0^*(y)| \, \mathrm{d}y \, \mathrm{d}x \nonumber\\
				&\quad + \int_{\mathbb{R}} \int_0^t \int_{\mathbb{R}} \left| H_x(x,t;y,s;D^k) \right| K \phi(1+\Theta^k) Z^k(y,s)\, \mathrm{d}y \, \mathrm{d}s \, \mathrm{d}x \nonumber\\
				&\leq O(1) \int_{\mathbb{R}} \frac{e^{-\frac{(x-y)^2}{C_* t}}}{\sqrt{t}} \, \mathrm{d}x \cdot \|z_0^*\|_{BV} \nonumber\\
				&\quad+ O(1) \int_{\mathbb{R}} \int_0^t \int_{\mathbb{R}} \frac{e^{-\frac{(x-y)^2}{C_*(t-s)}}}{t-s} K \phi(1+\Theta^k) Z^k(y,s)\, \mathrm{d}y \,\mathrm{d}s \, \mathrm{d}x \nonumber\\
				&\leq O(1) \delta + O(1) \int_0^t \left\|\frac{e^{-\frac{x^2}{C_* (t-s)}}}{t-s}\right\|_{L_x^1}\left\|Z^k\right\|_{L_x^1} \, \mathrm{d}s\nonumber\\
				&\leq O(1) \delta + O(1)\int_0^t \frac{1}{\sqrt{t-s}}\delta\,\mathrm{d}s \nonumber\\
				&\leq O(1)(\delta+\sqrt{t}\delta)\leq 2C_{\sharp}\delta.
			\end{align}
			
			\item \textbf{Estimate of $\|Z_x^{k+1}\|_{L_x^{\infty}}$}: 
			Similarly, for sufficiently small $\delta$
			\begin{align}
				\label{eq: Zx k L infty}
				\|Z_x^{k+1}\|_{L_x^{\infty}} 
				&\leq \int_{\mathbb{R}} \left| H_x(x,t;y,0;D^k) \right| |z_0^*(y)| \, \mathrm{d}y \nonumber\\
				&\quad + \int_0^t \int_{\mathbb{R}} \left| H_x(x,t;y,s;D^k) \right| K \phi(1+\Theta^k) Z^k\, \mathrm{d}y \, \mathrm{d}s \nonumber\\
				&\leq O(1) \int_{\mathbb{R}} \frac{e^{-\frac{(x-y)^2}{C_* t}}}{t} \, \mathrm{d}y \|z_0^*\|_{L_x^{\infty}} + O(1) \int_0^t \int_{\mathbb{R}} \frac{e^{-\frac{(x-y)^2}{C_*(t-s)}}}{t-s} \left\|Z^k\right\|_{L_x^{\infty}}\, \mathrm{d}y \, \mathrm{d}s \nonumber\\
				&\leq O(1) \frac{\delta}{\sqrt{t}} + O(1) \delta\int_0^t \frac{1}{\sqrt{t-s}} \, \mathrm{d}s\nonumber \\
				&\leq O(1) \left( \frac{\delta}{\sqrt{t}} + \delta \sqrt{t} \right) \leq \frac{2 C_{\sharp} \delta}{\sqrt{t}}.
			\end{align}
		\end{enumerate}
		
		Combining the estimates in steps (1)--(4), we conclude that $Z^{k+1}$ satisfies the bounds in \eqref{estimates n} for sufficiently small $\delta$ and $t_{\sharp}$. Similar estimates hold for $V^{k+1}$, $U^{k+1}$, and $\Theta^{k+1}$ analogously. This completes the induction.   \end{proof}

We next show the Lipschitz continuity of $V^{k+1}$ in time across the jump set.
	\begin{lemma}
		\label{lemma: lip V}
		For sufficiently small $\delta>0$ and $t_{\sharp}>0$, $V^{k+1}$ satisfies the following Lipschitz continuity estimate for $0\leq s<t<t_{\sharp}$:
		\begin{align*}
			&\sum_{\omega\in\mathcal{D}}\left|\left.V_x^{k+1}(\cdot,t)\right|_{\omega^-}^{\omega^+}-\left.V_x^{k+1}(\cdot,s)\right|_{\omega^-}^{\omega^+}\right|\\
			&\qquad\leq O(1)\int_s^t \left(1+\frac{1}{\sqrt{\tau}}\right)\,\mathrm{d}\tau\sum_{\omega\in\mathcal{D}}\left|\left.v_0^*(\cdot)\right|_{\omega^-}^{\omega^+}\right|\leq O(1)\frac{t-s}{\sqrt{t}}.
		\end{align*}
	\end{lemma}
	\begin{proof}
		Our arguments are directly adapted from \cite[Lemma 3.4]{WangHT2022}.  
		Since $$\frac{\mu U_x^k}{1+V^{k-1}}-p(1+V^{k-1},1+\Theta^{k-1})$$ is continuous in $x$ by Remark \ref{re: conti}, we have that
		\begin{align*}
			\left.V_t^{k+1}(\cdot, t)\right|_{\omega^-}^{\omega^+}=&\left.U_x^{k+1}(\cdot, t)\right|_{\omega^-}^{\omega^+}\\
			=&\left.\frac{V^k}{\mu}(\cdot, t)\right|_{\omega^-}^{\omega^+}\left(\frac{\mu U_x^k}{1+V^{k-1}}-p(1+V^{k-1},1+\Theta^{k-1})\right)+\left.\frac{a(1+\Theta^k)}{\mu}(\cdot, t)\right|_{\omega^-}^{\omega^+}.
		\end{align*}
		Thus, integrating this identity over $[0,t]$ and applying Lemma~\ref{lemma: V U Theta Z k+1}, we conclude that
		\begin{align}
			\label{Vk jump}
			&\left|\left.V^{k+1}(\cdot,t)\right|_{\omega^-}^{\omega^+}\right|\nonumber\\
			\leq &\left|\left.v_0^*(\cdot)\right|_{\omega^-}^{\omega^+}\right|+\frac{1}{\mu}\int_0^t\left(\frac{\mu\left|U_x^{k+1}\right|(\cdot,s)}{1-\left\|V^{k}\right\|_{L_x^{\infty}}}+\frac{2a\left(1+\left\|\Theta^{k}\right\|_{L_x^{\infty}}\right)}{1-\left\|V^{k}\right\|_{L_x^{\infty}}}\right)\,\mathrm{d}s\times \sup\limits_{0\leq s\leq t}\left|\left.V^{k}(\cdot,s)\right|_{\omega^-}^{\omega^+}\right|\nonumber\\
			\leq &O(1)\left(1+\sqrt{t}+t\right)\left|\left.v_0^*(\cdot)\right|_{\omega^-}^{\omega^+}\right|\nonumber\\
			\leq & 2\left|\left.v_0^*(\cdot)\right|_{\omega^-}^{\omega^+}\right|.
		\end{align}
	\end{proof}

The following result is a variant of \cite[Lemma~3.2]{WangHT2022}, but the presence of $Z^k$ in the $\mathcal{N}_2$ term  (see Eq.~\eqref{N1 N2}) brings about essential differences. We thus provide a detailed proof below.
    
	\begin{lemma}\label{lemma: diff Theta k+1}
		For sufficiently small $\delta$ and $t_{\sharp}$, the following time
		difference estimates hold for $\Theta^{k+1}$, whenever $0<s\leq t<t_{\sharp}$:
		\begin{align}
			& \label{eq: diff theta k+1}\left\|\Theta^{k+1}(\cdot, t)-\Theta^{k+1}(\cdot, s)\right\|_{L_x^{\infty}} \nonumber\\
			& \quad \leq O(1)\left((\delta+\delta^2) \frac{t-s}{\sqrt{s} \sqrt{t}}+\delta\frac{t-s}{\sqrt{t}}+\delta\sqrt{t-s}+\delta^2 \frac{\sqrt{t-s}}{\sqrt{s}}\right),\\
			& \left\|\Theta^{k+1}(\cdot, t)-\Theta^{k+1}(\cdot, s)\right\|_{L_x^1} \nonumber\\
			& \quad \leq O(1)\left(\delta \frac{t-s}{\sqrt{s} \sqrt{t}}+\delta(t-s)+\delta^2 \frac{t-s}{\sqrt{t}}+ \delta\sqrt{t-s} \sqrt{s}+\delta^2\sqrt{t-s} \right)\nonumber \\
			& \left\|\Theta_x^{k+1}(\cdot, t)-\Theta_x^{k+1}(\cdot, s)\right\|_{L_x^1} \nonumber\\
			& \quad \leq O(1)\left(\frac{\sqrt{t-s}|\log (t-s)|}{\sqrt{t} \sqrt{s}} \delta+\sqrt{t-s} \delta+\frac{\sqrt{t-s}}{\sqrt{t}} \delta^2+\sqrt{t-s}|\log (t-s)| \delta^2\right).\nonumber
		\end{align}
	\end{lemma}
	\begin{proof}
In view of Eq.~\eqref{estimates n},  $\nu^k$ satisfies Eq.~\eqref{f}, so we can apply Lemmas~\ref{lemma: Liu} and \ref{lemma: 2 derivative} to derive the required estimates. Indeed, by the representation formula for $\Theta^{k+1}$ in Eq.~\eqref{eq: Theta k+1}, we have
		\begin{align*}
			&\Theta^{k+1}(y,t)-\Theta^{k+1}(y,s) \\
			= & \int_{\mathbb{R}}\left(H(y,t;\omega,0;\nu^k)-H(y,s;\omega,0;\nu^k)\right) \theta_0^*(\omega)\,\mathrm{d}\omega \\
			& +\int_s^t \int_{\mathbb{R}} H(y,t;\omega,\tau;\nu^k) \mathcal{N}_2^k(\omega,\tau)\,\mathrm{d}\omega\,\mathrm{d} \tau\\&+\int_0^s \int_{\mathbb{R}}\left(H(y,t;\omega,\tau;\nu^k)-H(y,s;\omega,\tau;\nu^k)\right) \mathcal{N}_2^k(\omega, \tau)\,\mathrm{d}\omega \,\mathrm{d} \tau \\
			=&\mathcal{I}_1+\mathcal{I}_2+\mathcal{I}_3.
		\end{align*}
		\begin{enumerate}
			\item For $\mathcal{I}_1$, we use the estimate of $H_t$ from \eqref{es: H_t infty} in Lemma \ref{lemma: 2 derivative} to obtain
			\begin{align*}
				\left|\mathcal{I}_1\right| & =\left|\int_{\mathbb{R}} \int_s^t H_\sigma(y,\sigma;\omega,0;\nu^k) \theta_0^*(\omega)\,\mathrm{d} \sigma\,\mathrm{d}\omega\right| \\
				&\leq\int_{\mathbb{R}} \int_s^t \left|H_\sigma(y,\sigma;\omega,0;\nu^k)\right|\,\mathrm{d} \sigma\,\mathrm{d}\omega \|\theta_0^*\|_{L_x^{\infty}} \\
				&\leq O(1) \delta\int_{\mathbb{R}} \int_s^t \frac{e^{\frac{-(y-\omega)^2}{C_* \sigma}}}{\sigma^{\frac{3}{2}}} \,\mathrm{d} \sigma\,\mathrm{d}\omega\\
				& \leq O(1)\delta\int_s^t \frac{1}{\sigma}\,\mathrm{d} \sigma \leq O(1)\delta \int_s^t \frac{1}{\sqrt{s}\sqrt{\sigma}}\,\mathrm{d} \sigma\\
				&\leq O(1)\frac{\delta(t-s)}{\sqrt{s} \sqrt{t}}.
			\end{align*}
			
			\item For $\mathcal{I}_2$, using the expression for $\mathcal{N}_2^k$ in \eqref{N1 N2}, the estimate for $H$ from Lemma \ref{lemma: Liu}, and the ansatz \eqref{estimates n}, we have:
			\begin{align*}
				\left|\mathcal{I}_2\right| & 
				\leq \int_s^t \int_{\mathbb{R}} \left|H(y,t;\omega,\tau;\nu^k)\right|\left(\left|\frac{a(1+\Theta^k)}{c_v(1+V^k)}U_y^k(\omega,\tau)\right| +\left|\frac{\mu\left(U_y^k(\omega,\tau)^2\right)}{c_v(1+V^k)}\right|\right.\\
				&\qquad +\left.\frac{q}{c_v}K\phi(1+\Theta^k)Z^k(\omega,\tau)\right)\,\mathrm{d}\omega\,\mathrm{d}\tau\\
				&\leq O(1)\int_s^t \int_{\mathbb{R}} \frac{e^{\frac{-(y-\omega)^2}{C_* (t-\tau))}}}{\sqrt{t-\tau}}\left(\frac{\delta}{\sqrt{\tau}}+\frac{\delta^2}{\tau}\right)\,\mathrm{d}\omega\,\mathrm{d}\tau+O(1)\int_s^t \frac{1}{\sqrt{t-\tau}}\int_{\mathbb{R}} Z^k(\omega,\tau)\,\mathrm{d}\omega\,\mathrm{d}\tau\\
				&\leq O(1) \int_s^t \left(\frac{\delta}{\sqrt{\tau}}+\frac{\delta^2}{\tau}+\frac{\delta}{\sqrt{t-\tau}}\right)\,\mathrm{d}\tau\\
				& \leq O(1)\left(\delta\frac{t-s}{\sqrt{t}}+\delta^2\frac{t-s}{\sqrt{t}\sqrt{s}}+\delta\sqrt{t-s}\right).
			\end{align*}
			
			\item Similarly, for $\mathcal{I}_3$, combining the expression of $\mathcal{N}_2^k$ in \eqref{N1 N2}, the estimate for $H_t$ from Lemma \ref{lemma: 2 derivative}, and the ansatz \eqref{estimates n}, we obtain
			\begin{align*}
				\left|\mathcal{I}_3\right| & =\left|\int_0^s\int_{\mathbb{R}} \int_s^t H_{\sigma}(y,\sigma;\omega,\tau;\nu^k) \,\mathrm{d}\sigma \mathcal{N}_2^k(\omega,\tau)\,\mathrm{d} \omega\,\mathrm{d}\tau\right| \\
				&\leq O(1)\int_0^s \int_{\mathbb{R}} \int_s^t \frac{e^{\frac{-(y-\omega)^2}{C_* (\sigma-\tau))}}}{(\sigma-\tau)^{3/2}}\,\mathrm{d}\sigma \left[\left|U_\omega^k\right| \left(\delta+\frac{\delta}{\sqrt{\tau}}\right)+Z^k(\omega,\tau)\right]\,\mathrm{d}\omega\,\mathrm{d}\tau\\
				&\leq O(1) \int_0^s\int_{\mathbb{R}}\left(\frac{1}{\sqrt{s-\tau}}-\frac{1}{\sqrt{t-\tau}}\right)\left[\left|U_{\omega}^k\right|\delta+Z^k(\omega,\tau)\right]\,\mathrm{d}\omega\,\mathrm{d}\tau\\
				&\qquad +O(1)\int_0^{\frac{s}{2}}\int_{\mathbb{R}}\left(\frac{1}{\sqrt{s-\tau}}-\frac{1}{\sqrt{t-\tau}}\right)\left|U_{\omega}^k\right|\frac{\delta}{\sqrt{\tau}}\,\mathrm{d}\omega\,\mathrm{d}\tau+\int_{\frac{s}{2}}^{s}\int_s^t \frac{1}{\sigma-\tau}\frac{\delta^2}{\tau}\,\mathrm{d}\sigma\,\mathrm{d}\tau\\
				&\leq O(1)(\delta+\delta^2)\int_0^s\left(\frac{1}{\sqrt{s-\tau}}-\frac{1}{\sqrt{t-\tau}}\right)\,\mathrm{d}\tau+O(1)\delta^2\int_0^{\frac{s}{2}}\left(\frac{1}{\sqrt{s-\tau}}-\frac{1}{\sqrt{t-\tau}}\right)\frac{1}{\sqrt{\tau}}\,\mathrm{d}\tau\\
				&\qquad +O(1)\delta^2\int_{\frac{s}{2}}^{s}\frac{1}{\sqrt{s-\tau}}\int_s^t\frac{1}{\sqrt{\sigma-\tau}}\frac{1}{\tau}\,\mathrm{d}\sigma\,\mathrm{d}\tau\\
				&\leq O(1)(\delta+\delta^2)(\frac{t-s}{\sqrt{t}}+\sqrt{t-s})+O(1)\frac{\delta^2\sqrt{t-s}}{\sqrt{s}}+O(1)\frac{\delta^2}{\sqrt{s}}\int_{\frac{s}{2}}^{s}\frac{1}{\sqrt{s-\tau}}\frac{t-s}{\sqrt{t-\tau}}\frac{1}{\sqrt{\tau}}\,\mathrm{d}\tau\\
				&\leq O(1)(\delta+\delta^2)(\frac{t-s}{\sqrt{t}}+\sqrt{t-s})+O(1)\frac{\delta^2\sqrt{t-s}}{\sqrt{s}}.
			\end{align*}
		\end{enumerate}
		
	We combining the estimates for $\mathcal{I}_1$, $\mathcal{I}_2$, and $\mathcal{I}_3$ to conclude Eq.~\eqref{eq: diff theta k+1}. 
	\end{proof}

	\subsection{Convergence of the iteration scheme}
	Our goal is to establish the  convergence of the iteration scheme in Eq.~\eqref{eq: n+1} for $t \in (0,t_{\sharp})$. To this end, we need to estimate the differences between adjacent terms in the sequence of approximate solutions $\left(V^{n}, U^{n}, \Theta^{n}, Z^{n}\right)$, with respect to the norms $\left(\left\|\left|\cdot\right|\right\|_{\infty},\left\|\left|\cdot\right|\right\|_1,\left\|\left|\cdot\right|\right\|_{BV}\right)$ defined in Eq.~\eqref{norm sup t}.

    Observe that the PDE for  $(V^{n+1}-V^n, U^{n+1}-U^n, \Theta^{n+1}-\Theta^n, Z^{n+1}-Z^n)$ is as follows:
	\begin{equation}
		\label{diff equs}
		\left\{\begin{array}{l}
			\partial_t(V^{n+1}-V^n)-\partial_x(U^{n+1}-U^n)=0, \\
			\partial_t(U^{n+1}-U^n)-\partial_x\left(\frac{\mu(U^{n+1}-U^n)_x}{1+V^n}\right)=-\partial_x\left(\frac{\mu U_x^n(V^n-V^{n-1})}{(1+V^n)(1+V^{n-1})}\right)+\mathcal{N}_1^n-\mathcal{N}_1^{n-1}, \\
			\partial_t(\Theta^{n+1}-\Theta^n)-\partial_x\left(\frac{\nu(\Theta^{n+1}-\Theta^n)_x}{c_v(1+V^n)}\right)=-\partial_x\left(\frac{\nu\Theta_x^n(V^n-V^{n-1})}{c_v(1+V^n)(1+V^{n-1})}\right)+\mathcal{N}_2^n-\mathcal{N}_2^{n-1}, \\
			\partial_t\left(Z^{n+1}-Z^n\right)-\partial_x\left(\frac{D(Z^{n+1}-Z^n)_x}{(1+V^n)^2}\right)=-\partial_x\left(\frac{DZ_x^n(V^n-V^{n-1})(V^n+V^{n-1}+2)}{(1+V^n)^2(1+V^{n-1})^2}\right)+\mathcal{N}_3^n-\mathcal{N}_3^{n-1},\\
			V^{n+1}(x, 0)-V^n(x, 0)=U^{n+1}(x, 0)-U^n(x, 0)=\Theta^{n+1}(x, 0)-\Theta^n(x, 0)=0,\\
			Z^{n+1}(x, 0)-Z^n(x, 0)=0.
		\end{array}\right.
	\end{equation}
	Here,
	\begin{align}
		\label{N1n n-1}\mathcal{N}_1^n-\mathcal{N}_1^{n-1}= & -\partial_x\left(\frac{a(\Theta^n-\Theta^{n-1})}{1+V^n}-\frac{a(V^{n-1}-V^n)(1+\Theta^{n-1})}{(1+V^{n-1})(1+V^n)}\right), \\
		\label{N2n n-1}\mathcal{N}_2^n-\mathcal{N}_2^{n-1}= & -\frac{U_x^n}{c_v}\left(p^n-\frac{\mu U_x^n}{1+V^n}\right)+\frac{U_x^{n-1}}{c_v}\left(p^{n-1}-\frac{\mu U_x^{n-1}}{1+V^{n-1}}\right)\nonumber\\
		& +\frac{q}{c_v}K\phi(1+\Theta^n)Z^n-\frac{q}{c_v}K\phi(1+\Theta^{n-1})Z^{n-1}\nonumber\\
		= & O(1)\left[\left(\left|V^n-V^{n-1}\right|+\left|\Theta^n-\Theta^{n-1}\right|\right)\left|U_x^n\right|+\left|V^n-V^{n-1}\right|\left(U_x^n\right)^2\right.\nonumber \\
		& \left.+\left(1+\left|U_x^n+U_x^{n-1}\right|\right)\left|U_x^n-U_x^{n-1}\right|+\left|Z^n-Z^{n-1}\right|\right.\nonumber\\
		&+\left.\left|\Theta^n-\Theta^{n-1}\right|Z^{n-1}\right],\\
		\label{N3n n-1}\mathcal{N}_3^n-\mathcal{N}_3^{n-1}=& -K\phi(1+\Theta^n)Z^n+K\phi(1+\Theta^{n-1})Z^{n-1}\nonumber\\
		=&O(1)\left[\left|Z^n-Z^{n-1}\right|+ \left|\Theta^n-\Theta^{n-1}\right|Z^{n-1}\right].
	\end{align}
    Note that we have used the Lipschitz continuity of $\phi$, which accounts for the term $\left|\Theta^n-\Theta^{n-1}\right|$ on the right-hand side.  

    An application of the Duhamel principle to Eq.~\eqref{diff equs} yields that
	\begin{align}
		\label{eq: V n+1-n}
		\left(V^{n+1}-V^n\right)(x,t)=&\int_0^t \partial_x(U^{n+1}-U^n) \mathrm{d}\tau,\\
		\label{eq: U n+1-n}
		\left(U^{n+1}-U^n\right)(x,t)=&\int_0^t \int_{\mathbb{R}} H_y(x,t;y,\tau;\mu^n)\frac{\mu U_y^n(V^n-V^{n-1})}{(1+V^n)(1+V^{n-1})}(y,\tau)\,\mathrm{d}y\,\mathrm{d}\tau\nonumber\\
		&+\int_0^t\int_{\mathbb{R}}H(x,t;y,\tau;\mu^n)\left(\mathcal{N}_1^n-\mathcal{N}_1^{n-1}\right)(y,\tau)\,\mathrm{d}y\,\mathrm{d}\tau,\\
		\label{eq: Theta n+1-n}
		\left(\Theta^{n+1}-\Theta^n\right)(x,t)=&\int_0^t \int_{\mathbb{R}} H_y(x,t;y,\tau;\nu^n)\frac{\nu \Theta_y^n(V^n-V^{n-1})}{c_v(1+V^n)(1+V^{n-1})}(y,\tau)\,\mathrm{d}y\,\mathrm{d}\tau\nonumber\\
		&+\int_0^t\int_{\mathbb{R}}H(x,t;y,\tau;\nu^n)\left(\mathcal{N}_2^n-\mathcal{N}_2^{n-1}\right)(y,\tau)\,\mathrm{d}y\,\mathrm{d}\tau,\\
		\label{eq: Z n+1-n}
		\left(Z^{n+1}-Z^n\right)(x,t)=&\int_0^t \int_{\mathbb{R}} H_y(x,t;y,\tau;D^n)\frac{D Z_y^n(V^n-V^{n-1})(V^n+V^{n-1}+2)}{(1+V^n)^2(1+V^{n-1})^2}(y,\tau)\,\mathrm{d}y\,\mathrm{d}\tau \nonumber\\
		&+\int_0^t\int_{\mathbb{R}}H(x,t;y,\tau;D^n)\left(\mathcal{N}_3^n-\mathcal{N}_3^{n-1}\right)(y,\tau)\,\mathrm{d}y\,\mathrm{d}\tau,
	\end{align}
	where $\mu^n$, $\nu^n$ and $D^n$ are defined as in \eqref{N1 N2}.

We shall establish the following log-Lipschitz in-time bounds for $\Theta^{n+1}-\Theta^n$ and $\Theta^{n+1}_x-\Theta^n_x$.

	\begin{lemma}
		\label{lemma: Theta n diff}
		For sufficiently small $\delta$ and $t_{\sharp}$, there exists a positive constant $C_2$ such that for any $0<t<t_{\sharp}$, we have the estimates:
		\begin{align*}
			& \left\|\frac{\Theta^{n+1}(\cdot, t)-\Theta^n(\cdot, t)}{\left|\log t\right|}\right\|_{L_x^{\infty}} \\
			& \quad \leq C_2\left(\sqrt{t_{\sharp}}+\delta\right)\left(\left\|\left|V^n-V^{n-1}\right|\right\|_{\infty}+\left\|\left|\frac{\Theta^n-\Theta^{n-1}}{\left|\log \tau\right|}\right|\right\|_{\infty}+\left\|\left|\frac{U_x^n-U_x^{n-1}}{\left|\log \tau\right|}\right|\right\|_1\right.\\
			&\qquad \left.+\frac{1}{\left|\log t\right|}\left\|\left|Z^n-Z^{n-1}\right|\right\|_{\infty}\right), \\
			& \left\|\Theta^{n+1}(\cdot, t)-\Theta^n(\cdot, t)\right\|_{L_x^{1}} \\
			& \quad \leq C_2\left(\sqrt{t_{\sharp}}+\delta\right)\left(\left\|\left|V^n-V^{n-1}\right|\right\|_1+\left\|\left|V^n-V^{n-1}\right|\right\|_{\infty}+\left\|\left|\Theta^n-\Theta^{n-1}\right|\right\|_1\right.\\
			&\qquad \left.+\sqrt{t}\left|\log t\right|\left\|\left|\frac{U_x^n-U_x^{n-1}}{\left|\log \tau\right|}\right|\right\|_1+\sqrt{t}\left\|\left|Z^n-Z^{n-1}\right|\right\|_{1}\right), \\
			& \frac{\sqrt{t}}{\left|\log t\right|}\left\|\Theta_x^{n+1}(\cdot, t)-\Theta_x^n(\cdot, t)\right\|_{L_x^{\infty}} \\
			& \quad \leq C_2\left(\sqrt{t_{\sharp}}+\delta\right)\left(\left\|\left|V^n-V^{n-1}\right|\right\|_{\infty}+\frac{1}{\left|\log t\right|}\left\|\left| V^n-V^{n-1}\right|\right\|_{B V}+\frac{1}{\left|\log t\right|}\left\|\left|V^n-V^{n-1}\right|\right\|_1\right. \\
			& \left.\qquad+\left\|\left|\frac{\sqrt{\tau}}{\left|\log \tau\right|}\left(U_x^n-U_x^{n-1}\right)\right|\right\|_{\infty}+\left\|\left|\frac{U_x^n-U_x^{n-1}}{\left|\log \tau\right|}\right|\right\|_1+\left\|\left|\frac{\Theta^n-\Theta^{n-1}}{\left|\log \tau\right|}\right|\right\|_{\infty}+\frac{1}{\left|\log t\right|}\left\|\left|Z^n-Z^{n-1}\right|\right\|_{\infty}\right), \\
			& \frac{\left\|\Theta_x^{n+1}(\cdot, t)-\Theta_x^n(\cdot, t)\right\|_{L_x^{1}}}{\left|\log t\right|} \\
			& \quad \leq C_2\left(\sqrt{t_{\sharp}}+\delta\right)\left(\left\|\left| V^n-V^{n-1}\right|\right\|_{\infty}+\frac{1}{\left|\log t\right|}\left\|\left|V^n-V^{n-1}\right|\right\|_{B V}+\frac{1}{\left|\log t\right|}\left\|\left| V^n-V^{n-1}\right|\right\|_1\right. \\
			& \left.\qquad+\left\|\left|\frac{\sqrt{\tau}}{\left|\log \tau\right|}\left(U_x^n-U_x^{n-1}\right)\right|\right\|_{\infty}+\left\|\left|\frac{U_x^n-U_x^{n-1}}{\left|\log \tau\right|}\right|\right\|_1+\frac{1}{\left|\log t\right|}\left\|\left|\Theta^n-\Theta^{n-1}\right|\right\|_1+\frac{1}{\left|\log t\right|}\left\|\left|Z^n-Z^{n-1}\right|\right\|_1\right).
		\end{align*}
	\end{lemma}

Recall the norms $\left\|\left|\cdot\right|\right\|_{\infty}$, $\left\|\left|\cdot\right|\right\|_{1}$ and $\left\|\left|\cdot\right|\right\|_{BV}$ from Eq.~\eqref{norm sup t}.
    
	\begin{proof}
	Assume the estimates in Eq.~\eqref{estimates n}. Then  both $\nu^n$ and $\nu^{n-1}$ satisfy Eq.~\eqref{f}. We may thus apply Lemmas~\ref{lemma: Liu}--\ref{lemma: comparison2} to corresponding fundamental solution $H(x,t;y,0;\cdot)$.

    First, thanks to the representation formula~\eqref{eq: Theta n+1-n}, we have that
		\begin{align*}
			|(\Theta^{n+1}&-\Theta^n)(x,t)|\leq\int_0^t\int_{\mathbb{R}}\left|H_y(x,t;y,\tau;\nu^n)\right|\frac{\nu\left|\Theta_y^n\right|\left|V^n-V^{n-1}\right|}{c_v\left|(1+V^n)(1+V^{n-1})\right|}(y,\tau)\,\mathrm{d}y\,\mathrm{d}\tau \\
			&\quad +\int_0^t\int_{\mathbb{R}} \left|H(x,t;y,\tau;\nu^n)\right|\left|\mathcal{N}_2^n-\mathcal{N}_2^{n-1}\right|(y,\tau)\,\mathrm{d}y\,\mathrm{d}\tau.
		\end{align*}
		 For the first term on the right-hand side, we apply the estimate for $H_y$ in Lemma~\ref{lemma: Liu} and that for $|\Theta_y^n|$ in Eq.~\eqref{estimates n}. For the second term, we take the expression for $\mathcal{N}_2^n - \mathcal{N}_2^{n-1}$ from \eqref{N2n n-1}, and apply the bound on $H$ in Lemma~\ref{lemma: Liu} and Eq.~\eqref{estimates n}. In this way, one obtains that 
		\begin{align*}
			|(\Theta^{n+1}&-\Theta^n)(x,t)|\\
			\leq& O(1)\int_0^t\int_{\mathbb{R}}\frac{\frac{e^{-(x-y)^2}}{C_*(t-\tau)}}{t-\tau}\left|\Theta_y^n\right|(y,\tau)\,\mathrm{d}y\,\mathrm{d}\tau \left\|\left|V^n-V^{n-1}\right|\right\|_{\infty}\\
			&+O(1)\int_0^t\int_{\mathbb{R}}\frac{\frac{e^{-(x-y)^2}}{C_*(t-\tau)}}{\sqrt{t-\tau}}\left[\left(\left\|\left|V^n-V^{n-1}\right|\right\|_{\infty}+\left|\log\tau\right|\left\|\left|\frac{\Theta^n-\Theta^{n-1}}{\left|\log\tau\right|}\right|\right\|_{\infty}\right)\left|U_x^n\right|\right.\\
			&\qquad +\left.\left\|\left|V^n-V^{n-1}\right|\right\|_{\infty}\left|U_x^n\right|^2\right]\,\mathrm{d}y\,\mathrm{d}\tau\\
			&+O(1)\int_0^t \frac{1}{\sqrt{t-\tau}}\left|\log\tau\right|\left(1+\frac{\delta}{\sqrt{\tau}}\right)\,\mathrm{d}\tau\left\|\left|\frac{U_x^n-U_x^{n-1}}{\left|\log\tau\right|}\right|\right\|_{1}\\
			&+O(1)\int_0^t\int_{\mathbb{R}}\frac{\frac{e^{-(x-y)^2}}{C_*(t-\tau)}}{\sqrt{t-\tau}}\,\mathrm{d}y\,\mathrm{d}\tau \left\|\left|Z^n-Z^{n-1}\right|\right\|_{\infty}\\
			&+O(1)\int_0^t \frac{1}{\sqrt{t-\tau}}\|Z^{n-1}\|_{L_x^1}\left|\log\tau\right|\,\mathrm{d}\tau \left\|\left|\frac{\Theta^n-\Theta^{n-1}}{\left|\log\tau\right|}\right|\right\|_{\infty}\\
			\leq & O(1)\int_0^t\frac{1}{\sqrt{t-\tau}}\frac{\delta}{\sqrt{\tau}}\,\mathrm{d}\tau\left\|\left|V^n-V^{n-1}\right|\right\|_{\infty}+O(1)\int_0^t\left(\left\|\left|V^n-V^{n-1}\right|\right\|_{\infty}\right.\\
			&+\left.\left|\log\tau\right|\left\|\left|\frac{\Theta^n-\Theta^{n-1}}{\left|\log\tau\right|}\right|\right\|_{\infty}\right)\frac{\delta}{\sqrt{\tau}}\,\mathrm{d}\tau+O(1)\int_0^t\frac{1}{\sqrt{t-\tau}}\frac{\delta^2}{\sqrt{\tau}}\,\mathrm{d}\tau\left\|\left|V^n-V^{n-1}\right|\right\|_{\infty}\\
			&+O(1)\int_0^t \frac{1}{\sqrt{t-\tau}}\left|\log\tau\right|\left(1+\frac{\delta}{\sqrt{\tau}}\right)\,\mathrm{d}\tau\left\|\left|\frac{U_x^n-U_x^{n-1}}{\left|\log\tau\right|}\right|\right\|_{1}+O(1)t\left\|\left|Z^n-Z^{n-1}\right|\right\|_{\infty}\\
			&+O(1)\int_0^t \frac{1}{\sqrt{t-\tau}}\delta \left|\log\tau\right|\,\mathrm{d}\tau \left\|\left|\frac{\Theta^n-\Theta^{n-1}}{\left|\log\tau\right|}\right|\right\|_{\infty}\\
			\leq& O(1)(\sqrt{t}+\delta)\left(\left\|\left|V^n-V^{n-1}\right|\right\|_{\infty}+\left|
			\log t\right|\left\|\left|\frac{\Theta^n-\Theta^{n-1}}{\left|\log\tau\right|}\right|\right\|_{\infty}+\left|\log t\right|\left\|\left|\frac{U_x^n-U_x^{n-1}}{\left|\log\tau\right|}\right|\right\|_{1}\right.\\
			&\left.+\left\|\left|Z^n-Z^{n-1}\right|\right\|_{\infty}\right).
		\end{align*}

Next, we derive the $L^1$-estimate for $\left(\Theta^{n+1}-\Theta^n\right)$.
		\begin{align*}
			\int_{\mathbb{R}}\left|(\Theta^{n+1}-\Theta^n)(x,t)\right|\,\mathrm{d}x\leq&\int_{\mathbb{R}}\int_0^t\int_{\mathbb{R}}\left|H_y(x,t;y,\tau;\nu^n)\right|\frac{\nu\left|\Theta_y^n\right|\left|V^n-V^{n-1}\right|}{c_v\left|(1+V^n)(1+V^{n-1})\right|}(y,\tau)\,\mathrm{d}y\,\mathrm{d}\tau\,\mathrm{d}x\\
			&\quad +\int_{\mathbb{R}}\int_0^t\int_{\mathbb{R}} \left|H(x,t;y,\tau;\nu^n)\right|\left|\mathcal{N}_2^n-\mathcal{N}_2^{n-1}\right|(y,\tau)\,\mathrm{d}y\,\mathrm{d}\tau\,\mathrm{d}x\\
			\leq &O(1)\int_0^t\frac{1}{\sqrt{t-\tau}}\frac{\delta}{\sqrt{\tau}}\,\mathrm{d}\tau\left\|\left|V^n-V^{n-1}\right|\right\|_{1}\\
			&+O(1)\int_0^t \frac{\delta}{\sqrt{\tau}}\,\mathrm{d}\tau\left(\left\|\left|V^n-V^{n-1}\right|\right\|_{1}+\left\|\left|\Theta^n-\Theta^{n-1}\right|\right\|_{1}\right)\nonumber\\
			&+O(1)\int_0^t\frac{\delta^2}{\sqrt{\tau}}\,\mathrm{d}\tau\left\|\left|V^n-V^{n-1}\right|\right\|_{\infty}\\
			&+O(1)\int_0^t \left|\log\tau\right|\left(1+\frac{\delta}{\sqrt{\tau}}\right)\,\mathrm{d}\tau\left\|\left|\frac{U_x^n-U_x^{n-1}}{\left|\log\tau\right|}\right|\right\|_{1}\\
			&+O(1)t\left\|\left|Z^n-Z^{n-1}\right|\right\|_{1}+O(1)\int_0^t \left\|Z^{n-1}\right\|_{L_x^1}\,\mathrm{d}\tau \left\|\left|\Theta^n-\Theta^{n-1}\right|\right\|_{1}\\
			\leq&  O(1)(\sqrt{t}+\delta)\left(\left\|\left|V^n-V^{n-1}\right|\right\|_{1}+\left\|\left|V^n-V^{n-1}\right|\right\|_{\infty}+\left\|\left|\Theta^n-\Theta^{n-1}\right|\right\|_{1}\right.\\
			&+\left.\sqrt{t}\left|\log t\right|\left\|\left|\frac{U_x^n-U_x^{n-1}}{\left|\log\tau\right|}\right|\right\|_{1}+\sqrt{t}\left\|\left|Z^n-Z^{n-1}\right|\right\|_{1}\right).
		\end{align*}

To control $\Theta_x^{n+1}-\Theta_x^n$, we differentiate Eq.~\eqref{eq: Theta k+1} with respect to $x$ to deduce
		\begin{align*}
			(\Theta_x^{n+1}-\Theta_x^n)(x,t)=&\int_{\mathbb{R}}\left(H_x(x,t;y,0;\nu^n)-H_x(x,t;y,0;\nu^{n-1})\right)\theta_0^*(y)\,\mathrm{d}y\\
			&+\int_0^t\int_{\mathbb{R}}H_x(x,t;y,\tau;\nu^{n-1})\left(\mathcal{N}_2^n(y,\tau)-\mathcal{N}_2^{n-1}(y,\tau)\right)\,\mathrm{d}y\,\mathrm{d}\tau\\
			&+\int_0^t\int_{\mathbb{R}}\left(H_x(x,t;y,\tau;\nu^n)-H_x(x,t;y,\tau;\nu^{n-1})\right)\mathcal{N}_2^{n}(y,\tau)\,\mathrm{d}y\,\mathrm{d}\tau\\
			=&:\mathcal{I}_1+\mathcal{I}_2+\mathcal{I}_3.
		\end{align*}

        We estimate $\mathcal{I}_1$ by using the $L^\infty$-bound for $\theta_0^*$ and the comparison estimates for $H_x$ in Lemma~\ref{lemma: comparison}:
		\begin{align*}
			\left|\mathcal{I}_1\right|\leq& \int_{\mathbb{R}}\left|\left(H_x(x,t;y,0;\nu^n)-H_x(x,t;y,0;\nu^{n-1})\right)\right|\,\mathrm{d}y \left\|\theta_0^*\right\|_{L_x^{\infty}}\\
			\leq &O(1)\delta\int_{\mathbb{R}}\frac{e^{\frac{-(x-y)^2}{C_*t}}}{t}\,\mathrm{d}y\left[\left|\log t\right|\left\|\left|V^n-V^{n-1}\right|\right\|_{\infty}+\left\|\left|V^n-V^{n-1}\right|\right\|_{BV}+\sqrt{t}\left\|\left|V^n-V^{n-1}\right|\right\|_{1}\right.\\
			&\quad +\left.\sqrt{t}\left|\log t\right|\left\|\left|\frac{\sqrt{\tau}}{\left|\log \tau\right|}\left(U_x^n-U_x^{n-1}\right)\right|\right\|_{\infty}\right]\\
			\leq &O(1)\frac{\delta}{\sqrt{t}}\left[\left|\log t\right|\left\|\left|V^n-V^{n-1}\right|\right\|_{\infty}+\left\|\left|V^n-V^{n-1}\right|\right\|_{BV}+\sqrt{t}\left\|\left|V^n-V^{n-1}\right|\right\|_{1}\right.\\
			&\quad +\left.\sqrt{t}\left|\log t\right|\left\|\left|\frac{\sqrt{\tau}}{\left|\log \tau\right|}\left(U_x^n-U_x^{n-1}\right)\right|\right\|_{\infty}\right].
		\end{align*}
        
        For $\mathcal{I}_2$, we  apply the estimate of $H_x$ in Lemma~\ref{lemma: Liu} to deduce that
		\begin{align*}
			\left|\mathcal{I}_2\right|\leq& \int_0^t\int_{\mathbb{R}}\left|H_x(x,t;y,\tau;\nu^{n-1})\right|\left|\mathcal{N}_2^n(y,\tau)-\mathcal{N}_2^{n-1}(y,\tau)\right|\,\mathrm{d}y\,\mathrm{d}\tau\\
			\leq &O(1)\int_0^t\int_{\mathbb{R}}\frac{e^{\frac{-(x-y)^2}{C_*(t-\tau)}}}{t-\tau}   
			\left[\left(\left|V^n-V^{n-1}\right|+\left|\Theta^n-\Theta^{n-1}\right|\right)\left|U_x^n\right|+\left|V^n-V^{n-1}\right|\left(U_x^n\right)^2\right.\\
			& \left.+\left(1+\left|U_x^n+U_x^{n-1}\right|\right)\left|U_x^n-U_x^{n-1}\right|+\left|Z^n-Z^{n-1}\right|+ \left|\Theta^n-\Theta^{n-1}\right|Z^{n-1}\right]
			\,\mathrm{d}y\,\mathrm{d}\tau\\
			\leq& O(1)\int_0^t\frac{1}{\sqrt{t-\tau}}\frac{\delta}{\sqrt{\tau}}\,\mathrm{d}\tau \left\|\left|V^n-V^{n-1}\right|\right\|_{\infty}+O(1)\int_0^t\frac{1}{\sqrt{t-\tau}}\frac{\delta}{\sqrt{\tau}}\left|\log\tau\right|\,\mathrm{d}\tau \left\|\left|\frac{\Theta^n-\Theta^{n-1}}{\left|\log\tau\right|}\right|\right\|_{\infty}\\
			&+O(1)\left(\int_0^{\frac{t}{2}}\frac{1}{t-\tau}\frac{\delta^2}{\sqrt{\tau}}\,\mathrm{d}\tau +\int_{\frac{t}{2}}^t\frac{1}{\sqrt{t-\tau}}\frac{\delta^2}{\tau}\,\mathrm{d}\tau\right)\left\|\left|V^n-V^{n-1}\right|\right\|_{\infty}\nonumber\\
			&+O(1)\int_0^{\frac{t}{2}}\left(1+\frac{\delta}{\sqrt{\tau}}\right)\frac{1}{t-\tau}\left|\log\tau\right|\,\mathrm{d}\tau \left\|\left|\frac{U_x^n-U_x^{n-1}}{\left|\log\tau\right|}\right|\right\|_{1}\\
			&+O(1)\int_{\frac{t}{2}}^t\left(1+\frac{\delta}{\sqrt{\tau}}\right)\frac{1}{\sqrt{t-\tau}}\frac{\left|\log\tau\right|}{\sqrt{\tau}}\,\mathrm{d}\tau \left\|\left|\frac{\sqrt{\tau}\left(U_x^n-U_x^{n-1}\right)}{\left|\log\tau\right|}\right|\right\|_{\infty}\\
			&+O(1)\int_0^t\frac{1}{\sqrt{t-\tau}}\,\mathrm{d}\tau \left\|\left|Z^n-Z^{n-1}\right|\right\|_{\infty}+O(1)\int_0^t\frac{1}{\sqrt{t-\tau}}\left|\log\tau\right|\,\mathrm{d}\tau \left\|\left|\frac{\Theta^n-\Theta^{n-1}}{\left|\log\tau\right|}\right|\right\|_{\infty}\left\|Z^{n-1}\right\|_{L_x^{\infty}}\\
			\leq &O(1)\delta\left(\left\|\left|V^n-V^{n-1}\right|\right\|_{\infty}+\left|\log t\right|\left\|\left|\frac{\Theta^n-\Theta^{n-1}}{\left|\log\tau\right|}\right|\right\|_{\infty}+\frac{\delta}{\sqrt{t}}\left\|\left|V^n-V^{n-1}\right|\right\|_{\infty}\right)\\
			&+O(1)(\delta+\sqrt{t})\frac{\left|\log t\right|}{\sqrt{t}}\left(\left\|\left|\frac{U_x^n-U_x^{n-1}}{\left|\log\tau\right|}\right|\right\|_{1}+\left\|\left|\frac{\sqrt{\tau}\left(U_x^n-U_x^{n-1}\right)}{\left|\log\tau\right|}\right|\right\|_{\infty}\right)\\
			&+O(1)\sqrt{t}\left\|\left|Z^n-Z^{n-1}\right|\right\|_{\infty}+O(1)\delta \sqrt{t}\left|\log t\right|\left\|\left|\frac{\Theta^n-\Theta^{n-1}}{\left|\log\tau\right|}\right|\right\|_{\infty},
		\end{align*}
		where we split the time integral at $\frac{t}{2}$ to handle the singularity at $\tau=t$.
		
		Finally, we estimate $\mathcal{I}_3$ by combining the comparison estimates for $H_x$ in Lemma~\ref{lemma: comparison} with bounds for $\mathcal{N}_2^n$. We again split the integral at $\frac{t}{2}$. 
		\begin{align*}
			\left|\mathcal{I}_3\right| \leq& O(1)\int_0^t\int_{\mathbb{R}}\left|\left(H_x(x,t;y,\tau;\nu^n)-H_x(x,t;y,\tau;\nu^{n-1})\right)\right|\left[\left|U_y^n\right|\left(1+\frac{\delta}{\sqrt{\tau}}\right)+Z^n\right](y,\tau)\,\mathrm{d}y\,\mathrm{d}\tau\\
			\leq &O(1)\int_0^t\int_{\mathbb{R}}\frac{e^{\frac{-(x-y)^2}{C_*(t-\tau)}}}{t-\tau}\left[\left|\log (t-\tau)\right|\left\|\left|V^n-V^{n-1}\right|\right\|_{\infty}+\left\|\left|V^n-V^{n-1}\right|\right\|_{BV}+\sqrt{t-\tau}\left\|\left|V^n-V^{n-1}\right|\right\|_{1}\right.\\
			&\quad +\left.\sqrt{t-\tau}\left|\log t\right|\left\|\left|\frac{\sqrt{\tau}}{\left|\log \tau\right|}\left(U_x^n-U_x^{n-1}\right)\right|\right\|_{\infty}\right]\left[\left|U_y^n\right|\left(1+\frac{\delta}{\sqrt{\tau}}\right)+Z^n\right]\,\mathrm{d}y\,\mathrm{d}\tau\\
			\leq& O(1)\left[\int_0^t\frac{1}{\sqrt{t-\tau}}\left|\log(t-\tau)\right|\frac{\delta}{\sqrt{\tau}}\,\mathrm{d}\tau+\int_0^{\frac{t}{2}}\frac{1}{t-\tau}\left|\log(t-\tau)\right|\frac{\delta^2}{\sqrt{\tau}}\,\mathrm{d}\tau+\int_{\frac{t}{2}}^t\frac{1}{\sqrt{t-\tau}}\left|\log(t-\tau)\right|\frac{\delta^2}{\tau}\,\mathrm{d}\tau \right]\\
			&\times \left\|\left|V^n-V^{n-1}\right|\right\|_{\infty}+O(1)\left[\int_0^t\frac{1}{\sqrt{t-\tau}}\frac{\delta}{\sqrt{\tau}}\,\mathrm{d}\tau+\int_0^{\frac{t}{2}}\frac{1}{t-\tau}\frac{\delta^2}{\sqrt{\tau}}\,\mathrm{d}\tau+\int_{\frac{t}{2}}^t\frac{1}{\sqrt{t-\tau}}\frac{\delta^2}{\tau}\,\mathrm{d}\tau \right]\\ &\times\left\|\left|V^n-V^{n-1}\right|\right\|_{BV}+O(1)\left[\int_0^t\frac{\delta}{\sqrt{\tau}}\,\mathrm{d}\tau+\int_0^{\frac{t}{2}}\frac{1}{\sqrt{t-\tau}}\frac{\delta^2}{\sqrt{\tau}}\,\mathrm{d}\tau+\int_{\frac{t}{2}}^t\frac{\delta^2}{\tau}\,\mathrm{d}\tau \right]\left\|\left|V^n-V^{n-1}\right|\right\|_{1}\\
			&+O(1)\left[\int_0^t\frac{\delta}{\sqrt{\tau}}\,\mathrm{d}\tau+\int_0^{\frac{t}{2}}\frac{1}{\sqrt{t-\tau}}\frac{\delta^2}{\sqrt{\tau}}\,\mathrm{d}\tau+\int_{\frac{t}{2}}^t\frac{\delta^2}{\tau}\,\mathrm{d}\tau \right]\left|\log t\right|\left\|\left|\frac{\sqrt{\tau}}{\left|\log \tau\right|}\left(U_x^n-U_x^{n-1}\right)\right|\right\|_{\infty}\\
			&+O(1)\left\|Z^n\right\|_{L_x^{\infty}}\int_0^t\frac{1}{\sqrt{t-\tau}}\left[\left|\log (t-\tau)\right|\left\|\left|V^n-V^{n-1}\right|\right\|_{\infty}+\left\|\left|V^n-V^{n-1}\right|\right\|_{BV}\right.\\
			&\quad +\left.\sqrt{t-\tau}\left\|\left|V^n-V^{n-1}\right|\right\|_{1}+\sqrt{t-\tau}\left|\log t\right|\left\|\left|\frac{\sqrt{\tau}}{\left|\log \tau\right|}\left(U_x^n-U_x^{n-1}\right)\right|\right\|_{\infty}\right]\,\mathrm{d}\tau\\
			\leq& O(1)\frac{(\sqrt{t}+\delta)\delta}{\sqrt{t}}\left[\left|\log t\right|\left\|\left|V^n-V^{n-1}\right|\right\|_{\infty}+\left\|\left|V^n-V^{n-1}\right|\right\|_{BV}+\sqrt{t}\left\|\left|V^n-V^{n-1}\right|\right\|_{1}\right.\\
			&\qquad +\left.\sqrt{t}\left|\log t\right|\left\|\left|\frac{\sqrt{\tau}}{\left|\log \tau\right|}\left(U_x^n-U_x^{n-1}\right)\right|\right\|_{\infty}\right].
		\end{align*}

		Combining the estimates for $\mathcal{I}_1$, $\mathcal{I}_2$ and $\mathcal{I}_3$ and dividing both sides by $\frac{|\log t|}{\sqrt{t}}$, we obtain that 
		\begin{align}
			\frac{\sqrt{t}}{\left|\log t\right|}\left|\left(\Theta_x^{n+1}-\Theta_x^n\right)\right|\leq&C_2(\sqrt{t}+\delta)\left[ \left\|\left|V^n-V^{n-1}\right|\right\|_{\infty}+\left\|\left|V^n-V^{n-1}\right|\right\|_{BV}+ \left\|\left|V^n-V^{n-1}\right|\right\|_{1}\right.\nonumber\\
			&+\left.\left\|\left|\frac{\sqrt{\tau}}{\left|\log \tau\right|}\left(U_x^n-U_x^{n-1}\right)\right|\right\|_{\infty}+\left\|\left|\frac{U_x^n-U_x^{n-1}}{\left|\log\tau\right|}\right|\right\|_{1}+\left\|\left|\frac{\Theta^n-\Theta^{n-1}}{\left|\log\tau\right|}\right|\right\|_{\infty}\right.\nonumber\\
			&+\left.\left\|\left|Z^n-Z^{n-1}\right|\right\|_{\infty}\right].
		\end{align}

        Finally, via a similar argument, we integrate $\mathcal{I}_1$, $\mathcal{I}_2$, and $\mathcal{I}_3$ over $x \in \mathbb{R}$ to obtain that
		\begin{align}
			\frac{\left\|\Theta_x^{n+1}-\Theta_x^n\right\|_{L_x^1}}{\left|\log t\right|}\leq& C_2(\sqrt{t}+\delta)\left[ \left\|\left|V^n-V^{n-1}\right|\right\|_{\infty}+\left\|\left|V^n-V^{n-1}\right|\right\|_{BV}+ \left\|\left|V^n-V^{n-1}\right|\right\|_{1}\right.\nonumber\\
			&+\left.\left\|\left|\frac{\sqrt{\tau}}{\left|\log \tau\right|}\left(U_x^n-U_x^{n-1}\right)\right|\right\|_{\infty}+\left\|\left|\frac{U_x^n-U_x^{n-1}}{\left|\log\tau\right|}\right|\right\|_{1}+\left\|\left|\Theta^n-\Theta^{n-1}\right|\right\|_{1}\right.\nonumber\\
			&+\left.\left\|\left|Z^n-Z^{n-1}\right|\right\|_{1}\right].
		\end{align}
        This completes the proof.  
	\end{proof}

	\begin{lemma}
		\label{lemma: Z n diff}
		For sufficiently small $\delta$ and $t_{\sharp}$ and for $0<t<t_{\sharp}$, there exists a positive constant $C_2$ such that the differences $Z^{n+1}-Z^n$
		and their derivatives satisfy the following:
		\begin{align}
			& \left\|Z^{n+1}(\cdot, t)-Z^n(\cdot, t)\right\|_{L_x^{\infty}} \nonumber\\
			& \quad \leq C_2\left(\sqrt{t_{\sharp}}\left|\log t\right|+\delta\right)\left(\left\|\left|V^n-V^{n-1}\right|\right\|_{\infty}+\left\|\left|\frac{\Theta^n-\Theta^{n-1}}{\left|\log \tau\right|}\right|\right\|_{\infty}+\left\|\left|Z^n-Z^{n-1}\right|\right\|_{\infty}\right), \nonumber\\
			& \left\|Z^{n+1}(\cdot, t)-Z^n(\cdot, t)\right\|_{L_x^{1}} \nonumber\\
			& \quad \leq C_2\left(\sqrt{t_{\sharp}}+\delta\right)\left(\left\|\left|V^n-V^{n-1}\right|\right\|_1+\left\|\left|\Theta^n-\Theta^{n-1}\right|\right\|_1+\left\|\left|Z^n-Z^{n-1}\right|\right\|_{1}\right), \nonumber\\
			& \frac{\sqrt{t}}{\left|\log t\right|}\left\|Z_x^{n+1}(\cdot, t)-Z_x^n(\cdot, t)\right\|_{L_x^{\infty}}\nonumber \\
			& \quad \leq C_2\left(\sqrt{t_{\sharp}}+\delta\right)\left(\left\|\left|V^n-V^{n-1}\right|\right\|_{\infty}+\frac{1}{\left|\log t\right|}\left\|\left| V^n-V^{n-1}\right|\right\|_{B V}+\frac{1}{\left|\log t\right|}\left\|\left|V^n-V^{n-1}\right|\right\|_1\right.\nonumber \\
			& \left.\quad+\left\|\left|\frac{\sqrt{\tau}}{\left|\log \tau\right|}\left(U_x^n-U_x^{n-1}\right)\right|\right\|_{\infty}+\left\|\left|\frac{\Theta^n-\Theta^{n-1}}{\left|\log \tau\right|}\right|\right\|_{\infty}+\frac{1}{\left|\log t\right|}\left\|\left|Z^n-Z^{n-1}\right|\right\|_{\infty}\right), \nonumber\\
			& \frac{\left\|Z_x^{n+1}(\cdot, t)-Z_x^n(\cdot, t)\right\|_{L_x^{1}}}{\left|\log t\right|} \nonumber\\
			& \leq C_2\left(\sqrt{t_{\sharp}}+\delta\right)\left(\left\|\left| V^n-V^{n-1}\right|\right\|_{\infty}+\frac{1}{\left|\log t\right|}\left\|\left|V^n-V^{n-1}\right|\right\|_{B V}+\frac{1}{\left|\log t\right|}\left\|\left| V^n-V^{n-1}\right|\right\|_1\right. \nonumber\\
			& \left.\quad+\left\|\left|\frac{\sqrt{\tau}}{\left|\log \tau\right|}\left(U_x^n-U_x^{n-1}\right)\right|\right\|_{\infty}+\frac{1}{\left|\log t\right|}\left\|\left|\Theta^n-\Theta^{n-1}\right|\right\|_1+\frac{1}{\left|\log t\right|}\left\|\left|Z^n-Z^{n-1}\right|\right\|_1\right).
		\end{align}
	\end{lemma}
	\begin{proof}
		The proof follows a similar structure to that of Lemma \ref{lemma: Theta n diff}.

       Assume the bounds in Eq.~\eqref{estimates n}. Then, both $D^n$ and $D^{n-1}$ satisfy \eqref{f}, which enables us to apply Lemmas~\ref{lemma: Liu}--\ref{lemma: comparison2} to the fundamental solutions $H(x,t;y,\tau;D^n)$ and $H(x,t;y,\tau;D^{n-1})$. Indeed, we have the pointwise bound
       \begin{align*}
			|(Z^{n+1}&-Z^n)(x,t)|\leq\int_0^t\int_{\mathbb{R}}\left|H_y(x,t;y,\tau;D^n)\right|\frac{\left|Z_y^n\right|\left|(V^n-V^{n-1})(V^n+V^{n-1}+2)\right|}{(1+V^n)^2(1+V^{n-1})^2}(y,\tau)\,\mathrm{d}y\,\mathrm{d}\tau\nonumber\\
			&\quad +O(1)\int_0^t\int_{\mathbb{R}} \left|H(x,t;y,\tau;D^n)\right|\left[\left|Z^n-Z^{n-1}\right|+ \left|\Theta^n-\Theta^{n-1}\right|Z^{n-1}\right]\,\mathrm{d}y\,\mathrm{d}\tau,
		\end{align*}
       thanks to the representation formula~\eqref{eq: Z n+1-n} and the bound for $\mathcal{N}_3^n-\mathcal{N}_3^{n-1}$ in Eq.~\eqref{N3n n-1}. From here, we apply the estimates for $H_y(x,t;y,\tau;D^n)$ and $H(x,t;y,\tau;D^n)$ in Lemma~\ref{lemma: Liu}, together with the $L^{\infty}$-estimates for $Z_y^n$ in Eq.~\eqref{estimates n}, to obtain that
		\begin{align*}
			|(Z^{n+1}-Z^n)(x,t)|
			\leq & O(1)\left\{\int_0^t\int_{\mathbb{R}}\frac{e^{\frac{-(x-y)^2}{C_*(t-\tau)}}}{t-\tau}\left|Z_y^n\right|\,\mathrm{d}y\,\mathrm{d}\tau \right\}\left\|\left| V^n-V^{n-1}\right|\right\|_{\infty}\\
			&+O(1)\left\{\int_0^t\int_{\mathbb{R}}\frac{e^{\frac{-(x-y)^2}{C_*(t-\tau)}}}{\sqrt{t-\tau}}\,\mathrm{d}y\,\mathrm{d}\tau \right\}\left\|\left| Z^n-Z^{n-1}\right|\right\|_{\infty}\\
			&+O(1)\left\{\int_0^t\int_{\mathbb{R}}\frac{e^{\frac{-(x-y)^2}{C_*(t-\tau)}}}{\sqrt{t-\tau}}\left|\log \tau\right|\,\mathrm{d}y\,\mathrm{d}\tau \right\}\left\|\left|\frac{\Theta^n-\Theta^{n-1}}{\left|\log \tau\right|}\right|\right\|_{\infty}\left\|Z^{n-1}\right\|_{L_x^{\infty}}\\
			\leq & O(1)\left\{\int_0^t\frac{1}{\sqrt{t-\tau}}\frac{\delta}{\sqrt{\tau}}\,\mathrm{d}\tau \right\}\left\|\left| V^n-V^{n-1}\right|\right\|_{\infty}+O(1)t\left\|\left| Z^n-Z^{n-1}\right|\right\|_{\infty}\nonumber\\
			&+O(1)\delta \left\{\int_0^t \left|\log \tau\right|\,\mathrm{d}\tau\right\} \left\|\left|\frac{\Theta^n-\Theta^{n-1}}{\left|\log \tau\right|}\right|\right\|_{\infty}.
		\end{align*}
We thus arrive at 
\begin{align*}
    &|(Z^{n+1}-Z^n)(x,t)|\\
    &\quad \leq O(1) \left(\delta\left\|\left|V^n-V^{n-1}\right|\right\|_{\infty}+\delta t\left|\log t\right|\left\|\left|\frac{\Theta^n-\Theta^{n-1}}{\left|\log \tau\right|}\right|\right\|_{\infty}+t\left\|\left|Z^n-Z^{n-1}\right|\right\|_{\infty}\right).
\end{align*}

The $L^1$-estimate for $Z^{n+1}-Z^{n}$
		follows easily from the pointwise bound above:
		\begin{align*}
			\int_{\mathbb{R}}|(Z^{n+1}&-Z^n)|(x,t)\,\mathrm{d}x \nonumber\\ 
			\leq&  O(1)\left\{\int_0^t\frac{1}{\sqrt{t-\tau}}\frac{\delta}{\sqrt{\tau}}\,\mathrm{d}\tau\right\} \left\|\left| V^n-V^{n-1}\right|\right\|_{1}+O(1)t\left\|\left| Z^n-Z^{n-1}\right|\right\|_{1}\\
			&+O(1)\left\{\int_0^t \frac{1}{\sqrt{t-\tau}}\delta \,\mathrm{d}\tau\right\}\left\|\left|\Theta^n-\Theta^{n-1}\right|\right\|_{1}\\
			\leq &O(1) \left(\delta\left\|\left|V^n-V^{n-1}\right|\right\|_{1}+\sqrt{t}\delta\left\|\left|\Theta^n-\Theta^{n-1}\right|\right\|_{1}+t\left\|\left|Z^n-Z^{n-1}\right|\right\|_{1}\right).
		\end{align*}

		To estimate $Z_x^{n+1}-Z_x^n$, we differentiate the representation formula~\eqref{eq: Z k+1} of $Z^{n+1}$  with respect to $x$, and then decompose  into three terms:
		\begin{align*}
			\left(Z_x^{n+1}-Z_x^n\right)(x,t)= &\int_{\mathbb{R}} \left[H_x(x,t;y,0;D^n)-H_x(x,t;y,0;D^{n-1})\right] z_0^*(y) \,\mathrm{d}y \\
			& +\int_0^t \int_{\mathbb{R}} H_x(x,t;y,\tau; D^n) \left[\mathcal{N}_3^n-\mathcal{N}_3^{n-1}\right](y,\tau)\,\mathrm{d}y \,\mathrm{d}\tau\\
			&-\int_0^t \int_{\mathbb{R}}\left[H_x(x,t;y,\tau; D^n)-H_x(x,t;y,\tau; D^{n-1})\right]\\
			&\qquad \times K\phi(1+\Theta^{n-1})Z^{n-1}\,\mathrm{d}y \,\mathrm{d}\tau\\
			=&: \mathcal{I}_1+\mathcal{I}_2+\mathcal{I}_3.
		\end{align*}
		We shall bound each term $\mathcal{I}_1$, $\mathcal{I}_2$, and $\mathcal{I}_3$ separately. 
\begin{itemize}
    \item
 For $\mathcal{I}_1$, by the $L^{\infty}$-bound for $z_0^*$ and the comparison estimate for $H_x(x,t;y,\tau; D^n)$ in Lemma~\ref{lemma: comparison}, one arrives at
\begin{align*}
			\left|\mathcal{I}_1\right|
			\leq &O(1)\delta\left\{\int_{\mathbb{R}}\frac{e^{\frac{-(x-y)^2}{C_*t}}}{t}\,\mathrm{d}y\right\}\left[\left|\log t\right|\left\|\left|V^n-V^{n-1}\right|\right\|_{\infty}+\left\|\left|V^n-V^{n-1}\right|\right\|_{BV}+\sqrt{t}\left\|\left|V^n-V^{n-1}\right|\right\|_{1}\right.\\
			&\quad +\left.\sqrt{t}\left|\log t\right|\left\|\left|\frac{\sqrt{\tau}}{\left|\log \tau\right|}\left(U_x^n-U_x^{n-1}\right)\right|\right\|_{\infty}\right]\left\|z_0^*\right\|_{L_x^{\infty}}\\
			\leq &O(1)\frac{\delta}{\sqrt{t}}\left[\left|\log t\right|\left\|\left|V^n-V^{n-1}\right|\right\|_{\infty}+\left\|\left|V^n-V^{n-1}\right|\right\|_{BV}+\sqrt{t}\left\|\left|V^n-V^{n-1}\right|\right\|_{1}\right.\\
			&\quad +\left.\sqrt{t}\left|\log t\right|\left\|\left|\frac{\sqrt{\tau}}{\left|\log \tau\right|}\left(U_x^n-U_x^{n-1}\right)\right|\right\|_{\infty}\right].
		\end{align*}

        \item 
        For $\mathcal{I}_2$, we directly estimate that 	\begin{align*}
			\left|\mathcal{I}_2\right|
			\leq &O(1)\int_0^t\int_{\mathbb{R}}\frac{e^{\frac{-(x-y)^2}{C_*(t-\tau)}}}{t-\tau}   
			\left[\left|Z^n-Z^{n-1}\right|+ \left|\Theta^n-\Theta^{n-1}\right|Z^{n-1}\right]
			\,\mathrm{d}y\,\mathrm{d}\tau \\
			\leq& O(1)\left\{\int_0^t\frac{1}{\sqrt{t-\tau}}\,\mathrm{d}\tau \right\}\left\|\left|Z^n-Z^{n-1}\right|\right\|_{\infty}+O(1)\left\{\int_0^t\frac{1}{\sqrt{t-\tau}}\left|\log\tau\right|\,\mathrm{d}\tau \right\}\left\|\left|\frac{\Theta^n-\Theta^{n-1}}{\left|\log\tau\right|}\right|\right\|_{\infty}\left\|Z^{n-1}\right\|_{L_x^{\infty}}\\
			\leq &O(1)\sqrt{t}\left\|\left|Z^n-Z^{n-1}\right|\right\|_{\infty}+O(1)\delta \sqrt{t}\left|\log t\right|\left\|\left|\frac{\Theta^n-\Theta^{n-1}}{\left|\log\tau\right|}\right|\right\|_{\infty}.
		\end{align*}
 \item 
	For $\mathcal{I}_3$, we use the Lipschitz continuity of $\phi$, combined with the estimate for $H_x(x,t;y,\tau;D^n)$ and the bound for $\|Z^{n-1}\|_{L_x^{\infty}}$ in Eq.~\eqref{estimates n}, to obtain that
 \begin{align*}
			\left|\mathcal{I}_3\right| \leq& O(1)\int_0^t\int_{\mathbb{R}}\frac{e^{\frac{-(x-y)^2}{C_*(t-\tau)}}}{t-\tau}\left[\left|\log (t-\tau)\right|\left\|\left|V^n-V^{n-1}\right|\right\|_{\infty}+\left\|\left|V^n-V^{n-1}\right|\right\|_{BV}+\sqrt{t-\tau}\left\|\left|V^n-V^{n-1}\right|\right\|_{1}\right.\\
			&\quad +\left.\sqrt{t-\tau}\left|\log t\right|\left\|\left|\frac{\sqrt{\tau}}{\left|\log \tau\right|}\left(U_x^n-U_x^{n-1}\right)\right|\right\|_{\infty}\right]\left\|Z^{n-1}\right\|_{L_x^{\infty}}\,\mathrm{d}y\,\mathrm{d}\tau\\
			\leq& O(1)\delta\left[\int_0^t\frac{1}{\sqrt{t-\tau}}\left|\log(t-\tau)\right|\,\mathrm{d}\tau\left\|\left|V^n-V^{n-1}\right|\right\|_{\infty}+\int_0^{t}\frac{1}{\sqrt{t-\tau}}\,\mathrm{d}\tau\left\|\left|V^n-V^{n-1}\right|\right\|_{BV}\right.\\
			&\left.+t\left\|\left|V^n-V^{n-1}\right|\right\|_{1}+t\left|\log t\right|\left\|\left|\frac{\sqrt{\tau}}{\left|\log \tau\right|}\left(U_x^n-U_x^{n-1}\right)\right|\right\|_{\infty}\right]\\
			\leq& O(1)\delta\left[\frac{1}{\sqrt{t}}\left|\log t\right|\left\|\left|V^n-V^{n-1}\right|\right\|_{\infty}+\sqrt{t}\left\|\left|V^n-V^{n-1}\right|\right\|_{BV}+t\left\|\left|V^n-V^{n-1}\right|\right\|_{1}\right.\\
			&\qquad +\left.t\left|\log t\right|\left\|\left|\frac{\sqrt{\tau}}{\left|\log \tau\right|}\left(U_x^n-U_x^{n-1}\right)\right|\right\|_{\infty}\right].
		\end{align*}
\end{itemize}

Summarising that above pointwise bounds for $\mathcal{I}_1$, $ \mathcal{I}_2$, and $\mathcal{I}_3$, we arrive at 
		\begin{align*}
			&\frac{\sqrt{t}}{\left|\log t\right|}\left|Z_x^{n+1}-Z_x^n\right|\\
			&\qquad \leq C_2(\sqrt{t}+\delta)\bigg\{\left\|\left|V^n-V^{n-1}\right|\right\|_{\infty}+\frac{1}{\left|\log t\right|}\left\|\left|V^n-V^{n-1}\right|\right\|_{BV}+ \frac{1}{\left|\log t\right|}\left\|\left|V^n-V^{n-1}\right|\right\|_{1}\\
			&\qquad +\left\|\left|\frac{\sqrt{\tau}}{\left|\log \tau\right|}\left(U_x^n-U_x^{n-1}\right)\right|\right\|_{\infty}+\left\|\left|\frac{\Theta^n-\Theta^{n-1}}{\left|\log\tau\right|}\right|\right\|_{\infty}+\frac{1}{\left|\log t\right|}\left\|\left|Z^n-Z^{n-1}\right|\right\|_{\infty}\bigg\}.
		\end{align*}

        Next, let us derive the $L^1$-estimate for $Z_x^{n+1}-Z_x^n= \mathcal{I}_1+\mathcal{I}_2+\mathcal{I}_3$. 
       \begin{itemize}
           \item 
           For $\mathcal{I}_1$, combining the definition of $W$ in \eqref{W anti_deri}, the estimates in Lemma~\ref{lemma: comparison}, and the BV-norm bound for $z_0^*$, we obtain that
		\begin{align*}
			\int_{\mathbb{R}}\left|\mathcal{I}_1\right|\,\mathrm{d}x
			&\leq O(1)\left\{\int_{\mathbb{R}}\frac{e^{\frac{-(x-y)^2}{C_*t}}}{\sqrt{t}}\,\mathrm{d}y\right\}\left\|z_0^*\right\|_{BV} \bigg[\left|\log t\right|\left\|\left|V^n-V^{n-1}\right|\right\|_{\infty}+\left\|\left|V^n-V^{n-1}\right|\right\|_{BV}\\
			& +\sqrt{t}\left\|\left|V^n-V^{n-1}\right|\right\|_{1}+\sqrt{t}\left|\log t\right|\left\|\left|\frac{\sqrt{\tau}}{\left|\log \tau\right|}\left(U_x^n-U_x^{n-1}\right)\right|\right\|_{\infty}\bigg]\\
			&\quad\leq O(1)\delta\bigg[\left|\log t\right|\left\|\left|V^n-V^{n-1}\right|\right\|_{\infty}+\left\|\left|V^n-V^{n-1}\right|\right\|_{BV}+\sqrt{t}\left\|\left|V^n-V^{n-1}\right|\right\|_{1}\\
			&\quad +\sqrt{t}\left|\log t\right|\left\|\left|\frac{\sqrt{\tau}}{\left|\log \tau\right|}\left(U_x^n-U_x^{n-1}\right)\right|\right\|_{\infty}\bigg].
		\end{align*}
           
           \item 
For $\mathcal{I}_2$, a direct computation leads to 
		\begin{align*}
			\int_{\mathbb{R}}\left|\mathcal{I}_2\right|\,\mathrm{d}x
			&\leq O(1)\left\{\int_0^t\frac{1}{\sqrt{t-\tau}}\,\mathrm{d}\tau \right\}\left\|\left|Z^n-Z^{n-1}\right|\right\|_{1}\\
            &\qquad + O(1)\left\{\int_0^t\frac{1}{\sqrt{t-\tau}}\,\mathrm{d}\tau \left\|\left|\Theta^n-\Theta^{n-1}\right|\right\|_{1}\left\|Z^{n-1}\right\|_{L_x^{\infty}}\right\}\\
&\leq  O(1)\sqrt{t}\left\|\left|Z^n-Z^{n-1}\right|\right\|_{1}+O(1)\delta \sqrt{t}\left\|\left|\Theta^n-\Theta^{n-1}\right|\right\|_{1}.
		\end{align*}
           
		For $\mathcal{I}_3$, a similar argument as that for $\mathcal{I}_1$ leads to
\begin{align}
\int_{\mathbb{R}}\left|\mathcal{I}_3\right|\,\mathrm{d}x &\leq O(1)\left\{\int_0^t\frac{1}{\sqrt{t-\tau}}\left|\log(t-\tau)\right|\delta \,\mathrm{d}\tau\right\}\left\|\left|V^n-V^{n-1}\right|\right\|_{\infty}\nonumber\\
&\qquad
 +O(1)\left\{\int_0^{t}\frac{1}{\sqrt{t-\tau}}\delta \,\mathrm{d}\tau\right\}\left\|\left|V^n-V^{n-1}\right|\right\|_{BV}\nonumber\\
&\qquad
 +O(1)\left\{\int_0^t \delta \,\mathrm{d}\tau\right\}\left\|\left|V^n-V^{n-1}\right|\right\|_{1}\nonumber\\
&\qquad
 +O(1)\left\{\int_0^t \left|\log t\right|\delta\,\mathrm{d}\tau \right\} \left\|\left|\frac{\sqrt{\tau}}{\left|\log \tau\right|}\left(U_x^n-U_x^{n-1}\right)\right|\right\|_{\infty}\nonumber\\
		& \leq  O(1)\delta\Bigg[\sqrt{t}\left|\log t\right|\left\|\left|V^n-V^{n-1}\right|\right\|_{\infty}+\sqrt{t}\left\|\left|V^n-V^{n-1}\right|\right\|_{BV}\nonumber\\
			&\qquad+t\left\|\left|V^n-V^{n-1}\right|\right\|_{1} + t\left|\log t\right|\left\|\left|\frac{\sqrt{\tau}}{\left|\log \tau\right|}\left(U_x^n-U_x^{n-1}\right)\right|\right\|_{\infty}\Bigg].
		\end{align}
       \end{itemize}

		Summarising the above estimates, we infer that for sufficiently small $\delta$ and $t_{\sharp}$,
		\begin{align*}
			&\frac{\left\|Z_x^{n+1}-Z_x^n\right\|_{L_x^{1}}}{\left|\log t\right|}\\
			&\quad \leq C_2(\sqrt{t}+\delta)\Bigg[\left\|\left|V^n-V^{n-1}\right|\right\|_{\infty}+\frac{1}{\left|\log t\right|}\left\|\left|V^n-V^{n-1}\right|\right\|_{BV}+ \frac{1}{\left|\log t\right|}\left\|\left|V^n-V^{n-1}\right|\right\|_{1}\\
&\qquad +\left\|\left|\frac{\sqrt{\tau}}{\left|\log \tau\right|}\left(U_x^n-U_x^{n-1}\right)\right|\right\|_{\infty}+\frac{1}{\left|\log t\right|}\left\|\left|\Theta^n-\Theta^{n-1}\right|\right\|_{1}+\frac{1}{\left|\log t\right|}\left\|\left|Z^n-Z^{n-1}\right|\right\|_{1}\Bigg].
		\end{align*}

        The proof of Lemma~\ref{lemma: Z n diff} is now complete.  	\end{proof}

	The following two lemmas can be found in \cite{WangHT2022}. We safely omit the proof here.
	\begin{lemma}
		\label{lemma: U n diff}
		For sufficiently small $\delta$ and $t_{\sharp}$ and for $0<t<t_{\sharp}$, there exists a positive constant $C_2$ such that the differences $U^{n+1}-U^n$
		and their derivatives satisfy the following estimates:
		\begin{align}
			& \left\|U^{n+1}(\cdot,t)-U^n(\cdot, t)\right\|_{L_x^{\infty}} \nonumber\\
			& \quad \leq C_2\left(\sqrt{t_{\sharp}}\left|\log t\right|+\delta\right)\left(\left\|\left|V^n-V^{n-1}\right|\right\|_{\infty}+\left\|\left|\frac{\Theta^n-\Theta^{n-1}}{\left|\log \tau\right|}\right|\right\|_{\infty}\right), \nonumber\\
			& \left\|U^{n+1}(\cdot,t)-U^n(\cdot, t)\right\|_{L_x^{1}} \nonumber\\
			& \quad \leq C_2\left(\sqrt{t_{\sharp}}+\delta\right)\left(\left\|\left|V^n-V^{n-1}\right|\right\|_1+\left\|\left|\Theta^n-\Theta^{n-1}\right|\right\|_1\right), \nonumber\\
			& \frac{\sqrt{t}}{\left|\log t\right|}\left\|U_x^{n+1}(\cdot,t)-U_x^n(\cdot, t)\right\|_{L_x^{\infty}}\nonumber \\
			& \quad \leq C_2\left(\sqrt{t_{\sharp}}+\delta\right)\left(\left\|\left|V^n-V^{n-1}\right|\right\|_{\infty}+\frac{1}{\left|\log t\right|}\left\|\left| V^n-V^{n-1}\right|\right\|_{B V}+\frac{1}{\left|\log t\right|}\left\|\left|V^n-V^{n-1}\right|\right\|_1\right.\nonumber \\
			& \left.\quad+\left\|\left|\frac{\sqrt{\tau}}{\left|\log \tau\right|}\left(U_x^n-U_x^{n-1}\right)\right|\right\|_{\infty}+\left\|\left|\frac{\Theta^n-\Theta^{n-1}}{\left|\log \tau\right|}\right|\right\|_{\infty}\right), \nonumber\\
			& \frac{\left\|U_x^{n+1}(\cdot,t)-U_x^n(\cdot, t)\right\|_{L_x^{1}}}{\left|\log t\right|} \nonumber\\
			& \leq C_2\left(\sqrt{t_{\sharp}}+\delta\right)\left(\left\|\left| V^n-V^{n-1}\right|\right\|_{\infty}+\frac{1}{\left|\log t\right|}\left\|\left|V^n-V^{n-1}\right|\right\|_{B V}+\frac{1}{\left|\log t\right|}\left\|\left| V^n-V^{n-1}\right|\right\|_1\right. \nonumber\\
			& \left.\quad+\left\|\left|\frac{\sqrt{\tau}}{\left|\log \tau\right|}\left(U_x^n-U_x^{n-1}\right)\right|\right\|_{\infty}+\left\|\left|\frac{U_x^n-U_x^{n-1}}{\left|\log\tau\right|}\right|\right\|_{1}+\left\|\left|\frac{\Theta^n-\Theta^{n-1}}{\left|\log\tau\right|}\right|\right\|_{\infty}\right.\nonumber\\
			&\left.\quad +\frac{1}{\left|\log t\right|}\left\|\left|\Theta^n-\Theta^{n-1}\right|\right\|_1\right).
		\end{align}
	\end{lemma}
	
\begin{proof}
   See \cite{WangHT2022}.  
\end{proof}

	\begin{lemma}
		\label{lemma: V n diff}
		For sufficiently small $\delta$ and $t_{\sharp}$ and for $0<t<t_{\sharp}$, there exists a positive constant $C_2$ such that the differences $V^{n+1}-V^n$
		and their derivatives satisfy the following estimates:
		\begin{align}
			& \left\|V^{n+1}(\cdot,t)-V^n(\cdot, t)\right\|_{L_x^{\infty}} \nonumber\\
			& \quad \leq C_2\left(\sqrt{t_{\sharp}}+\delta\right)\left(\left\|\left|V^n-V^{n-1}\right|\right\|_{\infty}+\left\|\left|\frac{\Theta^n-\Theta^{n-1}}{\left|\log \tau\right|}\right|\right\|_{\infty}\right), \nonumber\\
			& \left\|V^{n+1}(\cdot,t)-V^n(\cdot, t)\right\|_{L_x^{1}} \nonumber\\
			& \quad \leq C_2\left(\sqrt{t_{\sharp}}+\delta\right)\left(\left\|\left|V^n-V^{n-1}\right|\right\|_1+\left\|\left|\Theta^n-\Theta^{n-1}\right|\right\|_1\right), \nonumber\\
			& \label{Vn cauchy}\left|\left.\left(V^{n+1}(\cdot,t)-V^n(\cdot, t)\right)\right|_{\omega^-}^{\omega^+}\right| \\
			& \quad \leq C_2\left(\sqrt{t_{\sharp}}+\delta\right)\sqrt{t} \sup_{0<\tau<t} \left|\left.V^n(\cdot,t)-V^{n-1}(\cdot, t)\right|_{\omega^-}^{\omega^+}\right|+C_2\delta\left(\left\|\left|V^n-V^{n-1}\right|\right\|_{\infty}\right.\nonumber\\
			&\qquad +\left.\left\|\left|\frac{\sqrt{\tau}}{\left|\log \tau\right|}\left(U_x^n-U_x^{n-1}\right)\right|\right\|_{\infty}+\left\|\left|\frac{\Theta^n-\Theta^{n-1}}{\left|\log \tau\right|}\right|\right\|_{\infty}\right), \nonumber\\
			& \left\|V^{n+1}(\cdot,t)-V^n(\cdot, t)\right\|_{BV}\nonumber\\
			& \quad \leq C_2\left(\sqrt{t_{\sharp}}+\delta\right)\left(\left\|\left| V^n-V^{n-1}\right|\right\|_{\infty}+\left\|\left|V^n-V^{n-1}\right|\right\|_{B V}+\left\|\left| V^n-V^{n-1}\right|\right\|_1\right. \nonumber\\
			& \left.\qquad+\left\|\left|\frac{\Theta^n-\Theta^{n-1}}{\left|\log\tau\right|}\right|\right\|_{\infty}+\left\|\left|\Theta^n-\Theta^{n-1}\right|\right\|_1+\left\|\left|\frac{\Theta_x^n-\Theta_x^{n-1}}{\left|\log\tau\right|}\right|\right\|_{\infty}\right).\nonumber
		\end{align}
	\end{lemma}
\begin{proof}
   See \cite{WangHT2022}.  
\end{proof}

Equipped with the estimates for the time differences of $(V^n, U^n, \Theta^n, Z^n)$ in Lemmas~\ref{lemma: Theta n diff}--\ref{lemma: V n diff}, we are now ready to establish the local existence of weak solutions to Eq.~\eqref{Cauchy pde}, provided that the smallness condition~\eqref{smallness condition} for the initial datum is verified.
	\begin{theorem}\label{thm: local exi}
There is a universal constant $\delta>0$ such that the following holds. Suppose that the initial datum $\left(v_0^*,u_0^*,\theta_0^*,z_0^*\right)$ satisfies the smallness condition~\eqref{smallness condition} with $\delta$. Then there exists a positive constant $t_{\sharp}$ such that system \eqref{Cauchy pde} admits a weak solution
		$$(v,u,\theta,z)=(v^*+1, u^*, \theta^*+1, z^*),\quad t<t_{\sharp}\ll 1,$$
		with the estimates
		\begin{equation}
			\label{es local sol}
			\left\{\begin{array}{l}
				\max \left\{\left\|u(\cdot, t)\right\|_{L_x^1},\left\|u(\cdot, t)\right\|_{L_x^{\infty}},\left\|u_x(\cdot, t)\right\|_{L_x^1}, \sqrt{t}\left\|u_x(\cdot, t)\right\|_{L_x^{\infty}}\right\} \leq 2 C_{\sharp} \delta, \\
				\max \left\{\left\|\theta(\cdot, t)-1\right\|_{L_x^1},\left\|\theta(\cdot, t)-1\right\|_{L_x^{\infty}},\left\|\theta_x(\cdot, t)\right\|_{L_x^1}, \sqrt{t}\left\|\theta_x(\cdot, t)\right\|_{L_x^{\infty}}\right\} \leq 2 C_{\sharp} \delta, \\
				\max \left\{\left\|z(\cdot, t)\right\|_{L_x^1},\left\|z(\cdot, t)\right\|_{L_x^{\infty}},\left\|z_x(\cdot, t)\right\|_{L_x^1},\sqrt{t}\left\|z_x(\cdot, t)\right\|_{L_x^{\infty}}\right\} \leq 2  C_{\sharp} \delta, \\
				\max \left\{\left\|v(\cdot, t)\right\|_{B V},\left\|v(\cdot, t)-1\right\|_{L_x^1},\left\|v-1(\cdot, t)\right\|_{L_x^{\infty}}, \sqrt{t}\left\|v_t(\cdot, t)\right\|_{L_x^{\infty}}\right\} \leq 2 C_{\sharp} \delta, \\
				v^*=v_{\tilde{a}}^*+v_j^*,\quad v_j^*(x,t)=\sum\limits_{\omega<x,\omega\in \mathcal{D}} \left.v^*\right|_{\omega^-}^{\omega^{+}} h(x-\omega), \quad v_{\tilde{a}}^* \text{ is continuous},\\
				\left|\left.v(\cdot, t)\right|_{x=\omega^{-}}^{x=\omega^{+}}\right|\leq 2\left|\left. v_0^*(\cdot)\right|_{x=\omega^{-}} ^{x=\omega^{+}} \right|, \quad \omega \in \mathcal{D}.
			\end{array}\right.
		\end{equation}
		Here, $C_{\sharp}$ is a universal constant and $h(x)$ is the Heaviside step function. 
        
        The fluxes of $u$, $\theta$ and $z$, namely that
        \begin{align*}
            \bar{u}:= \frac{\mu u_x}{v}-p,\qquad \bar{\theta}:=\frac{\nu}{c_v v}\theta_x-\int_{-\infty}^{x}(\frac{p}{c_v}u_y-\frac{\mu}{c_v v}u_y^2-\frac{q}{c_v}K\phi(\theta)z)\,\mathrm{d}y,\qquad \bar{z}:=\frac{D}{v^2}z_x,
        \end{align*} 
		are all continuous with respect to $x$.
	\end{theorem} 
    
	\begin{proof}
    We essentially follows the framework presented in \cite{WangHT2022}. Consider the functional:
		\begin{align}
			\label{map F}
			\mathcal{F}&\left[V^{n+1}-V^n, U^{n+1}-U^n, \Theta^{n+1}-\Theta^n, Z^{n+1}-Z^n \right]\nonumber\\
			&:= \left\|\left|V^n-V^{n-1}\right|\right\|_{\infty}+ \left\|\left|V^n-V^{n-1}\right|\right\|_{1}+\left\|\left|V^n-V^{n-1}\right|\right\|_{BV}\nonumber\\
			&\qquad+\left\|\left|U^n-U^{n-1}\right|\right\|_{\infty}+ \left\|\left|U^n-U^{n-1}\right|\right\|_{1}+\left\|\left|\frac{\sqrt{\tau}}{\left|\log \tau\right|}\left(U_x^n-U_x^{n-1}\right)\right|\right\|_{\infty}\nonumber\\
			&\qquad +\left\|\left|\frac{U_x^n-U_x^{n-1}}{\left|\log \tau\right|}\right|\right\|_{1} +\left\|\left|\frac{\Theta^n-\Theta^{n-1}}{\left|\log \tau\right|}\right|\right\|_{\infty}+\left\|\left|\Theta^n-\Theta^{n-1}\right|\right\|_{1}\nonumber\\
			&\qquad +\left\|\left|\frac{\sqrt{\tau}}{\left|\log \tau\right|}\left(\Theta_x^n-\Theta_x^{n-1}\right)\right|\right\|_{\infty}+\left\|\left|\frac{\Theta_x^n-\Theta_x^{n-1}}{\left|\log \tau\right|}\right|\right\|_{1} +\left\|\left|Z^n-Z^{n-1}\right|\right\|_{\infty}\nonumber\\
			&\qquad +\left\|\left|Z^n-Z^{n-1}\right|\right\|_{1}+\left\|\left|\frac{\sqrt{\tau}}{\left|\log \tau\right|}\left(Z_x^n-Z_x^{n-1}\right)\right|\right\|_{\infty}+\left\|\left|\frac{Z_x^n-Z_x^{n-1}}{\left|\log \tau\right|}\right|\right\|_{1}.
		\end{align}
		From Lemmas~\ref{lemma: Theta n diff}, \ref{lemma: Z n diff}, \ref{lemma: U n diff}, and \ref{lemma: V n diff}, we deduce that
		\begin{equation}
			\label{contraction}
			\begin{aligned}
				\mathcal{F}&\left[V^{n+1}-V^n, U^{n+1}-U^n, \Theta^{n+1}-\Theta^n, Z^{n+1}-Z^n \right]\\&\leq C_2\left(\sqrt{t_{\sharp}}\left|\log t_{\sharp}\right|+\delta\right)\mathcal{F}\left[V^n-V^{n-1}, U^n-U^{n-1}, \Theta^n-\Theta^{n-1}, Z^n-Z^{n-1} \right]
			\end{aligned}
		\end{equation}
		where $C_2\simeq O(1)$ is a uniform constant independent of $n$, $\delta$, and $t_{\sharp}$. As $\sqrt{t}|\log t|\rightarrow 0$, as $t\rightarrow 0^+$, one may choose positive constants $\delta, t_{\sharp}\ll 1$  such that
		\begin{equation*}
			C_2\left(\sqrt{t_{\sharp}}\left|\log t_{\sharp}\right|+\delta\right)< 1.
		\end{equation*} 
	We may thus apply the contraction mapping theorem to deduce a strong limit $(v^*, u^*, \theta^*, z^*)$ for the sequence $\left\{\left(V^n, U^n, \Theta^n, Z^n\right)\right\}$ that satisfies
		\begin{equation}
			\label{es: w-sol}
			\begin{cases}v^*(x, t) \in L^{\infty}\left(0, t_{\sharp} ; L^1(\mathbb{R})\right), & \\ u^*(x, t) \in L^{\infty}\left(0, t_{\sharp} ; L^1(\mathbb{R}) \cap L^{\infty}(\mathbb{R})\right), & \sqrt{t} u_x^*(x, t) \in L^{\infty}\left(0, t_{\sharp} ; L^{\infty}(\mathbb{R})\right), \\ \theta^*(x, t) \in L^{\infty}\left(0, t_{\sharp} ; L^1(\mathbb{R}) \cap L^{\infty}(\mathbb{R})\right), & \sqrt{t} \theta_x^*(x, t) \in L^{\infty}\left(0, t_{\sharp} ; L^{\infty}(\mathbb{R})\right),\\
				z^*(x, t) \in L^{\infty}\left(0, t_{\sharp} ; L^1(\mathbb{R}) \cap L^{\infty}(\mathbb{R})\right), & \sqrt{t} z_x^*(x, t) \in L^{\infty}\left(0, t_{\sharp} ; L^{\infty}(\mathbb{R})\right).
			\end{cases}
		\end{equation}
		Moreover, $(v^*, u^*, \theta^*, z^*)$ is a weak solution to Eq.~\eqref{Cauchy pde}, and clearly $\left(v, u, \theta, z\right)=\left(v^*+1, u^*, \theta^*+1, z^*\right)$ is a weak solution to Eq.~\eqref{Cauchy pde}.

        Meanwhile, as $(V^n, U^n, \Theta^n, Z^n)$ converges in the topology defined by $\mathcal{F}$, which is stronger than $[L^1(\mathbb{R})]^4$, we conclude that 
		\begin{equation}
			\label{eq: phiz}
				\left\|z^*(\cdot, t)\right\|_{L_x^1}\leq C_{\sharp}\delta, \quad \int_0^t\int_{\mathbb{R}}K \phi(\theta^*) z^* (y,\tau)\,\mathrm{d}y\,\mathrm{d}\tau\leq \delta.
		\end{equation}
		By Lemma~\ref{lemma: V U Theta Z k+1}, the norms $\|V^n\|_{BV}$, $\|U_x^n\|_{L_x^1}$, $\|\Theta_x^n\|_{L_x^1}$, and $\|Z_x^n\|_{L_x^1}$ are uniformly bounded. Thus, in view of Helly's selection theorem for BV functions, we may pass to a subsequence of $\left\{\left(V^n, U^n, \Theta^n, Z^n\right)\right\}$ to deduce that 
		\begin{equation*}
			\|v^*\|_{BV}\leq 2C_{\sharp}\delta,\quad \|u_x^*\|_{L_x^1}\leq 2C_{\sharp}\delta,\quad \|\theta_x^*\|_{L_x^1}\leq 2C_{\sharp}\delta,\quad \|z_x^*\|_{L_x^1}\leq 2C_{\sharp}\delta.
		\end{equation*}

        On the other hand, $V^n\in {\rm BV}(\R)$ by construction for each $n\in\mathbb{N}$, and
		\begin{equation*}
			V^n = V_{\tilde{a}}^n+ V_j^n
		\end{equation*}
        where $V_{\tilde{a}}^n$ is continuous and
\begin{align*}
    V_j^n(x,t) = \sum_{\omega<x,\omega \in \mathcal{D}} j^n(\omega, t) h(x-\omega)
\end{align*}
with \begin{align*}
    j^n(\omega, t) = \left.V^n(\cdot, t) \right|_{\omega^-}^{\omega^+}.
\end{align*}
From the estimates~\eqref{Vk jump} in Lemma~\ref{lemma: lip V} and \eqref{Vn cauchy} in Lemma~\ref{lemma: V n diff}, the sequence of jumps $\left\{ j^n(\omega, t)\right\}$ is Cauchy. Hence, it converges strongly to a limit $j(\omega, t)$ for each $\omega \in \mathcal{D}$. Define the step function
		\begin{equation*}
			v_j^*(x,t)=\sum_{\omega<x,\omega\in\mathcal{D}}j(\omega,t)h(x-\omega) \quad \text{with} \quad |j(\omega, t)|\leq 2 \left|\left.v_0^*(\cdot)\right|_{\omega^-}^{\omega^+}\right|.
		\end{equation*}
		Since $\left\{V_j^n\right\}_n$ converges pointwise to $v_j^*$ and $\left\{V_{\tilde{a}}^n\right\}_n$ converges uniformly to a  continuous function $v_{\tilde{a}}^*$, we conclude that
		\begin{equation*}
			v^*=\lim_{n\to\infty} V^n = v_{\tilde{a}}^*+ v_j^*.
		\end{equation*}
		Therefore, $v^*$ is a BV function with the same  set  $\mathcal{D}$ of discontinuities as the initial datum. Clearly, $v = 1 + v^*$ is also BV.

		Now we turn to the continuity of the fluxes for $\theta$, $u$, and $z$.
		First, by substituting $v$, $u$, and $z$ obtained just now into the equation
		for $\theta$, we arrive at an inhomogeneous linear heat equation: 
		\begin{align}
			\label{eq: bar theta}
			\bar{\theta}_t=&\left(\frac{\nu}{c_v v}\bar{\theta}_x-\int_{-\infty}^{x}\left(\frac{p}{c_v} u_y-\frac{\mu}{c_v v} u_y^2\right)\,\mathrm{d}y+\int_{-\infty}^{x}\frac{q}{c_v}K\phi(\theta)z \,\mathrm{d}y\right)_x\nonumber\\
			=&:\left(\frac{\nu}{c_v v}\bar{\theta}_x+ g(t,x)\right)_x.
		\end{align}
        This is of the same form as in Remark~\ref{re: conti}. By Eq.~\eqref{es: w-sol} we know that
		\begin{equation*}
			\int_{\mathbb{R}}\left|\frac{p}{c_v} u_y\right|+\left|\frac{\mu}{c_v v} u_y^2\right|\,\mathrm{d}y+\int_{\mathbb{R}}\frac{q}{c_v}K\phi(\theta)z \,\mathrm{d}y<+\infty,
		\end{equation*}
		which implies that $g(t,x)$ is a BV function in $x \in \R$. Remark~\ref{re: conti} ensures the existence of the solution $\bar{\theta}$, and from \eqref{es: w-sol} we infer the Lipschitz continuity of $g$ in $x$:
		\begin{align*}
			\|g_x(t,x)\|_{L_x^{\infty}}&\leq O(1)\|\theta\|_{L_x^{\infty}} \|u_x\|_{L_x^{\infty}}+O(1)\|u_x\|_{L_x^{\infty}}^2+O(1)\|z\|_{L_x^{\infty}}\\
            &\qquad< \frac{O(1)}{t}<+\infty\qquad \text{ for any $t\in (0,t_{\sharp})$}.
		\end{align*}
        By Remark~\ref{re: conti} once again, the flux of $\bar{\theta}$ is continuous in $x$. But Eq.~\eqref{eq: bar theta} is linear in $\bar{\theta}$, so the weak solution is unique. Therefore, the solution $\theta$ constructed from the aforementioned iteration scheme coincides with $\bar{\theta}$, and hence the flux of $\theta$ is continuous. 
		
		Next, we substitute $v$ and $\theta$ into the equation of $u$ to obtain 
		\begin{equation}
			\label{eq: bar u}
			\bar{u}_t=\left(\frac{\mu \bar{u}_x}{v}-p(v,\theta)\right)_x,
		\end{equation}
		which is of the same form as in Remark \ref{re: conti}. We have already shown that $\theta$ is  the weak solution to Eq.~\eqref{eq: bar theta}, and obtained regularity  estimates for $v$ and $u$. Thus, following the proof for Eq.~\eqref{eq: diff theta k+1} in Lemma~\ref{lemma: diff Theta k+1}, we easily deduce that
		\begin{equation*}
			|\theta(x,t)-\theta(x,s)|\leq O(1)\frac{\sqrt{t-s}}{\sqrt{s}},
		\end{equation*}
        which shows that $\theta(x,t)$ is H\"{o}lder continuous in $t$. As $v_t=u_x$ is uniformly bounded, $v$ is also Lipschitz in $t$. Thus, in light of Eq.~\eqref{eq: bar u}, $p(v,\theta)$ is H\"{o}lder in $t$. Therefore, repeating the same argument as for $\theta$ above, we obtain the uniqueness of weak solution $\bar{u}$ and the continuity of the flux of $\bar{u}$. But $u$ is also a weak solution, so $u=\bar{u}$ and hence  the flux of $u$ is continuous.
		
		Finally, we consider $\bar{z}=\frac{D}{v^2}z_x$,  the flux of $z$. It satisfies
		\begin{equation}
			\label{eq: bar z}
			\bar{z}_t+K\phi(\theta)\bar{z}=\left(\frac{D \bar{z}_x}{v^2}\right)_x.
		\end{equation}
		Since the homogeneous part of \eqref{eq: bar z} is of the same form as in Remark \ref{re: conti}, it admits a unique weak solution. By Duhamel's principle, the inhomogeneous equation \eqref{eq: bar z} also admits a unique weak solution. Since $z$ is also a weak solution to \eqref{eq: bar z}, uniqueness implies $\bar{z}=z$. Moreover, the continuity of the flux of $z$ follows directly from Remark \ref{re: conti} with $g\equiv 0$.

The proof of Theorem~\ref{thm: local exi} is now complete.  \end{proof}
	
	\subsection{Improved regularity in time}

In this subsection, we show that the weak solutions admit improved regularity properties with respect to time. In particular, we shall  establish H\"{o}lder continuity for $\theta$ and $z$. (See Theorem~\ref{thm: local regularity} below for the precise formulation.)

Similar arguments as in \cite[Lemma 4.1]{WangHT2022} shall be used, but we need to overcome the additional difficulties originated from the $z$ terms in the equations.

	\begin{lemma}
		\label{lemma: z holder}
		For sufficiently small $t_{\sharp}\ll 1$, the reactant ratio function $z$ satisfies the following H\"{o}lder continuity estimates with respect to $t$, provided that $0<s,t<t_{\sharp}$:
		\begin{equation}
			\left\{\begin{array}{l}
				|z(x, t)-z(x, s)|=\left|z^*(x, t)-z^*(x, s)\right| \leq O(1) \delta \frac{(t-s)|\log (t-s)|}{s}, \\
				|z(x, t)-z(x, s)|=\left|z^*(x, t)-z^*(x, s)\right| \leq O(1) \delta\left(\sqrt{s}\sqrt{t-s}+\frac{(t-s)}{s}\right),\\
				\int_{\mathbb{R}}|z(x, t)-z(x, s)| \,\mathrm{d} x=\int_{\mathbb{R}}\left|z^*(x, t)-z^*(x, s)\right| \,\mathrm{d} x \leq O(1) \delta \frac{(t-s)|\log (t-s)|}{\sqrt{s}},\\
				\int_{\mathbb{R}}|z(x, t)-z(x, s)| d x=\int_{\mathbb{R}}\left|z^*(x, t)-z^*(x, s)\right| \,\mathrm{d} x \leq O(1) \delta\left(\sqrt{t-s}+\frac{(t-s)}{\sqrt{s}}\right).
			\end{array}\right.
		\end{equation}
	\end{lemma}
	\begin{proof}
		By Eq.~\eqref{Cauchy pde} and Duhamel's principle, we have that
		\begin{align*}
			z(x,t)-z(x,s)=&\int_{\mathbb{R}} \left(H(x,t;y,0;\frac{D}{v^2})-H(x,s;y,0;\frac{D}{v^2})\right) z_0(y)\,\mathrm{d}y \\
			&-\int_s^t\int_{\mathbb{R}} H(x,t;y,\tau;\frac{D}{v^2})K\phi(\theta) z(y,\tau)\,\mathrm{d}y\,\mathrm{d}\tau\\
			&-\int_0^s\int_{\mathbb{R}}\left[H(x,t;y,\tau;\frac{D}{v^2})-H(x,s;y,\tau;\frac{D}{v^2})\right]K\phi(\theta)z(y,\tau)\,\mathrm{d}y\,\mathrm{d}\tau \\
			=&\int_{\mathbb{R}} \int_s^t H_{\sigma}(x,\sigma;y,0;\frac{D}{v^2}) z_0(y)\,\mathrm{d}y-\int_s^t\int_{\mathbb{R}} H(x,t;y,\tau;\frac{D}{v^2})K\phi(\theta) z(y,\tau)\,\mathrm{d}y\,\mathrm{d}\tau\\
			&-\int_0^s\int_{\mathbb{R}}\int_s^t H_{\sigma}(x,\sigma;y,\tau;\frac{D}{v^2}) K\phi(\theta)z(y,\tau)\,\mathrm{d}y\,\mathrm{d}\tau \\
			=&:\mathcal{I}_1+\mathcal{I}_2+\mathcal{I}_3.
		\end{align*}

        In view of Eq.~\eqref{es local sol} in Theorem~\ref{thm: local exi}, the term $\frac{D}{v^2}$ satisfies Eq.~\eqref{f}, so Lemmas~\ref{lemma: Liu} and \ref{lemma: 2 derivative} are applicable to the function $H\left(x,t;y,\tau;\frac{D}{v^2}\right)$. Thus, in view of the estimates for $H_{\sigma}$ in Lemma~\ref{lemma: 2 derivative}, the smallness condition~\eqref{smallness condition} for the initial datum, as well as the estimates for $z$ in Theorem~\ref{thm: local exi}, we deduce that
			\begin{align*}
			&\left|z(x,t)- z(x,s)\right|\nonumber\\
			\leq& \int_{\mathbb{R}}\int_s^t \left|H_{\sigma}(x,\sigma;y,0;\frac{D}{v^2})\right|\,\mathrm{d}\sigma z_0(y)\,\mathrm{d}y +O(1)\int_s^t\int_{\mathbb{R}}\left|H(x,t;y,\tau;\frac{D}{v^2})\right| z(y,\tau)\,\mathrm{d}y\,\mathrm{d}\tau\\
			&+O(1)\int_0^s\int_{\mathbb{R}}\int_s^t \left|H_{\sigma}(x,\sigma;y,\tau;\frac{D}{v^2})\right|\,\mathrm{d}\sigma z(y,\tau)\,\mathrm{d}y\,\mathrm{d}\tau\\
			\leq &O(1) \int_{\mathbb{R}}\int_s^t \frac{e^{\frac{-(x-y)^2}{C_*\sigma}}}{\sigma^{\frac{3}{2}}}\,\mathrm{d}\sigma \,\mathrm{d}y \left\|z_0\right\|_{L_x^{\infty}}+O(1) \int_s^t\int_{\mathbb{R}}\frac{e^{\frac{-(x-y)^2}{C_*(t-\tau)}}}{\sqrt{t-\tau}}\,\mathrm{d}y\,\mathrm{d}\tau \left\|z\right\|_{L_x^{\infty}}\\
			&+O(1)\int_0^s\int_{\mathbb{R}}\int_s^t\frac{e^{\frac{-(x-y)^2}{C_*(t-\tau)}}}{(t-\tau)^{\frac{3}{2}}} \,\mathrm{d}\sigma z(y,\tau)\,\mathrm{d}y\,\mathrm{d}\tau\\
			\leq &O(1)\delta \int_s^t \frac{1}{\sigma}\,\mathrm{d}\sigma+O(1)\int_s^t \delta \,\mathrm{d}\tau+O(1) \int_0^{\frac{s}{2}} \left(\frac{1}{\sqrt{s-\tau}}-\frac{1}{\sqrt{t-\tau}}\right) \left\|z\right\|_{L_x^1}\,\mathrm{d}\tau\\
			&+O(1)\int_{\frac{s}{2}}^s \int_s^t \frac{1}{\sigma-\tau}\,\mathrm{d}\sigma \,\mathrm{d}\tau \left\|z\right\|_{L_x^{\infty}}\\
			\leq &O(1) \delta\frac{t-s}{s} +O(1)\delta (t-s)+O(1)\delta\int_{0}^{\frac{s}{2}}\frac{t-s}{\sqrt{t-\tau}(s-\tau)}\,\mathrm{d}\tau+O(1)\delta(t-s)\left|\log(t-s)\right|\\
			\leq &O(1)\delta\frac{(t-s)\left|\log(t-s)\right|}{s},
		\end{align*}
		where we use the fact that $\frac{1}{\sqrt{s-\tau}}-\frac{1}{\sqrt{t-\tau}}\leq O(1)\frac{t-s}{\sqrt{t-\tau}(s-\tau)}$. On the other hand, since
		\begin{align*}
			\int_{\frac{s}{2}}^s \int_s^t \frac{1}{\sigma-\tau}\,\mathrm{d}\sigma \,\mathrm{d}\tau&\leq \int_{\frac{s}{2}}^s\frac{1}{\sqrt{s-\tau}} \int_s^t \frac{1}{\sqrt{\sigma-\tau}}\,\mathrm{d}\sigma \,\mathrm{d}\tau
			\\
            &\leq \int_{\frac{s}{2}}^s\frac{1}{\sqrt{s-\tau}} \frac{t-s}{\sqrt{t-\tau}}\,\mathrm{d}\tau\\
            &\leq O(1)\sqrt{s}\sqrt{t-s},
		\end{align*}
		we obtain another bound without the logarithmic term, namely that
		\begin{align*}
			\left|z(x,t)- z(x,s)\right|\leq O(1)\delta\left(\frac{t-s}{s}+\sqrt{s}\sqrt{t-s}\right).
		\end{align*}

		Next, we estimate the $L^1$-norm of $z(x,t)-z(x,s)$. Define the anti-derivative 
		\begin{equation*}
	W\left(x,\sigma;y,0;\frac{D}{v^2}\right)= \begin{cases}\int_{-\infty}^y H_\sigma(x,\sigma;\omega,0;\frac{D}{v^2}) \mathrm{d} \sigma & \text {for } y<x, \\ -\int_y^{\infty} H_\sigma(x,\sigma;\omega,0;\frac{D}{v^2}) \mathrm{d}\sigma & \text {for } y \geq x.\end{cases}
		\end{equation*}
		We then have
		\begin{align*}
			\int_{\mathbb{R}}\left|\mathcal{I}_1\right|\,\mathrm{d}x &\leq  \int_s^t\int_{\mathbb{R}}\left|\int_{\mathbb{R}} \partial_y W(x,\sigma;y,0;\frac{D}{v^2})z_0(y) \,\mathrm{d}y\right|\,\mathrm{d}x\,\mathrm{d}\sigma\\
			&\leq  \int_s^t\int_{\mathbb{R}}\left|W(x,\sigma;y,0;\frac{D}{v^2})\right|\int_{\mathbb{R}}\left|z_0'(y)\right|\,\mathrm{d}x\,\mathrm{d}\sigma\\
			&\leq O(1) \left\{\int_s^t \int_{\mathbb{R}}\frac{e^{\frac{-(x-y)^2}{C_*\sigma}}}{\sigma}\,\mathrm{d}x\,\mathrm{d}\sigma \right\}\left\|z_0\right\|_{BV}\\
            &\leq O(1)\delta \int_s^t\frac{1}{\sqrt{\sigma}}\,\mathrm{d}\sigma\\
            &\leq O(1)\delta \frac{t-s}{\sqrt{t}},
		\end{align*}
		For $\mathcal{I}_2$, using the estimate for $H$ in Lemma~\ref{lemma: Liu} and for $\|z\|_{L^1_x}$ in Theorem~\ref{thm: local exi}, we deduce that
			\begin{align*}
			\int_{\mathbb{R}}\left|\mathcal{I}_2\right|\,\mathrm{d}x\leq & \int_s^t\int_{\mathbb{R}}\int_{\mathbb{R}} \frac{e^{\frac{-(x-y)^2}{C_*(t-\tau)}}}{\sqrt{t-\tau}}z(y,\tau) \,\mathrm{d}y\,\mathrm{d}x\,\mathrm{d}\tau\nonumber\\
			\leq & O(1)\int_s^t \delta\,\mathrm{d}\tau
			\leq O(1) \delta (t-s),
		\end{align*}
		Similarly, for $\mathcal{I}_3$, we bound its $L^1$-norm by
		\begin{align*}
			\int_{\mathbb{R}}\left|\mathcal{I}_3\right|\,\mathrm{d}x\leq & \int_0^s\int_s^t\int_{\mathbb{R}}\int_{\mathbb{R}} \frac{e^{\frac{-(x-y)^2}{C_*(t-\tau)}}}{(t-\tau)^{\frac{3}{2}}}z(y,\tau) \,\mathrm{d}y\,\mathrm{d}x\,\mathrm{d}\sigma\,\mathrm{d}\tau\\
			\leq & O(1)\int_0^s\int_s^t \frac{1}{\sigma-\tau}\delta \,\mathrm{d}\sigma\,\mathrm{d}\tau \\
			\leq &O(1) \delta \int_0^{\frac{s}{2}}\int_s^t\frac{1}{\sigma-\frac{s}{2}}\,\mathrm{d}\sigma\,\mathrm{d}\tau \\
			&+O(1)\delta\int_{\frac{s}{2}}^s \left(\log(t-\tau)-\log(s-\tau)\right)\,\mathrm{d}\tau\\
			\leq &O(1)\delta(t-s)+O(1)\delta(t-s)\left|\log(t-s)\right|.
		\end{align*}
		We also note that
		\begin{align*}
			\int_{\mathbb{R}}\left|\mathcal{I}_3\right|\,\mathrm{d}x\leq O(1)\delta(t-s)+O(1)\delta\sqrt{s}\sqrt{t-s}.
		\end{align*}
		Putting all these estimates together, we complete the proof.  \end{proof}
        \begin{lemma}
		\label{lemma: theta holder}
		When $t_{\sharp}\ll 1$ is sufficiently small, $\theta$ satisfies the following H\"{o}lder continuity estimates with respect to $t$:
		\begin{equation}
			\left\{\begin{array}{l}
				|\theta(x, t)-\theta(x, s)|=\left|\theta^*(x, t)-\theta^*(x, s)\right| \leq O(1) \delta \frac{(t-s)|\log (t-s)|}{s} \\
				|\theta(x, t)-\theta(x, s)|=\left|\theta^*(x, t)-\theta^*(x, s)\right| \leq O(1) \delta\left(\frac{\sqrt{t-s}}{\sqrt{s}}+\frac{(t-s)}{s}\right), \\
				\int_{\mathbb{R}}|\theta(x, t)-\theta(x, s)| d x=\int_{\mathbb{R}}\left|\theta^*(x, t)-\theta^*(x, s)\right| \,{\rm d} x \leq O(1) \delta \frac{(t-s)|\log (t-s)|}{\sqrt{s}}, \\
				\int_{\mathbb{R}}|\theta(x, t)-\theta(x, s)|\,{\rm d} x=\int_{\mathbb{R}}\left|\theta^*(x, t)-\theta^*(x, s)\right| \,{\rm d} x \leq O(1) \delta\left(\sqrt{t-s}+\frac{(t-s)}{\sqrt{s}}\right)
			\end{array}\right.
		\end{equation}
     for arbitrary  $0<s,t<t_{\sharp}$.
	\end{lemma}

\begin{proof}
    Since $\theta$ is the weak solution of system \eqref{Cauchy pde}, we obtain by Duhamel's principle that   
		\begin{align*}
			&\theta(x,t)-\theta(x,s)\\=&\int_{\mathbb{R}} \left(H(x,t;y,0;\frac{\nu}{c_v v})-H(x,s;y,0;\frac{\nu}{c_v v})\right) \theta_0(y)\,\mathrm{d}y \\
			&-\int_s^t\int_{\mathbb{R}} H(x,t;y,\tau;\frac{\nu}{c_v v})\left[-\frac{p}{c_v} u_x+\frac{\mu}{c_v v}\left(u_x\right)^2+\frac{q}{c_v}K\phi(\theta)z\right](y,\tau) \,\mathrm{d}y\,\mathrm{d}\tau\\
			&-\int_0^s\int_{\mathbb{R}}\left[H(x,t;y,\tau;\frac{\nu}{c_v v})-H(x,s;y,\tau;\frac{\nu}{c_v v})\right]\left[-\frac{p}{c_v} u_x+\frac{\mu}{c_v v}\left(u_x\right)^2+\frac{q}{c_v}K\phi(\theta)z\right](y,\tau)\,\mathrm{d}y\,\mathrm{d}\tau \\
			=&\int_{\mathbb{R}} \int_s^t H_{\sigma}(x,\sigma;y,0;\frac{\nu}{c_v v}) \theta_0(y)\,\mathrm{d}y-\int_s^t\int_{\mathbb{R}} H(x,t;y,\tau;\frac{\nu}{c_v v})\left[-\frac{p}{c_v} u_x+\frac{\mu}{c_v v}\left(u_x\right)^2\right](y,\tau)\,\mathrm{d}y\,\mathrm{d}\tau\\
			&-\int_0^s\int_{\mathbb{R}}\int_s^t H_{\sigma}(x,\sigma;y,\tau;\frac{\nu}{c_v v})\left[-\frac{p}{c_v} u_x+\frac{\mu}{c_v v}\left(u_x\right)^2\right](y,\tau)\,\mathrm{d}y\,\mathrm{d}\tau \\
			&-\int_s^t\int_{\mathbb{R}} H(x,t;y,\tau;\frac{\nu}{c_v v})\frac{q}{c_v}K\phi(\theta)z\,\mathrm{d}y\,\mathrm{d}\tau-\int_0^s\int_{\mathbb{R}}\int_s^t H_{\sigma}(x,\sigma;y,\tau;\frac{\nu}{c_v v})\frac{q}{c_v}K\phi(\theta)z(y,\tau)\,\mathrm{d}y\,\mathrm{d}\tau\\
			=&:\mathcal{I}_1+\mathcal{I}_2+\mathcal{I}_3+\mathcal{I}_4+\mathcal{I}_5.
		\end{align*}

        The bounds for $\mathcal{I}_1$, $\mathcal{I}_2$, and $\mathcal{I}_3$ in $L^\infty \cap L^1$ follow from \cite[Lemma 4.1]{WangHT2022}, and the estimates for $\mathcal{I}_4$ and $\mathcal{I}_5$ follow from Lemma~\ref{lemma: z holder}.
\end{proof}

	\begin{lemma}
		\label{lemma: t deri}
The time-derivatives 		$u_t(x,t)$, $\theta_t(x,t)$, and $z_t(x,t)$ are well-defined for a.e. $x\in\mathbb{R}$ whenever $0<t<t_{\sharp}$. Moreover, we have the estimates
		\begin{align*}
			&\left\|u_t(\cdot, t)\right\|_{L_x^{\infty}} \leq O(1) \frac{\delta}{t}, \quad\left\|u_t(\cdot, t)\right\|_{L_x^1} \leq O(1) \frac{\delta}{\sqrt{t}},\\
			&\left\|\theta_t(\cdot, t)\right\|_{L_x^{\infty}} \leq O(1) \frac{\delta}{t}, \quad\left\|\theta_t(\cdot, t)\right\|_{L_x^1} \leq O(1) \frac{\delta}{\sqrt{t}},\\
			&\left\|z_t(\cdot, t)\right\|_{L_x^{\infty}} \leq O(1) \frac{\delta}{t}, \quad\left\|z_t(\cdot, t)\right\|_{L_x^1} \leq O(1) \frac{\delta}{\sqrt{t}}.
		\end{align*}
	\end{lemma}
	\begin{proof}
The bound for $u_t$ follows from \cite[Lemma 4.2]{WangHT2022}, while the bound for $\theta_t$ follows from \cite[Lemma 4.3]{WangHT2022} together with the estimates for $z_t$. It thus remains to estimate $z_t$. 
        
For this purpose, when $t_{\sharp}\ll 1$ is sufficiently small, we have that
			\begin{align*}
			z_t(x,t)=&\int_{\mathbb{R}}H_t\left(x,t;y,0;\frac{D}{v^2}\right)z_0(y)\,\mathrm{d}y\\
			& -\int_0^{\frac{t}{2}}\int_{\mathbb{R}}H_t\left(x,t;y,\tau;\frac{D}{v^2}\right)K\phi(\theta(y,\tau))z(y,\tau)\,\mathrm{d}y\,\mathrm{d}\tau \\
			&-\int_{\frac{t}{2}}^t\int_{\mathbb{R}} H_t\left(x,t;y,\tau;\frac{D}{v^2}\right)K\phi(\theta(y,t))z(y,t)\,\mathrm{d}y\,\mathrm{d}\tau\\
			&-\int_{\frac{t}{2}}^t\int_{\mathbb{R}} H_t\left(x,t;y,\tau;\frac{D}{v^2}\right)\left[K\phi(\theta(y,\tau))z(y,\tau)-K\phi(\theta(y,t))z(y,t)\right]\,\mathrm{d}y\,\mathrm{d}\tau\\
			=&:\mathcal{I}_1+\mathcal{I}_2+\mathcal{I}_3+\mathcal{I}_4.
		\end{align*}

\smallskip
\noindent
\textbf{$L^{\infty}$-estimate for $z_t$.} For $\mathcal{I}_1, \mathcal{I}_2$, and $\mathcal{I}_3$, we apply the estimate for $H_t$ in Eqs.~\eqref{es: H_t infty} and \eqref{es: int H_t} in Lemma~\ref{lemma: 2 derivative}, as well as the estimates in Theorem~\ref{thm: local exi}, Eq.~\eqref{es local sol}, to deduce that
			\begin{align*}
			\left|\mathcal{I}_1\right|\leq &O(1)\int_{\mathbb{R}} \frac{e^{\frac{-(x-y)^2}{C_*t}}}{t^{\frac{3}{2}}}\,\mathrm{d}y\left\|z_0\right\|_{L_x^{\infty}}\leq O(1)\frac{\delta}{t},\\
			\left|\mathcal{I}_2\right|\leq &O(1)\int_0^{\frac{t}{2}}\frac{1}{(t-\tau)^{\frac{3}{2}}}\int_{\mathbb{R}} K\phi(\theta(y,\tau))z(y,\tau)\,\mathrm{d}y\,\mathrm{d}\tau\\
			\leq& O(1)\frac{1}{t^{\frac{3}{2}}}\int_0^{\frac{t}{2}}\int_{\mathbb{R}} K\phi(\theta(y,\tau))z(y,\tau)\,\mathrm{d}y\,\mathrm{d}\tau\\
			\leq &O(1)\frac{\delta}{t^{\frac{3}{2}}}\leq O(1)\frac{\delta}{t},\\
			\left|\mathcal{I}_3\right|\leq
			&\int_{\mathbb{R}}\left|\int_{\frac{t}{2}}^t H_t(x,t;y,\tau;\frac{D}{v^2})\,\mathrm{d}\tau+\delta(x-y)\right| K\phi(\theta(y,t))z(y,t)\,\mathrm{d}y\\
			&\qquad + K\phi(\theta(x,t))z(x,t)\\
			\leq& O(1)\int_{\mathbb{R}}\left(\frac{e^{\frac{-(x-y)^2}{C_*(t-\tau)}}}{\sqrt{t-\tau}}+\delta_* e^{\frac{-(x-y)^2}{C_*(t-\tau)}}\right) z(y,t)\,\mathrm{d}y+O(1)\delta\\
			\leq &O(1) \left\|z\right\|_{L_x^{\infty}}+O(1)\delta
			\leq O(1)\delta.
		\end{align*}
For $\mathcal{I}_4$, we utilise additionally the Lipschitz continuity of $\phi=\phi(\theta)$ and the H\"{o}lder continuity for $z$, $\theta$ in Lemmas~\ref{lemma: z holder} and \ref{lemma: theta holder} to infer that
		\begin{align}
			\label{z_t L infty}
			\left|\mathcal{I}_4\right|\leq &O(1)\int_{\frac{t}{2}}^t\int_{\mathbb{R}} \left|H_t(x,t;y,\tau;\frac{D}{v^2})\right|\Big[\left|z(y,\tau)-z(y,t)\right|+z(y,\tau)\left|\theta(y,\tau)-\theta(y,t)\right|\Big]\,\mathrm{d}y\,\mathrm{d}\tau\nonumber\\
			\leq &O(1)\delta\int_{\frac{t}{2}}^t\int_{\mathbb{R}}\frac{e^{\frac{-(x-y)^2}{C_*(t-\tau)}}}{(t-\tau)^\frac{3}{2}}\left(\frac{t-\tau}{\tau}+\sqrt{\tau}\sqrt{t-\tau}+\frac{\sqrt{t-\tau}}{\sqrt{\tau}}\right)\,\mathrm{d}y\,\mathrm{d}\tau\nonumber\\
			\leq &O(1)\delta \int_{\frac{t}{2}}^t\frac{1}{t-\tau}\left(\frac{t-\tau}{\tau}+\sqrt{\tau}\sqrt{t-\tau}+\frac{\sqrt{t-\tau}}{\sqrt{\tau}}\right)\,\mathrm{d}\tau\nonumber\\
			\leq &O(1)\delta(1+t).
		\end{align}
In summary, we have established that
		\begin{equation*}
			\left\|z_t(\cdot,t)\right\|_{L_x^{\infty}}\leq O(1)\frac{\delta}{t}\qquad \text{for any } 0<t<t_{\sharp}.
		\end{equation*}
        
        \smallskip
		\noindent\textbf{$L^{1}$-estimate.}
		Following a similar argument to that in Lemma~\ref{lemma: V U Theta Z k+1}, define the anti-derivative
		\begin{equation*}
			W(x,t;y,0;\frac{D}{v^2})= \begin{cases}\int_{-\infty}^y H_t(x,t;w,0; \frac{D}{v^2}) \mathrm{d}\omega & \text {for}\quad y<x, \\ -\int_y^{\infty} H_t(x,t;w,0;\frac{D}{v^2}) \mathrm{d}\omega, & \text {for} \quad y \geq x.\end{cases}
		\end{equation*}
		It then follows from the smallness assumption~\eqref{smallness condition} for the initial datum that
		\begin{align*}
\int_{\mathbb{R}}\left|\mathcal{I}_1\right|\,\mathrm{d}x &=\int_{\mathbb{R}} \left|\int_{\mathbb{R}}\partial_y W\left(x,t;y,0;\frac{D}{v^2}\right)z_0(y)\,\mathrm{d}y\right|\,\mathrm{d}x \nonumber\\
			&\quad\leq  \int_{\mathbb{R}}\left|W\left(x,t;y,0;\frac{D}{v^2}\right)\right| \int_{\mathbb{R}}\left|\mathrm{d}z_0'(y)\right|\,\mathrm{d}x\nonumber\\
			&\quad\leq O(1)\left\{\int_{\mathbb{R}}\frac{e^{\frac{-(x-y)^2}{C_*t}}}{t}\,\mathrm{d}x\right\}\left\|z_0\right\|_{BV} \\
            &\quad\leq O(1)\frac{\delta}{\sqrt{t}}.
		\end{align*}
For $\mathcal{I}_2$, $\mathcal{I}_3$, we infer from Eqs.~\eqref{es: H_t infty} and \eqref{es: int H_t} in Lemma~\ref{lemma: 2 derivative} and \eqref{es local sol} in Theorem~\ref{thm: local exi} that 
		\begin{align*}
			\int_{\mathbb{R}}\left|\mathcal{I}_2\right|\,\mathrm{d}x\leq &O(1)\int_0^{\frac{t}{2}}\int_{\mathbb{R}}\int_{\mathbb{R}} \frac{e^{\frac{-(x-y)^2}{C_*(t-\tau)}}}{(t-\tau)^{\frac{3}{2}}}z(y,\tau)\,\mathrm{d}y\,\mathrm{d}x\,\mathrm{d}\tau\\
			\leq &O(1)\int_0^{\frac{t}{2}}\frac{1}{t-\tau}\left\|z\right\|_{L_x^{1}}\,\mathrm{d}\tau\\
			\leq &O(1)\int_0^{\frac{t}{2}}\frac{1}{t-\tau}\delta \,\mathrm{d}\tau\leq O(1)\delta,
		\end{align*}
		as well as
			\begin{align*}
			\int_{\mathbb{R}}\left|\mathcal{I}_3\right|\,\mathrm{d}x\leq
			&\int_{\mathbb{R}}\int_{\mathbb{R}}\left|\int_{\frac{t}{2}}^t H_t(x,t;y,\tau;\frac{D}{v^2})\,\mathrm{d}\tau+\delta(x-y)\right| K\phi(\theta)z(y,t)\,\mathrm{d}y\,\mathrm{d}x\\
			&\qquad + \int_{\mathbb{R}}K\phi(\theta)z(x,t)\,\mathrm{d}x\\
			\leq& O(1)\int_{\mathbb{R}}\int_{\mathbb{R}}\left(\frac{e^{\frac{-(x-y)^2}{C_*(t-\tau)}}}{\sqrt{t-\tau}}+\delta_* e^{\frac{-(x-y)^2}{C_*(t-\tau)}}\right) z(y,t)\,\mathrm{d}y\,\mathrm{d}x+O(1)\left\|z\right\|_{L_x^1}\\
			\leq &O(1)\left\|z\right\|_{L_x^1}\leq O(1)\delta.
		\end{align*}
Moreover, an adaptation of the arguments in Eq.~\eqref{z_t L infty} yields that
		\begin{align*}
			\int_{\mathbb{R}}\left|\mathcal{I}_4\right|\,\mathrm{d}x\leq& O(1)\int_{\frac{t}{2}}^t\int_{\mathbb{R}}\int_{\mathbb{R}} \left|H_t(x,t;y,\tau;\frac{D}{v^2})\right|\left[\left|z(y,\tau)-z(y,t)\right|\right.\\
			&\left.\qquad+z(y,\tau)\left|\theta(y,\tau)-\theta(y,t)\right|\right]\,\mathrm{d}y\,\mathrm{d}x\,\mathrm{d}\tau\\
			\leq&O(1) \delta \int_{\frac{t}{2}}^t\frac{1}{t-\tau}(\sqrt{t-\tau}+\frac{\sqrt{t-\tau}}{\sqrt{\tau}})\,\mathrm{d}\tau\\
			\leq &O(1) \delta(1+\sqrt{t}).
		\end{align*} 
		Putting the above estimates together, we conclude that $$\left\|z_t(\cdot,t)\right\|_{L_x^{1}}\leq O(1)\frac{\delta}{\sqrt{t}}$$ for any $0<t<t_{\sharp}$.  This completes the proof.  	\end{proof}

    Now we are ready to state and prove the main result of this subsection.
	
	\begin{theorem}
		\label{thm: local regularity}
        There exists a universal constant $\delta>0$ such that the following holds. Suppose that the initial datum $\left(v_0^*, u_0^*, \theta_0^*, z_0^*\right)$ satisfies the smallness condition in Eq.~\eqref{smallness condition} with this $\delta$. Let $\left(v, u, \theta, z\right)$ be the corresponding local weak solution constructed in Theorem~\ref{thm: local exi}. Then
		\begin{enumerate}
			\item There exists a positive constant $C_{\sharp}$ such that 
			\begin{equation}
				\label{es partial t}
				\begin{aligned}
					&\max \left\{\sqrt{t}\left\|u_t(\cdot, t)\right\|_{L_x^1}, t\left\|u_t(\cdot, t)\right\|_{L_x^{\infty}}, \sqrt{t}\left\|\theta_t(\cdot, t)\right\|_{L_x^1}, t\left\|\theta_t(\cdot, t)\right\|_{L_x^{\infty}},\right.\\
					& \left.\qquad \sqrt{t}\left\|z_t(\cdot, t)\right\|_{L_x^1}, t\left\|z_t(\cdot, t)\right\|_{L_x^{\infty}} \right\} 
					\leq 2 C_{\sharp} \delta,
				\end{aligned}
			\end{equation}
            in addition to the estimates in Eq.~\eqref{es local sol}. 
			\item The fluxes of $u$, $\theta$, and $z$ (defined in Theorem~\ref{thm: local exi}) are globally Lipschitz continuous with respect to $x$, provided that $0<t<t_{\sharp}$.
			\item The specific volume $v(x,t)$ satisfies 
			\begin{equation}
				\label{v Holder}
				\left\{\begin{array}{l}
					\|v(\cdot, t)-v(\cdot, s)\|_{B V} \leq O(1) \delta \frac{(t-s)|\log (t-s)|}{\sqrt{t}}, \\
					\|v(\cdot, t)-v(\cdot, s)\|_{L_x^{\infty}} \leq O(1) \delta \frac{t-s}{\sqrt{t}}, \\
					\|v(\cdot, t)-v(\cdot, s)\|_{L_x^1} \leq O(1) \delta(t-s)
				\end{array}\right.
			\end{equation}
            for any $0\leq s<t<t_{\sharp}$.
		\end{enumerate}
	\end{theorem}
	\begin{proof}
		(1) follows directly from the Lemma~\ref{lemma: t deri}. Meanwhile, the $L_x^{\infty}$-norm of $u_t(x,t)$, $\theta_t(x,t)$, and $z_t(x,t)$ are finite whenever $0<t<t_{\sharp}$, thanks to Lemma~\ref{lemma: t deri}. This proves (2).

        For (3), from the representation formula~\eqref{eq: V k+1} for $V^{k+1}$ and the bound  $\left\|U_x^{k+1}\right\|_{L_x^1}\leq 2C_{\sharp}\delta$, we deduce that
		\begin{align*}
			\left\|V^{k+1}(x,t)-V^{k+1}(x,s)\right\|_{L_x^1}=&\left\|\int_0^t U_x^{k+1}(x,\tau)\,\mathrm{d}\tau-\int_0^s U_x^{k+1}(x,\tau)\,\mathrm{d}\tau\right\|_{L_x^1}\\
			\leq &\int_s^t \left\|U_x^{k+1}(x,\tau)\right\|_{L_x^1}\,\mathrm{d}\tau\leq O(1)\delta (t-s).
		\end{align*}
		Furthermore, from Eq.~\eqref{estimates n} and Lemma~\ref{lemma: lip V}, we obtain 
		the first two properties in Eq.~\eqref{v Holder}  for $V^{k+1}$. We thus conclude  by applying the strong convergence result in Theorem~\ref{thm: local exi}. 	\end{proof}

Recall the total specific energy $E=c_v\theta+\frac{u^2}{2}+q z$. We may deduce from the above developments the existence of local weak solutions to Eq.~\eqref{PDE,1}.
        
	\begin{corollary}
		Suppose that the initial datum $\left(v_0^*, u_0^*, \theta_0^*, z_0^*\right)$ satisfy the smallness condition~\eqref{smallness condition} for a sufficiently small $\delta$. Let $\left(v, u, \theta, z\right)$ be the corresponding weak solution to Eq.~\eqref{PDE,2} on $\R \times [0,t_\sharp)$ constructed in
		Theorems~\ref{thm: local exi} and \ref{thm: local regularity}, where $t_{\sharp}\ll 1$ is sufficiently small. Then $\left(v, u, E, z\right)$ is a weak solution to the original system~\eqref{PDE,1} with initial datum
		$$\left(v_0, u_0, E_0, z_0\right)=\left(1+v_0^*, u_0^*, c_v(1+\theta_0^*)+\frac{(u_0^*)^2}{2}+q z_0^*, z_0^*\right).$$
	\end{corollary}
	\begin{proof}
		As $(v, u, \theta, z)$
		is the weak solution constructed in Theorems~\ref{thm: local exi} and \ref{thm: local regularity}, it also satisfies Eq.~\eqref{PDE,2} in the distributional sense. Since $u_t$ is defined in the strong sense, we can take $\varphi u$ as the test function in the second identity of Eq.~\eqref{def distri}. Additionally, we may take $c_v\varphi$ as the test function in the third identity of Eq.~\eqref{def distri} and $q\varphi$ as the test function in the fourth equation of Eq.~\eqref{Cauchy pde}, respectively. Thus $E$ satisfies Eq.~\eqref{PDE,1} in the distributional sense.
	\end{proof}
	\begin{remark}
		\label{re: v regularity}
		The regularity result in Eq.~\eqref{v Holder} for $v$ (obtained from Theorem \ref{thm: local regularity}) is the same as in that in Eq.~\eqref{weak solu}. Thus, for each $0\leq t<t_{\sharp}$, the mapping  $x\mapsto v(t,x)$ is continuous on $L_x^1(\mathbb{R})\cap L_x^{\infty}(\mathbb{R})\cap BV$. 
	\end{remark}
	\subsection{Stability and Uniqueness}
	This subsection focuses on the stability of weak solution constructed in Theorems~\ref{thm: local exi} and \ref{thm: local regularity}. From this one may conclude the continuous dependence of weak solution on the initial data, as well as the uniqueness of weak solution.
	\begin{lemma}\label{lemma: smallness} 
    There exists a universal constant $0 < \delta_* \ll 1$ such that the following holds.  Suppose that the initial data satisfy
		\begin{align}
			\label{small initial}
&	\left\|v_0-1\right\|_{L_x^1}+\left\|v_0\right\|_{B V}+\left\|u_0\right\|_{L_x^1}+\left\|u_0\right\|_{B V}\nonumber\\
&\qquad + \left\|\theta_0-1\right\|_{L_x^1}+\left\|\theta_0\right\|_{B V} +\left\|z_0\right\|_{L_x^1}+\left\|z_0\right\|_{B V}<\delta_* \ll 1.
		\end{align}
		Let $(v, u, \theta, z)$ be any weak solution to Eq.~\eqref{PDE,2} as in Definition~\ref{new def, weak sol, Nov25} with the above initial data. Let $C_{\sharp}$ and $\delta$ be the parameters given in Theorems~\ref{thm: local exi} and \ref{thm: local regularity}. Then there exists a small positive constant $t_{*}$ such that for any $t \in (0,t_*)$, one has that
		\begin{equation}
		\label{smallness for any}
			\begin{aligned}
				& \max \left\{\|u(\cdot, t)\|_{L_x^1},\|u(\cdot, t)\|_{L_x^{\infty}},\left\|u_x(\cdot, t)\right\|_{L_x^1}, \sqrt{t}\left\|u_x(\cdot, t)\right\|_{L_x^{\infty}}\right\} \leq 2 C_{\sharp} \delta, \\
				& \max \left\{\|\theta(\cdot, t)\|_{L_x^1},\|\theta(\cdot, t)\|_{L_x^{\infty}},\left\|\theta_x(\cdot, t)\right\|_{L_x^1}, \sqrt{t}\left\|\theta_x(\cdot, t)\right\|_{L_x^{\infty}}\right\}\leq 2 C_{\sharp} \delta,\\
				& \max \left\{\|z(\cdot, t)\|_{L_x^1}, \|z(\cdot, t)\|_{L_x^{\infty}},\left\|z_x(\cdot, t)\right\|_{L_x^1},\sqrt{t}\left\|z_x(\cdot, t)\right\|_{L_x^{\infty}}\right\}\leq 2 C_{\sharp} \delta.
			\end{aligned}
		\end{equation}
	\end{lemma}

	\begin{proof}
	By Theorems~\ref{thm: local exi} and \ref{thm: local regularity}, under the assumption~\eqref{small initial} for the initial data,    
    Eq.~\eqref{PDE,2} has at least one weak solution. Eq.~\eqref{smallness for any} follows directly from Eqs.~\eqref{es local sol} and \eqref{es partial t}.
		
		Next, by Remark \ref{re: v regularity}, there exists a small $t_{*}$ such that
		\begin{equation}
			\begin{aligned}
				& \|v(\cdot, t)-1\|_{L_x^1} \leq C_{\sharp} \delta_*, \quad\|v(\cdot, t)-1\|_{L_x^{\infty}} \leq C_{\sharp} \delta_*, \\
				& \|v(\cdot, t)-1\|_{B V} \leq C_{\sharp} \delta_*
			\end{aligned}
		\end{equation}
whenever $t \in (0,t_*)$. 	For $\delta_{*}$ and $t_{*}$ sufficiently small, the terms $\frac{\mu}{v}$, $\frac{\nu}{c_v v}$, and $\frac{D}{v^2}$ satisfy the condition in Eq.~\eqref{f}. This allows us to construct the corresponding fundamental solution $H(x,t;y,t_0;\cdot)$. Applying Duhamel's principle and integration by parts, we arrive at the representation formulae:
		\begin{align}
			u(x,t)=&\int_{\mathbb{R}} H(x,t;y,0;\frac{\mu}{v}) u(y,0)\,\mathrm{d}y+\int_0^t\int_{\mathbb{R}} H_y(x,t;y,\tau;\frac{\mu}{v})p(y, \tau)\,\mathrm{d}y\,\mathrm{d}\tau,\\
			\theta(x,t)=&\int_{\mathbb{R}} H(x,t;y,0;\frac{\nu}{c_vv}) \theta(y,0)\,\mathrm{d}y\nonumber\\
			&+\int_0^t\int_{\mathbb{R}} H(x,t;y,\tau;\frac{\nu}{c_vv})\left(-\frac{p u_y}{b}+\frac{\mu}{c_vv}(u_y)^2+\frac{q}{c_v}K\phi(\theta)z\right)(y,\tau)\,\mathrm{d}y\,\mathrm{d}\tau,\\
			\label{z}
			z(x,t)=&\int_{\mathbb{R}} H(x,t;y,0;\frac{D}{v^2}) z(y,0)\,\mathrm{d}y+\int_0^t\int_{\mathbb{R}} H(x,t;y,\tau;\frac{D}{v^2})K\phi(\theta)z(y,\tau)\,\mathrm{d}y\,\mathrm{d}\tau.
		\end{align}
		Note that the smallness of $u$ and $\theta$ without the $z$-term can be established as in \cite[proof of Lemma~5.1]{WangHT2022}, so here we focus on proving the smallness of $z(x,t)$.
		
		Using the smallness condition on the initial data as in \eqref{small initial}, we integrate both sides of the fourth equation in \eqref{PDE,2} over $[0,t] \times \mathbb{R}$ to obtain
		\begin{align*}
			\int_{\mathbb{R}} z(x,t) \, \mathrm{d}x 
			+\int_0^t\int_{\mathbb{R}} K \phi(\theta) z(y,\tau) \, \mathrm{d}y\, \mathrm{d}\tau \leq \int_{\mathbb{R}} z_0^*(x) \, \mathrm{d}x.
		\end{align*}
	  Since $K \phi(\theta) \geq 0$, the second integral on the left-hand side is non-negative, which implies
		\begin{equation*}
			\left\| z\right\|_{L_x^1} \leq \|z_0^*\|_{L_x^1} \leq O(1)\delta_*.
		\end{equation*}
		Then, using the estimate for $|H(x,t;y,\tau;D^k)|$ in Lemma \ref{lemma: Liu} and the above $L^1$-norm of $z$, for sufficiently small $\delta$
		\begin{align}
			\|z\|_{L_x^{\infty}} 
			&\leq \int_{\mathbb{R}} \left| H(x,t;y,0;D^k) \right| |z_0^*(y)| \, \mathrm{d}y \nonumber\\
			&\quad + \int_0^t \int_{\mathbb{R}} \left| H(x,t;y,\tau;D^k) \right| K \phi(\theta) z(y,\tau) \, \mathrm{d}y \, \mathrm{d}s \nonumber\\
			&\leq O(1) \int_{\mathbb{R}} \frac{e^{-\frac{(x-y)^2}{C_* t}}}{\sqrt{t}} \, \mathrm{d}y \cdot \|z_0^*\|_{L_x^{\infty}} 
			+ O(1) \int_0^t \int_{\mathbb{R}} \frac{e^{-\frac{(x-y)^2}{C_*(t-\tau)}}}{\sqrt{t-\tau}}z(y,\tau)\, \mathrm{d}y \, \mathrm{d}\tau \nonumber\\
			&\leq O(1) \delta_* + O(1) \int_0^t \frac{1}{\sqrt{t-\tau}} \left\|z\right\|_{L_x^1}\, \mathrm{d}\tau \nonumber\\
			&\leq O(1)\delta_*+O(1)\delta_*\sqrt{t} \leq O(1)\delta_*.
		\end{align}
		Next, we differentiate \eqref{z} with respect to $x$ to get the representation of $z_x$ and then apply the smallness condition for initial data \eqref{small initial}, the estimate for $H_x(x,t;y,\tau;\frac{D}{v^2})$ in Lemma \ref{lemma: Liu}, the Lipschitz continuity of $\phi$, and the regularity in \eqref{weak solu} to derive the $L^{\infty}$-estimate for $z_x(x,t)$
		\begin{align*}
			\left|z_x(x,t)\right|\leq&\int_{\mathbb{R}}\left|H_x(x,t;y,0;\frac{D}{v^2})\right|\left|z(y,0)\right|\,\mathrm{d}y\\
			&\qquad +\int_0^t\int_{\mathbb{R}} \left|H_{x}(x,t;y,\tau;\frac{D}{v^2})\right|\left|K\phi(\theta)z(y,\tau)\right|\,\mathrm{d}y\,\mathrm{d}\tau\\
			\leq &O(1)\left\{\int_{\mathbb{R}}\frac{e^{\frac{-(x-y)^2}{C_*t}}}{t}\,\mathrm{d}y\right\}\left\|z_0\right\|_{L_x^{\infty}}+O(1)\left\{\int_0^t\int_{\mathbb{R}}\frac{e^{\frac{-(x-y)^2}{C_*(t-\tau}}}{t-\tau}\,\mathrm{d}y\,\mathrm{d}\tau\right\}\|z\|_{L_x^{\infty}}\\
			\leq&O(1)\frac{\delta_{*}}{\sqrt{t}}+O(1)\delta_*\int_0^t\frac{1}{\sqrt{t-\tau}}\,\mathrm{d}\tau \leq O(1)\delta_{*}(\frac{1}{\sqrt{t}}+\sqrt{t}).
		\end{align*}
		For the $L^1$-norm of $z_x$, similarly to Eq.~\eqref{W anti_deri}, we define anti-derivative of $H_x(x,t;y,0;\frac{D}{v^2})$:
		\begin{equation*}
			W\left(x,t;y,0;\frac{D}{v^2}\right)= \begin{cases}\int_{-\infty}^y H_x\left(x,t;w,0; \frac{D}{v^2}\right) \mathrm{d} w & \text {for} \quad y<x, \\ -\int_y^{\infty} H_x\left(x,t;w,0;\frac{D}{v^2}\right) \mathrm{d} w & \text {for}\quad y \geq x .\end{cases}
		\end{equation*}
		Using a similar argument as for $\left\|Z_x^{k+1}\right\|_{L_x^1}$ in Lemma~\ref{lemma: V U Theta Z k+1}, we find that
		\begin{align*}
			\int_{\mathbb{R}}\left|z_x(x,t)\right|\,\mathrm{d}x\leq & \int_{\mathbb{R}}\int_{\mathbb{R}}\left|H_x(x,t;y,0;\frac{D}{v^2})\right|\left|z(y,0)\right|\,\mathrm{d}y\,\mathrm{d}x\\
			&+\int_{\mathbb{R}}\int_0^t\int_{\mathbb{R}} \left|H_x(x,t;y,\tau;\frac{D}{v^2})\right|K\phi(\theta)z(y,\tau)\,\mathrm{d}y\,\mathrm{d}\tau\,\mathrm{d}x\\
			\leq & \int_{\mathbb{R}} \int_{\mathbb{R}}\left|W\left(x,t;y,0;\frac{D}{v^2}\right)\right|\left|\mathrm{d}z_0(y)\right|\,\mathrm{d}x\\
			&+O(1)\int_0^t\int_{\mathbb{R}}\int_{\mathbb{R}}\frac{e^{\frac{-(x-y)^2}{C_*(t-\tau)
			}}}{t-\tau}z(y,\tau)\,\mathrm{d}y\,\mathrm{d}x\,\mathrm{d}\tau\\
			\leq&O(1) \int_{\mathbb{R}}\frac{e^{\frac{-(x-y)^2}{C_* t
			}}}{\sqrt{t}}\,\mathrm{d}y\left\|z_0\right\|_{BV}+O(1)\int_0^t\frac{1}{\sqrt{t-\tau}}\,\mathrm{d}\tau\|z\|_{L_x^1}\\
			\leq &O(1)\delta_{*}(1+\sqrt{t}).
		\end{align*}

        Therefore, for sufficiently small $\delta_{*}$ and $t_{*}$ such that $$\delta_{*}(1+\sqrt{t_{*}})\leq \delta,$$ we conclude that
	\begin{align*}
			\max \left\{\|z(\cdot, t)\|_{L_x^1}, \|z(\cdot, t)\|_{L_x^{\infty}},\left\|z_x(\cdot, t)\right\|_{L_x^1}, \sqrt{t}\left\|z_x(\cdot, t)\right\|_{L_x^{\infty}}\right\}\leq 2 C_{\sharp}\delta
		\end{align*}
whenever $0<t<t_*$. This completes the proof.   
    \end{proof}

	\begin{lemma}
		\label{lemma: sta}
        There is a univeral positive constant $\delta_0$ such that the following holds. 	Suppose that the initial data $(v_0^{\epsilon},u_0^{\epsilon},\theta_0^{\epsilon},z_0^{\epsilon})$ and $(v_0^{\iota},u_0^{\iota},\theta_0^{\iota},z_0^{\iota})$ both satisfy the smallness condition:
		\begin{equation*}
			\left\|v_0-1\right\|_{L_x^1}+\left\|v_0\right\|_{B V}+\left\|u_0\right\|_{L_x^1}+\left\|u_0\right\|_{B V}+\left\|\theta_0-1\right\|_{L_x^1}+\left\|\theta_0\right\|_{B V} +\left\|z_0\right\|_{L_x^1}+\left\|z_0\right\|_{B V}\leq\delta_*.
		\end{equation*}
		Let $(v^{\epsilon},u^{\epsilon},\theta^{\epsilon},z^{\epsilon})$ and $(v^{\iota},u^{\iota},\theta^{\iota},z^{\iota})$ be the corresponding weak solutions to the system \eqref{Cauchy pde} as constructed in Theorem \ref{thm: local exi}. Then  there
		exists a positive constant $C_2$ such that
		\begin{equation*}
			\begin{aligned}
				&\mathcal{F} \left[v^{\epsilon}-v^{\iota}, u^{\epsilon}-u^{\iota}, \theta^{\epsilon}-\theta^{\iota}, z^{\epsilon}-z^{\iota}\right] \\
				&\qquad \leq
				C_2\left(\left\|v_0^{\epsilon}-v_0^{\iota}\right\|_{L_x^1}+\left\|v_0^{\epsilon}-v_0^{\iota}\right\|_{L_x^{\infty}}+\left\|v_0^{\epsilon}-v_0^{\iota}\right\|_{B V}+\left\|u_0^{\epsilon}-u_0^{\iota}\right\|_{L_x^{\infty}}\right.\\
				&\left.\qquad+\left\|u_0^{\epsilon}-u_0^{\iota}\right\|_{L_x^1}+\left\|\theta_0^{\epsilon}-\theta_0^{\iota}\right\|_{L_x^{\infty}}+\left\|\theta_0^{\epsilon}-\theta_0^{\iota}\right\|_{L_x^1}+\left\|z_0^{\epsilon}-z_0^{\iota}\right\|_{L_x^{\infty}}+\left\|z_0^{\epsilon}-z_0^{\iota}\right\|_{L_x^1}\right),
			\end{aligned}
		\end{equation*}
		where $\mathcal{F}$ is a functional defined in \eqref{map F}.

        In particular, we have that
		\begin{equation*}
			\begin{aligned}
				&\left\|v^{\epsilon}-v^{\iota}\right\|_{L_x^1}+\left\| u^{\epsilon}-u^{\iota}\right\|_{L_x^1}+\left\| \theta^{\epsilon}-\theta^{\iota} \right\|_{L_x^1}+\left\| z^{\epsilon}-z^{\iota}\right\|_{L_x^1}\\
				&\qquad \leq C_2\left(\left\|v_0^{\epsilon}-v_0^{\iota}\right\|_{L_x^1}+\left\|v_0^{\epsilon}-v_0^{\iota}\right\|_{L_x^{\infty}}+\left\|v_0^{\epsilon}-v_0^{\iota}\right\|_{B V}+\left\|u_0^{\epsilon}-u_0^{\iota}\right\|_{L_x^{\infty}}\right.\\
				&\left.\qquad+\left\|u_0^{\epsilon}-u_0^{\iota}\right\|_{L_x^1}+\left\|\theta_0^{\epsilon}-\theta_0^{\iota}\right\|_{L_x^{\infty}}+\left\|\theta_0^{\epsilon}-\theta_0^{\iota}\right\|_{L_x^1}+\left\|z_0^{\epsilon}-z_0^{\iota}\right\|_{L_x^{\infty}}+\left\|z_0^{\epsilon}-z_0^{\iota}\right\|_{L_x^1}\right) .
			\end{aligned}
		\end{equation*}
	\end{lemma}
	\begin{proof}
The existence and continuity of the fluxes of the weak solutions $(v^{\epsilon},u^{\epsilon},\theta^{\epsilon},z^{\epsilon})$ and $(v^{\iota},u^{\iota},\theta^{\iota},z^{\iota})$ follows from  Theorem~\ref{thm: local exi}. In addition, the integral representation formulae for both solutions can also be derived via  Duhamel's principle.

        
Following the proof of convergence for approximate solutions (Lemmas~\ref{lemma: Theta n diff}--\ref{lemma: V n diff}), we have
		\begin{align*}
\left(z_x^{\epsilon}-z_x^{\iota}\right)(x,t) &= \int_{\mathbb{R}} \left[H_x(x,t;y,0;D^{\epsilon})-H_x(x,t;y,0;D^{\iota})\right] z_0^{\epsilon}(y) \,\mathrm{d}y \\
			&\qquad+\int_{\mathbb{R}} H_x(x,t;y,0;D^{\iota})\left(z_0^{\epsilon}-z_0^{\iota}\right)(y) \,\mathrm{d}y\\
			&\qquad-\int_0^t \int_{\mathbb{R}}\left[H_x(x,t;y,\tau; D^{\epsilon})-H_x(x,t;y,\tau; D^{\iota})\right]K\phi(\theta^{\epsilon})z^{\epsilon}(y,\tau)\,\mathrm{d}y \,\mathrm{d}\tau\\
			&\qquad-\int_0^t \int_{\mathbb{R}} H_x(x,t;y,\tau; D^{\iota}) \left[K\phi(\theta^{\epsilon})z^{\epsilon}(y,\tau)-K\phi(\theta^{\iota})z^{\iota}(y,\tau)\right]\,\mathrm{d}y \,\mathrm{d}\tau
		\end{align*}
Note, however, that the set of discontinuities $\mathcal{D}_{\epsilon}$ and $\mathcal{D}_{\iota}$ may not coincide; compare with Lemma~\ref{lemma: Z n diff}. We define instead $\mathcal{D} = \mathcal{D}_{\epsilon} \cup \mathcal{D}_{\iota}$, and note that all the previous estimates remain valid on $\mathcal{D}$. 

Using the $L^{\infty}$-bound for $H_x$ from Lemma~\ref{lemma: Liu}, we estimate that
		\begin{align*}
			\left|\int_{\mathbb{R}} H_x(x,t;y,0;D^{\iota})\left(z_0^{\epsilon}-z_0^{\iota}\right)(y) \,\mathrm{d}y\right|\leq& O(1)\int_{\mathbb{R}} \left|H_x(x,t;y,0;D^{\iota})\right| \,\mathrm{d}y \left\|z_0^{\epsilon}-z_0^{\iota}\right\|_{L_x^{\infty}}\\
			\leq &O(1)\frac{1}{\sqrt{t}}\left\|z_0^{\epsilon}-z_0^{\iota}\right\|_{L_x^{\infty}}.
		\end{align*}
This together with Lemma~\ref{lemma: Z n diff} implies that, for sufficiently small $\delta$ and $t_{\sharp}$, 
		\begin{align*}
			&\frac{\sqrt{t}}{\left|\log t\right|}\left\|z_x^{\epsilon}(\cdot, t)-z_x^{\iota}(\cdot, t)\right\|_{L_x^{\infty}}\\
			& \quad \leq \frac{O(1)}{|\log t|}\left\|z_0^{\epsilon}-z_0^{\iota}\right\|_{L_x^{\infty}}+ C_2\left(\sqrt{t_{\sharp}}+\delta\right)\left(\left\|\left|v^{\epsilon}-v^{\iota}\right|\right\|_{\infty}+\frac{1}{\left|\log t\right|}\left\|\left| v^{\epsilon}-v^{\iota}\right|\right\|_{BV}\right.\\
			&\qquad\left.+\frac{1}{\left|\log t\right|}\left\|\left|v^{\epsilon}-v^{\iota}\right|\right\|_1+\left\|\left|\frac{\sqrt{\tau}}{\left|\log \tau\right|}\left(u_x^{\epsilon}-u_x^{\iota}\right)\right|\right\|_{\infty}+\left\|\left|\frac{\theta^{\epsilon}-\theta^{\iota}}{\left|\log \tau\right|}\right|\right\|_{\infty}+\frac{1}{\left|\log t\right|}\left\|\left|z^{\epsilon}-z^{\iota}\right|\right\|_{\infty}\right).
		\end{align*}
		Following the similar arguments as in Lemmas~\ref{lemma: Theta n diff}, \ref{lemma: Z n diff}, \ref{lemma: U n diff}, and \ref{lemma: V n diff}, we find that
		\begin{align*}
			&\mathcal{F}\left[v^{\epsilon}-v^{\iota}, u^{\epsilon}-u^{\iota}, \theta^{\epsilon}-\theta^{\iota}, z^{\epsilon}-z^{\iota}\right] \\
			\leq & 15 C_2 \left(\delta+\sqrt{t_{\sharp}}\left|\log t_{\sharp}\right|\right) \mathcal{F}\left[v^{\epsilon}-v^{\iota}, u^{\epsilon}-u^{\iota}, \theta^{\epsilon}-\theta^{\iota}, z^{\epsilon}-z^{\iota}\right] \\
			& +\frac{O(1)}{|\log (t)|}\left\|\theta_0^{\epsilon}-\theta_0^{\iota}\right\|_{L_x^{\infty}}+O(1)\left\|\theta_0^{\epsilon}-\theta_0^{\iota}\right\|_{L_x^1} +\frac{O(1) \sqrt{t}}{|\log (t)|}\left\|\theta_0^{\epsilon}-\theta_0^{\iota}\right\|_{L_x^{\infty}}\\
			&+\frac{O(1)}{|\log (t)|}\left\|u_0^{\epsilon}-u_0^{\iota}\right\|_{L_x^{\infty}} +O(1)\left\|u_0^{\epsilon}-u_0^{\iota}\right\|_{L_x^{\infty}}+O(1)\left\|u_0^{\epsilon}-u_0^{\iota}\right\|_{L_x^1} \\
			& +\left\|v_0^{\epsilon}-v_0^{\iota}\right\|_{L_x^1}+\left\|v_0^{\epsilon}-v_0^{\iota}\right\|_{L_x^{\infty}}+\left\|v_0^{\epsilon}-v_0^{\iota}\right\|_{BV}\\
			&+\frac{O(1)}{|\log (t)|}\left\|z_0^{\epsilon}-z_0^{\iota}\right\|_{L_x^{\infty}}+ +O(1)\left\|z_0^{\epsilon}-z_0^{\iota}\right\|_{L_x^{\infty}}+O(1)\left\|z_0^{\epsilon}-z_0^{\iota}\right\|_{L_x^1}.
		\end{align*}
		By choosing $\delta$, $t_{\sharp}$ sufficiently small such that $$15 C_2 \left(\delta+\sqrt{t_{\sharp}}\left|\log t_{\sharp}\right|\right)<1,$$  we obtain the desired results. This completes the proof.  	\end{proof}

Now, by collecting the previous results in this section, we are at the stage of concluding the main theorem on stability and uniqueness of weak solutions.
        
	\begin{theorem}
		\label{thm: sta}
	There exists a univeral positive constant $\delta_0$ such that the following holds. 		Suppose that the initial data $(v_0^{\epsilon},u_0^{\epsilon},\theta_0^{\epsilon},z_0^{\epsilon})$ and $(v_0^{\iota},u_0^{\iota},\theta_0^{\iota},z_0^{\iota})$ both satisfy the smallness conditions:
		\begin{equation*}
			\left\|v_0-1\right\|_{L_x^1}+\left\|v_0\right\|_{B V}+\left\|u_0\right\|_{L_x^1}+\left\|u_0\right\|_{B V}+\left\|\theta_0-1\right\|_{L_x^1}+\left\|\theta_0\right\|_{B V} +\left\|z_0\right\|_{L_x^1}+\left\|z_0\right\|_{B V}\leq\delta_*.
		\end{equation*}
		Let $(v^{\epsilon},u^{\epsilon},\theta^{\epsilon},z^{\epsilon})$ and $(v^{\iota},u^{\iota},\theta^{\iota},z^{\iota})$ be the corresponding weak solutions in the sense of Definition~\ref{new def, weak sol, Nov25}. Then there
		exists a positive constant $C_2$ such that 
		\begin{equation*}
			\begin{aligned}
				&\mathcal{F} \left[v^{\epsilon}-v^{\iota}, u^{\epsilon}-u^{\iota}, \theta^{\epsilon}-\theta^{\iota}, z^{\epsilon}-z^{\iota}\right] \\
				&\qquad \leq
				C_2\left(\left\|v_0^{\epsilon}-v_0^{\iota}\right\|_{L_x^1}+\left\|v_0^{\epsilon}-v_0^{\iota}\right\|_{L_x^{\infty}}+\left\|v_0^{\epsilon}-v_0^{\iota}\right\|_{B V}+\left\|u_0^{\epsilon}-u_0^{\iota}\right\|_{L_x^{\infty}}\right.\\
				&\left.\qquad+\left\|u_0^{\epsilon}-u_0^{\iota}\right\|_{L_x^1}+\left\|\theta_0^{\epsilon}-\theta_0^{\iota}\right\|_{L_x^{\infty}}+\left\|\theta_0^{\epsilon}-\theta_0^{\iota}\right\|_{L_x^1}+\left\|z_0^{\epsilon}-z_0^{\iota}\right\|_{L_x^{\infty}}+\left\|z_0^{\epsilon}-z_0^{\iota}\right\|_{L_x^1}\right),
			\end{aligned}
		\end{equation*}
		where $\mathcal{F}$ is the functional defined in Eq.~\eqref{map F}. 
        
        In particular, for $\delta_*$ sufficiently small, there exists  $t_{*}>0$ such that Eq.~\eqref{PDE,2} admits a unique weak solution in the sense of Definition~\ref{new def, weak sol, Nov25} for $t\in[0,t_{*})$.
	\end{theorem}

	\begin{proof}
The existence of weak solutions has been established in Theorems~\ref{thm: local exi} and \ref{thm: local regularity}, whenever $\delta_*<\delta$ and $t<t_{\sharp}$ with $\delta$ and $t_{\sharp}$ in Theorem~\ref{thm: local exi}. Then, by Lemma~\ref{lemma: smallness}, the weak solution remains small over $t \in (0,t_*)$. We thus deduce from Lemma~\ref{lemma: sta} the  estimate:
		\begin{equation*}
			\begin{aligned}
				&\mathcal{F} \left[v^{\epsilon}-v^{\iota}, u^{\epsilon}-u^{\iota}, \theta^{\epsilon}-\theta^{\iota}, z^{\epsilon}-z^{\iota}\right] \\
				&\qquad \leq
				C_2\left(\left\|v_0^{\epsilon}-v_0^{\iota}\right\|_{L_x^1}+\left\|v_0^{\epsilon}-v_0^{\iota}\right\|_{L_x^{\infty}}+\left\|v_0^{\epsilon}-v_0^{\iota}\right\|_{B V}+\left\|u_0^{\epsilon}-u_0^{\iota}\right\|_{L_x^{\infty}}\right.\\
				&\left.\qquad+\left\|u_0^{\epsilon}-u_0^{\iota}\right\|_{L_x^1}+\left\|\theta_0^{\epsilon}-\theta_0^{\iota}\right\|_{L_x^{\infty}}+\left\|\theta_0^{\epsilon}-\theta_0^{\iota}\right\|_{L_x^1}+\left\|z_0^{\epsilon}-z_0^{\iota}\right\|_{L_x^{\infty}}+\left\|z_0^{\epsilon}-z_0^{\iota}\right\|_{L_x^1}\right), \quad 0<t<t_*.
			\end{aligned}
		\end{equation*}
		In particular, if the two solutions have the same initial data, then
		\begin{equation}
			\mathcal{F} \left[v^{\epsilon}-v^{\iota}, u^{\epsilon}-u^{\iota}, \theta^{\epsilon}-\theta^{\iota}, z^{\epsilon}-z^{\iota}\right]=0,
		\end{equation}
		which implies the two solutions coincide almost everywhere. Moreover, since both solutions belong to the class \eqref{weak solu}, they satisfy the following continuity properties:
		\begin{itemize}
			\item $v(x,t)$ has both left and right limits at $x\in\mathbb{R}$ for $0<t<t_*$,
			\item $u(x,t)$, $\theta(x,t)$ and $z(x,t)$ are continuous on $\mathbb{R}$ for $0<t<t_*$.
		\end{itemize}
		Using these continuity properties, we conclude that
		\begin{equation*}
			(v^{\epsilon},u^{\epsilon},\theta^{\epsilon},z^{\epsilon})=(v^{\iota},u^{\iota},\theta^{\iota},z^{\iota}),\quad x\in\mathbb{R},\quad 0<t<t_*.
		\end{equation*}
        This completes the proof.  	\end{proof}

	\section{Green's Function}\label{sec: Green func}

This section is devoted to deriving the pointwise estimates for the Green's function $\mathbb{G}$ associated to the PDE of the form Eq.~\eqref{new, balance law, Dec25}, which is a  balance law of divergence form with the drift coefficient in BV. In this low regularity regime, $\mathbb{G}$ has both singular and regular parts, whose detailed estimates are crucial for the analysis of the large-time behaviour of weak solutions to Eq.~\eqref{PDE,1} in \S\ref{sec: global wp, final} below. 

The notations in the subsequent parts of the paper follows Wang--Yu--Zhang~\cite{WangHT2021}. In particular, we denote by $\mathbb{G}$, $\lambda_j$, and $\lambda_j^*$ the Green's function, the eigenvalues, and the corresponding approximated eigenvalues. The symbol $\hat{\mathbb{M}}_j$ is reserved for the modes of the Green's function at frequency $\lambda_j$; see Eq.~\eqref{G} for the precise definition. 
	
	We first introduce the variables
	\begin{align*}
		&E=e+\frac{u^2}{2}+qz,\qquad U=(v,u,E,z),\\
    &p(v,e(E,u,z))=\frac{a}{c_v v}(E-\frac{u^2}{2}-qz)\simeq \frac{E-\frac{u^2}{2}-qz}{v}.
	\end{align*}
	Note that $e_{u}=-u$, $e_{E}=1$, and $e_z=-q$. Eq.~\eqref{PDE,1} can be expressed as	\begin{equation}\label{PDE, g}
		\left\{\begin{aligned}
			&v_t-u_x=0, \\
			&u_t+p_v v_x+p_e e_u u_x+p_e e_E E_x+p_e e_z z_x=\left(\frac{\mu u_x}{v}\right)_x, \\
			&E_t+u p_v v_x+\left(p+u p_e e_u\right) u_x+u p_e e_E E_x+u p_e e_z z_x \\&\qquad\qquad \qquad =\left(\left(\frac{\mu u}{v}-\frac{\nu u}{c_vv}\right) u_x+\frac{\nu}{c_vv}E_x\right)_x +\left(\left(-\frac{q\nu}{c_vv}+\frac{qD}{v^2}\right)z_x\right)_x,\\
			&z_t+K\phi(\theta)z=\left(\frac{D}{v^2}z_x\right)_x.
		\end{aligned}\right.
	\end{equation}
	This system can also be expressed in vector form as
	\begin{equation}\label{new, balance law, Dec25}
	U_t+F(U)_x=\left(B(U) U_x\right)_x +A \, \Longleftrightarrow \, U_t+F'(U)U_x=\left(B(U) U_x\right)_x +A,
 	\end{equation}
	where
	\begin{equation}
		\begin{aligned}
			U & =\left(\begin{array}{l}
				v \\
				u \\
				E \\
				z
			\end{array}\right), \quad F(U)=\left(\begin{array}{c}
				-u \\
				p \\
				p u\\
				0
			\end{array}\right), \quad F^{\prime}(U)=\left(\begin{array}{cccc}
				0 & -1 & 0 & 0 \\
				p_v & -p_e u & p_e  &-q p_e\\
				p_v u & p-p_e u^2 & p_e u & -q p_e u\\
				0 & 0 & 0 & 0
			\end{array}\right), \\
			B(U) & =\left(\begin{array}{cccc}
				0 & 0 & 0 & 0 \\
				0 & \frac{\mu}{v} & 0 & 0 \\
				0 & \left(\frac{\mu}{v}-\frac{\nu}{c_vv}\right) u & \frac{\nu}{c_vv} & -\frac{q\nu}{c_vv}+\frac{qD}{v^2}\\
				0 & 0 & 0 & \frac{D}{v^2}
			\end{array}\right), \quad A=\left(\begin{array}{c}
				0 \\
				0 \\
				0 \\
				-K\phi(\theta)z
			\end{array}\right).
		\end{aligned}
	\end{equation}

	Consider $U=\bar{U}+V$, where $\bar{U}$ is a constant state. Then, the linearization of Eq.~\eqref{PDE, g} around $\bar{U}$ reads:
    \begin{equation}\label{eq: V g}
		V_t+F^{\prime}(\bar{U})V_x-B(\bar{U})V_{xx}=\left[N_1(V;\bar{U})+N_2(V;\bar{U})\right]_x +A,
	\end{equation}
	where $N_1$ and $N_2$ are the nonlinear terms originating from the hyperbolic and parabolic parts, respectively:
	\begin{equation}
		N_1(V;\bar{U})=-\left[F(\bar{U}+V)-F(\bar{U})-F^{\prime}(\bar{U}) V\right], \quad N_2(V;\bar{U})=\left(B(U)-B(\bar{U})\right)V_x.
	\end{equation}

\begin{definition}\label{def, G, Dec25}
We write $\mathbb{G}(x,t;\bar{U})$ for the Green's function of the linearised equation \eqref{eq: V g}:
\begin{equation}\label{1}
		\left\{\begin{array}{l}
			\partial_t \mathbb{G}(x,t;\bar{U})=\left(-F^{\prime}(\bar{U}) \partial_x+B(\bar{U})\partial_{xx}\right) \mathbb{G}(x,t;\bar{U}), \\
			\mathbb{G}(x,0;\bar{U})=\delta(x)I,
		\end{array}\right.
	\end{equation}
	where $I$ is the identity matrix, $\delta(x)$ is the Dirac delta function, and $F^{\prime}(\overline{U})$ and $ B(\overline{U})$ are the coefficient matrices after linearization around the constant state $\bar{U}$.
\end{definition}

Over the frequency domain, the Fourier-transformed Green's function $\hat{\mathbb{G}}(\eta, t;\bar{U})$ satisfies
\begin{equation}\label{green fourier}
		\left\{\begin{array}{l}
			\partial_t \hat{\mathbb{G}}(x,t;\bar{U})=\left(-i\eta F^{\prime}(\bar{U}) -\eta^2 B(\bar{U})\right) \hat{\mathbb{G}}(x,t;\bar{U}), \\
			\hat{\mathbb{G}}(\eta,0;\bar{U})=I.
		\end{array}\right.
	\end{equation}

	The eigenvalues $\lambda_j$ are obtained by solving the characteristic equation:
	\begin{align}\label{det}
		0 &= \operatorname{det}(\lambda I+i\eta F^{\prime}(\bar{U})+\eta^2 B(\bar{U}))\nonumber\\
        &=\left(\lambda +\eta^2 \frac{D}{v^2}\right)\left[\lambda\eta^2 pp_e+\left(\lambda\left(\lambda+\eta^2\frac{\mu}{v}\right)-\eta^2p_v\right)\left(\lambda+\eta^2\frac{\nu}{c_vv}\right)\right].
	\end{align}
	\begin{lemma}
		There exists a sufficiently small positive constant $\sigma_0$ such that the matrix $$-i\eta F^{\prime}(\bar{U})-\eta^2 B(\bar{U})$$ has 4 distinct eigenvalues $\lambda_1$ --  $\lambda_4$ whenever $\left|\eta\right|>0$ and $\left|\operatorname{Im} (\eta)\right|<\sigma_0$.
	\end{lemma}
	\begin{proof}
		The second factor on the right-most term of Eq.~\eqref{det} coincides with the characteristic equation in \cite[Lemma 3.1]{WangHT2021}, wherein it is shown that it has 3 distinct roots $\lambda_1, \lambda_2, \lambda_3$ provided that $\left|\eta\right|>0$ and $\left|\operatorname{Im}(\eta)\right|<\sigma_0$ for some $\sigma_0>0$ sufficiently small.  Clearly, $\lambda_4=-\frac{D}{v^2}\eta^2$ is another eigenvalue. It remains to verify that $\lambda_4$ is distinct from $\lambda_j (j=1,2,3)$.

        Substituting $\lambda_4 = -\frac{D}{v^2}\eta^2$ into the second factor of \eqref{det}, we obtain that
		\begin{equation}
			\label{lambda neq}
			\begin{aligned}
				&\lambda_4\eta^2 pp_e+\left(\lambda_4\left(\lambda_4+\eta^2\frac{\mu}{v}\right)-\eta^2p_v\right)\left(\lambda_4+\eta^2\frac{\nu}{c_vv}\right)\\
				=&-\frac{\eta^4}{c_v v^4}\left[\eta^2\frac{D(D-\mu v)(\nu v-c_vD)}{v^2}+a\theta(v\nu-c_vD-aD)\right].
			\end{aligned}
		\end{equation}
		For $\lambda_4$ to coincide with any of $\lambda_j(j=1,2,3)$, the right-most term in \eqref{lambda neq} must constantly vanish. In this case, one of the following two conditions must hold:
		\begin{itemize}
			\item $[a]$ $D=\mu v$ and $\nu v=(a+c_v)D=\gamma c_v D$;
			\item $[b]$ There exists a  special $\eta$ such that $\eta^2=-\frac{a\theta(v\nu-c_vD-aD) v^2}{D(D-\mu v)(\nu v-c_vD)}.$
		\end{itemize}
		Condition [b] depends on a specific $\eta$, but we are concerned with all $\eta$ with $|\eta|>0$ and $\left|\operatorname{Im} (\eta)\right|<\sigma_0$. Thus $[b]$ is impossible. On the other hand, it is known that the Prandtl number $$\operatorname{Pr}=\frac{\mu c_p}{\nu}$$ for ideal gas is smaller than one. Hence,
		\begin{equation}
			\label{Pr}
			\frac{\mu c_v}{\nu}=\operatorname{Pr}\frac{c_v}{c_p}=\operatorname{Pr}\frac{1}{\gamma}<\frac{1}{\gamma},
		\end{equation}
		which implies that $\nu>\mu c_v\gamma$. This rules out $[a]$. 
	\end{proof}
	\begin{lemma}\label{spectral gap}
		For any $0<r<R$, there exists a positive constant $\tilde{b}>0$ such that $$\operatorname{Re}(\lambda_j(\eta))\leq-\tilde{b}$$ for all real $\eta$ with $r<|\eta|<R$, $j=1,2,3,4$.
	\end{lemma}
	\begin{proof}
		For $j=1,2,3$, it follows from \cite[Lemma 3.2]{WangHT2021} that there exists $b>0$ such that $\operatorname{Re}(\lambda_j(\eta))\leq-b$ for all real $\eta$ with $r<|\eta|<R$. For $j=4$, since $D>0$ and $v>0$, we have $\operatorname{Re}(\lambda_4)= -\frac{D}{v^2}|\eta|^2\leq -\frac{D}{v^2}r^2$. Take $\tilde{b}=\min(b,\frac{D}{v^2}r^2)$ to conclude.
	\end{proof}

To proceed, we represent the Green's function in the frequency domain as follows:
	\begin{equation}\label{G}
		\hat{\mathbb{G}}(\eta,t;\bar{U})=\sum_{j=1}^4 e^{\lambda_j t}\hat{\mathbb{M}}_j,\quad \hat{\mathbb{M}}_j=\frac{\left.\operatorname{adj}\left(s+i\eta F^{\prime}(\bar{U})+\eta^2 B(\bar{U})\right)\right|_{s=\lambda_j}}{\prod_{k\neq j}(\lambda_j-\lambda_k)},
	\end{equation}
	where $\operatorname{adj}(A)$ is the adjugate matrix of $A$. The standard Green's function $\mathbb{G}(x,t;\bar{U})$ can be derived via the inverse Fourier transform: $$\mathbb{G}(x,t;\bar{U})=\mathcal{F}^{-1}(\hat{\mathbb{G}}(\eta,t;\bar{U})).$$

	It has been shown in \cite{LiuTP2022, WangHT2021} that $\mathbb{G}(x,t;\bar{U})$ is decomposed into singular and regular parts:
    \begin{equation}\label{decomposition into reg and sing of G, Dec25}
        \mathbb{G}(x,t;\bar{U}) = \underbrace{\mathbb{G}^{*}(x,t;\bar{U})}_{\text{regular part}} + \underbrace{\mathbb{G}^{\dagger}(x,t;\bar{U})}_{\text{singular part}}.  
    \end{equation}
We shall estimate these two parts in detail below.
	
	We first expand the eigenvalues $\lambda_j$ as Laurent polynomials in terms of $\eta^{-1}$. For sufficiently large $\eta$, we have the following asymptotic expansions (see \cite{WangHT2021}):
	\begin{align}
		\label{asy infty}
		\lambda_1= & \frac{v p_v}{\mu}-\frac{v^3\left(\nu \theta_e p_v^2+\mu p p_e p_v\right)}{\nu \mu^3 \theta_e} \eta^{-2}+\frac{v^5 p_v\left(\mu^2 p^2 p_e^2+2 \nu^2 \theta_e^2 p_v^2+\mu p p_e p_v\left(3 \nu \theta_e+\mu\right)\right)}{\nu^2 \mu^5 \theta_e^2} \eta^{-4} \nonumber\\
		& -\frac{v^7 p_v\left(\mu^3 p^3 p_e^3+3 \mu^2 p^2 p_e^2 p_v\left(2 \nu \theta_e+\mu\right)+5 \nu^3 \theta_e^3 p_v^3+\mu p p_e p_v^2\left(10 \nu^2 \theta_e^2+4 \nu \mu \theta_e+\mu^2\right)\right)}{\nu^3 \mu^7 \theta_e^3} \eta^{-6}\nonumber\\
		&+O(1) \eta^{-8},\nonumber \\
		\lambda_2= & -\frac{\eta^2 \mu}{v}+\frac{v\left(\mu p p_e+\nu \theta_e p_v-\mu p_v\right)}{\mu\left(\mu-\nu \theta_e\right)}+\eta^{-2}\left(\frac{v^3\left(\mu^2 p p_e-p_v\left(\mu-\nu \theta_e\right)^2\right)\left(\mu p p_e+p_v\left(\nu \theta_e-\mu\right)\right)}{\mu^3\left(\mu-\nu \theta_e\right)^3}\right)\nonumber \\
		& +\eta^{-4}\left(\frac{v^5\left(2 \mu^5 p^3 p_e^3+\mu^2 p^2 p_e^2 p_v\left(\nu^3 \theta_e^2-5 \nu^2 \mu \theta_e^2+10 \nu \mu^2 \theta_e-6 \mu^3\right)\right)}{\mu^5\left(\mu-\nu \theta_e\right)^5}\right) \nonumber\\
		& \left.+\frac{v^5\left(\mu p p_e p_v^2\left(\mu-\nu \theta_e\right)^2\left(3 \nu^2 \theta_e^2-8 \nu \mu \theta_e+6 \mu^2\right)-2 p_v^3\left(\mu-\nu \theta_e\right)^5\right)}{\mu^5\left(\mu-\nu \theta_e\right)^5}\right) \nonumber\\
		& +\eta^{-6}\left(\frac{v^7\left(5 \mu^7 p^4 p_e^3-\mu^3 p^3 p_e^3 p_v\left(\nu^4 \theta_e^3-7 \nu^3 \mu \theta_e^3+21 \nu^2 \mu^2 \theta_e^2-35 \nu \mu^3 \theta_e+20 \mu^4\right)\right)}{\mu^7\left(\mu-\nu \theta_e\right)^7}\right)\nonumber\\
		& +\frac{v^7\left(3 \mu^2 p^2 p_e^2 p_v^2\left(\mu-\nu \theta_e\right)^2\left(-2 \nu^3 \theta_e^3+9 \nu^2 \mu \theta_e^2-15 \nu \mu^2 \theta_e+10 \mu^3\right)\right)}{\mu^7\left(\mu-\nu \theta_e\right)^7} \nonumber\\
		& \left.-\frac{v^7\left(\mu p p_e p_v^3\left(\mu-\nu \theta_e\right)^3\left(-10 \nu^3 \theta_e^3+36 \nu^2 \mu \theta_e^2-45 \nu \mu^2 \theta_e+20 \mu^3\right)+5 p_v^4\left(\mu-\nu \theta_e\right)^7\right)}{\mu^7\left(\mu-\nu \theta_e\right)^7}\right)+O(1)\eta^{-8},\nonumber \\
		\lambda_3= & -\frac{\eta^2 \nu \theta_e}{v}+\frac{p v p_e}{\nu \theta_e-\mu}+\eta^{-2}\left(\frac{p v^3 p_e\left(\nu \theta_e\left(p p_e+p_v\right)-\mu p_v\right)}{\nu \theta_e\left(\nu \theta_e-\mu\right)^3}\right) \nonumber\\
		& +\eta^{-4}\left(\frac{p v^5 p_e\left(2 \nu^2 p^2 \theta_e^2 p_e^2+p p_e p_v\left(4 \nu^2 \theta_e^2-5 \nu \mu \theta_e+\mu^2\right)+p_v^2\left(\mu-\nu \theta_e\right)^2\right)}{\nu^2 \theta_e^2\left(\nu \theta_e-\mu\right)^5}\right) \nonumber\\
		& +\eta^{-6}\left(\frac{p v^7 p_e\left(5 \nu^3 p^3 \theta_e^3 p_e^2-p^2 p_e^2 p_v\left(-15 \nu^3 \theta_e^3+21 \nu^2 \mu \theta_e^2-7 \nu \mu^2 \theta_e+\mu^3\right)\right)}{\nu^3 \theta_e^3\left(\nu \theta_e-\mu\right)^7}\right.\nonumber \\
		& \left.-\frac{3 p p_e p_v^2\left(\mu-3 \nu \theta_e\right)\left(\mu-\nu \theta_e\right)^2}{\nu^3 \theta_e^3\left(\nu \theta_e-\mu\right)^7}-\frac{p v^7 p_e\left(p_v^3\left(\mu-\nu \theta_e\right)^3\right)}{\eta^6 \nu^3 \theta_e^3\left(\nu \theta_e-\mu\right)^7}\right)+O(1) \eta^{-8},\nonumber\\
		\lambda_4=&-\eta^2\frac{D}{v^2}.
	\end{align}

	From \cite{WangHT2021}, the real parts of the above high-frequency eigenvalues satisfy  that
	\begin{equation*}
		0>\operatorname{Re}(\lambda_1)>\operatorname{Re}(\lambda_2)>\operatorname{Re}(\lambda_3).
	\end{equation*}
    On the other hand, the Lewis number $\operatorname{Le}$ is approximately 1 for ideal gases, while the   Prandtl number $\operatorname{Pr} <1$. Here, $$\operatorname{Le}=\frac{\nu v}{c_p D},$$ and recall that $\gamma = \frac{c_p}{c_v}>1$. We thus have
	\begin{equation*}
		\frac{\nu \theta_e v}{D}=\frac{\nu v}{c_vD} =\frac{\nu v}{c_p D} \frac{c_p}{c_v}=\operatorname{Le} \gamma >1
	\end{equation*}
    as well as
	\begin{equation*}
		\frac{\mu v}{D}=\operatorname{Pr}\cdot \operatorname{Le} <1.
	\end{equation*}
	From this we conclude that 
	\begin{equation}
		\label{Re lambda dis}
		0>\operatorname{Re}(\lambda_1)>\operatorname{Re}(\lambda_2)>\operatorname{Re}(\lambda_4)>\operatorname{Re}(\lambda_3).
	\end{equation}

 We also have the following low-frequency asymptotic expansions ($|\eta|\to 0$, see \cite{WangHT2021}):
	\begin{align}
		\label{lambda zero}
		\lambda_1= & \frac{\eta^2 \nu \theta_e p_v}{v\left(p p_e-p_v\right)}-\frac{\eta^4 \nu^2 p \theta_e^2 p_e p_v\left(\mu p p_e+p_v\left(\nu \theta_e-\mu\right)\right)}{v^3\left(p_v-p p_e\right)^4}\nonumber \\
		& -\frac{\eta^6 \nu^3 p \theta_e^3 p_e p_v\left(\mu^2\left(-p^3\right) p_e^3+\mu p^2 p_e^2 p_v\left(\mu-3 \nu \theta_e\right)+p p_e p_v^2\left(-2 \nu^2 \theta_e^2+\nu \mu \theta_e+\mu^2\right)-p_v^3\left(\mu-\nu \theta_e\right)^2\right)}{v^5\left(p p_e-p_v\right)^7}\nonumber\\
		&+O(1) \eta^6,\nonumber\\
		\lambda_2= & -i \eta \sqrt{p p_e-p_v}-\frac{\eta^2\left(\nu p \theta_e p_e+\mu p p_e-\mu p_v\right)}{2 v\left(p p_e-p_v\right)}\nonumber\\
		&-\frac{i \eta^3\left(-p^2 p_e^2\left(\mu-\nu \theta_e\right)^2+2 p p_e p_v\left(2 \nu^2 \theta_e^2-\nu \mu \theta_e+\mu^2\right)-\mu^2 p_v^2\right)}{8 v^2\left(p p_e-p_v\right)^{5 / 2}}\nonumber \\
		& +\frac{\eta^4 \nu^2 p \theta_e^2 p_e p_v\left(\mu p p_e+p_v\left(\nu \theta_e-\mu\right)\right)}{2 v^3\left(p_v-p p_e\right)^4}+O(1) \eta^4,\nonumber\\
		\lambda_3= & i \eta \sqrt{p p_e-p_v}-\frac{\eta^2\left(\nu p \theta_e p_e+\mu p p_e-\mu p_v\right)}{2 v\left(p p_e-p_v\right)}\nonumber\\
		&-\frac{i \eta^3\left(p^2 p_e^2\left(\mu-\nu \theta_e\right)^2-2 p p_e p_v\left(2 \nu^2 \theta_e^2-\nu \mu \theta_e+\mu^2\right)+\mu^2 p_v^2\right)}{8 v^2\left(p p_e-p_v\right)^{5/2}},\nonumber \\
		& +\frac{\eta^4 \nu^2 p \theta_e^2 p_e p_v\left(\mu p p_e+p_v\left(\nu \theta_e-\mu\right)\right)}{2 v^3\left(p_v-p p_e\right)^4}+O(1) \eta^4,\nonumber\\
		\lambda_4=&-\eta^2\frac{D}{v^2}.
	\end{align}
	\subsection{High-Frequency Analysis ($\eta\rightarrow \infty$), Singular Part}\label{sec: high freq}
	In this part, we focus on analyzing the singular part of Green's function, which captures high-frequency behaviours and singularities of the Green's function.
	
	Due to the inverse power of $\eta$ in the asymptotic expansion ~\eqref{asy infty} as $|\eta|\to\infty$, $\lambda_j$ is not analytic at $\eta=0$. To resolve this,
	we construct analytic approximations $\lambda_j^*$ in the neighbourhood of the real axis:
	\begin{align}
		\label{appro lambda}
			& \lambda_1^*=\beta_1^*+\sum_{k=1}^3 \frac{A_{1, k}}{\left(1+\eta^2\right)^k}-\frac{K_1}{\left(\eta^2+1\right)^4}, \nonumber\\
			& \lambda_2^*=-\alpha_2^* \eta^2+\beta_2^*+\sum_{k=1}^3 \frac{A_{2, k}}{\left(1+\eta^2\right)^k}-\frac{K_2}{\left(\eta^2+1\right)^4}, \nonumber\\
			& \lambda_3^*=-\alpha_3^* \eta^2+\beta_3^*+\sum_{k=1}^3 \frac{A_{3, k}}{\left(1+\eta^2\right)^k}-\frac{K_3}{\left(\eta^2+1\right)^4},\nonumber\\
			& \lambda_4^*=-\alpha_4^* \eta^2=-\frac{D}{v^2}\eta^2,
	\end{align}
	where the coefficients $\alpha_j^*$, $\beta_j^*$, and $A_{j,k}$ are listed in Appendix~\ref{appendix, new}. Moreover,  $K_1$, $K_2$, and $K_3$ are sufficiently large positive constants ensuring the following Lemma.
	\begin{lemma}\label{lemma: appro eigen}
		Fix $\bar{U}$. We can find positive constants $K_1, K_2, K_3,\sigma_0, \sigma_0^*$ and $\sigma_1^*$ such that
		\begin{enumerate}
			\item  The approximate eigenvalues $\lambda_j^*$;  $j=1,2,3,4$, are analytic  in \begin{equation*}
				\left\{\eta:|\operatorname{Im}(\eta)|<\sigma_0\right\}.
			\end{equation*}
			\item $\lambda_j^*$ (j=1,2,3) is an approximation of $\lambda_j$  accurate up to the power $(\eta^2+1)^{-3}$:\begin{equation*}
				\lambda_j^*=\lambda_j+O(1)\left|\eta\right|^{-8}.
			\end{equation*}
			\item In the region $\left\{\eta:|\operatorname{Im}(\eta)|<\sigma_0\right\}$, all the approximated eigenvalues $\lambda_j^*$ (j=1,2,3,4) have distinct negative real part. Indeed,
			\begin{align*}
				&\sup _{|\operatorname{Im}(\eta)|<\sigma_0} \operatorname{Re}\left(\lambda_1^*\right)<-\sigma_0^*, \\
                &\sup _{|\operatorname{Im}(\eta)|<\sigma_0} \operatorname{Re}\left(\lambda_2^*+\frac{\mu}{2 v}\left(\eta^2+1\right)\right)<-\sigma_0^*, \\
				&\sup _{|\operatorname{Im}(\eta)|<\sigma_0} \operatorname{Re}\left(\lambda_3^*+\frac{\nu \theta_e}{2 v}\left(\eta^2+1\right)\right)<-\sigma_0^*,\\
                &\sup _{|\operatorname{Im}(\eta)|<\sigma_0} \operatorname{Re}\left(\lambda_4^*+\frac{D}{ 2v^2}\left(\eta^2+1\right)\right)<-\sigma_0^*,\\
				&\min_{j,k} \inf _{|\operatorname{Im}(\eta)|<\sigma_0}\left|\operatorname{Re}\left(\lambda_k^*-\lambda_j^*\right)\right|=\sigma_1^*.
			\end{align*}
		\end{enumerate}
	\end{lemma}
	\begin{proof}
		
		\begin{itemize}
    \item 
    For $\lambda_1^*,\lambda_2^*,\lambda_3^*$, it is shown in \cite[Lemma 3.3]{WangHT2021} that they satisfy (1) and (2).
			\item Since $\lambda_4^*=-\frac{D}{v^2}\eta^2$ is a polynomial in $\eta$, it is analytic on $\mathbb{C}$. 
			\item For (3), applying \eqref{Re lambda dis}, we may find sufficiently large $C_1>0$ and define 
			\begin{equation*}
				\sigma_0^* = \min\left\{\frac{1}{2}|\beta_1^*|,\, \frac{1}{2}\alpha_2^*C_1^2,\,\frac{1}{2}\frac{D}{v^2}\eta^2,\, \frac{1}{2}\alpha_3^*C_1^2\right\}
			\end{equation*}
			so that $\operatorname{Re}(\lambda_j^*)<-\sigma_0^*$.
			
			Next, for $|\eta|\leq C_1$, we know that $\operatorname{Re}(\lambda_j^*) \approx \beta_j^* - K_j$. Then, we can choose $K_1, K_2, K_3$ sufficiently large with $K_1<K_2<K_3$ to ensure $0>\beta_1^*-K_1>\beta_2^*-K_2>\beta_3^*-K_3$. For $\lambda_4$, since $\operatorname{Re}(\lambda_4) \approx -\frac{D}{v^2}|\eta|^2$, we make  additional assumptions that $K_2<\beta_2^*+\frac{D}{v^2}|\eta|^2$ and $K_3>\beta_3^*+\frac{D}{v^2}|\eta|^2$, which ensure $\beta_2^*-K_2>-\frac{D}{v^2}|\eta|^2>\beta_3^*-K_3$. Therefore, we obtain
			\begin{equation*}
				0>\operatorname{Re}(\lambda_1^*)>\operatorname{Re}(\lambda_2^*)> \operatorname{Re}(\lambda_4^*) > \operatorname{Re}(\lambda_3^*),\quad |\eta|\leq C_1.
			\end{equation*}
			
			On the other hand, for $|\eta|>C_1$, we know that $\operatorname{Re}(\lambda_2^*)\approx-\alpha_2^*|\eta|^2$, $\operatorname{Re}(\lambda_4^*)=-\frac{D}{v^2}|\eta|^2$ and $\operatorname{Re}(\lambda_3^*)\approx-\alpha_3^*|\eta|^2$ are negative. We then apply \eqref{Re lambda dis} to obtain that \begin{equation*}
				0>\operatorname{Re}(\lambda_1^*)>\operatorname{Re}(\lambda_2^*)>\operatorname{Re}(\lambda_4^*)>\operatorname{Re}(\lambda_3^*),\quad |\eta|> C_1.
			\end{equation*}
			Thus, there exist positive gaps between any two $\operatorname{Re}(\lambda_j^*)$. Set $\sigma_1^*$ to be the minimum gap
			\begin{equation*}
				\sigma_1^*= \min_{j, k} \inf_{|\operatorname{Im}(\eta)|<\sigma_0}\left|\operatorname{Re}\left(\lambda_k^*-\lambda_j^*\right)\right|
			\end{equation*}
            to conclude the proof.
		\end{itemize}
	\end{proof}
	Now, we are ready to construct the singular part of Green's function using $\lambda_j$ and $\hat{\mathbb{M}}_j$. A direct computation yields the expression of $\hat{\mathbb{M}}_j:=\hat{\mathbb{M}}_j(\eta;\lambda)$ when $\eta$ tends to infinity:
	\begin{equation}\label{adj M}
		\hat{\mathbb{M}}_j=\frac{\left(\begin{array}{cccc}
				A_{11}(\lambda_j) & A_{21}(\lambda_j) & A_{31}(\lambda_j) & A_{41}(\lambda_j) \\
				A_{12}(\lambda_j) & A_{22}(\lambda_j) & A_{32}(\lambda_j) & A_{42}(\lambda_j) \\
				A_{13}(\lambda_j) & A_{23}(\lambda_j) & A_{33}(\lambda_j) & A_{43}(\lambda_j) \\
				A_{14}(\lambda_j) & A_{24}(\lambda_j) & A_{34}(\lambda_j) & A_{44}(\lambda_j)
			\end{array}\right)}{\prod_{k \neq j}\left(\lambda_j-\lambda_k\right)} ,
	\end{equation}
	where 
	\begin{align*}
		&A_{11}(\lambda_j)=\left(\lambda_j+\eta^2\frac{D}{v^2}\right)\left[\left(\lambda_j+\eta^2\frac{\mu}{v}\right)\left(\lambda_j+\eta^2\frac{\nu}{c_vv}\right)+\eta^2 p p_e\right],\\
		&A_{21}(\lambda_j)=i\eta \left(\lambda_j+\eta^2\frac{\nu}{c_v v}\right)\left(\lambda_j+\eta^2\frac{D}{v^2}\right),\\
		&A_{31}(\lambda_j)=\eta^2 p_e\left(\lambda_j+\eta^2\frac{D}{v^2}\right),\\
		&A_{41}(\lambda_j)=-q\eta^2 p_e\left(\lambda_j+\eta^2\frac{D}{v^2}\right),\\
		&A_{12}(\lambda_j)=-i\eta p_v\left(\lambda_j+\eta^2\frac{D}{v^2}\right)\left(\lambda_j+\eta^2\frac{\nu}{c_vv}\right),\\
		&A_{22}(\lambda_j)=\lambda_j\left(\lambda_j+\eta^2\frac{D}{v^2}\right)\left(\lambda_j+i\eta p_e+\eta^2\frac{\nu}{c_v v}\right),\\
		&A_{32}(\lambda_j)=-\lambda_j\left(i\eta p_e\right)\left(\lambda_j+\eta^2\frac{D}{v^2}\right),\\
		&A_{42}(\lambda_j)=q\lambda_j\left(i\eta p_e\right)\left(\lambda_j+\eta^2\frac{D}{v^2}\right),\\
		&A_{13}(\lambda_j)=-\left(\lambda_j+\eta^2\frac{D}{v^2}\right)\eta p_v\left[\eta p+i u\left(\lambda_j+\eta^2\frac{\nu}{c_vv}\right)\right],\\
		&A_{23}(\lambda_j)=-\left(\lambda_j+\eta^2\frac{D}{v^2}\right)\left[\lambda_j\left(i\eta p+\eta^2\frac{\mu u}{v}-\eta^2\frac{\nu u}{c_vv}\right)-i\eta \lambda_j p_e u^2-\eta^2 p_v u\right],\\
		&A_{33}(\lambda_j)=\left(\lambda_j+\eta^2\frac{D}{v^2}\right)\left[\lambda_j\left(\lambda_j+\eta^2\frac{\mu}{v}\right)-i\eta \lambda_j p_e u-\eta^2 p_v \right],\\
		&A_{43}(\lambda_j)=q\left[\left(\lambda_j+i\eta p_e+\eta^2\frac{\nu}{c_vv}\right)\left(\lambda_j\left(\lambda_j+\eta^2\frac{\mu}{v}\right)-\eta^2 p_v\right)+\lambda_j\eta^2 p p_e\right]-qA_{33}(\lambda_j),\\
		&A_{14}(\lambda_j)=A_{24}(\lambda_j)=A_{34}(\lambda_j)=0,\\
		&A_{44}(\lambda_j)=\lambda\eta^2 pp_e+\left(\lambda\left(\lambda+\eta^2\frac{\mu}{v}\right)-\eta^2p_v\right)\left(\lambda+\eta^2\frac{\nu}{c_vv}\right).
	\end{align*}

	As $\hat{\mathbb{M}}_j(\eta;\lambda_j)$ may be non-analytic, we consider approximate eigenvalues $\lambda_j^*$ \eqref{appro lambda}  to obtain analytic matrix $\hat{\mathbb{M}}_j(\eta;\lambda_j^*)$, which satisfies:
	\begin{lemma}\label{lemma: M star}
		The matrix $\hat{\mathbb{M}}_j(\eta;\lambda_j^*)$ is analytic in the region $\left\{\eta:|\operatorname{Im}(\eta)|<\sigma_0\right\}$ around the real axis, and it has the following expansion at infinity:
		\begin{equation}
			\label{M star ex}
			\hat{\mathbb{M}}_j^*=M_j^{*, 0}+i \eta^{-1} M_j^{*, 1}+\eta^{-2} M_j^{*, 2}+i \eta^{-3} M_j^{*, 3}+\eta^{-4} M_j^{*, 4}+O(1) \eta^{-5}, \quad j=1,2,3,4 \quad|\eta| \rightarrow \infty,
		\end{equation}
		where $M_j^{*,k}$ are listed in Appendix A.
	\end{lemma}

	\begin{proof}
		For $j=1,2,3$, it follows from \cite[Lemma 3.4]{WangHT2021}
		that
		\begin{align*}
			\hat{\mathbb{M}}_j^*= \frac{\operatorname{adj}(s+i\eta F'+\eta^2 B)|_{s=\lambda_j^*}}{\prod_{k\neq j}(\lambda_j^*-\lambda_k^*)}
		\end{align*}
		is analytic in $|\operatorname{Im}(\eta)|<\sigma_0$ and admits the asymptotic expansion in Eq.~\eqref{M star ex}. This relies on the analyticity of the approximate eigenvalues $\lambda_j^*$ and the presence of the uniform spectral gap shown in Lemma~\ref{lemma: appro eigen}.
		
		For $j=4$, the analyticity of $\hat{\mathbb{M}}_4^*$ follows from that of $\lambda_4^*$, as well as the spectral gap $|\operatorname{Re}(\lambda_4^*-\lambda_j^*)|\geq \sigma_1^*$ established in Lemma~\ref{lemma: appro eigen}. \end{proof}

Observe that the top-left $3\times 3$-minor of the matrix $\hat{\mathbb{M}}_j^*$ has already been computed in  \cite{WangHT2021}. The other entries of $\hat{\mathbb{M}}_j^*$ ($j=1,2,3,4$) are derived via the observation \eqref{M 34} below.

	By using Lemma \ref{lemma: appro eigen}, we derive the following expansions for $j=1,2,3$:
	\begin{align*}
		&(\lambda_j^*-\lambda_1)(\lambda_j^*-\lambda_2)(\lambda_j^*-\lambda_3)\\
		&\quad=\left[\lambda_j^*-\lambda_1^*-O(1)|\eta|^{-8}\right]\left[\lambda_j^*-\lambda_2^*-O(1)|\eta|^{-8}\right]\left[\lambda_j^*-\lambda_3^*-O(1)|\eta|^{-8}\right]\\
		&\quad =(\lambda_j^*-\lambda_1^*)(\lambda_j^*-\lambda_2^*)(\lambda_j^*-\lambda_3^*)+O(1)|\eta|^{-8}(\lambda_j^*-\lambda_1^*)(\lambda_j^*-\lambda_2^*)\\
		&\qquad +O(1)|\eta|^{-8}(\lambda_j^*-\lambda_1^*)(\lambda_j^*-\lambda_3^*)+O(1)|\eta|^{-16}(\lambda_j^*-\lambda_1^*)\\
		&\qquad +O(1)|\eta|^{-8}(\lambda_j^*-\lambda_2^*)(\lambda_j^*-\lambda_3^*)+O(1)|\eta|^{-16}(\lambda_j^*-\lambda_2^*)\\
		&\qquad +O(1)|\eta|^{-16}(\lambda_j^*-\lambda_3^*)+O(1)|\eta|^{-24}.		
	\end{align*}
	By Lemma \ref{lemma: appro eigen}, $\lambda_j^*=\lambda_j+O(1)\left|\eta\right|^{-8}$ and $\min\limits_{j, k} \inf \limits_{|\operatorname{Im}(\eta)|<\sigma_0}\left|\operatorname{Re}\left(\lambda_k^*-\lambda_j^*\right)\right|=\sigma_1^*$ for a positive number $\sigma_1^*$. These imply that the terms of order $|\eta|^{-16}$ or higher decay much faster than those of order $|\eta|^{-8}$. Thus, 
	\begin{equation*}
		(\lambda_j^*-\lambda_1)(\lambda_j^*-\lambda_2)(\lambda_j^*-\lambda_3)\simeq (\lambda_j^*-\lambda_1^*)(\lambda_j^*-\lambda_2^*)(\lambda_j^*-\lambda_3^*)+O(1)|\eta|^{-8}.
	\end{equation*}
	Moreover, direct computation yields that, for $j=1,2,3$,
	\begin{equation*}
		\begin{aligned}
			\frac{A_{43}(\lambda_j^*)}{\prod\limits_{k\neq j}\left(\lambda_j^*-\lambda_k^*\right)}&=q\frac{\operatorname{det}(\lambda_j^* I+i\eta F^{\prime}(\bar{U})+\eta^2 B(\bar{U}))}{\left(\lambda_j^*-\lambda_4\right)\prod\limits_{k\neq j}\left(\lambda_j^*-\lambda_k^*\right)}-q\frac{A_{33}(\lambda_j^*)}{\prod\limits_{k\neq j}\left(\lambda_j^*-\lambda_k^*\right)}\\
			&=O(1)|\eta|^{-8}-q\frac{A_{33}(\lambda_j^*)}{\prod\limits_{k\neq j}\left(\lambda_j^*-\lambda_k^*\right)}
		\end{aligned}
	\end{equation*}
	and 
	\begin{equation*}
		\frac{A_{44}(\lambda_j^*)}{\prod\limits_{k\neq j}\left(\lambda_j^*-\lambda_k^*\right)}=\frac{\operatorname{det}(\lambda_j^* I+i\eta F^{\prime}(\bar{U})+\eta^2 B(\bar{U}))}{\left(\lambda_j^*-\lambda_4\right)\prod\limits_{k\neq j}\left(\lambda_j^*-\lambda_k\right)}=O(1)|\eta|^{-8}.
	\end{equation*}
    This term vanishes as  $\eta \rightarrow\infty$,  which result in two approximations:
	\begin{equation*}
		\frac{A_{43}(\lambda_j^*)}{\prod\limits_{k\neq j}(\lambda_j^*-\lambda_k^*)}\simeq -q\frac{A_{33}(\lambda_j^*)}{\prod\limits_{k\neq j}\left(\lambda_j^*-\lambda_k^*\right)},\qquad \frac{A_{44}(\lambda_j^*)}{\prod\limits_{k\neq j}\left(\lambda_j^*-\lambda_k^*\right)}\simeq 0\qquad (j=1,2,3).
	\end{equation*}
	As a consequence,
	\begin{equation}
		\label{M 34}
		\left(\hat{\mathbb{M}}_j^*\right)_{k4}=-q\left(\hat{\mathbb{M}}_j^*\right)_{k3}, \qquad 	\left(\hat{\mathbb{M}}_j^*\right)_{44}=0 \qquad (j,k \in \{1,2,3\}).
	\end{equation}
	Similar arguments also lead to
	\begin{equation}
		\hat{\mathbb{M}}_4^*=\hat{\mathbb{M}}_4=\left(\begin{array}{cccc}
			0 & 0 & 0 & 0 \\
			0 & 0 & 0 & 0 \\ 
			0 & 0 & 0 & q \\ 
			0 & 0 & 0 & 1
		\end{array}\right).
	\end{equation}
	For simplicity, we take $M_4^{*,k}=0$ for $k=1,2,3,4$.

    Now, substituting the previous formulae for $\lambda_j^*$ and $\hat{\mathbb{M}}_j^*$ into the Green's function \eqref{G}, we obtain the expression for the singular part of Green's function as follows: 	\begin{equation}\label{G star}
		\hat{\mathbb{G}}^*(\eta,t;\bar{U})=\sum_{j=1}^4 \hat{\mathbb{G}}^{*,j}(\eta,t;\bar{U}),\quad  \hat{\mathbb{G}}^{*,j}= e^{\lambda_j^* t}\hat{\mathbb{M}}_j^*,\quad j=1,2,3,4.
	\end{equation}
	Now, we give the point-wise estimates for $\hat{\mathbb{G}}^{*,4}$. According to Lemma \ref{lemma: appro eigen}, we know that the real part of $\lambda_4^*$ has a negative upper bound. For $t\geq 1$, we have
	\begin{align*}
		e^{\lambda_4^* t+\sigma_0^* t}&=e^{\left(-\eta^2\frac{D}{v^2}+\frac{D}{ 2v^2}(\eta^2+1)\right)t+\sigma_0^* t}\, e^{-\frac{D}{ 2v^2}(\eta^2+1)t}\\
		&=O(1)e^{-\frac{D}{ 2v^2}(\eta^2+1)t}=O(1)\frac{e^{-t/C}}{(1+\eta^2)^m}, \quad t\geq 1, \quad m\in \mathbb{Z}^{+},
	\end{align*}
	which is analytic in the region $\left\{\eta:|\operatorname{Im}(\eta)|<\sigma_0\right\}$. Therefore, we apply the Lemma \ref{lemma f} to obtain that
	\begin{equation*}
		e^{\lambda_4^* t}=e^{-\sigma_0^* t-\sigma_0|x|}, \quad t\geq 1.
	\end{equation*}
	On the other hand, for $0<t<1$, it is clear that $e^{\lambda_4^* t}=e^{-\eta^2\frac{D}{v^2} t}$. Similar to the argument in \cite[Lemma 3.7]{WangHT2021}, we introduce the symbol $\Lambda_4^*(x,t)=\mathcal{F}^{-1}\left[e^{\lambda_4^* t}\right]$. Then 
	\begin{align*}
		\Lambda_4^*(x,t)=\begin{cases}
			O(1)e^{-\sigma_0^* t-\sigma_0|x|}, \quad & t\geq 1,\\
			\frac{1}{\sqrt{4\pi \alpha_4^* t}}e^{-\frac{x^2}{4\alpha_4^* t}},\quad & 0<t<1.
		\end{cases}
	\end{align*}
	Hence, we have 
	\begin{align*}
		\mathbb{G}^{*,4}(x,t)=&\Lambda_4^*(x,t) \star_x \mathbb{M}_4^*(x)=\Lambda_4^*(x,t) \star_x \left(\delta(x) \mathbb{M}_4^{*,0}\right)\\
		=&\begin{cases}
			O(1)e^{-\sigma_0^* t-\sigma_0|x|}, \quad & t\geq 1,\\
			\left[\frac{1}{\sqrt{4\pi \alpha_4^* t}}e^{-\frac{x^2}{4\alpha_4^* t}}\right]M_4^{*,0},\quad & 0<t<1.
		\end{cases}
	\end{align*}
The symbol $\star_x$ denotes the convolution in the $x$-variable.

	Finally, we combine the bounds for $\mathbb{G}^{*,j}(x,t) (j=1,2,3)$ given in \cite[Theorem 3.1]{WangHT2021} and the above estimates for $\mathbb{G}^{*,4}(x,t)$ to arrive at the following estimates of $\mathbb{G}^*(x,t;\bar{U})$.
	\begin{theorem}
		\label{thm: es for singular G}
		The singular part of the Green's function defined in Eq.~\eqref{G star} satisfies the following bounds.
        
        For $\mathbb{G}^{*,1}$, whenever $t>0$, 
		\begin{equation*}
			\left\{\begin{array}{l}
				\mathbb{G}^{*, 1}(x, t)=e^{\frac{\nu p_\nu}{\mu} t} \delta(x) M_1^{*, 0}+O(1) e^{-\sigma_0^* t-\sigma_0|x|} \\
				\partial_x \mathbb{G}^{*, 1}(x, t)=e^{\frac{\nu p_\nu}{\mu} t}\left(\frac{d}{d x} \delta(x) M_1^{*, 0}-\delta(x) M_1^{*, 1}\right)+O(1) e^{-\sigma_0^* t-\sigma_0|x|} \\
				\partial_x^2 \mathbb{G}^{*, 1}(x, t)=e^{\frac{\nu p_\nu}{\mu} t}\left(\frac{d^2}{d x^2} \delta(x) M_1^{*, 0}-\frac{d}{d x} \delta(x) M_1^{*, 1}+\delta(x)\left(M_1^{*, 2}-A_{1,1} t M_1^{*, 0}\right)\right)+O(1) e^{-\sigma_0^* t-\sigma_0|x|}.
			\end{array}\right.
		\end{equation*}

		For $\mathbb{G}^{*,2}(x,t)$, $\mathbb{G}^{*,3}(x,t)$, and $\mathbb{G}^{*,4}(x,t)$, if $t\geq 1$ then
		\begin{equation*}
			\partial_x^k \mathbb{G}^{*, j}(x, t)=O(1) e^{-\sigma_0^* t-\sigma_0|x|}, \quad k=0,1,2,3; \quad j=2,3,4.
		\end{equation*}
		If $0<t<1$, then 
		\begin{equation*}
			\left\{\begin{aligned}
				\mathbb{G}^{*, j}(x, t)= & O(1) e^{-\sigma_0^* t-\sigma_0|x|}+\frac{e^{\beta_j^* t}}{\sqrt{4 \pi \alpha_j^* t}} e^{-\frac{x^2}{4 \alpha_j^* t}} M_j^{*, 0}, \\
				\partial_x \mathbb{G}^{*, j}(x, t)= & O(1) e^{-\sigma_0^* t-\sigma_0|x|}+\partial_x\left[\frac{e^{\beta_j^* t}}{\sqrt{4 \pi \alpha_j^* t}} e^{-\frac{x^2}{4 \alpha_j^* t}}\right] M_j^{*, 0}-\frac{e^{\beta_j^* t}}{\sqrt{4 \pi \alpha_j^* t}} e^{-\frac{x^2}{4 \alpha_j^* t}} M_j^{*, 1}, \\
				\partial_x^2 \mathbb{G}^{*, j}(x, t)= & O(1) e^{-\sigma_0^* t-\sigma_0|x|}+\partial_x^2\left[\frac{e^{\beta_j^* t}}{\sqrt{4 \pi \alpha_j^* t}} e^{-\frac{x^2}{4 \alpha_j^* t}}\right] M_j^{*, 0}-\partial_x\left[\frac{e^{\beta_j^* t}}{\sqrt{4 \pi \alpha_j^* t}} e^{-\frac{x^2}{4 \alpha_j^* t}}\right] M_j^{*, 1}, \\
				& +\frac{e^{\beta_j^* t}}{\sqrt{4 \pi \alpha_j^* t}} e^{-\frac{x^2}{4 \alpha_j^* t}} M_j^{*, 2}, \\
				\partial_x^3 \mathbb{G}^{*, j}(x, t)= & O(1) e^{-\sigma_0^* t-\sigma_0|x|}+\partial_x^3\left[\frac{e^{\beta_j^* t}}{\sqrt{4 \pi \alpha_j^* t}} e^{-\frac{x^2}{4 \alpha_j^* t}}\right] M_j^{*,0}-\partial_x^2\left[\frac{e^{\beta_j^* t}}{\sqrt{4 \pi \alpha_j^* t}} e^{-\frac{x^2}{4 \alpha_j^* t}}\right] M_j^{*, 1} \\
				& +\partial_x\left[\frac{e^{\beta_j^* t}}{\sqrt{4 \pi \alpha_j^* t}} e^{-\frac{x^2}{4 \alpha_j^* t}}\right] M_j^{*, 2}+\frac{e^{\beta_j^* t}}{\sqrt{4 \pi \alpha_j^* t}} e^{-\frac{x^2}{4 \alpha_j^* t}} M_j^{*, 3},\\
				\partial_x^k \mathbb{G}^{*,4}(x,t)=& \partial_x^k\left[\frac{1}{\sqrt{4\pi\alpha_4^* t}} e^{-\frac{x^2}{4 \alpha_4^* t}}\right] M_4^{*,0},\quad k=0,1,2,3.
			\end{aligned}\right.
		\end{equation*}
	\end{theorem}

The same approach can be directly adapted to bound the time derivatives.
	Note that
	\begin{align*}
		&\mathcal{F}\left[\partial_t \mathbb{G}^{*,4}(x,t)\right]=\lambda_4^* e^{\lambda_4^* t} \hat{\mathbb{M}}_4^*=-\eta^2\frac{D}{v^2} e^{\lambda_4^* t} \hat{\mathbb{M}}_4^*,\\
		&\mathcal{F}\left[\partial_{xt} \mathbb{G}^{*,4}(x,t)\right]=i\eta \lambda_4^* e^{\lambda_4^* t} \hat{\mathbb{M}}_4^*= -i\eta^3\frac{D}{v^2} e^{\lambda_4^* t} \hat{\mathbb{M}}_4^*.
	\end{align*}
	Taking the inverse Fourier transform, we deduce that 
	\begin{align*}
		\partial_t \mathbb{G}^{*,4}(x,t)=& \frac{D}{v^2} \partial_x^2 \left[\frac{1}{\sqrt{4\pi \alpha_4^*t}}e^{-\frac{x^2}{4\alpha_4^* t}}\right] \star_x \left(\delta(x)\mathbb{M}_4^{*,0}\right)\\
		=& \alpha_4^*\partial_x^2 \left[\frac{1}{\sqrt{4\pi \alpha_4^*t}}e^{-\frac{x^2}{4\alpha_4^* t}}\right]\mathbb{M}_4^{*,0},\\
		\partial_{xt} \mathbb{G}^{*,4}(x,t)=& \frac{D}{v^2} \partial_x^3 \left[\frac{1}{\sqrt{4\pi \alpha_4^*t}}e^{-\frac{x^2}{4\alpha_4^* t}}\right] \star_x \left(\delta(x)\mathbb{M}_4^{*,0}\right)\\
		=& \alpha_4^*\partial_x^3 \left[\frac{1}{\sqrt{4\pi \alpha_4^*t}}e^{-\frac{x^2}{4\alpha_4^* t}}\right]\mathbb{M}_4^{*,0}
	\end{align*}
    for $0<t<1$. Together with \cite[Theorem 3.2]{WangHT2021}, it yields the following.
	\begin{theorem}
	The time-derivatives of the singular part of  Green's function satisfy that
		\begin{align*}
				&\partial_t \mathbb{G}^{*, 1}(x, t)=\frac{v p_v}{\mu} e^{\frac{v p_v}{\mu} t} \delta(x) M_1^{*, 0}+O(1) e^{-\sigma_0^* t-\sigma_0|x|}, \\
				&\partial_{x t} \mathbb{G}^{*, 1}(x, t)=\frac{v p_v}{\mu} e^{\frac{v p_v}{\mu} t} \frac{d}{d x} \delta(x) M_1^{*, 0}-\frac{v p_v}{\mu} e^{\frac{v p_w}{\mu} t} \delta(x) M_1^{*, 1}+O(1) e^{-\sigma_0^* t-\sigma_0|x|}, \\
				&\partial_t \mathbb{G}^{*, j}(x, t)=O(1) e^{-\sigma_0^* t-\sigma_0|x|}+\begin{cases}
					\alpha_j^* \partial_x^2\left[\frac{e^{\beta_j^* t}}{\sqrt{4 \pi \alpha_j^* t}} e^{-\frac{x^2}{4 \alpha_j^* t}}\right] M_j^{*, 0}-\alpha_j^* \partial_x\left[\frac{e^{\beta_j^* t}}{\sqrt{4 \pi \alpha_j^* t}} e^{-\frac{x^2}{4 \alpha_j^* t}}\right] M_j^{*, 1} \\
					+\frac{e^{\beta_j^* t}}{\sqrt{4 \pi \alpha_j^* t}} e^{-\frac{x^2}{4 \alpha_j^* t}}\left(\alpha_j^* M_j^{*, 2}+\beta_j^* M_j^{*, 0}\right), \quad 0<t<1,\\
					0, \quad t\geq 1.
				\end{cases}\\
				&\partial_{x t} \mathbb{G}^{*, j}(x, t)=O(1) e^{-\sigma_0^* t-\sigma_0|x|}+ \begin{cases}\alpha_j^* \partial_x^3\left[\frac{e^{\beta_j^* t}}{\sqrt{4 \pi \alpha_j^* t}} e^{-\frac{x^2}{4 \alpha_j^* t}}\right] M_j^{*, 0}-\alpha_j^* \partial_x^2\left[\frac{e^{\beta_j^* t}}{\sqrt{4 \pi \alpha_j^* t}} e^{-\frac{x^2}{4 \alpha_j^* t}}\right] M_j^{*, 1} \\
					+\partial_x\left[\frac{e^{\beta_j^* t}}{\sqrt{4 \pi \alpha_j^* t}} e^{-\frac{x^2}{4 \alpha_j^* t}}\right]\left(\alpha_j^* M_j^{*, 2}+\beta_j^* M_j^{*, 0}\right) \\
					+\frac{e^{\beta_j^* t}}{\sqrt{4 \pi \alpha_j^* t}} e^{-\frac{x^2}{4 \alpha_j^* t}}\left(\alpha_j^* M_j^{*, 3}-\beta_j^* M_j^{*, 1}\right), \quad 0<t<1,\\
					0,\quad t\geq 1.\end{cases}
				\\
				&\partial_t \mathbb{G}^{*,4}(x,t)=\begin{cases}
					\alpha_4^* \partial_x^2\left[\frac{1}{\sqrt{4 \pi \alpha_4^* t}} e^{-\frac{x^2}{4 \alpha_4^* t}}\right] M_4^{*, 0}, \quad 0<t<1,\\
					0, \quad t\geq 1.
				\end{cases}\\
				&\partial_{x t} \mathbb{G}^{*,4}(x, t)=\begin{cases}\alpha_4^* \partial_x^3\left[\frac{1}{\sqrt{4 \pi \alpha_4^* t}} e^{-\frac{x^2}{4 \alpha_4^* t}}\right] M_4^{*, 0}, \quad 0<t<1,\\
					0,\quad t\geq 1.\end{cases}
		\end{align*}
        In the above, $j \in \{2,3\}.$
	\end{theorem}
	\subsection{Low-Frequency Analysis ($\eta\rightarrow 0$), Regular Part}
	This subsection is devoted to the pointwise estimates for the regular part of Green's function, $\mathbb{G}^{\dagger}=\mathbb{G}-\mathbb{G}^*$.

	In view of Eqs.~\eqref{G}, \eqref{G star} and the fact that $\lambda_4=\lambda_4^*$, the Fourier transform of the regular part $\mathbb{G}^{\dagger}$ can be expressed as
	\begin{equation}\label{G dagger}
		\hat{\mathbb{G}}^{\dagger}(\eta,t;\bar{U})=\sum_{j=1}^4 e^{\lambda_j t}\hat{\mathbb{M}}_j-\sum_{j=1}^4 e^{\lambda_j^* t}\hat{\mathbb{M}}_j^*=\sum_{j=1}^3 e^{\lambda_j t}\hat{\mathbb{M}}_j-\sum_{j=1}^3 e^{\lambda_j^* t}\hat{\mathbb{M}}_j^*.
	\end{equation}
	As $\lambda_4$ does not contribute to the regular part, Eq.~\eqref{G dagger} is consistent with Eq.~(3.35) in \cite{WangHT2021}. We may thus directly adapt the arguments therein.

    The estimates we need for $\hat{\mathbb{G}}^{\dagger}$ in the low frequency regime are summarised below. We refer the reader to \cite{WangHT2021} for the proofs.

	\begin{enumerate}
		\item \textbf{Weighted Energy Estimates:} Weighted energy estimates are used to characterise $\mathbb{G}^{\dagger}$ in two key regimes: the initial layer $0<t<1$ and the region outside the wave cone $t\geq 1$ and $|x|>2Mt$, where $M$ is a constant related to the Mach number.

		\begin{itemize}
			\item \textbf{Estimates for $0<t<1$:} For the initial layer,  multiplying the equation of $\mathbb{G}^{\dagger}$ by a positive definite matrix and applying integration by parts, one obtains (see \cite[Lemma 3.9]{WangHT2021}) that
			\begin{lemma}
				Under the assumption 
				\begin{equation}
					\label{state}
					\bar{U}=(\bar{v}, \bar{u}, \bar{E}, \bar{z}), \quad\|\bar{U}-(1,0, c_v, 0)\|<\varepsilon, \quad 0<\varepsilon \ll 1,
				\end{equation}
				there exists a positive constant $\sigma$ such that  for $0<t<1$, one has that	\begin{equation*}
					\left\|e^{\sigma|x|} \mathbb{G}^{\dagger}(x, t)\right\|_{H^4(\mathbb{R})} \leq O(1) t, \quad \sum_{k=0}^3\left|\partial_x^k \mathbb{G}^{\dagger}(x, t)\right| \leq O(1) t e^{-\sigma|x|}.
				\end{equation*}
			\end{lemma}
			\item \textbf{Estimates for $t\geq 1$ and $|x|>2Mt$:} For large $t$ outside the wave cone, by choosing $M$ sufficiently larger than the sound speed and $\sigma>0$ sufficiently small, the following decay estimates are established as Lemma 3.10 in \cite{WangHT2021}:
			\begin{lemma}
				Under the assumption \eqref{state}, there exists a positive constant  $\sigma$ and $M$ such that the following estimates hold when $t\geq 1$ and $|x|>2Mt$:
				\begin{align*}
&\left\|e^{ \pm \sigma(x-M t)} \mathbb{G}^{\dagger}(x, t)\right\|_{H^4(\mathbb{R})} \leq O(1) e^{-\sigma_0^* t}, \\
&\sum_{k=0}^3\left|\partial_x^k \mathbb{G}^{\dagger}(x, t)\right| \leq O(1) e^{-\frac{\sigma}{2}|x|-\sigma_0^* t}.\end{align*}
			\end{lemma}
		\end{itemize}
		\item \textbf{Long-Wave and Short-Wave Decomposition:} For large time $(t\geq 1)$ within the wave cone $(|x|<2Mt)$, the regular part is decomposed into long wave part $\mathbb{G}_L^{\dagger}$ $(|\eta|<\delta)$ and short wave part $\mathbb{G}_S^{\dagger}$ $(|\eta|\geq \delta)$.

		\begin{itemize}
			\item \textbf{Long wave estimates:} For low frequencies ($|\eta|<\delta$),  eigenvalues $\lambda_j(j=1,2,3)$ have the following asymptotic expansions around $\eta=0$:
			\begin{align*}
				&\lambda_1=i\eta \beta_1-\alpha_1 \eta^2+O(1)\eta^3,\quad \lambda_2=i\eta \beta_2-\alpha_2\eta^2+O(1)\eta^3,\\
				&\lambda_3=i\eta \beta_3 -\alpha_3\eta^2+O(1)\eta^3,
			\end{align*}
			where $\beta_j$ are wave speeds and $\alpha_j>0$ are diffusion coefficients, defined as
			\begin{align*}
				&\alpha_1=\frac{-\nu\theta_e p_v}{v(p p_e-p_v)},\quad \alpha_2=\alpha_3=\frac{\nu p \theta_e p_e+\mu pp_e-\mu p_v}{2v(pp_e-p_v)},\\
				&\beta_1=0,\quad \beta_2=-\sqrt{pp_e-p_v},\quad \beta_3=\sqrt{pp_e-p_v}.
			\end{align*}
			Using these expansions and inverse Fourier transforms, the key estimate for $\mathbb{G}_L^\dagger$ is given as Lemma 3.13 in \cite{WangHT2021}:
			\begin{lemma}
				\label{lemma: regular M}
				Suppose $t\geq 1$ and $|x|\leq 2Mt$. There exists a  sufficiently small positive constant $\sigma_0^*$ and a  large constant $C$ such that, the long wave of the regular part $\mathbb{G}_L^{\dagger}$ has the following estimates
				\begin{align*}
					&\left|\partial_x^k\mathbb{G}_L^{\dagger}(x, t)-\sum_{j=1}^{3}\partial_x^k\left(\frac{e^{-\frac{\left(x+\beta_j t\right)^2}{4 \alpha_j t}}}{2 \sqrt{\pi \alpha_j t}}\right) M_j^0-\sum_{j=1}^{3}\partial_x^{k+1}\left(\frac{e^{-\frac{\left(x+\beta_j t\right)^2}{4 \alpha_j t}}}{2 \sqrt{\pi \alpha_j t}}\right) M_j^1\right|\\
					&\leq \frac{O(1) \sum_{j=1}^{3}e^{-\frac{\left(x+\beta_j t\right)^2}{4 C t}}}{t^{\frac{k+2}{2}}} M_j^0+\sum_{j=1}^{3}\frac{O(1) e^{-\frac{\left(x+\beta_j t\right)^2}{4 C t}}}{t^{\frac{k+3}{2}}}+O(1) e^{-\sigma_0^*\left(\frac{t}{2}+\frac{|x|}{4 M}\right)},
				\end{align*}
				where $M_j^0$ and $M_j^1$ are listed in Appendix A.
			\end{lemma}
			\item For high frequencies ($|\eta|\geq\delta$), applying Duhamel's principle for the inhomogeneous equation about $\mathbb{G}^{\dagger}$, the following estimates are proven as Lemma 3.14 in \cite{WangHT2021}:
			\begin{lemma}
				There exists a sufficiently small positive constant $\sigma_0^*$ such that, when $t\geq 1$ and $|x|\leq 2Mt$, the short wave parts $\mathbb{G}_S^{\dagger}$ has the following estimates
				\begin{align*}
					&\left|\hat{\mathbb{G}}_S^{\dagger}(\eta,t)\right|=O(1)\frac{e^{-\sigma_0^* t}}{(1+|\eta|)^6},\quad \left\|\mathbb{G}_S^{\dagger}(x,t)\right\|_{H^4(\mathbb{R})}\leq O(1)e^{-\sigma_0^* t},\\
					&\sum\limits_{k=0}^{3}\left|\partial_x^k \mathbb{G}_S^{\dagger}(x,t)\right|\leq O(1)e^{-\sigma_0^*t}\leq O(1)e^{-\sigma_0^*(\frac{t}{2}+\frac{|x|}{4M})}.
				\end{align*}
			\end{lemma}
		\end{itemize}
	\end{enumerate}

We deduce the following result from the four lemmas above.
	\begin{theorem}
		\label{thm: es for regular G}
		Assume the condition~\eqref{state}. There exists a sufficiently large constant C and sufficiently small constants $\sigma_0$, $\sigma_0^*$ such that
\begin{equation*}
    \sum\limits_{k=0}^3\left|\partial_x^k \mathbb{G}^{\dagger}(x, t)\right| \leq O(1) t e^{-\sigma_0|x|} \qquad\text{for } 0<t<1,
\end{equation*}
as well as for $k \in \{0,1,2,3\}$ that
\begin{align*}
&\left|\partial_x^k \mathbb{G}^{\dagger}(x, t)-\sum\limits_{j=1}^3 \partial_x^k\left(\frac{e^{-\frac{\left(x+\beta_j t\right)^2}{4 \alpha_j t}}}{2 \sqrt{\pi \alpha_j t}}\right) M_j^0-\sum\limits_{j=1}^3 \partial_x^{k+1}\left(\frac{e^{-\frac{\left(x+\beta_j t\right)^2}{4 \alpha_j t}}}{2 \sqrt{\pi \alpha_j t}}\right) M_j^1\right|  \\
				&\quad\leq\sum\limits_{j=1}^3\frac{O(1) e^{-\frac{\left(x+\beta_j t\right)^2}{4Ct}}}{t^{\frac{k+2}{2}}} M_j^0+\sum\limits_{j=1}^3 \frac{O(1) e^{-\frac{\left(x+\beta_jt\right)^2}{4Ct}}}{t^{\frac{k+3}{2}}}+O(1) e^{-\sigma_0^* t-\sigma_0|x|}, \qquad \text{for } t \geq 1.    
\end{align*} 
The explicit expressions for $M_j^0$ and $M_j^1$ are presented in Appendix~\ref{appendix, new}.
	\end{theorem}

See \cite[Theorem 3.3]{WangHT2021} for a proof.

	\section{Global Well-posedness}\label{sec: global wp, final}
	In this final section, we show that the local-in-time weak solution constructed in Theorem~\ref{thm: local exi} and \ref{thm: local regularity} can be
	extended to arbitrary time, provided that  the initial data $(v_0, u_0, \theta_0, z_0)$ are a small perturbation of the constant state $(1,0,1,0)$. (Equivalently,  $(v_0, u_0, E_0, z_0)$ is a small perturbation around $\bar{U}=(1,0,c_v,0)$, where $E$ is the total energy.) 
    
    Our proof relies on a delicate analysis of the properties of Green's function $\mathbb{G}$ constructed in \S\ref{sec: Green func}. As in Eq.~\eqref{decomposition into reg and sing of G, Dec25}, we consider \begin{align*}
          \mathbb{G}(x,t;\bar{U}) = \underbrace{\mathbb{G}^{*}(x,t;\bar{U})}_{\text{regular part}} + \underbrace{\mathbb{G}^{\dagger}(x,t;\bar{U})}_{\text{singular part}}, 
    \end{align*}
    where
	\begin{equation*}
		\mathbb{G}^*=\left(\begin{array}{cccc}
			\mathbb{G}_{11}^* & \mathbb{G}_{12}^* & \mathbb{G}_{13}^* &\mathbb{G}_{14}^* \\
			\mathbb{G}_{21}^* & \mathbb{G}_{22}^* & \mathbb{G}_{23}^* &\mathbb{G}_{24}^*\\
			\mathbb{G}_{31}^* & \mathbb{G}_{32}^* & \mathbb{G}_{33}^* &\mathbb{G}_{34}^*\\
			\mathbb{G}_{41}^* & \mathbb{G}_{42}^* & \mathbb{G}_{43}^* &\mathbb{G}_{44}^*
		\end{array}\right), \quad \mathbb{G}^{\dagger}=\left(\begin{array}{cccc}
			\mathbb{G}_{11}^{\dagger} & \mathbb{G}_{12}^{\dagger} & \mathbb{G}_{13}^{\dagger} &\mathbb{G}_{14}^{\dagger} \\
			\mathbb{G}_{21}^{\dagger} & \mathbb{G}_{22}^{\dagger} & \mathbb{G}_{23}^{\dagger} &\mathbb{G}_{24}^{\dagger}\\
			\mathbb{G}_{31}^{\dagger} & \mathbb{G}_{32}^{\dagger} & \mathbb{G}_{33}^{\dagger} &\mathbb{G}_{34}^{\dagger}\\
			\mathbb{G}_{41}^{\dagger} & \mathbb{G}_{42}^{\dagger} & \mathbb{G}_{43}^{\dagger} &\mathbb{G}_{44}^{\dagger}
		\end{array}\right).
	\end{equation*}
It satisfies the backward heat equation  
	\begin{equation}
		\partial_\tau \mathbb{G}(x-y, t-\tau ; \bar{U})+\partial_y \mathbb{G}(x-y, t-\tau ; \bar{U}) F^{\prime}(\bar{U})+\partial_y^2 \mathbb{G}(x-y, t-\tau ; \bar{U}) B(\bar{U})=0
	\end{equation}
with the coefficient matrices 
\begin{equation}
		F^{\prime}(\bar{U})=\left(\begin{array}{cccc}
			0 & -1 & 0 & 0\\
			-a & 0 & \frac{a}{c_v} &-\frac{qa}{c_v}\\
			0 & a & 0 & 0\\
			0 & 0 & 0 & 0
		\end{array}\right), \quad B(\bar{U})=\left(\begin{array}{cccc}
			0 & 0 & 0 & 0\\
			0 & \mu & 0 & 0 \\
			0 & 0 & \frac{\nu}{c_v} & -\frac{q\nu}{c_v}+qD\\
			0 & 0 & 0 & D
		\end{array}\right).
	\end{equation}
In components, the PDEs for $\mathbb{G}_{jk}$ are as follows:

	\begin{align}
		\label{array G}
		&\left(\begin{array}{ccc}
			\partial_\tau \mathbb{G}_{11}-a \partial_y \mathbb{G}_{12} & \partial_\tau \mathbb{G}_{12}-\partial_y \mathbb{G}_{11}+a \partial_y \mathbb{G}_{13}+\mu \partial_y^2 \mathbb{G}_{12} & \partial_\tau \mathbb{G}_{13}+\frac{a}{c_v} \partial_y \mathbb{G}_{12}+\frac{\nu}{c_v} \partial_y^2 \mathbb{G}_{13} \\
			\partial_\tau \mathbb{G}_{21}-a \partial_y \mathbb{G}_{22} & \partial_\tau \mathbb{G}_{22}-\partial_y \mathbb{G}_{21}+a \partial_y \mathbb{G}_{23}+\mu \partial_y^2 \mathbb{G}_{22} & \partial_\tau \mathbb{G}_{23}+\frac{a}{c_v} \partial_y \mathbb{G}_{22}+\frac{\nu}{c_v} \partial_y^2 \mathbb{G}_{23} \\
			\partial_\tau \mathbb{G}_{31}-a \partial_y \mathbb{G}_{32} & \partial_\tau \mathbb{G}_{32}-\partial_y \mathbb{G}_{31}+a \partial_y \mathbb{G}_{33}+\mu \partial_y^2 \mathbb{G}_{32} & \partial_\tau \mathbb{G}_{33}+\frac{a}{c_v} \partial_y \mathbb{G}_{32}+\frac{\nu}{c_v} \partial_y^2 \mathbb{G}_{33} \\
			\partial_\tau \mathbb{G}_{41}-a \partial_y \mathbb{G}_{42} & \partial_\tau \mathbb{G}_{42}-\partial_y \mathbb{G}_{41}+a \partial_y \mathbb{G}_{43}+\mu \partial_y^2 \mathbb{G}_{42} & \partial_\tau \mathbb{G}_{43}+\frac{a}{c_v} \partial_y \mathbb{G}_{42}+\frac{\nu}{c_v} \partial_y^2 \mathbb{G}_{43} 
		\end{array}\right.\nonumber\\
		&\qquad \left.\begin{array}{c}
			\partial_\tau \mathbb{G}_{14}-\frac{qa}{c_v} \partial_y \mathbb{G}_{12}+\left(-\frac{q\nu}{c_v}+qD\right)\partial_y^2 \mathbb{G}_{13}+D\partial_y^2 \mathbb{G}_{14}\\
			\partial_\tau \mathbb{G}_{24}-\frac{qa}{c_v} \partial_y \mathbb{G}_{22}+\left(-\frac{q\nu}{c_v}+qD\right)\partial_y^2 \mathbb{G}_{23}+D\partial_y^2 \mathbb{G}_{24} \\
			\partial_\tau \mathbb{G}_{34}-\frac{qa}{c_v} \partial_y \mathbb{G}_{32}+\left(-\frac{q\nu}{c_v}+qD\right)\partial_y^2 \mathbb{G}_{33}+D\partial_y^2 \mathbb{G}_{34}\\
			\partial_\tau \mathbb{G}_{44}-\frac{qa}{c_v} \partial_y \mathbb{G}_{42}+\left(-\frac{q\nu}{c_v}+qD\right)\partial_y^2 \mathbb{G}_{43}+D\partial_y^2 \mathbb{G}_{44}  	  	
		\end{array}\right)=\mathbb{O}.
	\end{align}

In the sequel, we shall systematically drop the background state $\bar{U}$ when there is no danger of confusion. For instance, we shall write $\mathbb{G}(x-y,t-\tau)\equiv \mathbb{G}(x-y,t-\tau;\bar{U})$.

	\subsection{Representation by Green's function} 
	The heat kernels introduced in \S\ref{sec: pre} is convenient for constructing local solutions but less suitable for analysing large-time behaviour. To address this issue, as in \cite{WangHT2021, LiuTP2022} we introduce an \emph{effective Green's function} $G$ that combines the local heat kernel $H$ and the global Green's function $\mathbb{G}$.

    First, fix a smooth, non-increasing cutoff function $\mathcal{X}$, such that
	\begin{equation}
		\mathcal{X}(t) \in C^{\infty}\left(\mathbb{R}_{+}\right), \quad \mathcal{X}^{\prime}(t) \leq 0, \quad\left\|\mathcal{X}^{\prime}\right\|_{L^{\infty}\left(\mathbb{R}_{+}\right)} \leq 2, \quad \mathcal{X}(t)= \begin{cases}1, & \text { for } t \in(0,1], \\ 0, & \text { for } t>2.\end{cases}
	\end{equation}
	Next, let $\nu_0$ be a sufficiently small positive constant such that the heat kernel  $H(x,t;y,\tau;\frac{1}{v})$ and the local solution $(v(x,\tau),u(x,\tau),E(x,\tau),z(x,\tau))$ exist for $\tau\in (t-2\nu_0,t)$. Then define
	\begin{equation}
		\label{eff G}
		\left\{\begin{array}{l}
			G_{22}(x,t;y,\tau)=\mathcal{X}\left(\frac{t-\tau}{\nu_0}\right) H\left(x,t;y,\tau ; \frac{\mu}{v}\right)+\left(1-\mathcal{X}\left(\frac{t-\tau}{\nu_0}\right)\right) \mathbb{G}_{22}(x-y;t-\tau), \\
			G_{33}(x,t;y,\tau)=\mathcal{X}\left(\frac{t-\tau}{\nu_0}\right) H\left(x,t;y,\tau ;\frac{\nu}{c_v v}\right)+\left(1-\mathcal{X}\left(\frac{t-\tau}{\nu_0}\right)\right) \mathbb{G}_{33}(x-y;t-\tau),\\
			G_{44}(x,t;y,\tau)=\mathcal{X}\left(\frac{t-\tau}{\nu_0}\right) H\left(x,t;y,\tau;\frac{D}{v^2}\right)+\left(1-\mathcal{X}\left(\frac{t-\tau}{\nu_0}\right)\right) \mathbb{G}_{44}(x-y;t-\tau).
		\end{array}\right.
	\end{equation}
	\begin{lemma}\label{lemma: z rep}
		Suppose that the weak solution $(v(x,\tau),u(x,\tau),E(x,\tau),z(x,\tau))$ to Eq.~\eqref{PDE,1} exists for $\tau\in [0,t]$. Let $\nu_0$ be a sufficiently small positive constant with $2\nu_0<t$ such that the heat kernel $H(x,t;y,\tau; \frac{D}{v^2})$ exists for $\tau\in (t-2\nu_0,t)$. Then 
		\begin{equation}
			\begin{aligned}\label{z rep}
				z(x,t)
				=&\int_{\mathbb{R}} \mathbb{G}_{41}(x-y,t)(v(y,0)-1)\mathrm{d}y+\int_{\mathbb{R}} \mathbb{G}_{42}(x-y,t) u(y, 0)\,\mathrm{d}y\\
				&+\int_{\mathbb{R}} \mathbb{G}_{43}(x-y,t)(E(y,0)-c_v)\,\mathrm{d}y+\int_{\mathbb{R}} G_{44}(x,t;y,0)z(y,0)\,\mathrm{d}y\\
				&-\int_0^t\int_{\mathbb{R}} G_{44}(x,t;y,\tau)K\phi(\theta)z(y,\tau)\,\mathrm{d}y\,\mathrm{d}\tau+\sum_{i=1}^3 \mathcal{R}_i^{z},
			\end{aligned}
		\end{equation}
		where the inhomogeneous remainders $\mathcal{R}_i^z$, $i = 1,2,3$, are given by
		\begin{align*}
			\mathcal{R}_1^{z}= & \int_0^{t-2\nu_0}\int_{\mathbb{R}} \partial_y \mathbb{G}_{42}(x-y, t-\tau)\left[\frac{a(v-1)^2}{v}+\frac{a(\theta-1)(1-v)}{v}-\frac{a u^2}{2 c_v}+\frac{\mu u_y(v-1)}{v}\right] \,\mathrm{d}y \,\mathrm{d}\tau \\
			& +\int_0^{t-2\nu_0}\int_{\mathbb{R}}\partial_y \mathbb{G}_{43}(x-y, t-\tau)\left[\left(\frac{a(\theta-1)+a(1-v)}{v}\right) u+\frac{\nu \theta_y(v-1)}{v}+\left(\frac{\nu}{c_v}-\frac{\mu}{v}\right)u u_y\right.\\
			&+\left.\frac{qD(v^2-1)}{v^2}z_y\right]\,\mathrm{d}y\,\mathrm{d}\tau +\int_0^{t-2\nu_0} \int_{\mathbb{R}} \partial_y \mathbb{G}_{44}(x-y,t-\tau)\frac{D(v^2-1)}{v^2}z_y \,\mathrm{d}y \,\mathrm{d}\tau,\\
			\mathcal{R}_2^{z}= &\int_{t-2\nu_0}^{t-\nu_0}\int_{\mathbb{R}} \partial_y \mathbb{G}_{42}(x-y, t-\tau)\left[\frac{a(v-1)^2}{v}+\frac{a(\theta-1)(1-v)}{v}-\frac{au^2}{2c_v}+\frac{\mu u_y(v-1)}{v}\right.\\
			&\left.-\mathcal{X}(\frac{t-\tau}{\nu_0})\frac{qa}{c_v}z\right](y,\tau)\,\mathrm{d}y\,\mathrm{d}\tau+\int_{t-2\nu_0}^{t-\nu_0} \int_{\mathbb{R}} \partial_y\mathbb{G}_{43}(x-y;t-\tau) \left[\frac{a(\theta-1)+a(1-v)}{v} u\right.\\
			&\left.+\frac{\nu(v-1)}{v}\theta_y+\left(\frac{\nu}{c_v v}-\frac{\mu}{v}\right)uu_y+\frac{qD(v^2-1)}{v^2}z_y\right](y,\tau)\,\mathrm{d}y\,\mathrm{d}\tau\\
			&+\int_{t-2\nu_0}^{t-\nu_0} \int_{\mathbb{R}}\left(1-\mathcal{X}(\frac{t-\tau}{\nu_0})\right) \partial_y\mathbb{G}_{44}(x-y;t-\tau) \frac{D(v^2-1)}{v^2}z_y(y,\tau)\,\mathrm{d}y\,\mathrm{d}\tau\\
			& +\int_{t-2\nu_0}^{t-\nu_0} \int_{\mathbb{R}} \frac{1}{\nu_0} \mathcal{X}^{\prime}\left(\frac{t-\tau}{\nu_0}\right)\left[\mathbb{G}_{44}(x-y;t-\tau)-H\left(x,t;y,\tau;\frac{D}{v^2}\right)\right]z(y, \tau)\,\mathrm{d}y\,\mathrm{d}\tau \\
			&+\int_{t-2\nu_0}^{t-\nu_0} \int_{\mathbb{R}}\mathcal{X}(\frac{t-\tau}{\nu_0}) \partial_y\mathbb{G}_{43}(x-y, t-\tau) \left(\frac{q\nu}{c_v}-qD\right)z_y(y,\tau)\,\mathrm{d}y\,\mathrm{d}\tau,\\
			\mathcal{R}_3^{z}= &\int_{t-\nu_0}^t \int_{\mathbb{R}} \partial_y \mathbb{G}_{42}(x-y,t-\tau)\left[\frac{a(v-1)^2}{v}+\frac{a(\theta-1)(1-v)}{v}-\frac{au^2}{2c_v}+\frac{\mu u_y(v-1)}{v}-\frac{qa}{c_v}z\right]\,\mathrm{d}y \,\mathrm{d}\tau\\
			&+\int_{t-2\nu_0}^{t-\nu_0} \int_{\mathbb{R}} \partial_y \mathbb{G}_{43}(x-y;t-\tau)\left[\frac{a(\theta-1)+a(1-v)}{v} u+\frac{\nu(v-1)}{v}\theta_y+\left(\frac{\nu}{c_v v}-\frac{\mu}{v}\right)u u_y\right.\\
			& \left.-\left(\frac{q\nu}{c_v}-qD\right)z_y\right](y,\tau)\,\mathrm{d}y\,\mathrm{d}\tau.
		\end{align*}
	\end{lemma}
	\begin{proof}
		Multiplying the vector $\left[\mathbb{G}_{41}(x-y,t-\tau),\mathbb{G}_{42}(x,t;y,\tau),\mathbb{G}_{43}(x-y,t-\tau),G_{44}(x,t;y,\tau)\right]^\top$ to the system~\eqref{PDE,1} and integrating over $\tau\in[0,t]$ and $y \in \mathbb{R}$ by parts, we arrive at
		\begin{align*}
			0= & \int_0^t \int_{\mathbb{R}} \mathbb{G}_{41}(x-y, t-\tau)\left(v_\tau-u_y\right)(y,\tau)\,\mathrm{d}y\,\mathrm{d}\tau \\
			& +\int_0^t \int_{\mathbb{R}} \mathbb{G}_{42}(x-y,t-\tau)\left(u_\tau+p_y-\left(\frac{\mu u_y}{v}\right)_y\right)(y,\tau)\,\mathrm{d}y\,\mathrm{d}\tau \\
			& +\int_0^t \int_{\mathbb{R}} \mathbb{G}_{43}(x-y, t-\tau)\left(E_{\tau}+(p u)_y-\left(\frac{\nu}{v} \theta_y+\frac{\mu}{v} u u_y+\frac{qDz_y}{v^2}\right)_y\right)(y,\tau)\,\mathrm{d}y\,\mathrm{d}\tau \\
			&+\int_0^t \int_{\mathbb{R}} G_{44}(x,t;y,\tau)\left(z_{\tau}+K\phi(\theta)z-\left(\frac{Dz_y}{v^2}\right)_y\right)(y,\tau)\,\mathrm{d}y\,\mathrm{d}\tau \\
			= & \int_0^t \int_{\mathbb{R}} \mathbb{G}_{41}(x-y, t-\tau)\left((v-1)_\tau-u_y\right)(y,\tau)\,\mathrm{d}y\,\mathrm{d}\tau \\
			& +\int_0^t \int_{\mathbb{R}} \mathbb{G}_{42}(x-y,t-\tau)\left(u_\tau+(p(v, \theta)-p(1,1))_y-\left(\frac{\mu u_y}{v}\right)_y\right)(y,\tau)\,\mathrm{d}y\,\mathrm{d}\tau \\
			& +\int_0^t \int_{\mathbb{R}} \mathbb{G}_{43}(x-y, t-\tau)\left(\left(E-c_v\right)_\tau+(p u)_y-\left(\frac{\nu}{v} \theta_y+\frac{\mu}{v} u u_y+\frac{qDz_y}{v^2}\right)_y\right)(y,\tau)\,\mathrm{d}y\,\mathrm{d}\tau \\
			&+\int_0^t \int_{\mathbb{R}} G_{44}(x,t;y,\tau)\left(z_{\tau}+K\phi(\theta)z-\left(\frac{Dz_y}{v^2}\right)_y\right)(y,\tau)\,\mathrm{d}y\,\mathrm{d}\tau \\
			= & -\int_{\mathbb{R}} \mathbb{G}_{41}(x-y,t)\left(v(y,0)-1\right) d y-\int_0^t \int_{\mathbb{R}} \partial_\tau \mathbb{G}_{41}(x-y, t-\tau)(v(y,\tau)-1)\,\mathrm{d}y\,\mathrm{d}\tau \\
			& +\int_0^t \int_{\mathbb{R}} \partial_y\mathbb{G}_{41}(x-y, t-\tau) u(y,\tau)\,\mathrm{d}y\,\mathrm{d}\tau-\int_{\mathbb{R}}\mathbb{G}_{42}(x-y,t)u(y,0)\,\mathrm{d}y\\
			&-\int_0^t \int_{\mathbb{R}}\partial_\tau \mathbb{G}_{42}(x-y,t-\tau)u(y,\tau)\,\mathrm{d}y\,\mathrm{d}\tau -\int_0^t \int_{\mathbb{R}}\partial_y \mathbb{G}_{42}(x-y,t- \tau)\left((p-p(1,1))-\frac{\mu u_y}{v}\right)\,\mathrm{d}y \,\mathrm{d}\tau \\
			& -\int_{\mathbb{R}} \mathbb{G}_{43}(x-y,t)\left(E(y, 0)-c_v\right)\,\mathrm{d} y-\int_0^t \int_{\mathbb{R}} \partial_\tau \mathbb{G}_{43}(x-y, t-\tau)\left(E(y,\tau)-c_v\right)\,\mathrm{d}y\,\mathrm{d}\tau \\
			& -\int_0^t \int_{\mathbb{R}} \partial_y \mathbb{G}_{43}(x-y, t-\tau)\left((pu)-\left(\frac{\nu}{v} \theta_y+\frac{\mu}{v} u u_y+\frac{qD}{v^2}z_y\right)\right)(y,\tau)\,\mathrm{d}y\,\mathrm{d}\tau\\
			&+z(x,t)-\int_{\mathbb{R}} G_{44}(x,t;y,0)z(y, 0)\,\mathrm{d} y-\int_0^t \int_{\mathbb{R}} \partial_\tau G_{44}(x,t;y,\tau)z(y,\tau)\,\mathrm{d}y \,\mathrm{d}\tau \\
			& +\int_0^t G_{44}(x,t;y,\tau)K\phi(\theta)z(y,\tau)\,\mathrm{d}y \,\mathrm{d}\tau+\int_0^t \int_{\mathbb{R}} \partial_y G_{44}(x,t;y,\tau)\frac{D}{v^2}z_y (y,\tau)\,\mathrm{d}y \,\mathrm{d}\tau.
		\end{align*}
In the above, we make use of  $\mathbb{G}(x-y;0)=\delta(x-y)$ and that
		\begin{equation*}
			\int_{\mathbb{R}}G_{44}(x,t;y,t)z(y,t)\,\mathrm{d}y=\int_{\mathbb{R}}\mathbb{G}_{44}(x-y;0)z(y,t)\,\mathrm{d}y=z(x,t).
		\end{equation*}
Rearrangement gives us
		\begin{align}
			\label{z Ii}
			z(x, t)
			&= \int_{\mathbb{R}}\mathbb{G}_{41}(x-y,t)(v(y,0)-1)\,\mathrm{d} y+\int_{\mathbb{R}} \mathbb{G}_{42}(x-y,t) u(y,0)\,\mathrm{d}y\nonumber\\
			&\quad+\int_{\mathbb{R}}\mathbb{G}_{43}(x-y, t)\left(E(y, 0)-c_v\right)\,\mathrm{d}y+\int_{\mathbb{R}} G_{44}(x,t;y,0) z(y,0)\,\mathrm{d}y\nonumber\\
			&\quad+\int_0^t\int_{\mathbb{R}} \partial_\tau \mathbb{G}_{41}(x-y,t-\tau)(v(y,\tau)-1)\,\mathrm{d}y\,\mathrm{d}\tau-\int_0^t \int_{\mathbb{R}} \partial_y \mathbb{G}_{41}(x-y,t-\tau) u(y,\tau)\,\mathrm{d}y\,\mathrm{d}\tau \nonumber\\
			&\quad+\int_0^t \int_{\mathbb{R}} \partial_\tau \mathbb{G}_{42}(x-y,t-\tau) u(y,\tau)\,\mathrm{d}y\,\mathrm{d} \tau\nonumber\\
			&\quad+\int_0^t \int_{\mathbb{R}} \partial_y \mathbb{G}_{42}(x-y,t-\tau)\left(p-p(1,1)-\frac{\mu u_y}{v}\right)(y,\tau)\,\mathrm{d}y\,\mathrm{d}\tau \nonumber\\
			&\quad+\int_0^t \int_{\mathbb{R}}\partial_\tau \mathbb{G}_{43}(x-y,t-\tau)\left(E(y, \tau)-c_v\right)\,\mathrm{d}y\,\mathrm{d}\tau\nonumber\\
			&\quad+\int_0^t \int_{\mathbb{R}} \partial_y \mathbb{G}_{43}(x-y, t-\tau)\left((p u)-\left(\frac{\nu}{v} \theta_y+\frac{\mu}{v} u u_y+\frac{qD}{v^2}z_y\right)\right)(y,\tau)\,\mathrm{d}y\,\mathrm{d}\tau \nonumber\\
			&\quad+\int_0^t \int_{\mathbb{R}}\partial_\tau G_{44}(x,t;y,\tau)z(y, \tau)\,\mathrm{d}y\,\mathrm{d}\tau-\int_0^t \int_{\mathbb{R}}G_{44}(x,t;y,\tau)K\phi(\theta)z(y,\tau)\,\mathrm{d}y \,\mathrm{d}\tau\nonumber\\
			&\quad-\int_0^t \int_{\mathbb{R}} \partial_y G_{44}(x,t;y,\tau)\frac{D}{v^2}z_y(y,\tau)\,\mathrm{d}y\,\mathrm{d}\tau \nonumber\\
			&=: \sum_{i=1}^{13} \mathcal{I}_i.
		\end{align}

        To proceed, we use Eq.~\eqref{array G} to replace the term $\mathbb{G}_{\tau}$ by $\mathbb{G}_y$. Hence,
		\begin{align*}
			\mathcal{I}_5=&\int_0^t\int_{\mathbb{R}}\partial_\tau \mathbb{G}_{41}(x-y,t-\tau)(v(y,\tau)-1)\,\mathrm{d}y\,\mathrm{d}\tau\\
			=&\int_0^t\int_{\mathbb{R}}a\partial_y \mathbb{G}_{42}(x-y,t-\tau)(v(y,\tau)-1)\,\mathrm{d}y\,\mathrm{d}\tau,\\
			\mathcal{I}_7=&\int_0^t\int_{\mathbb{R}}\partial_\tau \mathbb{G}_{42}(x-y,t-\tau)u(y,\tau)\,\mathrm{d}y\,\mathrm{d}\tau\\
			=&\int_0^t\int_{\mathbb{R}}\left(\partial_y\mathbb{G}_{41}-a\partial_y\mathbb{G}_{43}-\mu\partial_y^2\mathbb{G}_{42}\right)(x-y,t-\tau)u(y,\tau)\,\mathrm{d}y\,\mathrm{d}\tau,\\
			\mathcal{I}_9=&\int_0^t\int_{\mathbb{R}}\partial_\tau \mathbb{G}_{43}(x-y,t-\tau)(E(y,\tau)-c_v)\,\mathrm{d}y\,\mathrm{d}\tau\\
			=&\int_0^t\int_{\mathbb{R}}\left(-\frac{a}{c_v}\partial_y\mathbb{G}_{42}-\frac{\nu}{c_v}\partial_y^2\mathbb{G}_{43}\right)(x-y,t-\tau)(E(y,\tau)-c_v)\,\mathrm{d}y\,\mathrm{d}\tau.
		\end{align*}
		For $\mathcal{I}_{11}$, using the definition of $G_{44}(x,t;y,\tau)$ in \eqref{eff G}, we split the integral into three time intervals
		\begin{align*}
			\mathcal{I}_{11}=&\int_0^t\int_{\mathbb{R}}\partial_\tau G_{44}(x,t;y,\tau)z(y,\tau)\,\mathrm{d}y\,\mathrm{d}\tau\\
			=&\int_0^t\int_{\mathbb{R}}\left[-\frac{\mathcal{X}^{\prime}\left(\frac{t-\tau}{\nu_0}\right)}{\nu_0}H(x,t;y,\tau;\frac{D}{v^2})+\mathcal{X}\left(\frac{t-\tau}{\nu_0}\right)H_{\tau}(x,t;y,\tau;\frac{D}{v^2})\right.\\
			& \left.+\frac{\mathcal{X}^{\prime}\left(\frac{t-\tau}{\nu_0}\right)}{\nu_0}\mathbb{G}_{44}(x-y,t-\tau)+\left(1-\mathcal{X}\left(\frac{t-\tau}{\nu_0}\right)\right)\partial_\tau \mathbb{G}_{44}(x-y,t-\tau)\right]z(y,\tau)\,\mathrm{d}y\,\mathrm{d}\tau\\
			=&\int_{0}^{t-2\nu_0}\int_{\mathbb{R}}\partial_{\tau}\mathbb{G}_{44}(x-y,t-\tau)z(y,\tau)\,\mathrm{d}y\,\mathrm{d}\tau\\
			& +\int_{t-2\nu_0}^{t-\nu_0}\int_{\mathbb{R}}\frac{-\mathcal{X}^{\prime}\left(\frac{t-\tau}{\nu_0}\right)}{\nu_0}H(x,t;y,\tau;\frac{D}{v^2}) z(y,\tau)\,\mathrm{d}y\,\mathrm{d}\tau\\
			&\quad+\int_{t-2\nu_0}^{t-\nu_0}\int_{\mathbb{R}}\mathcal{X}\left(\frac{t-\tau}{\nu_0}\right)H_{\tau}(x,t;y,\tau;\frac{D}{v^2}) z(y,\tau)\,\mathrm{d}y\,\mathrm{d}\tau\\
			& +\int_{t-2\nu_0}^{t-\nu_0}\int_{\mathbb{R}}\frac{\mathcal{X}^{\prime}\left(\frac{t-\tau}{\nu_0}\right)}{\nu_0}\mathbb{G}_{44}(x-y,t-\tau)z(y,\tau)\,\mathrm{d}y\,\mathrm{d}\tau\\
			& +\int_{t-2\nu_0}^{t-\nu_0}\int_{\mathbb{R}}\left(1-\mathcal{X}\left(\frac{t-\tau}{\nu_0}\right)\right)\partial_{\tau}\mathbb{G}_{44}(x-y,t-\tau)z(y,\tau)\,\mathrm{d}y\,\mathrm{d}\tau\\
			& +\int_{t-\nu_0}^{t}\int_{\mathbb{R}}H_{\tau}(x,t;y,\tau;\frac{D}{v^2}) z(y,\tau)\,\mathrm{d}y\,\mathrm{d}\tau\\
			=&\int_{0}^{t-2\nu_0}\int_{\mathbb{R}}\left(\frac{qa}{c_v}\partial_y\mathbb{G}_{42}-\left(\frac{-q\nu}{c_v}+qD\right)\partial_y^2\mathbb{G}_{43}-D\partial_y^2\mathbb{G}_{44}\right)(x-y,t-\tau)z(y,\tau)\,\mathrm{d}y\,\mathrm{d}\tau\\
			& +\int_{t-2\nu_0}^{t-\nu_0}\int_{\mathbb{R}}\frac{\mathcal{X}^{\prime}\left(\frac{t-\tau}{\nu_0}\right)}{\nu_0}\left[\mathbb{G}_{44}(x-y,t-\tau)-H(x,t;y,\tau;\frac{D}{v^2})\right] z(y,\tau)\,\mathrm{d}y\,\mathrm{d}\tau\\
			& -\int_{t-2\nu_0}^{t-\nu_0}\int_{\mathbb{R}}\mathcal{X}\left(\frac{t-\tau}{\nu_0}\right)\left(\frac{D}{v^2}H_y(x,t;y,\tau;\frac{D}{v^2})\right)_y z(y,\tau)\,\mathrm{d}y \,\mathrm{d}\tau\\
			& +\int_{t-2\nu_0}^{t-\nu_0}\int_{\mathbb{R}}\left(1-\mathcal{X}\left(\frac{t-\tau}{\nu_0}\right)\right)\left(\frac{qa}{c_v}\partial_y\mathbb{G}_{42}-\left(\frac{-q\nu}{c_v}+qD\right)\partial_y^2\mathbb{G}_{43}-D\partial_y^2\mathbb{G}_{44}\right)\\
			&\quad (x-y,t-\tau)z(y,\tau)\,\mathrm{d}y\,\mathrm{d}\tau -\int_{t-\nu_0}^{t}\int_{\mathbb{R}}\left(\frac{D}{v^2}H_y(x,t;y,\tau;\frac{D}{v^2})\right)_y z(y,\tau)\,\mathrm{d}y\,\mathrm{d}\tau,
		\end{align*}
		in which we use the backward Eq.~\eqref{eq: H backward} for $H$ and Eq.~\eqref{array G} for $\mathbb{G}$.
		
		Finally, for $\mathcal{I}_{13}$, we use the definition of $G_{44}(x,t;y,\tau)$ to obtain
		\begin{align*}
			\mathcal{I}_{13}=&-\int_0^t\int_{\mathbb{R}}\partial_y G_{44}(x,t;y,\tau)\frac{D}{v^2} z_y(y,\tau)\,\mathrm{d}y\,\mathrm{d}\tau\\
			=&-\int_{0}^{t-2\nu_0}\int_{\mathbb{R}}\partial_y\mathbb{G}_{44}(x-y,t-\tau)\frac{D}{v^2}z_y(y,\tau)\,\mathrm{d}y\,\mathrm{d}\tau\\
			& -\int_{t-2\nu_0}^{t-\nu_0}\int_{\mathbb{R}}\mathcal{X}\left(\frac{t-\tau}{\nu_0}\right)H_y(x,t;y,\tau;\frac{D}{v^2}) \frac{D}{v^2}z_y(y,\tau)\,\mathrm{d}y\,\mathrm{d}\tau\\
			&-\int_{t-2\nu_0}^{t-\nu_0}\int_{\mathbb{R}}\left(1-\mathcal{X}\left(\frac{t-\tau}{\nu_0}\right)\right)\mathbb{G}_{44}(x-y,t-\tau)\frac{D}{v^2}z_y(y,\tau)\,\mathrm{d}y\,\mathrm{d}\tau\\
			&-\int_{t-\nu_0}^{t}\int_{\mathbb{R}}H_y(x,t;y,\tau;\frac{D}{v^2})\frac{D}{v^2}z_y(y,\tau)\,\mathrm{d}y\,\mathrm{d}\tau.
		\end{align*}
		Substituting the expressions for $\mathcal{I}_5,\mathcal{I}_7,\mathcal{I}_9,\mathcal{I}_{11}$ and $\mathcal{I}_{13}$ into \eqref{z Ii}, we obtain the desired representation formula for $z(x,t)$.  \end{proof}

	\begin{lemma}
		\label{lemma: formula uvt}
		Suppose the weak solution $(v(x,\tau),u(x,\tau),E(x,\tau),z(x,\tau))$ to Eq.~\eqref{PDE,1} exists for $\tau\in [0,t]$, and let $\nu_0 $ be a sufficiently small positive constant with $2\nu_0<t$ such that the heat kernels $H(x,t;y,\tau; \frac{\mu}{v})$ and $H(x,t;y,\tau; \frac{\nu}{c_v v})$ exist for $\tau\in (t-2\nu_0,t)$. Then 
		\begin{align}
			u(x, t)
			=&\int_{\mathbb{R}} \mathbb{G}_{21}(x-y,t)(v(y,0)-1)\mathrm{d}y+\int_{\mathbb{R}} G_{22}(x,t;y,0) u(y, 0)\,\mathrm{d}y\nonumber\\
			&+\int_{\mathbb{R}} \mathbb{G}_{23}(x-y,t)\left(E(y,0)-c_v\right)\,\mathrm{d}y
			+\sum_{i=1}^3 \mathcal{R}_i^u,\nonumber\\
			v(x,t)-1
			=&\int_{\mathbb{R}} \mathbb{G}_{11}(x-y,t)(v(y,0)-1)\mathrm{d}y+\int_{\mathbb{R}} \mathbb{G}_{12}(x-y,t) u(y,0)\,\mathrm{d}y\nonumber\\
			&+\int_{\mathbb{R}} \mathbb{G}_{13}(x-y,t)\left(E(y,0)-c_v\right)\,\mathrm{d}y\nonumber\\
			&+\int_0^{t}\int_{\mathbb{R}} \partial_y \mathbb{G}_{12}(x-y, t-\tau)\left[\frac{a(v-1)^2}{v}+\frac{a(\theta-1)(1-v)}{v}-\frac{a u^2}{2 c_v}+\frac{\mu u_y(v-1)}{v}\right.\nonumber\\
			&\left.-\frac{qa}{c_v}z\right](y,\tau)\,\mathrm{d}y \,\mathrm{d}\tau +\int_0^{t}\int_{\mathbb{R}}\partial_y \mathbb{G}_{13}(x-y, t-\tau)\left[\left(\frac{a(\theta-1)+a(1-v)}{v}\right) u\right.\nonumber\\
			&\left.+\frac{\nu \theta_y(v-1)}{v}+\left(\frac{\nu}{c_v}-\frac{\mu}{v}\right)u u_y+\frac{q\nu}{c_v}z_y-\frac{qD}{v^2}z_y\right](y,\tau)\,\mathrm{d}y\,\mathrm{d}\tau,\nonumber\\
			\label{E rep}
			E(x,t)-c_v
			=&\int_{\mathbb{R}} \mathbb{G}_{31}(x-y,t)(v(y,0)-1)\mathrm{d}y+\int_{\mathbb{R}} \mathbb{G}_{32}(x-y,t) u(y, 0)\,\mathrm{d}y\nonumber\\
			&+\int_{\mathbb{R}} G_{33}(x,t;y,0)(E(y,0)-c_v)\,\mathrm{d}y+\sum_{i=1}^3 \mathcal{R}_i^{\theta}.
		\end{align}
The inhomogeneous remainder terms $\mathcal{R}_i^{u}$ and $\mathcal{R}_i^{\theta}$ are given by
		\begin{align*}
			\mathcal{R}_1^u= & \int_0^{t-2 \nu_0} \int_{\mathbb{R}} \partial_y \mathbb{G}_{22}(x-y, t-\tau)\left[\frac{a(v-1)^2}{v}+\frac{a(\theta-1)(1-v)}{v}-\frac{a u^2}{2 c_v}+\frac{\mu u_y(v-1)}{v}\right.\\
			&\left.-\frac{qa}{c_v}z\right](y,\tau) \,\mathrm{d}y \,\mathrm{d}\tau+\int_0^{t-2\nu_0}\int_{\mathbb{R}}\partial_y \mathbb{G}_{23}(x-y, t-\tau)\left[\frac{a(\theta-1)+a(1-v)}{v} u\right.\\
			&\left.+\frac{\nu \theta_y(v-1)}{v}+\left(\frac{\nu}{c_v}-\frac{\mu}{v}\right)u u_y+\frac{q\nu}{c_v}z_y-\frac{qD}{v_2} z_y\right](y,\tau)\,\mathrm{d}y\,\mathrm{d}\tau,\\
			\mathcal{R}_2^u= & \int_{t-2\nu_0}^{t-\nu_0}\int_{\mathbb{R}} \partial_y \mathbb{G}_{22}(x-y, t-\tau)\left[\frac{a(v-1)^2}{v}+\frac{a(\theta-1)(1-v)}{v}+\frac{\mu u_y(v-1)}{v}(1-\mathcal{X})\right.\\
			&\left. -\frac{au^2}{2c_v}\right](y,\tau)\,\mathrm{d}y\,\mathrm{d}\tau +\int_{t-2\nu_0}^{t-\nu_0} \int_{\mathbb{R}} \partial_y \mathbb{G}_{23}(x-y;t-\tau)\left[\left(\frac{a(\theta-1)+a(1-v)}{v}\right) u\right.\\
			&\left.+\frac{\nu \theta_y(v-1)}{v}+\left(\frac{\nu}{c_v}-\frac{\mu}{v}\right)u u_y+\frac{q\nu}{c_v}z_y-\frac{qD}{v^2}z_y\right](y,\tau)\,\mathrm{d}y\,\mathrm{d}\tau \\
			& +\int_{t-2\nu_0}^{t-\nu_0} \int_{\mathbb{R}} \frac{1}{\nu_0} \mathcal{X}^{\prime}\left(\frac{t-\tau}{\nu_0}\right)\left[\mathbb{G}_{22}(x-y;t-\tau)-H\left(x,t;y,\tau;\frac{\mu}{v}\right)\right]u(y, \tau)\,\mathrm{d}y\,\mathrm{d}\tau \\
			& +\int_{t-2\nu_0}^{t-\nu_0} \int_{\mathbb{R}} \mathcal{X}\left(\frac{t-\tau}{\nu_0}\right)\left[H_y\left(x,t;y, \tau;\frac{\mu}{v}\right)-\partial_y\mathbb{G}_{22}(x-y ; t-\tau)\right] \frac{a(\theta-v)}{v}\,\mathrm{d}y\,\mathrm{d}\tau \\
			& +\int_{t-2\nu_0}^{t-\nu_0} \int_{\mathbb{R}} \mathcal{X}\left(\frac{t-\tau}{\nu_0}\right)\left[a\partial_y \mathbb{G}_{23}(x-y, t-\tau)-\partial_y \mathbb{G}_{21}(x-y, t-\tau)\right]u(y,\tau)\,\mathrm{d}y \,\mathrm{d}\tau, \\
			\mathcal{R}_3^u= & -\int_{t-\nu_0}^t \int_{\mathbb{R}} \partial_y \mathbb{G}_{21}(x-y,t-\tau) u(y,\tau)\,\mathrm{d}y \,\mathrm{d}\tau\\
			&+\int_{t-\nu_0}^t \int_{\mathbb{R}} \partial_y \mathbb{G}_{22}(x-y,t-\tau)\left[a(v-\theta)-\frac{au^2}{2c_v}-\frac{qa}{c_v}z\right](y,\tau)\,\mathrm{d}y \,\mathrm{d}\tau\\
			&+\int_{t-\nu_0}^t \int_{\mathbb{R}} \partial_y \mathbb{G}_{23}(x-y,t-\tau)\left[pu+\frac{\nu(v-1)}{v}\theta_y+\left(\frac{\nu}{c_v}-\frac{\mu}{v}\right)uu_y+\frac{q\nu}{c_v}z_y-\frac{qD}{v^2}z_y\right]\,\mathrm{d}y \,\mathrm{d}\tau\\
			&+\int_{t-\nu_0}^t \int_{\mathbb{R}} H_y(x,t;y,\tau;\frac{\mu}{v})\frac{a(\theta-v)}{v}(y,\tau)\,\mathrm{d}y \,\mathrm{d}\tau,
		\end{align*}
		\begin{align*}
			\mathcal{R}_1^{\theta}= & \int_0^{t-2\nu_0}\int_{\mathbb{R}} \partial_y \mathbb{G}_{32}(x-y, t-\tau)\left[\frac{a(v-1)^2}{v}+\frac{a(\theta-1)(1-v)}{v}-\frac{a u^2}{2 c_v}+\frac{\mu u_y(v-1)}{v}-\frac{qa}{c_v}z\right] \,\mathrm{d}y \,\mathrm{d}\tau \\
			& +\int_0^{t-2\nu_0}\int_{\mathbb{R}}\partial_y \mathbb{G}_{33}(x-y, t-\tau)\left[\left(\frac{a(\theta-1)+a(1-v)}{v}\right) u+\frac{\nu \theta_y(v-1)}{v}+\left(\frac{\nu}{c_v}-\frac{\mu}{v}\right)u u_y\right.\\
			&+\left.\frac{q\nu}{c_v}z_y-\frac{qD}{v^2}z_y\right](y,\tau)\,\mathrm{d}y\,\mathrm{d}\tau,\\
			\mathcal{R}_2^{\theta}= &\int_{t-2\nu_0}^{t-\nu_0}\int_{\mathbb{R}} \partial_y \mathbb{G}_{32}(x-y, t-\tau)\left[\frac{a(\theta-1)}{v}+\frac{\mu u_y(v-1)}{v}+\frac{a(v-1)^2}{v}\right](y,\tau)\,\mathrm{d}y\,\mathrm{d}\tau \\
			&-\int_{t-2\nu_0}^{t-\nu_0} \int_{\mathbb{R}}\left(1-\mathcal{X}(\frac{t-\tau}{\nu_0})\right)\frac{a}{c_v}\partial_y\mathbb{G}_{32}(x-y;t-\tau)(E(y, \tau)-c_v)\,\mathrm{d}y\,\mathrm{d}\tau\\
			& +\int_{t-2\nu_0}^{t-\nu_0} \int_{\mathbb{R}} \frac{1}{\nu_0} \mathcal{X}^{\prime}\left(\frac{t-\tau}{\nu_0}\right)\left[\mathbb{G}_{33}(x-y;t-\tau)-H\left(x,t;y,\tau;\frac{\nu}{c_v v}\right)\right](E(y, \tau)-c_v)\,\mathrm{d}y\,\mathrm{d}\tau \\
			& +\int_{t-2\nu_0}^{t-\nu_0} \int_{\mathbb{R}} \mathcal{X}(\frac{t-\tau}{\nu_0})H_y(x,t;y,\tau;\frac{\nu}{c_v v})\left[pu+\left(\frac{\nu}{c_v v}-\frac{\mu}{v}\right)u u_y+\left(\frac{q\nu}{c_v v}-\frac{qD}{v^2}z_y\right)\right]\,\mathrm{d}y\,\mathrm{d}\tau \\
			& +\int_{t-2\nu_0}^{t-\nu_0} \int_{\mathbb{R}}\left(1-\mathcal{X}(\frac{t-\tau}{\nu_0})\right) \partial_y\mathbb{G}_{33}(x-y;t-\tau) \left[pu+\frac{\nu(v-1)}{v}\theta_y+\left(\frac{\nu}{c_v v}-\frac{\mu}{v}\right)uu_y\right.\\
			&\left.\qquad+\left(\frac{q\nu}{c_v v}-\frac{qD}{v^2}\right)z_y\right](y,\tau)\,\mathrm{d}y\,\mathrm{d}\tau-\int_{t-2\nu_0}^{t-\nu_0} \int_{\mathbb{R}} \partial_y \mathbb{G}_{33}(x-y;t-\tau)a u(y,\tau)\,\mathrm{d}y\,\mathrm{d}\tau,\\
			\mathcal{R}_3^{\theta}= &\int_{t-\nu_0}^t \int_{\mathbb{R}} \partial_y \mathbb{G}_{32}(x-y,t-\tau)\left[a(v-1)+\frac{a(\theta-1)+a(1-v)}{v}+\frac{\mu u_y(v-1)}{v}\right](y,\tau)\,\mathrm{d}y \,\mathrm{d}\tau\\
			&-\int_{t-2\nu_0}^{t-\nu_0} \int_{\mathbb{R}} \partial_y \mathbb{G}_{33}(x-y;t-\tau)a u(y,\tau)\,\mathrm{d}y\,\mathrm{d}\tau\\
			&+\int_{t-\nu_0}^t \int_{\mathbb{R}} H_y(x,t;y,\tau;\frac{\nu}{c_v v})\left[pu+\left(\frac{\nu}{c_v v}-\frac{\mu}{v}\right)u u_y+\left(\frac{q\nu}{c_v v}-\frac{qD}{v^2}\right)z_y\right](y,\tau)\,\mathrm{d}y \,\mathrm{d}\tau.
		\end{align*}
	\end{lemma}
	\begin{proof}
		The proof follows  similarly to that of Lemma~\ref{lemma: z rep}. There is, however, a key difference: only the first three equations of Eq.~\eqref{PDE,1} are used here. This is because the fourth column of the Green's function is $q$ times the third column.

		We begin by multiplying the first three equations of Eq.~\eqref{PDE,1} by the vector
		\begin{equation*}
			\left[\mathbb{G}_{11}(x-y,t-\tau),\mathbb{G}_{12}(x-y,t-\tau),\mathbb{G}_{13}(x-y,t-\tau)\right]^\top,
		\end{equation*}
		and then integrate with respect to $\tau$ and $y$. Via integration by parts, we obtain the representation formula for $v(x,t)$. 
        
        Similarly, for $u(x,t)$ and $E(x,t)$ we multiply the first three equation of Eq.~\eqref{PDE,1} by the vectors
		\begin{equation*}
			\left[\mathbb{G}_{21}(x-y,t-\tau),G_{22}(x,t;y,\tau),\mathbb{G}_{23}(x-y,t-\tau)\right]^\top
		\end{equation*}
		and
		\begin{equation*}
			\left[\mathbb{G}_{31}(x-y,t-\tau),\mathbb{G}_{32}(x-y,t-\tau),G_{33}(x,t;y,\tau)\right]^\top,
		\end{equation*}
		respectively, and then  follow the similar procedure.
	\end{proof}
	\begin{corollary}
		The first order derivatives of $\mathcal{R}_3^{u}$ and $\mathcal{R}_3^{\theta}$ are expressed below:
		\begin{align*}
			\partial_x \mathcal{R}_3^u&=  \int_{-\infty}^x \int_{-\infty}^y \frac{1}{\mu}\left(\partial_x \mathbb{G}_{22}\left(x-\omega, \nu_0\right)-H_x\left(x,t;\omega,t-\nu_0;\frac{\mu}{v}\right)\right) \,\mathrm{d}\omega\left(a v-a\theta\right)(y,t-\nu_0) \,\mathrm{d}y \\
			& -\int_{t-\nu_0}^t \int_{-\infty}^x \int_{-\infty}^y-\frac{1}{\mu}\left(\partial_x \mathbb{G}_{22}(x-\omega,t-\tau)-H_x\left(x,t;\omega,\tau;\frac{\mu}{v}\right)\right)\,\mathrm{d}\omega\left(a v_\tau-a\theta_\tau\right)(y,\tau)\,\mathrm{d}y\,\mathrm{d}\tau \\
			& -\int_x^{+\infty} \int_y^{+\infty} \frac{1}{\mu}\left(\partial_x \mathbb{G}_{22}\left(x-\omega,\nu_0\right)-H_x\left(x,t;\omega,t-\nu_0;\frac{\mu}{v}\right)\right)\,\mathrm{d}\omega\left(a v-a \theta\right)(y, t-\nu_0)\,\mathrm{d} y \\
			& +\int_{t-\nu_0}^t \int_x^{+\infty} \int_y^{+\infty}-\frac{1}{\mu}\left(\partial_x \mathbb{G}_{22}(x-\omega,t-\tau)-H_x\left(x,t;\omega,\tau;\frac{\mu}{v}\right)\right)\,\mathrm{d}\omega\left(a v_\tau-a \theta_\tau\right)(y,\tau)\,\mathrm{d}y\,\mathrm{d} \tau \\
			& +\int_{t-\nu_0}^t \int_{\mathbb{R}}\frac{1}{\mu}\left(\partial_x \mathbb{G}_{21}(x-y,t-\tau)-a\partial_x \mathbb{G}_{23}(x-y, t-\tau)\right)(a v-a \theta)(y,\tau)\,\mathrm{d}y\,\mathrm{d}\tau \\
			& +\int_{t-\nu_0}^t \int_{\mathbb{R}} \partial_x \mathbb{G}_{21}(x-y, t-\tau)u_y(y,\tau)\,\mathrm{d}y\,\mathrm{d}\tau+\int_{t-\nu_0}^t \int_{\mathbb{R}} \partial_x \mathbb{G}_{22}(x-y, t-\tau)\frac{au u_y}{c_v}(y,\tau)\,\mathrm{d}y\,\mathrm{d}\tau\\
			&+\int_{t-\nu_0}^t \int_{\mathbb{R}} \partial_x \mathbb{G}_{22}(x-y, t-\tau)\frac{qa}{c_v}z_y(y,\tau)\,\mathrm{d}y\,\mathrm{d}\tau\\
			&+\int_{t-\nu_0}^t \int_{\mathbb{R}} \partial_{xy} \mathbb{G}_{23}(x-y,t-\tau)\left(p u+\frac{\nu(v-1)}{v}\theta_y+\left(\frac{\nu}{c_v}-\frac{\mu}{v}\right)u u_y+\frac{q\nu}{c_v}z_y-\frac{qD}{v^2}z_y\right)\,\mathrm{d}y\,\mathrm{d}\tau,
		\end{align*}
		and
		\begin{align*}
			\partial_x \mathcal{R}_3^\theta= & \int_{t-\nu_0}^t \int_{\mathbb{R}} \partial_{x y} \mathbb{G}_{32}(x-y,t-\tau)\left[a(v-1)+\mu u_y\frac{v-1}{v}+\frac{a \theta-a v}{v}\right](y,\tau)\,\mathrm{d}y\,\mathrm{d}\tau \\
			& +\int_{t-\nu_0}^t \int_{\mathbb{R}}\left(\partial_x \mathbb{G}_{33}(x-y,t-\tau)-H_x\left(x,t;y,\tau;\frac{\nu}{c_v v}\right)\right) au_y(y, \tau) d y d \tau \\
			& +\int_{t-\nu_0}^t \int_{\mathbb{R}}H_x\left(x,t;y,\tau;\frac{\nu}{c_v v}\right)au_y(y, \tau) d y d \tau \\
			& +\int_{t-\nu_0}^t \int_{\mathbb{R}}H_x\left(x,t;y,\tau;\frac{\nu}{c_v v}\right)\left(1-\frac{\nu}{c_v\mu}\right)\left(uu_{\tau}-pu_y+\frac{\mu u_y}{v}u_y\right)(y,\tau)d y d \tau \\
			& +\int_{-\infty}^x \int_{-\infty}^y H_x\left(x,t;\omega,t-\nu_0 ; \frac{\nu}{c_v v}\right)\,\mathrm{d}\omega \,\frac{a}{\mu} \theta(y,t-\nu_0)u(y, t-\nu_0) d y \\
			& +\int_{t-\nu_0}^t \int_{-\infty}^x \int_{-\infty}^y H_x\left(x, t ; \omega,\tau;\frac{\nu}{c_v v}\right) d\omega \,\frac{a }{\mu}\left(\theta_\tau u+\theta u_\tau\right)(y,\tau)d y d \tau \\
			& -\int_x^{+\infty} \int_y^{+\infty} H_x\left(x,t;\omega,t-\nu_0 ;\frac{\nu}{c_v v}\right) d\omega\,\frac{a}{\mu} \theta(y,t-\nu_0)u(y, t-\nu_0) d y \\
			& -\int_{t-\nu_0}^t \int_x^{+\infty} \int_y^{+\infty} H_x\left(x,t;\omega,\tau;\frac{\nu}{c_v v}\right) d \omega \,\frac{a }{\mu}\left(\theta_\tau u+\theta u_\tau\right)(y,\tau)d y d \tau\\
			&+qz_x(x,t)-q\int_{\mathbb{R}} H_x\left(x,t;y,t-\nu_0;\frac{\nu}{c_v v}\right) z(y,t-\nu_0)\,\mathrm{d}y\\
			&+q\int_{t-\nu_0}^{t}\int_{\mathbb{R}}H_x\left(x,t;y,t-\nu_0;\frac{\nu}{c_v v}\right)K\phi(\theta) z\,\mathrm{d}y\,\mathrm{d}\tau.
		\end{align*}
	\end{corollary}
	\begin{proof}
		The derivation for $\partial_x \mathcal{R}_3^u$ can be found in \cite[Corollary 4.1]{WangHT2021}. Here we focus on  $\partial_x \mathcal{R}_3^{\theta}$.

		Recall the evolution equation for $u$, namely,
		\begin{equation*}
			uu_{\tau}=u\left(\frac{\mu u_y}{v}-p\right)_y.
		\end{equation*}
		Multiplying  both sides by $H(x,t;y,\tau;\frac{\nu}{c_v v})$ and integrating over $\mathbb{R}\times [t-\nu_0,t]$, we obtain that
		\begin{equation*}
			\int_{t-\nu_0}^t\int_{\mathbb{R}}H(x,t;y,\tau;\frac{\nu}{c_v v})\left(uu_{\tau}-u\left(\frac{\mu u_y}{v}-p\right)_y\right)\,\mathrm{d}y\,\mathrm{d}\tau=0.
		\end{equation*}
		Applying integration by parts, we derive that
		\begin{align*}
			&\int_{t-\nu_0}^t\int_{\mathbb{R}}H_y(x,t;y,\tau;\frac{\nu}{c_v v})\left(\frac{\nu}{c_v v}-\frac{\mu}{v}\right)u u_y(y,\tau)\,\mathrm{d}y\,\mathrm{d}\tau\\
			=&-\int_{t-\nu_0}^t\int_{\mathbb{R}}H_y(x,t;y,\tau;\frac{\nu}{c_v v})\left(1-\frac{\nu}{c_v v}\right)\frac{\mu}{v}u u_y(y,\tau)\,\mathrm{d}y\,\mathrm{d}\tau\\
			=&\int_{t-\nu_0}^t\int_{\mathbb{R}}H(x,t;y,\tau;\frac{\nu}{c_v v})\left(1-\frac{\nu}{c_v v}\right)\left[uu_{\tau}-pu_y+\frac{\mu u_y}{v}u_y\right](y,\tau)\,\mathrm{d}y\,\mathrm{d}\tau\\
			&\quad -\int_{t-\nu_0}^t\int_{\mathbb{R}}H_y(x,t;y,\tau;\frac{\nu}{c_v v})\left(1-\frac{\nu}{c_v v}\right)p u(y,\tau)\,\mathrm{d}y\,\mathrm{d}\tau.
		\end{align*}
		Using $p=\frac{a\theta}{v}$ and the backward equation for $H$:
		\begin{equation}
			\label{eq: H back nu}
			\left\{\begin{array}{l}
				\partial_{\omega}\left(\frac{\nu}{c_v v}(\omega,\tau) H_{\omega}(x,t;\omega,\tau;\frac{\nu}{c_v v})\right)=-H_{\tau}(x,t;\omega,\tau;\frac{\nu}{c_v v}),\\
				H(x,t;\omega,t;\frac{\nu}{c_v v})=\delta(x-\omega),
			\end{array}\right.
		\end{equation}
		we further obtain via integration by parts that
		\begin{align*}
			&\int_{t-\nu_0}^t\int_{\mathbb{R}}H_y(x,t;y,\tau;\frac{\nu}{c_v v})\left(1-\frac{\nu}{c_v v}\right)\frac{\nu}{c_v \mu}\frac{a\theta}{v}u(y,\tau)\,\mathrm{d}y\,\mathrm{d}\tau\\
			=&\int_{t-\nu_0}^t\int_{-\infty}^{x}\int_{-\infty}^{y}\partial_{\omega}\left(\frac{\nu}{c_v v} H_{\omega}(x,t;\omega,\tau;\frac{\nu}{c_v v})\right)\,\mathrm{d}\omega \frac{a}{\mu}\theta u(y,\tau)\,\mathrm{d}y\,\mathrm{d}\tau\\
			&\quad -\int_{t-\nu_0}^t\int_{x}^{+\infty}\int_{y}^{+\infty}\partial_{\omega}\left(\frac{\nu}{c_v v} H_{\omega}(x,t;\omega,\tau;\frac{\nu}{c_v v})\right)\,\mathrm{d}\omega \frac{a}{\mu}\theta u(y,\tau)\,\mathrm{d}y\,\mathrm{d}\tau\\
			=&-\int_{t-\nu_0}^t\int_{-\infty}^{x}\int_{-\infty}^{y}H_{\tau}(x,t;\omega,\tau;\frac{\nu}{c_v v})\,\mathrm{d}\omega \frac{a}{\mu}\theta u(y,\tau)\,\mathrm{d}y\,\mathrm{d}\tau\\
			&\quad +\int_{t-\nu_0}^t\int_{x}^{+\infty}\int_{y}^{+\infty}H_{\tau}(x,t;\omega,\tau;\frac{\nu}{c_v v})\,\mathrm{d}\omega \frac{a}{\mu}\theta u(y,\tau)\,\mathrm{d}y\,\mathrm{d}\tau\\
			=&\int_{-\infty}^{x}\int_{-\infty}^y H(x,t;\omega,t-\nu_0;\frac{\nu}{c_v v})\,\mathrm{d}\omega \frac{a}{\mu}\theta(y,t-\nu_0) u(y,t-\nu_0)\,\mathrm{d}y\\
			&\quad -\int_{x}^{+\infty}\int_{y}^{+\infty} H(x,t;\omega,t-\nu_0;\frac{\nu}{c_v v})\,\mathrm{d}\omega \frac{a}{\mu}\theta(y,t-\nu_0) u(y,t-\nu_0)\,\mathrm{d}y\\
			&\quad +\int_{t-\nu_0}^t\int_{-\infty}^{x}\int_{-\infty}^{y}H(x,t;\omega,\tau;\frac{\nu}{c_v v})\,\mathrm{d}\omega \frac{a}{\mu}\left(\theta_{\tau}u+\theta u_{\tau}\right)(y,\tau)\,\mathrm{d}y\,\mathrm{d}\tau\\
			&\quad -\int_{t-\nu_0}^t\int_{x}^{+\infty}\int_{y}^{+\infty}H(x,t;\omega,\tau;\frac{\nu}{c_v v})\,\mathrm{d}\omega \frac{a}{\mu}\left(\theta_{\tau}u+\theta u_{\tau}\right)(y,\tau)\,\mathrm{d}y\,\mathrm{d}\tau.
		\end{align*}

		In an analogous manner, multiplying 
		\begin{equation*}
			z_{\tau}+K \phi(\theta)z=\left(\frac{D}{v^2} z_y\right)_y
		\end{equation*}
		by the kernel $H(x,t;y,\tau;\frac{\nu}{c_v v})$ and integrating over $\mathbb{R}\times [t-\nu_0,t]$, we deduce that
		\begin{align*}
			\int_{t-\nu_0}^t \int_{\mathbb{R}}H(x,t;y,\tau;\frac{\nu}{c_v v}) \left(z_{\tau}+K \phi(\theta) z-\left(\frac{D}{v^2} z_y\right)_y\right)\,\mathrm{d}y\,\mathrm{d}\tau=0.
		\end{align*}
		Note that
		\begin{align*}
			\int_{t-\nu_0}^t \int_{\mathbb{R}}H(x,t;y,\tau;\frac{\nu}{c_v v}) \left(z_{\tau}+K \phi(\theta) z\right)\,\mathrm{d}y\,\mathrm{d}\tau=-\int_{t-\nu_0}^t \int_{\mathbb{R}}H_y(x,t;y,\tau;\frac{\nu}{c_v v}) \frac{D}{v^2} z_y \,\mathrm{d}y \,\mathrm{d}\tau.
		\end{align*}
        Hence, using the backward Eq.~\eqref{eq: H back nu} for $H$  and integrating by parts in $\tau$, we obtain that
		\begin{align*}
			&\int_{t-\nu_0}^t \int_{\mathbb{R}}H_y(x,t;y,\tau;\frac{\nu}{c_v v}) \frac{\nu}{c_v v} z_y(y,\tau) \,\mathrm{d}y \,\mathrm{d}\tau\\
			=&-\int_{t-\nu_0}^t \int_{\mathbb{R}}\partial_y\left(\frac{\nu}{c_v v}H_y(x,t;y,\tau;\frac{\nu}{c_v v}) \right) z(y,\tau) \,\mathrm{d}y \,\mathrm{d}\tau\\
			=&\int_{t-\nu_0}^t \int_{\mathbb{R}}H_{\tau}(x,t;y,\tau;\frac{\nu}{c_v v}) z(y,\tau) \,\mathrm{d}y \,\mathrm{d}\tau\\
			=&z(x,t)-\int_{\mathbb{R}}H(x,t;y,t-\nu_0;\frac{\nu}{c_v v})z(y,t-\nu_0)\,\mathrm{d}y-\int_{t-\nu_0}^t \int_{\mathbb{R}}H(x,t;y,\tau;\frac{\nu}{c_v v}) z_{\tau}(y,\tau) \,\mathrm{d}y \,\mathrm{d}\tau.
		\end{align*}

		Finally, substituting the above results into the representation of $\mathcal{R}_3^{\theta}$ given in Lemma \ref{lemma: formula uvt} and differentiating with respect to $x$, we complete the proof. 	    	
	\end{proof}
	\subsection{Global existence, uniqueness and large time behaviour}
  As in \cite{LiuTP2022}, we introduce 
	\begin{align}
		\label{stop time}
			T & =\sup _{t \geq 0}\{t :\,\mathcal{G}(\tau)<\delta \text { for all } 0<\tau<t\}, 
	\end{align}
viewed as the \emph{stopping time}, where
    \begin{align*}
\mathcal{G}  (\tau) &:=\|\sqrt{\tau+1}(v(\cdot, \tau)-1)\|_{L_x^{\infty}}+\|\sqrt{\tau+1} u(\cdot, \tau)\|_{L_x^{\infty}}+\|\sqrt{\tau+1}(\theta(\cdot, \tau)-1)\|_{L_x^{\infty}}\nonumber \\
			&+\|\sqrt{\tau+1} z(\cdot, \tau)\|_{L_x^{\infty}}+\|v(\cdot, \tau)-1\|_{L_x^1}+\|u(\cdot, \tau)\|_{L_x^1}+\|\theta(\cdot, \tau)-1\|_{L_x^1}\nonumber \\
			&+\|z(\cdot, \tau)\|_{L_x^1}+\|v(\cdot, \tau)-1\|_{BV}+\|u(\cdot, \tau)\|_{BV}+\|\theta(\cdot,\tau)-1\|_{BV}+\|z(\cdot,\tau)\|_{BV}\nonumber\\
			& +\left\|\sqrt{\tau}u_x(\cdot, \tau)\right\|_{L_x^{\infty}}+\left\|\sqrt{\tau}\theta_x(\cdot, \tau)\right\|_{L_x^{\infty}}+\left\|\sqrt{\tau}z_x(\cdot, \tau)\right\|_{L_x^{\infty}}.
    \end{align*}
	\begin{lemma}
		\label{lemma: T+t}
		Let $(v, u, E, z)$, $C_{\sharp}$, $t_{\sharp}$ and $\delta$ be the local solution
		and corresponding parameters constructed in Theorem \ref{thm: local exi}. We further suppose that 
		\begin{equation}
			\label{300C}
			\left\|v_0-1\right\|_{L_x^1}+\left\|v_0\right\|_{B V}+\left\|u_0\right\|_{L_x^1}+\left\|u_0\right\|_{B V}+\left\|\theta_0-1\right\|_{L_x^1}+\left\|\theta_0\right\|_{B V}+\left\|z_0\right\|_{L_x^1}+\left\|z_0\right\|_{B V} \leq \delta^*<\frac{\delta}{300C_{\sharp}}.
		\end{equation}
		Then the stopping time defined in \eqref{stop time} satisfies $$T>t_{\sharp}.$$  Moreover, for $\delta>0$ and $t<T+t_{\sharp}$, the solution satisfies
		\begin{equation}
			\label{es for T+t}
			\left\{\begin{array}{l}
				\max \left\{\begin{array}{l}
					\|u(\cdot, t)\|_{L_x^1},\|u(\cdot, t)\|_{L_x^{\infty}},\left\|u_x(\cdot, t)\right\|_{L_x^1}, \\
					\sqrt{\min \left(t, t_{\sharp}\right)}\left\|u_x(\cdot, t)\right\|_{L_x^{\infty}}, \sqrt{\min \left(t, t_{\sharp}\right)}\left\|u_t(\cdot, t)\right\|_{L_x^1}, \min \left(t, t_{\sharp}\right)\left\|u_t(\cdot, t)\right\|_{L_x^{\infty}}
				\end{array}\right\} \leq 2 C_{\sharp} \delta, \\
				\max \left\{\begin{array}{l}
					\|\theta(\cdot, t)-1\|_{L_x^1},\|\theta(\cdot, t)-1\|_{L_x^{\infty}},\left\|\theta_x(\cdot, t)\right\|_{L_x^1}, \\
					\sqrt{\min \left(t, t_{\sharp}\right)}\left\|\theta_x(\cdot, t)\right\|_{L_x^{\infty}}, \sqrt{\min \left(t, t_{\sharp}\right)}\left\|\theta_t(\cdot, t)\right\|_{L_x^1}, \min \left(t, t_{\sharp}\right)\left\|\theta_t(\cdot, t)\right\|_{L_x^{\infty}}
				\end{array}\right\} \leq 2 C_{\sharp} \delta, \\
				\max \left\{\int_{\mathbb{R} \backslash \mathcal{D}}\left|v_x(x, t)\right|\,\mathrm{d}x,\|v(\cdot, t)-1\|_{L_x^1},\|v(\cdot, t)-1\|_{L_x^{\infty}}, \sqrt{\min \left(t, t_{\sharp}\right)}\left\|v_t(\cdot, t)\right\|_{L_x^{\infty}}\right\} \leq 2 C_{\sharp} \delta, \\
				\max \left\{\begin{array}{l}
					\|z(\cdot, \tau)\|_{L_x^1},\|z(\cdot, t)\|_{L_x^{\infty}},\left\|z_x(\cdot, t)\right\|_{L_x^1}, \\
					\sqrt{\min \left(t, t_{\sharp}\right)}\left\|z_x(\cdot, t)\right\|_{L_x^{\infty}}, \sqrt{\min \left(t, t_{\sharp}\right)}\left\|z_t(\cdot, t)\right\|_{L_x^1}, \min \left(t, t_{\sharp}\right)\left\|z_t(\cdot, t)\right\|_{L_x^{\infty}}
				\end{array}\right\} \leq 2 C_{\sharp} \delta, \\
				v^*=v_{\tilde{a}}^*+v_j^*, \quad v_j^*(x,t)=\sum\limits_{\omega<x, \omega\in \mathcal{D}} \left.v^*\right|_{\omega^{-}} ^{\omega^{+}} h(x-\omega), \quad v_{\tilde{a}}^* \text { is continuous} , \\
				\left|\left.v(\cdot, t)\right|_{x=\omega^{-}}^{x=\omega^{+}}\right|\leq 2\left|\left. v_0^*(\cdot)\right|_{x=\omega^{-}}^{x=\omega^{+}}\right|, \quad \omega \in \mathcal{D},
			\end{array}\right.
		\end{equation}
		where $v^*=v-1$, $\mathcal{D}$ is the discontinuity set of $v$ and $h(\cdot)$ is the Heaviside function as before.
	\end{lemma}
	\begin{proof}
		By Theorem~\ref{thm: local exi}, the local solution can be constructed with the same constants $C_{\sharp}$ and $t_{\sharp}$ therein, when the initial data satisfies the condition~\eqref{300C}. Specifically, for all $\tau < t_{\sharp}$, the solution satisfies
		\begin{equation*}
			\left\{\begin{array}{l}
				\max \left\{\begin{array}{l}
					\|v(\cdot, \tau)-1\|_{BV},\|u(\cdot, \tau)\|_{BV},\|\theta(\cdot, \tau)-1\|_{BV},\|z(\cdot, \tau)\|_{BV},\\
					\|v(\cdot, \tau)-1\|_{L_x^1},\|u(\cdot, \tau)\|_{L_x^1},\|\theta(\cdot, \tau)-1\|_{L_x^1},\|z(\cdot, \tau)\|_{L_x^1}
				\end{array}\right\} \leq 2C_{\sharp}\delta^*, \\
				\max\left\{\begin{array}{l}
					\|\sqrt{1+\tau} u(\cdot, \tau)\|_{L_x^{\infty}},\|\sqrt{1+\tau}(\theta(\cdot, \tau)-1)\|_{L_x^{\infty}},\\
					\|\sqrt{1+\tau}(v(\cdot, \tau)-1)\|_{L_x^{\infty}},
					\|\sqrt{1+\tau} z(\cdot, \tau)\|_{L_x^{\infty}}
				\end{array}\right\}
				\leq 2 C_{\sharp} \sqrt{1+t_{\sharp}} \delta^*, \\
				\max \left\{\left\|\sqrt{\tau} u_x(\cdot, \tau)\right\|_{L_x^{\infty}},\left\|\sqrt{\tau} \theta_x(\cdot, \tau)\right\|_{L_x^{\infty}},\left\|\sqrt{\tau} z_x(\cdot, \tau)\right\|_{L_x^{\infty}}\right\} \leq 2 C_{\sharp} \delta^* .
			\end{array}\right.
		\end{equation*}
		Combining the above estimates, the definition of $\delta^*$ in \eqref{300C}, and the expression of $\mathcal{G}(\tau)$ in \eqref{stop time}, we can directly deduce that
		\begin{equation}
			\label{T geq t}
			\mathcal{G}(\tau)\leq 30C_{\sharp}\sqrt{1+t_{\sharp}} \delta^* <\frac{\delta}{5},\quad \text{for} \quad\tau<t_{\sharp}.
		\end{equation}
		Since $\mathcal{G}(\tau)<\delta$ for all $\tau<t_{\sharp}$, and the solution $(v, u, E, z)$ is continuous in $t$, we immediately conclude that $T>t_{\sharp}$.

        Next, we prove the estimate~\eqref{es for T+t} for $t<T+t_{\sharp}$. By definition of the stopping time $T$ in \eqref{stop time}, for any $t<T+t_{\sharp}$ we set $\tau_0=\max(0,t-t_{\sharp})$. Then clearly $\tau_0 < T$, which implies that $\mathcal{G}(\tau_0)<\delta$. Thus, if we take $(v(\cdot,\tau_0), u(\cdot,\tau_0), \theta(\cdot,\tau_0), z(\cdot,\tau_0))$ as the new initial data and apply the local existence and estimate results in Theorem \ref{thm: local exi} again, we may continue the solution up to time $T+t^\sharp$ with the estimates in Eq.~\eqref{es for T+t}. This completes the proof.  
	\end{proof}
	
    Equipped with the key property above for the stopping time $T$, now we turn to the large-time behaviour of the solution. It relies on refined estimates for Green's function around the constant equilibrium
	state $(v,u, \theta, z)=(1,0,1,0)$.

    First, obverse that by using Theorems~\ref{thm: es for singular G} and \ref{thm: es for regular G}, we arrive at simplified expressions for Green's function.
	\begin{lemma}\label{lemma: G es}
		Let $\mathbb{G}(x,t)$ be the Green's function of the linearized equation of Eq.~\eqref{PDE,1} around the constant equilibrium state $(v, u, \theta, z)=(1, 0, 1, 0)$. For $t\leq 1$, we have
		\begin{align}
			& \left\lvert\, \mathbb{G}(x-y, t)-e^{-\frac{K}{\mu} t} \delta(x-y)\left(\begin{array}{cccc}
				1 & 0 & 0 & 0 \\
				0 & 0 & 0 & 0\\
				0 & 0 & 0 & 0\\
				0 & 0 & 0 & 0
			\end{array}\right)-\frac{e^{\beta_2^* t}}{\sqrt{4 \pi \mu t}} e^{-\frac{(x-y)^2}{4 \mu t}}\left(\begin{array}{cccc}
				0 & 0 & 0 & 0\\
				0 & 1 & 0 & 0\\
				0 & 0 & 0 & 0\\
				0 & 0 & 0 & 0
			\end{array}\right)\right. \nonumber\\
			& \left.-\frac{e^{\beta_3^* t}}{\sqrt{4 \pi \frac{\nu}{c_v} t}} e^{-\frac{(x-y)^2}{4 \frac{\nu}{c_v} t}}\left(\begin{array}{cccc}
				0 & 0 & 0 & 0 \\
				0 & 0 & 0 & 0\\
				0 & 0 & 1 & q\\
				0 & 0 & 0 & 0
			\end{array}\right)-\frac{1}{\sqrt{4 \pi Dt}} e^{-\frac{(x-y)^2}{4 D t}}\left(\begin{array}{cccc}
				0 & 0 & 0 & 0 \\
				0 & 0 & 0 & 0\\
				0 & 0 & 0 & q\\
				0 & 0 & 0 & 1
			\end{array}\right)\right\rvert \nonumber\\
			&\qquad \qquad \leq O(1) e^{-\sigma_0^* t-\sigma_0|x-y|}+O(1) t e^{-\sigma_0|x-y|}, \quad t \leq 1, \nonumber\\
			& \left\lvert\, \mathbb{G}_x(x-y, t)-e^{-\frac{K}{\mu} t} \delta^{\prime}(x-y)\left(\begin{array}{cccc}
				1 & 0 & 0 & 0\\
				0 & 0 & 0 & 0\\
				0 & 0 & 0 & 0\\
				0 & 0 & 0 & 0
			\end{array}\right)+e^{-\frac{K}{\mu} t} \delta(x-y)\left(\begin{array}{cccc}
				0 & \frac{1}{\mu} & 0 & 0\\
				\frac{a}{\mu} & 0 & 0 & 0\\
				0 & 0 & 0 & 0\\
				0 & 0 & 0 & 0
			\end{array}\right)\right. \nonumber\\
			& \left.-\partial_x\left(\frac{e^{\beta_2^* t}}{\sqrt{4 \pi \mu t}} e^{-\frac{(x-y)^2}{4 \mu t}}\right)\left(\begin{array}{cccc}
				0 & 0 & 0 & 0\\
				0 & 1 & 0 & 0\\
				0 & 0 & 0 & 0\\
				0 & 0 & 0 & 0
			\end{array}\right)+\frac{e^{\beta_2^* t}}{\sqrt{4 \pi \mu t}} e^{-\frac{(x-y)^2}{4 \mu t}}\left(\begin{array}{cccc}
				0 & -\frac{1}{\mu} & 0 & 0\\
				-\frac{a}{\mu} & 0 & \frac{a}{c_v \mu-\nu} & \frac{-qa}{c_v \mu-\nu}\\
				0 & \frac{c_v a}{c_v \mu-\nu} & 0 & 0\\
				0 & 0 & 0 & 0
			\end{array}\right) \right.\nonumber\\
			& \left.-\partial_x\left(\frac{e^{\beta_3^* t}}{\sqrt{4 \pi \frac{\nu}{c_v} t}} e^{-\frac{(x-y)^2}{4 \frac{\nu}{c_v} t}}\right)\left(\begin{array}{cccc}
				0 & 0 & 0 & 0\\
				0 & 0 & 0 & 0\\
				0 & 0 & 1 & q\\
				0 & 0 & 0 & 0
			\end{array}\right)+\frac{e^{\beta_3^* t}}{\sqrt{4 \pi \frac{\nu}{c_v} t}} e^{-\frac{(x-y)^2}{4 \frac{\nu}{c_v} t}}\left(\begin{array}{cccc}
				0 & 0 & 0 & 0\\
				0 & 0 & -\frac{a}{c_v \mu-\nu} & \frac{qa}{c_v \mu-\nu}\\
				0 & -\frac{c_v a}{c_v \mu-\nu} & 0 & 0\\
				0 & 0 & 0 & 0
			\end{array}\right) \right.\nonumber\\
			&\left.-\partial_x\left(\frac{1}{\sqrt{4\pi Dt}}e^{\frac{-(x-y)^2}{4 Dt}}\right)\left(\begin{array}{cccc}
				0 & 0 & 0 & 0\\
				0 & 0 & 0 & 0\\
				0 & 0 & 0 & q\\
				0 & 0 & 0 & 1
			\end{array}\right)\right\rvert\, \nonumber\\
			& \leq O(1) e^{-\sigma_0^* t-\sigma_0|x-y|}+O(1) t e^{-\sigma_0|x-y|}, \quad t \leq 1, \nonumber\\
			& \left\lvert\, \mathbb{G}_{xx}(x-y, t)-e^{-\frac{a}{\mu} t} \delta^{\prime \prime}(x-y)\left(\begin{array}{cccc}
				1 & 0 & 0 & 0\\
				0 & 0 & 0 & 0\\
				0 & 0 & 0 & 0\\
				0 & 0 & 0 & 0
			\end{array}\right)+e^{-\frac{K}{\mu} t} \delta^{\prime}(x-y)\left(\begin{array}{cccc}
				0 & \frac{1}{\mu} & 0 & 0\\
				\frac{a}{\mu} & 0 & 0 & 0\\
				0 & 0 & 0 & 0\\
				0 & 0 & 0 & 0
			\end{array}\right)\right. \nonumber\\
			& \left.-e^{-\frac{K}{\mu} t} \delta(x-y)\left(\left(\begin{array}{cccc}
				\frac{a}{\mu^2} & 0 & \frac{a}{\nu \mu} & -\frac{qa}{\nu \mu} \\
				0 & -\frac{a}{\mu^2} & 0 & 0 \\
				\frac{c_\nu a^2}{\nu \mu} & 0 & 0 & 0\\
				0 & 0 & 0 & 0
			\end{array}\right)-t\left(\begin{array}{cccc}
				-\frac{\left(\nu a^2-\mu a^3\right)}{\nu \mu^3} & 0 & 0 & 0\\
				0 & 0 & 0 & 0\\
				0 & 0 & 0 & 0\\
				0 & 0 & 0 & 0
			\end{array}\right)\right)\right. \nonumber\\
			&\left. -\partial_{xx}\left(\frac{e^{\beta_2^* t}}{\sqrt{4 \pi \mu t}} e^{-\frac{(x-y)^2}{4 \mu t}}\right)\left(\begin{array}{cccc}
				0 & 0 & 0 & 0\\
				0 & 1 & 0 & 0\\
				0 & 0 & 0 & 0\\
				0 & 0 & 0 & 0
			\end{array}\right)\right.\nonumber\\
			&\left.+\partial_x\left(\frac{e^{\beta_2^* t}}{\sqrt{4 \pi \mu t}} e^{-\frac{(x-y)^2}{4 \mu t}}\right)\left(\begin{array}{cccc}
				0 & -\frac{1}{\mu} & 0 & 0\\
				-\frac{a}{\mu} & 0 & \frac{a}{c_v \mu-\nu} & -\frac{qa}{c_v \mu-\nu}\\
				0 & \frac{c_v a}{c_v \mu-\nu} & 0 & 0\\
				0 & 0 & 0 & 0
			\end{array}\right) \right.\nonumber\\
			& \left.-\frac{e^{\beta_2^* t}}{\sqrt{4 \pi \mu t}} e^{-\frac{(x-y)^2}{4 \mu t}}\left(\begin{array}{cccc}
				-\frac{a}{\mu^2} & 0 & \frac{a}{c_v \mu^2-\nu \mu} & -\frac{qa}{c_v \mu^2-\nu \mu}\\
				0 & \frac{c_v a^2}{\left(c_v \mu-\nu\right)^2}+\frac{a}{\mu^2} & 0  & 0\\
				\frac{c_v a^2}{c_v \mu^2-\nu \mu} & 0 & -\frac{c_v a^2}{\left(c_v \mu-\nu\right)^2} & \frac{q c_v a^2}{\left(c_v \mu-\nu\right)^2}\\
				0 & 0 & 0 & 0
			\end{array}\right) \right.\nonumber\\
			& \left.-\partial_{xx}\left(\frac{e^{\beta_3^* t}}{\sqrt{4 \pi \frac{\nu}{c_v} t}} e^{-\frac{(x-y)^2}{4 \frac{\nu}{c_v} t}}\right)\left(\begin{array}{cccc}
				0 & 0 & 0 & 0\\
				0 & 0 & 0 & 0\\
				0 & 0 & 1 & q\\
				0 & 0 & 0 & 0
			\end{array}\right)\right.\nonumber\\
			&\left.+\partial_x\left(\frac{e^{\beta_3^* t}}{\sqrt{4 \pi \frac{\nu}{c_v} t}} e^{-\frac{(x-y)^2}{4 \frac{\nu}{c_v} t}}\right)\left(\begin{array}{cccc}
				0 & 0 & 0 & 0\\
				0 & 0 & -\frac{a}{c_v \mu-\nu} & \frac{qa}{c_v \mu-\nu}\\
				0 & -\frac{c_v a}{c_v \mu-\nu} & 0 & 0\\
				0 & 0 & 0 & 0
			\end{array}\right)\right. \nonumber\\
			& \left.+\frac{e^{\beta_3^* t}}{\sqrt{4 \pi \frac{\nu}{c_v} t}} e^{-\frac{(x-y)^2}{4 \frac{\nu}{c_v} t}}\left(\begin{array}{cccc}
				0 & 0 & \frac{c_v a}{\nu\left(c_v \mu-\nu\right)} & -\frac{qc_v a}{\nu\left(c_v \mu-\nu\right)}\\
				0 & \frac{c_v a^2}{\left(c_v \mu-\nu\right)^2} & 0 & 0\\
				\frac{c_v^2 a^2}{\nu\left(c_v \mu-\nu\right)} & 0 & -\frac{c_v a^2}{\left(c_v \mu-\nu\right)^2} & \frac{q c_v a^2}{\left(c_v \mu-\nu\right)^2}\\
				0 & 0 & 0 & 0
			\end{array}\right) \right.\nonumber\\
			&\left.-\partial_{xx}\left(\frac{1}{\sqrt{4\pi Dt}}e^{\frac{-(x-y)^2}{4 Dt}}\right)\left(\begin{array}{cccc}
				0 & 0 & 0 & 0\\
				0 & 0 & 0 & 0\\
				0 & 0 & 0 & q\\
				0 & 0 & 0 & 1
			\end{array}\right)\right\rvert \nonumber\\
			&\leq O(1) e^{-\sigma_0^* t-\sigma_0|x-y|}+O(1) t e^{-\sigma_0|x-y|}, \quad t \leq 1 .
		\end{align}

		On the other hand, for large time $t\geq 1$ it holds that
		\begin{align}
			& \left\lvert\, \mathbb{G}(x-y, t)-e^{-\frac{a}{\mu} t} \delta(x-y)\left(\begin{array}{cccc}
				1 & 0 & 0 & 0\\
				0 & 0 & 0 & 0\\
				0 & 0 & 0 & 0\\
				0 & 0 & 0 & 0
			\end{array}\right)\right. \nonumber\\
			& \left.-\sum_{j=1}^4 \frac{e^{-\frac{\left(x-y+\beta_j t\right)^2}{4 \alpha_j t}}}{2 \sqrt{\pi \alpha_j t}} M_j^0-\sum_{j=1}^3 \partial_x\left(\frac{e^{-\frac{\left(x-y+\beta_j t\right)^2}{4 \alpha_j t}}}{2 \sqrt{\pi \alpha_j t}}\right) M_j^1 \right\rvert\, \nonumber\\
			& \leq \sum_{j=1}^3 \frac{O(1) e^{-\frac{\left(x-y+\beta_j t\right)^2}{4 C t}}}{t} M_j^0+\sum_{j=1}^3 \frac{O(1) e^{-\frac{\left(x-y+\beta_j t\right)^2}{4 C t}}}{t^{\frac{3}{2}}}+O(1) e^{-\sigma_0^* t-\sigma_0|x-y|}, \quad 1 \leq t, \nonumber\\
			& \left\lvert\, \mathbb{G}_x(x-y, t)-e^{-\frac{a}{\mu} t} \delta^{\prime}(x-y)\left(\begin{array}{cccc}
				1 & 0 & 0 & 0\\
				0 & 0 & 0 & 0\\
				0 & 0 & 0 & 0\\
				0 & 0 & 0 & 0
			\end{array}\right)+e^{-\frac{a}{\mu} t} \delta(x-y)\left(\begin{array}{cccc}
				0 & \frac{1}{\mu} & 0 & 0\\
				\frac{K}{\mu} & 0 & 0 & 0\\
				0 & 0 & 0 & 0\\
				0 & 0 & 0 & 0
			\end{array}\right)\right. \nonumber\\
			& \left.-\sum_{j=1}^4 \partial_x\left(\frac{e^{-\frac{\left(x-y+\beta_j t\right)^2}{4 \alpha_j t}}}{2 \sqrt{\pi \alpha_j t}}\right) M_j^0-\sum_{j=1}^3 \partial_x^2\left(\frac{e^{-\frac{\left(x-y+\beta_j t\right)^2}{4 \alpha_j t}}}{2 \sqrt{\pi \alpha_j t}}\right) M_j^1 \right\rvert\, \nonumber\\
			& \leq \sum_{j=1}^3 \frac{O(1) e^{-\frac{\left(x-y+\beta_j t\right)^2}{4 C t}}}{t^{\frac{3}{2}}} M_j^0+\sum_{j=1}^3 \frac{O(1) e^{-\frac{\left(x-y+\beta_j t\right)^2}{4 C t}}}{t^2}+O(1) e^{-\sigma_0^* t-\sigma_0|x-y|}, \quad 1 \leq t, \nonumber\\
			& \left\lvert\, \mathbb{G}_{x x}(x-y, t)-e^{-\frac{a}{\mu} t} \delta^{\prime \prime}(x-y)\left(\begin{array}{cccc}
				1 & 0 & 0 & 0\\
				0 & 0 & 0 & 0\\
				0 & 0 & 0 & 0\\
				0 & 0 & 0 & 0
			\end{array}\right)+e^{-\frac{a}{\mu} t} \delta^{\prime}(x-y)\left(\begin{array}{cccc}
				0 & \frac{1}{\mu} & 0 & 0\\
				\frac{K}{\mu} & 0 & 0 & 0\\
				0 & 0 & 0 & 0\\
				0 & 0 & 0 & 0
			\end{array}\right)\right. \nonumber\\
			& -e^{-\frac{a}{\mu} t} \delta(x-y)\left(\left(\begin{array}{cccc}
				\frac{a}{\mu^2} & 0 & \frac{a}{\nu \mu} & 0\\
				0 & -\frac{a}{\mu^2} & 0 & 0\\
				\frac{c_v a^2}{\nu \mu} & 0 & 0 & 0\\
				0 & 0 & 0 & 0
			\end{array}\right)-t\left(\begin{array}{cccc}
				-\frac{\left(\nu a^2-\mu a^3\right)}{\nu \mu^3} & 0 & 0 & 0\\
				0 & 0 & 0 & 0\\
				0 & 0 & 0 & 0\\
				0 & 0 & 0 & 0
			\end{array}\right)\right)\nonumber \\
			& \left.-\sum_{j=1}^4 \partial_x^2\left(\frac{e^{-\frac{\left(x-y+\beta_j t\right)^2}{4 \alpha_j t}}}{2 \sqrt{\pi \alpha_j t}}\right) M_j^0-\sum_{j=1}^3 \partial_x^3\left(\frac{e^{-\frac{\left(x-y+\beta_j t\right)^2}{4 \alpha_j t}}}{2 \sqrt{\pi \alpha_j t}}\right) M_j^1 \right\rvert\,\nonumber \\
			& \leq \sum_{j=1}^3 \frac{O(1) e^{-\frac{\left(x-y+\beta_j t\right)^2}{4 C t}}}{t^2} M_j^0+\sum_{j=1}^3 \frac{O(1) e^{-\frac{\left(x-y+\beta_j t\right)^2}{4 C t}}}{t^{\frac{5}{2}}}+O(1) e^{-\sigma_0^* t-\sigma_0|x-y|}, \quad 1 \leq t.
		\end{align}
	\end{lemma}
	\begin{remark}
		\label{re: nu scale}
		The above result follows from Remark 4.2 in \cite{WangHT2021}. By rescaling time  $t\mapsto \frac{t}{\nu_0}$ for $\nu_0>0$, we may obtain analogous estimates for  $t\geq \nu_0$. The only modification is that all the $O(1)$ terms now depend on $\nu_0$ in the rescaling.
	\end{remark}
	With the estimates for the Green's function, we proceed to derive some \emph{a priori} estimates for the solutions constructed in Theorem \ref{thm: local exi} and Theorem \ref{thm: local regularity}.
	Based on the results in \cite{WangHT2021}, we already have the following estimates
	\begin{lemma}
		\label{lemma: es u large t}
		Let $(v, u, \theta, z)$ be the local solution constructed in Theorem~\ref{thm: local exi}. Further assume that
		\begin{equation}
			\label{tau condition}
			\left\{\begin{array}{l}
				\left\|v_0-1\right\|_{B V}+\left\|u_0\right\|_{B V}+\left\|\theta_0-1\right\|_{B V}+\left\|v_0-1\right\|_{L_x^1}+\left\|u_0\right\|_{L_x^1}+\left\|\theta_0-1\right\|_{L_x^1}+\left\|z_0\right\|_{L_x^1}<\delta^*, \\
				\mathcal{G}(\tau)<\delta, \quad  \forall \tau<t, \\
				t_{\sharp} \geq 4 \nu_0,
			\end{array}\right.
		\end{equation}
		Then, for any $t>t_{\sharp}$ we have 
		\begin{equation*}
			\left\{\begin{array}{l}
				\|u(\cdot, t)\|_{L_x^1} \leq C(\nu_0) \delta^*+O(1)\left(\sqrt{\nu_0}\delta+\delta^2\right), \\
				\|\sqrt{1+t}\,u(\cdot, t)\|_{L_x^{\infty}} \leq C\left(\nu_0\right) \delta^*+O(1)\left(\sqrt{\nu_0}\delta+\delta^2\right),\\
				\|u_x(\cdot, t)\|_{L_x^1} \leq C(\nu_0) \delta^*+O(1)\left(\frac{|\log(\nu_0)|\delta^2}{\sqrt{\nu_0}}+\sqrt{\nu_0}\delta\right), \\
				\|\sqrt{t}\,u_x(\cdot, t)\|_{L_x^{\infty}} \leq C\left(\nu_0\right) \delta^*+O(1)\left(\frac{|\log(\nu_0)|\delta^2}{\sqrt{\nu_0}}+\sqrt{\nu_0}\delta\right).
			\end{array}\right.
		\end{equation*}
	\end{lemma}
	Using the above estimates for $u(x,t)$, we next derive estimates for $\theta(x,t)$ and $z(x,t)$. The argument for $\theta$ is similar to the proof of Lemma 4.7 in \cite{WangHT2021}, with an additional term $q z(x,t)$. Therefore, we estimate $\theta$ and $z$ together.
	\begin{lemma}
		\label{lemma: es z theta large t}
		Let $(v, u, \theta, z)$ be the local solution constructed in Theorem \ref{thm: local exi}, and assume the condition \eqref{tau condition}. 
		Then, for $t>t_{\sharp}$, the following estimates hold:
		\begin{equation*}
			\left\{\begin{array}{l}
				\left\|z(x,t)\right\|_{L_x^1}\leq \delta^*,\\
				\|\sqrt{1+t}\,z(\cdot, t)\|_{L_x^{\infty}} \leq C\left(\nu_0\right) \delta^*+O(1)\left(\delta+\delta^2\right),\\
				\|\theta(\cdot, t)-1\|_{L_x^1} \leq O(1)\left(C(\nu_0) \delta^*+\sqrt{\nu_0}\delta+\delta^2\right)+O(1)\left(C(\nu_0)\delta^*+\sqrt{\nu_0}\delta+\delta^2\right)^2, \\
				\|\sqrt{1+t}\,(\theta(\cdot, t)-1)\|_{L_x^{\infty}} \leq O(1)\left(C(\nu_0) \delta^*+\sqrt{\nu_0}\delta+\delta^2\right)+O(1)\left(C(\nu_0)\delta^*+\sqrt{\nu_0}\delta+\delta^2\right)^2.
			\end{array}\right.
		\end{equation*}
	\end{lemma}
	\begin{proof}
		By the integral representation of $z(x,t)$ given in Lemma \ref{lemma: z rep}, we have
		\begin{align*}
			z(x,t)
			=&\int_{\mathbb{R}} \mathbb{G}_{41}(x-y,t)(v(y,0)-1)\mathrm{d}y+\int_{\mathbb{R}} \mathbb{G}_{42}(x-y,t) u(y, 0)\,\mathrm{d}y\\
			&+\int_{\mathbb{R}} \mathbb{G}_{43}(x-y,t)(E(y,0)-c_v)\,\mathrm{d}y+\int_{\mathbb{R}} G_{44}(x,t;y,0)z(y,0)\,\mathrm{d}y\\
			&-\int_0^t\int_{\mathbb{R}} G_{44}(x,t;y,\tau)K\phi(\theta)z(y,\tau)\,\mathrm{d}y\,\mathrm{d}t+\sum_{i=1}^3 \mathcal{R}_i^{z},
		\end{align*}
		where $\mathcal{R}_i^z$ are given in Lemma \ref{lemma: z rep}. 
		Since the Green's function $\mathbb{G}$ has different estimates for short time $t\leq 1$ and large time $t\geq 1$, we split the proof into these two cases. However, the large-time estimates already incorporate the main ideas of the short-time analysis, so we only present the detailed proof for $t\geq 1$.

		In view of Lemma~\ref{lemma: G es}, for $t\geq 1$, $\mathbb{G}_{4 k}$ has the following estimates:
		\begin{align}
			\label{G4kx,tg1}
			& \left|\mathbb{G}_{4 k}(x, t)\right| \leq \sum_{j=1}^4 \frac{O(1) e^{-\frac{\left(x+\beta_j t\right)^2}{4 C t}}}{\sqrt{t}}+O(1) e^{-\sigma_0^* t-\sigma_0|x|}, \quad k=1,2,3,4, \nonumber\\
			& \left|\partial_x \mathbb{G}_{4k}(x, t)\right| \leq \sum_{j=1}^4 \frac{O(1) e^{-\frac{\left(x+\beta_j t\right)^2}{4 t}}}{t}+O(1) e^{-\sigma_0^* t-\sigma_0|x|}, \quad k=1,2,3,4,
		\end{align}
		and for $t\leq 1$
		\begin{align}
			\label{G4kx,tl1}
			& \left|\mathbb{G}_{4k}(x, t)\right| \leq O(1) e^{-\sigma_0^* t-\sigma_0 |x|}+O(1)t e^{-\sigma_0|x|}, \quad k=1,2,3, \nonumber\\
			& \left|\mathbb{G}_{44}(x, t)\right| \leq O(1) \frac{e^{-\frac{x^2}{4Dt}}}{\sqrt{4\pi D t}}+O(1) e^{-\sigma_0^* t-\sigma_0 |x|}+O(1)te^{-\sigma_0|x|},\nonumber\\
			&\left|\partial_x \mathbb{G}_{4k}(x, t)\right| \leq O(1) e^{-\sigma_0^* t-\sigma_0 |x|}+O(1)t e^{-\sigma_0|x|}, \quad k=1,2,3,\nonumber\\
			&\left|\partial_{xx} \mathbb{G}_{4k}(x, t)\right| \leq O(1) e^{-\sigma_0^* t-\sigma_0 |x|}+O(1)t e^{-\sigma_0|x|}, \quad k=1,2,3,\nonumber\\
			&\left|\partial_x \mathbb{G}_{44}(x,t)\right|\leq O(1) \left|\partial_x\left(\frac{e^{-\frac{x^2}{4Dt}}}{\sqrt{4\pi D t}}\right)\right|+ O(1) e^{-\sigma_0^* t-\sigma_0 |x|}+O(1)te^{-\sigma_0|x|}.
		\end{align}

         As per Remark~\ref{re: nu scale}, the $O(1)$ terms in the above estimates may depend on $\nu_0$.\\
		\noindent\textbf{Estimates of $\|z\|_{L_x^1}$:} The $L^1$ estimate is derived directly from the fourth equation of Eq.~\eqref{PDE,1}. Integrating both sides over $(0,t)\times \mathbb{R}$, we have
		\begin{align*}
			\int_{\mathbb{R}} z(x,t) \, \mathrm{d}x 
			+\int_0^t\int_{\mathbb{R}} K \phi(\theta) z(y,\tau) \, \mathrm{d}y\, \mathrm{d}\tau \leq \int_{\mathbb{R}} z_0(x) \, \mathrm{d}x.
		\end{align*}
		Under the initial condition $\|z_0\|_{L_x^1}<\delta^*$ given in \eqref{tau condition}, and since $K\phi(\theta) \geq 0$, it follows that
		\begin{align*}
			\|z(\cdot,t)\|_{L_x^1}\leq \delta^*,
		\end{align*}
		and 
		\begin{equation}
			\label{phiz Ltx}
			\int_0^t\int_{\mathbb{R}} K \phi(\theta) z(y,\tau) \, \mathrm{d}y\, \mathrm{d}\tau \leq \delta^*.
		\end{equation}
        Then, from the proof in \cite[Lemma 4.7]{WangHT2021} and above estimates, we infer that
		\begin{align*}
			\|\theta-1\|_{L_x^{1}}\leq& O(1)\left(C(\nu_0)\delta^*+\sqrt{\nu_0}\delta+\delta^2\right)+\frac{1}{2}\|u(\cdot,t)\|_{L_x^{\infty}}\|u(\cdot,x)\|_{L_x^1}+q\|z(\cdot,t)\|_{L_x^1}\\
			\leq &\left(C(\nu_0)\delta^*+\sqrt{\nu_0}\delta+\delta^2\right)+O(1)\frac{\left(C(\nu_0)\delta^*+\sqrt{\nu_0}\delta+\delta^2\right)^2}{\sqrt{1+t}}+O(1)\delta^*.
		\end{align*}

            We now make an important observation. Define $$M(\tau)=\int_{\mathbb{R}}K\phi(\theta)z(y,\tau)\,\mathrm{d}y.$$ Then
		\begin{align}
			\label{eq: dM}
			\frac{\mathrm{d}}{\mathrm{d}\tau}M(\tau)&=\int_{\mathbb{R}}K\phi^{\prime}(\theta)\theta_{\tau}z(y,\tau)\,\mathrm{d}y+\int_{\mathbb{R}}K\phi(\theta) z_{\tau}(y,\tau)\,\mathrm{d}y\nonumber\\
			&=\int_{\mathbb{R}}K\phi^{\prime}(\theta)\theta_{\tau}z(y,\tau)\,\mathrm{d}y+\int_{\mathbb{R}}K\phi(\theta) \left[\left(\frac{D}{v^2}z_y\right)_y-K\phi(\theta)z\right](y,\tau)\,\mathrm{d}y\nonumber\\
			&=\int_{\mathbb{R}}K\phi^{\prime}(\theta)\theta_{\tau}z(y,\tau)\,\mathrm{d}y-\int_{\mathbb{R}}K\phi^{\prime}(\theta)\theta_y \frac{D}{v^2}z_y\,\mathrm{d}y-\int_{\mathbb{R}}K^2\phi^{2}(\theta)z(y,\tau)\,\mathrm{d}y\nonumber\\
			&=:\mathcal{I}_1+\mathcal{I}_2+\mathcal{I}_3.
		\end{align}
	    By Lemma \ref{lemma: T+t}, we have the estimate
	    \begin{equation*}
	    	\left\|\theta_{\tau}(\cdot,\tau)\right\|_{L_x^1}\leq \frac{O(1)\delta}{\sqrt{\operatorname{min}(\tau,t_{\sharp})}},\quad \operatorname{max}(\tau,t_{\sharp})<t.
	    \end{equation*}
	    For $\tau\geq \nu_0$, as $\phi$ is Lipschitz, its derivative $\phi^{\prime}$ exists and is bounded a.e.. Hence,
	    \begin{align*}
	      &\left|\mathcal{I}_1\right|\leq O(1)\delta\left\|\theta_{\tau}(\cdot,\tau)\right\|_{L_x^1}\leq \frac{O(1)\delta^2}{\sqrt{\nu_0}},\\
	      &\left|\mathcal{I}_2\right|\leq O(1)\delta\left\|\theta_{y}(\cdot,\tau)\right\|_{L_x^{\infty}}\left\|z_{y}(\cdot,\tau)\right\|_{L_x^1}\leq \frac{O(1)\delta^2}{\sqrt{\tau}},\\
	      &\left|\mathcal{I}_3\right|\leq O(1)\delta.
	    \end{align*}
	    For sufficiently small $\delta\ll 1$, the leading-order term in \eqref{eq: dM} is $-O(\delta)$, while the higher-order term $O(\frac{\delta^2}{\sqrt{\nu_0}})$ is negligible. Thus, for all $\tau>\nu_0$, $M$ is monotone:
	    \begin{equation*}
	    	\frac{\mathrm{d}}{\mathrm{d}\tau}M(\tau)\leq 0.
	    \end{equation*}
	    It follows that
	    \begin{equation*}
	    	(\tau-\nu_0)M(\tau)=\int_{\nu_0}^{\tau}M(\tau)\,\mathrm{d}s\leq \int_{\nu_0}^{\tau}M(s)\,\mathrm{d}s\leq \delta^*.
	    \end{equation*}
	    If $\tau-\nu_0\geq 1$, since $\frac{1}{\tau-\nu_0}\leq \frac{2}{1+\tau-\nu_0}$, we have
	    \begin{equation*}
	    	M(\tau)\leq\frac{\delta^*}{\tau-\nu_0}\leq \frac{O(1)\delta^*}{1+\tau},
	    \end{equation*}
	    If $0<\tau-\nu_0<1$, then
	    \begin{equation*}
	    	M(\tau)\leq M(\nu_0)\leq \frac{2M(\nu_0)}{1+\tau-\nu_0}\leq \frac{O(1)\delta^*}{1+\tau}.
	    \end{equation*}
	    Therefore, we conclude for all $\nu_0<\tau<t$ that
	    \begin{equation}
	    	\label{es: M}
	    	M(\tau)\leq \frac{O(1)\delta^*}{1+\tau}.
	    \end{equation}
		
		\noindent \textbf{Estimate of $\left\|z(\cdot,t)\right\|_{L_x^{\infty}}$:}
		Since $t\geq t_{\sharp}\geq 4 \nu_0$, using estimates \eqref{G4kx,tg1} and the smallness condition \eqref{tau condition}, we have
		\begin{align*}
			\left|\int_{\mathbb{R}}\mathbb{G}_{41}(x-y,t)\left(v(y,0)-1\right)\,\mathrm{d}y\right| \leq & C(\nu_0)\int_{\mathbb{R}}\sum_{j=1}^4\frac{e^{-\frac{\left(x-y+\beta_j t\right)^2}{4 C t}}}{\sqrt{t}}\left|v(y,0)-1\right|\,\mathrm{d}y\\
			&+ C(\nu_0) \int_{\mathbb{R}}e^{-\sigma_0^* t-\sigma_0|x-y|}\left|v(y,0)-1\right|\,\mathrm{d}y\\
			\leq &C(\nu_0)\frac{\delta^*}{\sqrt{t}}\leq C(\nu_0)\frac{\delta^*}{\sqrt{1+t}}.
		\end{align*}
		By repeating this procedure, we derive similar bounds for the remaining initial integral terms
		\begin{align*}
			&\left|\int_{\mathbb{R}}\mathbb{G}_{42}(x-y,t)u(y,0)\,\mathrm{d}y\right| \leq C(\nu_0)\frac{\delta^*}{\sqrt{1+t}},\\
			&\left|\int_{\mathbb{R}}\mathbb{G}_{43}(x-y,t)\left(E(y,0)-c_v\right)\,\mathrm{d}y\right| \leq C(\nu_0)\frac{\delta^*}{\sqrt{1+t}},\\
			&\left|\int_{\mathbb{R}}G_{44}(x,t;y,0)z(y,0)\,\mathrm{d}y\right|=\left|\int_{\mathbb{R}}\mathbb{G}_{44}(x-y,t)z(y,0)\,\mathrm{d}y\right|\leq C(\nu_0)\frac{\delta^*}{\sqrt{1+t}}.
		\end{align*}
		Next, we estimate the integral remainder term. Using the definition of $\mathcal{X}(\frac{t-\tau}{\nu_0})$, we write
		\begin{align*}
			&\int_0^t\int_{\mathbb{R}} G_{44}(x,t;y,\tau)K\phi(\theta)z(y,\tau) \,\mathrm{d}y\,\mathrm{d}\tau\\=&\int_0^t\int_{\mathbb{R}} \mathcal{X}\left(\frac{t-\tau}{\nu_0}\right)H(x,t;y,\tau;\frac{D}{v^2})K\phi(\theta)z(y,\tau)\,\mathrm{d}y\,\mathrm{d}\tau\\
			&+\int_0^t\int_{\mathbb{R}}\left(1- \mathcal{X}\left(\frac{t-\tau}{\nu_0}\right)\right) \mathbb{G}_{44}(x-y,t-\tau)K\phi(\theta)z(y,\tau)\,\mathrm{d}y\,\mathrm{d}\tau\\
			=&\int_{t-2\nu_0}^{t-\nu_0}\int_{\mathbb{R}} \mathcal{X}\left(\frac{t-\tau}{\nu_0}\right)H(x,t;y,\tau;\frac{D}{v^2})K\phi(\theta)z(y,\tau)\,\mathrm{d}y\,\mathrm{d}\tau\\
			&+\int_{t-\nu_0}^{t}\int_{\mathbb{R}} H(x,t;y,\tau;\frac{D}{v^2})K\phi(\theta)z(y,\tau) \,\mathrm{d}y\,\mathrm{d}\tau\\
			&+\int_{0}^{t-2\nu_0}\int_{\mathbb{R}} \mathbb{G}_{44}(x-y,t-\tau)K\phi(\theta)z(y,\tau) \,\mathrm{d}y\,\mathrm{d}\tau\\
			&+\int_{t-2\nu_0}^{t-\nu_0}\int_{\mathbb{R}}\left(1- \mathcal{X}\left(\frac{t-\tau}{\nu_0}\right)\right) \mathbb{G}_{44}(x-y,t-\tau)K\phi(\theta)z(y,\tau) \,\mathrm{d}y\,\mathrm{d}\tau.
		\end{align*}
		To apply the Green's function estimates, we further split the interval ($0,t-2\nu_0$) into ($0,\frac{t-1+2\nu_0}{2}$), ($\frac{t-1+2\nu_0}{2},t-1$) and ($t-1,t-2\nu_0$). Using the estimates for $H$ in Lemma \ref{lemma: Liu}, the bounds for $\mathbb{G}_{44}$ \eqref{G4kx,tg1}\eqref{G4kx,tl1}, and $\mathcal{G}(\tau)<\delta$, each integral part can be bounded as follows
		\begin{align*}
			&\left|\int_0^t\int_{\mathbb{R}} G_{44}(x,t;y,\tau)K\phi(\theta)z(y,\tau)\,\mathrm{d}y\,\mathrm{d}\tau\right|\\ 
			\leq&O(1)\int_{t-2\nu_0}^{t-\nu_0}\int_{\mathbb{R}} \mathcal{X}\left(\frac{t-\tau}{\nu_0}\right)\left|H(x,t;y,\tau;\frac{D}{v^2})\right|z(y,\tau)\,\mathrm{d}y\,\mathrm{d}\tau\\
			&+O(1)\int_{t-\nu_0}^{t}\int_{\mathbb{R}} \left|H(x,t;y,\tau;\frac{D}{v^2})\right|z(y,\tau) \,\mathrm{d}y\,\mathrm{d}\tau\\
			&+O(1)\int_{0}^{\frac{t-1+2\nu_0}{2}}\int_{\mathbb{R}} \left|\mathbb{G}_{44}(x-y,t-\tau)\right|K\phi(\theta)z(y,\tau) \,\mathrm{d}y\,\mathrm{d}\tau\\
			&+O(1)\int_{\frac{t-1+2\nu_0}{2}}^{t-1}\int_{\mathbb{R}} \left|\mathbb{G}_{44}(x-y,t-\tau)\right|K\phi(\theta)z(y,\tau) \,\mathrm{d}y\,\mathrm{d}\tau\\
			&+O(1)\int_{t-1}^{t-2\nu_0}\int_{\mathbb{R}} \left|\mathbb{G}_{44}(x-y,t-\tau)\right|z(y,\tau)\,\mathrm{d}y\,\mathrm{d}\tau\\
			&+O(1)\int_{t-2\nu_0}^{t-\nu_0}\int_{\mathbb{R}}\left(1- \mathcal{X}\left(\frac{t-\tau}{\nu_0}\right)\right) \left|\mathbb{G}_{44}(x-y,t-\tau)\right|z(y,\tau) \,\mathrm{d}y\,\mathrm{d}\tau\\
			\leq & O(1)\int_{t-2\nu_0}^{t-\nu_0}\int_{\mathbb{R}}\frac{e^{\frac{-(x-y)^2}{C_*(t-\tau)}}}{\sqrt{t-\tau}}\frac{\delta}{\sqrt{1+\tau}}\,\mathrm{d}y\,\mathrm{d}\tau +O(1)\int_{t-\nu_0}^{t}\int_{\mathbb{R}}\frac{e^{\frac{-(x-y)^2}{C_*(t-\tau)}}}{\sqrt{t-\tau}}\frac{\delta}{\sqrt{1+\tau}}\,\mathrm{d}y\,\mathrm{d}\tau\\
			&+O(1)\int_{0}^{\frac{t-1+2\nu_0}{2}}\int_{\mathbb{R}} \left(\sum_{j=1}^4 \frac{O(1) e^{-\frac{\left(x-y+\beta_j(t-\tau)\right)^2}{4 C (t-\tau)}}}{\sqrt{t-\tau}}+ e^{-\sigma_0^* (t-\tau)-\sigma_0|x-y|}\right)z(y,\tau)\,\mathrm{d}y\,\mathrm{d}\tau\\
			&+O(1)\int_{\frac{t-1+2\nu_0}{2}}^{t-1}\int_{\mathbb{R}} \left(\sum_{j=1}^4 \frac{O(1) e^{-\frac{\left(x-y+\beta_j (t-\tau)\right)^2}{4 C (t-\tau)}}}{\sqrt{t-\tau}}+ e^{-\sigma_0^* (t-\tau)-\sigma_0|x-y|}\right)K\phi(\theta)z(y,\tau) \,\mathrm{d}y\,\mathrm{d}\tau\\
			&+O(1)\int_{t-1}^{t-\nu_0}\int_{\mathbb{R}} \left(\frac{e^{-\frac{(x-y)^2}{4D(t-\tau)}}}{\sqrt{4\pi D(t-\tau)}}+ e^{-\sigma_0^*(t-\tau)-\sigma_0 |x-y|}+ (t-\tau)e^{-\sigma_0|x-y|}\right)z(y,\tau)\,\mathrm{d}y\,\mathrm{d}\tau\\
			\leq &O(1)\int_{t-2\nu_0}^{t}\frac{\delta}{\sqrt{1+\tau}}\,\mathrm{d}\tau+O(1)\int_{0}^{\frac{t-1+2\nu_0}{2}} \left(\frac{1}{\sqrt{t-\tau}}+e^{-\sigma_0^* (t-\tau)}\right)M(\tau)\,\mathrm{d}\tau\\
			&+O(1)\int_{\frac{t-1+2\nu_0}{2}}^{t-1} \left(\frac{1}{\sqrt{t-\tau}}+e^{-\sigma_0^* (t-\tau)}\right)M(\tau)\,\mathrm{d}\tau+O(1)\int_{t-1}^{t-\nu_0}\left(1+e^{-\sigma_0^* (t-\tau)}\right)\frac{\delta}{\sqrt{1+\tau}}\,\mathrm{d}\tau\\
			\leq &O(1)\frac{\nu_0\delta}{\sqrt{1+t}}+\frac{O(1)}{\sqrt{1+t-2\nu_0}}\int_{0}^{\frac{t-1+2\nu_0}{2}} M(\tau)\,\mathrm{d}\tau+O(1)\int_{\frac{t-1+2\nu_0}{2}}^{t-1} \frac{1}{\sqrt{t-\tau}}\frac{\delta^*}{1+\tau}\,\mathrm{d}\tau\\
			\leq &O(1)\left(\frac{\nu_0\delta}{\sqrt{1+t}}+\frac{\delta}{\sqrt{1+t}}\right),
		\end{align*}
		in which we use the estimates $\int_{0}^{t}M(\tau)d\tau<\delta^*$ \eqref{phiz Ltx} and $M(\tau)\leq \frac{O(1)\delta^*}{1+\tau}$ \eqref{es: M}.\\
        
		Next, we estimate the remainder terms $\mathcal{R}_i^z$. We split the time integral at $\tau=\frac{t-1}{2}$, $\tau=t-1$ and $\tau=t-2\nu_0$ to apply derivative estimates separately to the short and long time scales. Using the bounds for $\partial_y \mathbb{G}_{4k}$ from \eqref{G4kx,tg1}\eqref{G4kx,tl1} and \emph{a priori} bounds $\|v-1\|_{L^\infty}, \|u\|_{L^\infty}, \dots, \|z_y\|_{L^\infty} \leq \frac{\delta}{\sqrt{1+\tau}}$ from $\mathcal{G}(\tau)<\delta$, we find
		\begin{align*}
			\left|\mathcal{R}_1^z\right|\leq & \left|\int_0^{\frac{t-1}{2}} \int_{\mathbb{R}} \partial_y \mathbb{G}_{42}(x-y, t-\tau)\left[\frac{a(v-1)^2}{v}+\frac{a(\theta-1)(1-v)}{v}-\frac{a u^2}{2 c_v}+\frac{\mu u_y(v-1)}{v}\right]\,\mathrm{d}y\,\mathrm{d}\tau\right| \\
			& +\left|\int_0^{\frac{t-1}{2}} \int_{\mathbb{R}} \partial_y \mathbb{G}_{43}(x-y, t-\tau)\left[\left(\frac{a(\theta-1)+a(1-v)}{v}\right) u+\frac{\nu \theta_y(v-1)}{v}+\left(\frac{\nu}{c_v}-\frac{\mu}{v}\right) u u_y \right.\right.\\
			&\left.\left.+\frac{qD(v^2-1)}{v^2}z_y\right]\,\mathrm{d}y\,\mathrm{d}\tau\right| +\left|\int_0^{\frac{t-1}{2}} \int_{\mathbb{R}} \partial_y \mathbb{G}_{44}(x-y, t-\tau)\frac{D(v^2-1)}{v^2}z_y\,\mathrm{d}y\,\mathrm{d}\tau\right|\\
			& +\left|\int_{\frac{t-1}{2}}^{t-1} \int_{\mathbb{R}} \partial_y \mathbb{G}_{42}(x-y, t-\tau)\left(\frac{a(v-1)^2}{v}+\frac{a(\theta-1)(1-v)}{v}-\frac{a u^2}{2 c_v}+\frac{\mu u_y(v-1)}{v}\right)\,\mathrm{d}y \,\mathrm{d}\tau\right| \\
			& +\left|\int_{\frac{t-1}{2}}^{t-1} \int_{\mathbb{R}} \partial_y \mathbb{G}_{43}(x-y, t-\tau)\left[\left(\frac{a(\theta-1)+a(1-v)}{v}\right) u+\frac{\nu \theta_y(v-1)}{v}+\left(\frac{\nu}{c_v}-\frac{\mu}{v}\right) u u_y \right.\right.\\
			&\left.\left.+\frac{qD(v^2-1)}{v^2}z_y\right]\,\mathrm{d}y\,\mathrm{d}\tau\right| +\left|\int_{\frac{t-1}{2}}^{t-1} \int_{\mathbb{R}} \partial_y \mathbb{G}_{44}(x-y, t-\tau)\frac{D(v^2-1)}{v^2}z_y\,\mathrm{d}y\,\mathrm{d}\tau\right|\\
			&+\left|\int_{t-1}^{t-2 \nu_0} \int_{\mathbb{R}} \partial_y \mathbb{G}_{42}(x-y, t-\tau)\left(\frac{a(v-1)^2}{v}+\frac{a(\theta-1)(1-v)}{v}-\frac{a u^2}{2 c_v}+\frac{\mu u_y(v-1)}{v}\right)\,\mathrm{d}y\,\mathrm{d}\tau\right| \\
			& +\left|\int_{t-1}^{t-2\nu_0} \int_{\mathbb{R}} \partial_y \mathbb{G}_{43}(x-y, t-\tau)\left[\left(\frac{a(\theta-1)+a(1-v)}{v}\right) u+\frac{\nu \theta_y(v-1)}{v}+\left(\frac{\nu}{c_v}-\frac{\mu}{v}\right) u u_y \right.\right.\\
			&\left.\left.+\frac{qD(v^2-1)}{v^2}z_y\right]\,\mathrm{d}y\,\mathrm{d}\tau\right| +\left|\int_{t-1}^{t-2\nu_0} \int_{\mathbb{R}} \partial_y \mathbb{G}_{44}(x-y, t-\tau)\frac{D(v^2-1)}{v^2}z_y\,\mathrm{d}y\,\mathrm{d}\tau\right|\\
			\leq & O(1)\int_0^{\frac{t-1}{2}}\int_{\mathbb{R}}\left(\sum_{j=1}^4 \frac{O(1) e^{-\frac{\left(x-y+\beta_j (t-\tau)\right)^2}{4 (t-\tau)}}}{t-\tau}+C(\nu_0) e^{-\sigma_0^*(t-\tau)-\sigma_0|x-y|}\right) \frac{\delta}{\sqrt{1+\tau}}\left(|v-1|\right.\\
			&\left.\qquad+|\theta-1|+|u|+|u_y|+|\theta_y|+|z_y|\right)\,\mathrm{d}y\,\mathrm{d}\tau\\
			&+O(1) \int_{\frac{t-1}{2}}^{t-1}\int_{\mathbb{R}} \left(\sum_{j=1}^4 \frac{O(1) e^{-\frac{\left(x-y+\beta_j (t-\tau)\right)^2}{4 (t-\tau)}}}{t-\tau}+C(\nu_0) e^{-\sigma_0^*(t-\tau)-\sigma_0|x-y|}\right) \frac{\delta}{\sqrt{1+\tau}}\left(|v-1|\right.\\
			&\left.\qquad+|\theta-1|+|u|+|u_y|+|\theta_y|+|z_y|\right) \,\mathrm{d}y\,\mathrm{d}\tau\\
			&+O(1) \int_{t-1}^{t-2 \nu_0} \int_{\mathbb{R}}\left(e^{-\sigma_0^*(t-\tau)-\sigma_0 |x-y|}+(t-\tau) e^{-\sigma_0|x-y|}\right)\frac{\delta}{\sqrt{1+\tau} }\left(|v-1|+|\theta-1|+|u|\right.\\
			&\left.\qquad+|u_y|+|\theta_y|+|z_y|\right)\,\mathrm{d}y \,\mathrm{d}\tau \\
			&+O(1) \int_{t-1}^{t-2 \nu_0}\int_{\mathbb{R}} \left(\frac{e^{-\frac{(x-y)^2}{4D(t-\tau)}}}{t-\tau}+e^{-\sigma_0^*(t-\tau)-\sigma_0 |x-y|}+O(1)(t-\tau)e^{-\sigma_0|x-y|}\right)\frac{\delta}{\sqrt{1+\tau} }|z_y|\,\mathrm{d}y\,\mathrm{d}\tau \\
			\leq & O(1) \int_0^{\frac{t-1}{2}}\left(\frac{1}{t-\tau}+e^{-\sigma_0^* (t-\tau)}\right) \frac{\delta^2}{\sqrt{1+\tau}} \,\mathrm{d}\tau+O(1) \int_{\frac{t-1}{2}}^{t-1}\frac{1}{\sqrt{t-\tau}} \frac{\delta^2}{\sqrt{1+\tau}\sqrt{\tau}} \,\mathrm{d}\tau\\
			&+O(1) \int_{\frac{t-1}{2}}^{t-1} e^{-\sigma_0^* (t-\tau)} \frac{\delta^2}{\sqrt{1+\tau}} \,\mathrm{d}\tau+O(1)\int_{t-1}^{t-2 \nu_0}\frac{1}{\sqrt{t-\tau}}\frac{\delta^2}{\sqrt{1+\tau} \sqrt{\tau}} \,\mathrm{d}\tau \\
			&+O(1)\int_{t-1}^{t-2 \nu_0} e^{-\sigma_0^* (t-\tau)}\frac{\delta^2}{\sqrt{1+\tau}}\,\mathrm{d}\tau\\
			\leq & O(1) \frac{\delta^2}{\sqrt{1+t}} .
		\end{align*} 
		For a more accurate estimate, since 
		$$\partial_\tau \mathbb{G}_{44}-\frac{qa}{c_v} \partial_y \mathbb{G}_{42}+\left(-\frac{q\nu}{c_v}+qD\right)\partial_y^2 \mathbb{G}_{43}+D\partial_y^2 \mathbb{G}_{44}=0,$$
		with the initial data $\mathbb{G}_{44}(x,0)=\delta(x)$, we have the explicit representation of $\mathbb{G}_{44}$
		\begin{align}
			\label{G44 rep}
			\mathbb{G}_{44}(x-y,t-\tau)=&H(x,t;y,\tau;D)-\int_{\tau}^t\int_{\mathbb{R}}H(\omega,s;y,\tau;D)\\
			&\qquad \left(\frac{qa}{c_v}\partial_{\omega}\mathbb{G}_{42}(x-\omega,t-s)+\left(\frac{q\nu}{c_v}-qD\right)\partial_{\omega}^2\mathbb{G}_{43}(x-\omega,t-s)\right)\,\mathrm{d}\omega\,\mathrm{d}s.\nonumber
		\end{align}
		Therefore, we substitute this into $\mathcal{R}_2^z$ and then follow the same strategy as for $\mathcal{R}_1^z$. Additionally, we use the comparison estimates for $H(x,t;y,\tau;D)-H(x,t;y,\tau;\frac{D}{v^2})$ in Lemma \ref{lemma: comparison} to obtain
		\begin{align*}
			&\left|\mathcal{R}_2^z(x, t)\right|  \\
			\leq & \left|\int_{t-2 \nu_0}^{t-\nu_0} \int_{\mathbb{R}} \partial_y \mathbb{G}_{42}(x-y, t-\tau)\left[\frac{a(v-1)^2}{v}+\frac{a(\theta-1)(1-v)}{v}-\frac{au^2}{2c_v}+\frac{\mu u_y(v-1)}{v}\right.\right.\\
			&\left.\left.\qquad-\mathcal{X}(\frac{t-\tau}{\nu_0})\frac{qa}{c_v}z\right](y,\tau)\,\mathrm{d}y\,\mathrm{d}\tau\right| \\
			&+\left|\int_{t-2 \nu_0}^{t-\nu_0} \int_{\mathbb{R}} \partial_y\mathbb{G}_{43}(x-y;t-\tau) \left[\frac{a(\theta-1)+a(1-v)}{v} u+\frac{\nu(v-1)}{v}\theta_y+\left(\frac{\nu}{c_v v}-\frac{\mu}{v}\right)uu_y\right.\right.\\
			&\left.\left.+\frac{qD(v^2-1)}{v^2}z_y\right](y,\tau)\,\mathrm{d}y\,\mathrm{d}\tau\right|+\left|\int_{t-2 \nu_0}^{t-\nu_0} \int_{\mathbb{R}} \partial_y\mathbb{G}_{43}(x-y;t-\tau)\mathcal{X}(\frac{t-\tau}{\nu_0})\left(\frac{q\nu}{c_v}-qD\right)z_y\,\mathrm{d}y\,\mathrm{d}\tau\right|\\
			&+\left|\int_{t-2 \nu_0}^{t-\nu_0} \int_{\mathbb{R}}\partial_y \mathbb{G}_{44}(x-y;t-\tau) \frac{D(v^2-1)}{v^2}z_y(y,\tau)\,\mathrm{d}y\,\mathrm{d}\tau\right| \\
			&+\int_{t-2 \nu_0}^{t-\nu_0} \int_{\mathbb{R}} \frac{1}{\nu_0}\left|H(x,t;y,\tau;D)-H(x,t;y,\tau;\frac{D}{v^2}) \right|z(y,\tau)\,\mathrm{d}y\,\mathrm{d} \tau\\
			& +\int_{t-2 \nu_0}^{t-\nu_0} \int_{\mathbb{R}} \frac{1}{\nu_0}\left|\int_\tau^t \int_{\mathbb{R}} H(\omega,s;y, \tau;D)\frac{qa}{c_v}\partial_{\omega}\mathbb{G}_{42}(x-\omega,t-s) \,\mathrm{d}\omega \mathrm{d}s\right|z(y,\tau)\,\mathrm{d}y\,\mathrm{d} \tau \\
			& +\int_{t-2 \nu_0}^{t-\nu_0} \int_{\mathbb{R}} \frac{1}{\nu_0}\left|\int_\tau^t \int_{\mathbb{R}} H(\omega,s;y, \tau;D)\left(\frac{q\nu}{c_v}-qD\right)\partial_{\omega}^2\mathbb{G}_{43}(x-\omega,t-s)\,\mathrm{d}\omega\,\mathrm{d}s\right|z(y, \tau)\,\mathrm{d}y \,\mathrm{d}\tau\\
			\leq & O(1) \int_{t-2\nu_0}^{t-\nu_0} \int_{\mathbb{R}}\left(e^{-\sigma_0^*(t-\tau)-\sigma_0 |x-y|}+(t-\tau) e^{-\sigma_0|x-y|}\right)\left[\frac{\delta}{\sqrt{1+\tau} }\left(|v-1|+|\theta-1|+|u|+|u_y|\right.\right.\\
			&\left.\left.\qquad+|\theta_y|+|z_y|\right)+z\right]\,\mathrm{d}y \,\mathrm{d}\tau \\
			&+O(1) \int_{t-2\nu_0}^{t-\nu_0}\int_{\mathbb{R}} \left(\frac{e^{-\frac{(x-y)^2}{4D(t-\tau)}}}{t-\tau}+e^{-\sigma_0^*(t-\tau)-\sigma_0 |x-y|}+O(1)(t-\tau)e^{-\sigma_0|x-y|}\right)\frac{\delta}{\sqrt{1+\tau} }|z_y|\,\mathrm{d}y\,\mathrm{d}\tau\\
			&+O(1)\int_{t-2\nu_0}^{t-\nu_0}\int_{\mathbb{R}} \frac{1}{\nu_0}\left\|D-\frac{D}{v^2}\right\|_{L_x^{\infty}}\frac{e^{-\frac{(x-y)^2}{t-\tau}}}{\sqrt{t-\tau}} z(y,\tau)\,\mathrm{d}y\,\mathrm{d}\tau\\
			&+O(1)\int_{t-2\nu_0}^{t-\nu_0}\int_{\mathbb{R}} \frac{1}{\nu_0}\int_{\tau}^{t}\int_{\mathbb{R}}\frac{e^{-\frac{(\omega-y)^2}{s-\tau}}}{\sqrt{s-\tau}}\left(e^{-\sigma_0^*(t-s)-\sigma_0 |x-\omega|}+(t-s) e^{-\sigma_0|x-\omega|}\right) z(y,\tau)\,\mathrm{d}\omega\,\mathrm{d}s \,\mathrm{d}y\,\mathrm{d}\tau\\
			\leq &O(1)\int_{t-2\nu_0}^{t-\nu_0}\left(e^{-\sigma_0^*(t-\tau)}+(t-\tau)\right)\frac{\delta^2}{\sqrt{1+\tau}}\,\mathrm{d}\tau+O(1)\int_{t-2\nu_0}^{t-\nu_0}\left(e^{-\sigma_0^*(t-\tau)}+(t-\tau)\right)\frac{\delta^2}{\sqrt{1+\tau}}\,\mathrm{d}\tau\\
			&+O(1)\int_{t-2\nu_0}^{t-\nu_0}\left(\frac{1}{t-\tau}+e^{-\sigma_0^*(t-\tau)}+(t-\tau)\right)\frac{\delta^2}{\sqrt{1+\tau}}\,\mathrm{d}\tau\\
			&+O(1)\int_{t-2\nu_0}^{t-\nu_0}\frac{1}{\nu_0}\delta \frac{\delta}{\sqrt{1+\tau}}\,\mathrm{d}\tau+O(1)\int_{t-2\nu_0}^{t-\nu_0}\int_{\tau}^{t}\frac{1}{\nu_0}\left(e^{-\sigma_0^*(t-s)}+(t-s)\right)\frac{\delta}{\sqrt{1+\tau}}\,\mathrm{d}s\,\mathrm{d}\tau\\
			\leq &O(1)\left(\frac{\nu_0\delta^2}{\sqrt{1+t}}+\frac{\nu_0^2\delta^2}{\sqrt{1+t}}+\frac{\delta^2}{\sqrt{1+t}}+\frac{\delta}{\sqrt{1+t}}+\frac{\nu_0\delta}{\sqrt{1+t}}\right).
		\end{align*}
		For $\mathcal{R}_3^z$, we repeat the arguments for $\mathcal{R}_1^z$ to obtain
		\begin{align*}
			&\left|\mathcal{R}_3^{z}(x,t)\right|\\
			\leq&O(1)\int_{t-2\nu_0}^{t-\nu_0}\int_{\mathbb{R}}\left(e^{-\sigma_0^*(t-\tau)-\sigma_0|x-y|}+(t-\tau)e^{-\sigma_0|x-y|}\right)\frac{\delta}{\sqrt{1+\tau}}\left(|v-1|+|\theta-1|+|u|+|u_y|\right)\,\mathrm{d}y \,\mathrm{d}\tau\\
			&+O(1)\int_{t-2\nu_0}^{t-\nu_0}\int_{\mathbb{R}}\left(e^{-\sigma_0^*(t-\tau)-\sigma_0|x-y|}+(t-\tau)e^{-\sigma_0|x-y|}\right)z(y,\tau)\,\mathrm{d}y \,\mathrm{d}\tau\\
			\leq &O(1)\int_{t-2\nu_0}^{t-\nu_0}\left(e^{-\sigma_0^*(t-\tau)}+(t-\tau)\right)\frac{\delta^2}{\sqrt{1+\tau}}\,\mathrm{d}\tau+\int_{t-2\nu_0}^{t-\nu_0}\left(e^{-\sigma_0^*(t-\tau)}+(t-\tau)\right)\frac{\delta}{\sqrt{1+\tau}}\,\mathrm{d}\tau\\
			\leq &O(1)\left(\frac{\nu_0\delta^2}{\sqrt{1+t}}+\frac{\nu_0^2\delta^2}{\sqrt{1+t}}+\frac{\nu_0\delta}{\sqrt{1+t}}+\frac{\nu_0^2\delta}{\sqrt{1+t}}\right).
		\end{align*}
		Combining the above estimates, for sufficiently small $\delta$ and $t\geq t_{\sharp}\geq 4\nu_0$, we obtain that
		\begin{equation}
			\label{z infty}
			\left\|z(\cdot,t)\right\|_{L_x^{\infty}}\leq C(\nu_0)\frac{\delta^*}{\sqrt{1+t}}+ O(1)\frac{\delta+\delta^2}{\sqrt{1+t}}.
		\end{equation}
		\noindent\textbf{Estimate of $\left\|\theta(\cdot,t)\right\|_{L_x^{\infty}}$:} 
		Using the representation of $E(x,t)$ from \eqref{E rep}, we express $\theta$ as
		\begin{align*}
			c_v(\theta(x,t)-1)
			=&\int_{\mathbb{R}} \mathbb{G}_{31}(x-y,t)(v(y,0)-1)\mathrm{d}y+\int_{\mathbb{R}} \mathbb{G}_{32}(x-y,t) u(y, 0)\,\mathrm{d}y\nonumber\\
			&+\int_{\mathbb{R}} G_{33}(x,t;y,0)(E(y,0)-c_v)\,\mathrm{d}y+\sum_{i=1}^3 \mathcal{R}_i^{\theta}-\frac{(u(x,t))^2}{2}-qz(x,t)
		\end{align*}
		Following the proof of \cite[Lemma 4.7]{WangHT2021}, we have that
		\begin{align*}
			&\left|\int_{\mathbb{R}}\mathbb{G}_{31}(x-y,t)\left(v(y,0)-1\right)\,\mathrm{d}y\right|\leq C(\nu_0)\frac{\delta^*}{\sqrt{1+t}},\\
			&\left|\int_{\mathbb{R}}\mathbb{G}_{32}(x-y,t)u(y,0)\,\mathrm{d}y\right| \leq C(\nu_0)\frac{\delta^*}{\sqrt{1+t}},\\
			&\left|\int_{\mathbb{R}}G_{33}(x,t;y,0)\left(E(y,0)-c_v\right)\,\mathrm{d}y\right| \leq C(\nu_0)\frac{\delta^*}{\sqrt{1+t}}.
		\end{align*}
		For the remainder term $\mathcal{R}_i^{\theta}$, we follow a similar procedure as for the estimates of $z(x,t)$ and apply the bounds in \cite{WangHT2021}, obtaining that
		\begin{align*}
			&|\mathcal{R}_1^{\theta}|\leq O(1)\frac{\delta^2}{\sqrt{1+t}},\\
			&|\mathcal{R}_2^{\theta}|\leq O(1)(\sqrt{\nu_0}\delta+\delta^2)\frac{1}{\sqrt{1+t}},\\
			&|\mathcal{R}_3^{\theta}|\leq O(1)\frac{\sqrt{\nu_0}\delta}{\sqrt{1+t}}.
		\end{align*}
		Therefore, we combine the above estimates with $\|u(\cdot,t)\|_{L_x^{\infty}}$ in Lemma~\ref{lemma: es u large t} and the bound for $\|z(\cdot,t)\|_{L_x^{\infty}}$ to derive that
		\begin{align}
			\label{theta infty}
			\|\theta-1\|_{L_x^{\infty}}\leq& \frac{C(\nu_0)\delta^*}{\sqrt{1+t}}+O(1)\frac{\sqrt{\nu_0}\delta+\delta^2}{\sqrt{1+t}}+\frac{\|u(\cdot,t)\|_{L_x^{\infty}}^2}{2}+q\|z(\cdot,t)\|_{L_x^{\infty}}\nonumber\\
			\leq &\frac{C(\nu_0)\delta^*}{\sqrt{1+t}}+O(1)\frac{\sqrt{\nu_0}\delta+\delta^2}{\sqrt{1+t}}+\left(\frac{C(\nu_0) \delta^*}{\sqrt{1+t}}+O(1)\frac{\sqrt{\nu_0}\delta+\delta^2}{\sqrt{1+t}}\right)^2+\frac{C(\nu_0)\delta^*}{\sqrt{1+t}}+ O(1)\frac{\delta+\delta^2}{\sqrt{1+t}}\nonumber\\
			\leq &\frac{C(\nu_0)\delta^*}{\sqrt{1+t}}+O(1)\frac{\sqrt{\nu_0}\delta+\delta^2}{\sqrt{1+t}}+\left(\frac{C(\nu_0) \delta^*}{\sqrt{1+t}}+O(1)\frac{\sqrt{\nu_0}\delta+\delta^2}{\sqrt{1+t}}\right)^2.
		\end{align}		
	\end{proof}
	\begin{lemma}
		\label{lemma: es z_x theta_x large t}
		Let $(v, u, \theta, z)$ be the local solution constructed in Theorem \ref{thm: local exi}. Also assume the condition~\eqref{tau condition}.
		Then, for $t>t_{\sharp}$, $z(x,t)$ and $\theta(x,t)$ satisfy the following first-order estimates:
		\begin{equation*}
			\left\{\begin{array}{l}
				\left\|z_x(\cdot,t)\right\|_{L_x^{1}}\leq C(\nu_0)\delta^*+ O(1)\left(\frac{|\log(\nu_0)|\delta^2}{\sqrt{\nu_0}}+\sqrt{\nu_0}\delta\right),\\
				\left\|\sqrt{t}z_x(\cdot,t)\right\|_{L_x^{\infty}}\leq C(\nu_0)\delta^*+ O(1)\left(\frac{|\log(\nu_0)|\delta^2}{\sqrt{\nu_0}}+\sqrt{\nu_0}\delta\right),\\
				\|\theta_x(\cdot, t)\|_{L_x^1} \leq C(\nu_0) \delta^*+O(1)\left(\frac{|\log(\nu_0)|\delta^2}{\sqrt{\nu_0}}+\sqrt{\nu_0}\delta\right), \\
				\|\sqrt{t}\,\theta_x(\cdot, t)\|_{L_x^{\infty}} \leq C(\nu_0) \delta^*+O(1)\left(\frac{|\log(\nu_0)|\delta^2}{\sqrt{\nu_0}}+\sqrt{\nu_0}\delta\right).
			\end{array}\right.
		\end{equation*}
	\end{lemma}
	\begin{proof}
	From the representation formula~\eqref{z rep} for $z(x,t)$, we deduce that
    \begin{align}\label{z_x rep}
			z_x(x,t)
			=&\int_{\mathbb{R}} \partial_x\mathbb{G}_{41}(x-y,t)(v(y,0)-1)\mathrm{d}y+\int_{\mathbb{R}} \partial_x \mathbb{G}_{42}(x-y,t) u(y, 0)\,\mathrm{d}y\nonumber\\
			&+\int_{\mathbb{R}} \partial_x\mathbb{G}_{43}(x-y,t)(E(y,0)-c_v)\,\mathrm{d}y+\int_{\mathbb{R}} \partial_xG_{44}(x,t;y,0)z(y,0)\,\mathrm{d}y\nonumber\\
			&-\int_0^t\int_{\mathbb{R}} \partial_xG_{44}(x,t;y,\tau)K\phi(\theta)z(y,\tau)\,\mathrm{d}y\,\mathrm{d}\tau+\sum_{i=1}^3 \partial_x\mathcal{R}_i^{z},
		\end{align}
		where $\mathcal{R}_i^z$ defined as in Lemma \ref{lemma: z rep}. Recall the second-order derivative estimates of the Green's function from Lemma \ref{lemma: G es}. For large time $t\geq 1$, the second derivative $\partial_{xx}\mathbb{G}_{4 k}$ satisfies
		\begin{align}
			\label{G4kxx,tg1}
\left|\partial_{xx}\mathbb{G}_{4 k}(x, t)\right| &\leq \sum_{j=1}^4 \frac{O(1) e^{-\frac{\left(x+\beta_j t\right)^2}{4 C t}}}{t^{\frac{3}{2}}} +\sum_{j=1}^3 \frac{O(1) e^{-\frac{\left(x+\beta_j t\right)^2}{4 C t}}}{t^2}\nonumber\\& \qquad +\sum_{j=1}^3 \frac{O(1) e^{-\frac{\left(x+\beta_j t\right)^2}{4 C t}}}{t^{\frac{5}{2}}} + O(1) e^{-\sigma_0^* t-\sigma_0|x|}\nonumber\\
& \leq       \sum_{j=1}^4 \frac{O(1) e^{-\frac{\left(x+\beta_j t\right)^2}{4 C t}}}{t^{\frac{3}{2}}} + O(1) e^{-\sigma_0^* t-\sigma_0|x|} \qquad \text{for } k=1,2,3,4,
		\end{align}

		while for short time $t\leq 1$,
		\begin{align}
			\label{G4kxx,tl1}
			&\left|\partial_{xx} \mathbb{G}_{4k}(x, t)\right| \leq O(1) e^{-\sigma_0^* t-\sigma_0 |x|}+O(1)t e^{-\sigma_0|x|}, \quad k=1,2,3,\nonumber\\
			&\left|\partial_{xx} \mathbb{G}_{44}(x,t)\right|\leq O(1) \left|\partial_{xx}\left(\frac{e^{-\frac{x^2}{4Dt}}}{\sqrt{4\pi D t}}\right)\right|+ O(1) e^{-\sigma_0^* t-\sigma_0 |x|}+O(1)te^{-\sigma_0|x|}.
		\end{align}
		\noindent \textbf{Estimate of $\left\|z_x(\cdot,t)\right\|_{L_x^{\infty}}$:} For $t\geq t_{\sharp}\geq 4 \nu_0$, we start with the integral term involving the initial perturbation $v_0 - 1$. Using the large-time estimate of $\partial_x\mathbb{G}_{41}$ from \eqref{G4kx,tg1} and the initial smallness condition \eqref{tau condition}, we obtain that
		\begin{align*}
			\left|\int_{\mathbb{R}}\partial_x\mathbb{G}_{41}(x-y,t)\left(v(y,0)-1\right)\,\mathrm{d}y\right| \leq & C(\nu_0)\int_{\mathbb{R}}\sum_{j=1}^4\frac{e^{-\frac{\left(x-y+\beta_j t\right)^2}{4 C t}}}{t}\left|v(y,0)-1\right|\,\mathrm{d}y\\
			&+ C(\nu_0) \int_{\mathbb{R}}e^{-\sigma_0^* t-\sigma_0|x-y|}\left|v(y,0)-1\right|\,\mathrm{d}y\\
			\leq &C(\nu_0)\frac{\delta^*}{\sqrt{t}}.
		\end{align*}
		We also have
		\begin{align*}
			&\left|\int_{\mathbb{R}}\partial_x\mathbb{G}_{42}(x-y,t)u(y,0)\,\mathrm{d}y\right| \leq C(\nu_0)\frac{\delta^*}{\sqrt{t}},\\
			&\left|\int_{\mathbb{R}}\partial_x\mathbb{G}_{43}(x-y,t)\left(E(y,0)-c_v\right)\,\mathrm{d}y\right| \leq C(\nu_0)\frac{\delta^*}{\sqrt{t}},\\
			&\left|\int_{\mathbb{R}}\partial_x \mathbb{G}_{44}(x,t;y,0)z(y,0)\,\mathrm{d}y\right|=\left|\int_{\mathbb{R}}\partial_x\mathbb{G}_{44}(x-y,t)z(y,0)\,\mathrm{d}y\right|\leq C(\nu_0)\frac{\delta^*}{\sqrt{t}},
		\end{align*}
		following the similar argument as in Lemma~\ref{lemma: es z theta large t}.
		We split the interval ($0,t-2\nu_0$) into ($0,\frac{t-1+2\nu_0}{2}$), ($\frac{t-1+2\nu_0}{2},t-1$) and ($t-1,t-2\nu_0$) and use the Lipschitz continuity of $\phi$, the Green's function bounds in \eqref{G4kx,tg1} and \eqref{G4kx,tl1}, and \emph{a priori} bound $\mathcal{G}(\tau)<\delta$ from \eqref{tau condition}. In this way, we obtain that
		\begin{align*}
			&\left|\int_0^t\int_{\mathbb{R}}\partial_x G_{44}(x,t;y,\tau)K\phi(\theta)z(y,\tau)\,\mathrm{d}y\,\mathrm{d}\tau\right|\\ \leq&O(1)\int_{t-2\nu_0}^{t-\nu_0}\int_{\mathbb{R}} \mathcal{X}\left(\frac{t-\tau}{\nu_0}\right) \left|H_x(x,t;y,\tau;\frac{D}{v^2})\right|z(y,\tau)\,\mathrm{d}y\,\mathrm{d}\tau\\
			&+O(1)\int_{t-\nu_0}^{t}\int_{\mathbb{R}}  \left|H_x(x,t;y,\tau;\frac{D}{v^2})\right|z(y,\tau)\,\mathrm{d}y\,\mathrm{d}\tau\\
			&+O(1)\int_{0}^{\frac{t-1+2\nu_0}{2}}\int_{\mathbb{R}} \left|\partial_x\mathbb{G}_{44}(x-y,t-\tau)\right|K\phi(\theta)z(y,\tau)\,\mathrm{d}y\,\mathrm{d}\tau\\
			&+O(1)\int_{\frac{t-1+2\nu_0}{2}}^{t-1}\int_{\mathbb{R}} \left|\partial_x\mathbb{G}_{44}(x-y,t-\tau)\right|K\phi(\theta)z(y,\tau)\,\mathrm{d}y\,\mathrm{d}\tau\\
			&+O(1)\int_{t-1}^{t-2\nu_0}\int_{\mathbb{R}} \left|\partial_x\mathbb{G}_{44}(x-y,t-\tau)\right|z(y,\tau) \,\mathrm{d}y\,\mathrm{d}\tau\\
			&+O(1)\int_{t-2\nu_0}^{t-\nu_0}\int_{\mathbb{R}}\left(1- \mathcal{X}\left(\frac{t-\tau}{\nu_0}\right)\right) \left|\partial_x\mathbb{G}_{44}(x-y,t-\tau)\right|z(y,\tau) \,\mathrm{d}y\,\mathrm{d}\tau\\
			\leq & O(1)\int_{t-2\nu_0}^{t-\nu_0}\int_{\mathbb{R}}\frac{e^{\frac{-(x-y)^2}{C_*(t-\tau)}}}{t-\tau}\frac{\delta}{\sqrt{1+\tau}}\,\mathrm{d}y\,\mathrm{d}\tau +O(1)\int_{t-\nu_0}^{t}\int_{\mathbb{R}}\frac{e^{\frac{-(x-y)^2}{C_*(t-\tau)}}}{t-\tau}\frac{\delta}{\sqrt{1+\tau}}\,\mathrm{d}y\,\mathrm{d}\tau\\
			&+O(1)\int_{0}^{\frac{t-1+2\nu_0}{2}}\int_{\mathbb{R}} \left(\sum_{j=1}^4 \frac{O(1) e^{-\frac{\left(x-y+\beta_j(t-\tau)\right)^2}{4 C (t-\tau)}}}{t-\tau}+ e^{-\sigma_0^* (t-\tau)-\sigma_0|x-y|}\right)K\phi(\theta)z(y,\tau) \,\mathrm{d}y\,\mathrm{d}\tau\\
			&+O(1)\int_{\frac{t-1+2\nu_0}{2}}^{t-1}\int_{\mathbb{R}} \left(\sum_{j=1}^4 \frac{O(1) e^{-\frac{\left(x-y+\beta_j (t-\tau)\right)^2}{4 C (t-\tau)}}}{t-\tau}+ e^{-\sigma_0^* (t-\tau)-\sigma_0|x-y|}\right)K\phi(\theta)z(y,\tau)\,\mathrm{d}y\,\mathrm{d}\tau\\
			&+O(1)\int_{t-1}^{t-\nu_0}\int_{\mathbb{R}} \left(\frac{e^{-\frac{(x-y)^2}{4D(t-\tau)}}}{t-\tau}+ e^{-\sigma_0^*(t-\tau)-\sigma_0 |x-y|}+ (t-\tau)e^{-\sigma_0|x-y|}\right)z(y,\tau)\,\mathrm{d}y\,\mathrm{d}\tau\\
			\leq &O(1)\int_{t-2\nu_0}^{t}\frac{1}{\sqrt{t-\tau}}\frac{\delta}{\sqrt{1+\tau}}\,\mathrm{d}\tau+O(1)\int_{0}^{\frac{t-1+2\nu_0}{2}} \left(\frac{1}{t-\tau}+e^{-\sigma_0^* (t-\tau)}\right)M(\tau)\,\mathrm{d}\tau\\
			&+O(1)\int_{\frac{t-1+2\nu_0}{2}}^{t-1} \left(\frac{1}{\sqrt{t-\tau}}+e^{-\sigma_0^* (t-\tau)}\right) M(\tau)\,\mathrm{d}\tau\\
			&+O(1)\int_{t-1}^{t-\nu_0}\left(\frac{1}{\sqrt{t-\tau}}+e^{-\sigma_0^*(t-\tau)}+(t-\tau)\right)\frac{\delta}{\sqrt{1+\tau}}\,\mathrm{d}\tau
			\\
			\leq& \frac{O(1)\sqrt{\nu_0}\delta}{\sqrt{1+t}}+\frac{O(1)}{1+t-2\nu_0}\int_{0}^{\frac{t-1+2\nu_0}{2}}M(\tau)\,\mathrm{d}\tau+O(1)\int_{\frac{t-1+2\nu_0}{2}}^{t-1} \left(\frac{1}{\sqrt{t-\tau}}+e^{-\sigma_0^* (t-\tau)}\right)\frac{\delta^*}{1+\tau}\,\mathrm{d}\tau\\
			\leq& O(1)\left(\frac{\sqrt{\nu_0}\delta}{\sqrt{t}}+\frac{\delta}{\sqrt{t}}+\frac{\delta^*}{\sqrt{t}}\right).
		\end{align*}

		For the remainder terms $\partial_x\mathcal{R}_1^z$, we split the time integral at $\tau=\frac{t-1}{2}$, $\tau=t-1$ and $\tau = t-2\nu_0$, which allows us to apply short-time and long-time derivative estimates from \eqref{G4kxx,tg1} and \eqref{G4kxx,tl1} respectively. Using the \emph{a priori} estimates $\mathcal{G}(\tau)<\delta$ in \eqref{tau condition}, we find that
		\begin{align*}
			\left|\partial_x\mathcal{R}_1^z\right|\leq & \left|\int_0^{\frac{t-1}{2}} \int_{\mathbb{R}} \partial_{xy} \mathbb{G}_{42}(x-y, t-\tau)\left[\frac{a(v-1)^2}{v}+\frac{a(\theta-1)(1-v)}{v}-\frac{a u^2}{2 c_v}+\frac{\mu u_y(v-1)}{v}\right]\,\mathrm{d}y\,\mathrm{d}\tau\right| \\
			& +\left|\int_0^{\frac{t-1}{2}} \int_{\mathbb{R}} \partial_{xy} \mathbb{G}_{43}(x-y, t-\tau)\left[\left(\frac{a(\theta-1)+a(1-v)}{v}\right) u+\frac{\nu \theta_y(v-1)}{v}+\left(\frac{\nu}{c_v}-\frac{\mu}{v}\right) u u_y \right.\right.\\
			&\left.\left.+\frac{qD(v^2-1)}{v^2}z_y\right]\,\mathrm{d}y\,\mathrm{d}\tau\right| +\left|\int_0^{\frac{t-1}{2}} \int_{\mathbb{R}} \partial_{xy} \mathbb{G}_{44}(x-y, t-\tau)\frac{D(v^2-1)}{v^2}z_y\,\mathrm{d}y\,\mathrm{d}\tau\right|\\
			& +\left|\int_{\frac{t-1}{2}}^{t-1} \int_{\mathbb{R}} \partial_{xy} \mathbb{G}_{42}(x-y, t-\tau)\left(\frac{a(v-1)^2}{v}+\frac{a(\theta-1)(1-v)}{v}-\frac{a u^2}{2 c_v}+\frac{\mu u_y(v-1)}{v}\right)\,\mathrm{d}y \,\mathrm{d}\tau\right| \\
			& +\left|\int_{\frac{t-1}{2}}^{t-1} \int_{\mathbb{R}} \partial_{xy} \mathbb{G}_{43}(x-y, t-\tau)\left[\left(\frac{a(\theta-1)+a(1-v)}{v}\right) u+\frac{\nu \theta_y(v-1)}{v}+\left(\frac{\nu}{c_v}-\frac{\mu}{v}\right) u u_y \right.\right.\\
			&\left.\left.+\frac{qD(v^2-1)}{v^2}z_y\right]\,\mathrm{d}y\,\mathrm{d}\tau\right| +\left|\int_{\frac{t-1}{2}}^{t-1} \int_{\mathbb{R}} \partial_{xy} \mathbb{G}_{44}(x-y, t-\tau)\frac{D(v^2-1)}{v^2}z_y\,\mathrm{d}y\,\mathrm{d}\tau\right|\\
			&+\left|\int_{t-1}^{t-2 \nu_0} \int_{\mathbb{R}} \partial_{xy} \mathbb{G}_{42}(x-y, t-\tau)\left(\frac{a(v-1)^2}{v}+\frac{a(\theta-1)(1-v)}{v}-\frac{a u^2}{2 c_v}+\frac{\mu u_y(v-1)}{v}\right)\,\mathrm{d}y\,\mathrm{d}\tau\right| \\
			& +\left|\int_{t-1}^{t-2\nu_0} \int_{\mathbb{R}} \partial_{xy} \mathbb{G}_{43}(x-y, t-\tau)\left[\left(\frac{a(\theta-1)+a(1-v)}{v}\right) u+\frac{\nu \theta_y(v-1)}{v}+\left(\frac{\nu}{c_v}-\frac{\mu}{v}\right) u u_y \right.\right.\\
			&\left.\left.+\frac{qD(v^2-1)}{v^2}z_y\right]\,\mathrm{d}y\,\mathrm{d}\tau\right| +\left|\int_{t-1}^{t-2\nu_0} \int_{\mathbb{R}} \partial_{xy} \mathbb{G}_{44}(x-y, t-\tau)\frac{D(v^2-1)}{v^2}z_y\,\mathrm{d}y\,\mathrm{d}\tau\right|\\
			\leq & O(1)\int_0^{\frac{t-1}{2}}\int_{\mathbb{R}}\left(\sum_{j=1}^4 \frac{O(1) e^{-\frac{\left(x-y+\beta_j (t-\tau)\right)^2}{4 (t-\tau)}}}{(t-\tau)^{\frac{3}{2}}}+C(\nu_0) e^{-\sigma_0^*(t-\tau)-\sigma_0|x-y|}\right) \frac{\delta}{\sqrt{1+\tau}}\left(|v-1|+|\theta-1|\right.\\
			&\left.\qquad+|u|+|u_y|+|\theta_y|+|z_y|\right)\,\mathrm{d}y\,\mathrm{d}\tau\\
			&+O(1) \int_{\frac{t-1}{2}}^{t-1}\int_{\mathbb{R}} \left(\sum_{j=1}^4 \frac{O(1) e^{-\frac{\left(x-y+\beta_j (t-\tau)\right)^2}{4 (t-\tau)}}}{(t-\tau)^{\frac{3}{2}}}+C(\nu_0) e^{-\sigma_0^*(t-\tau)-\sigma_0|x-y|}\right) \frac{\delta}{\sqrt{1+\tau}}\left(|v-1|+|\theta-1|\right.\\
			&\left.\qquad+|u|+|u_y|+|\theta_y|+|z_y|\right) \,\mathrm{d}y\,\mathrm{d}\tau\\
			&+O(1) \int_{t-1}^{t-2 \nu_0} \int_{\mathbb{R}}\left(e^{-\sigma_0^*(t-\tau)-\sigma_0 |x-y|}+(t-\tau) e^{-\sigma_0|x-y|}\right)\frac{\delta}{\sqrt{1+\tau} }\left(|v-1|+|\theta-1|+|u|\right.\\
			&\left.\qquad+|u_y|+|\theta_y|+|z_y|\right)\,\mathrm{d}y \,\mathrm{d}\tau \\
			&+O(1) \int_{t-1}^{t-2 \nu_0}\int_{\mathbb{R}} \left(\frac{e^{-\frac{(x-y)^2}{4D(t-\tau)}}}{(t-\tau)^{\frac{3}{2}}}+e^{-\sigma_0^*(t-\tau)-\sigma_0 |x-y|}+O(1)(t-\tau)e^{-\sigma_0|x-y|}\right)\frac{\delta}{\sqrt{1+\tau} }|z_y|\,\mathrm{d}y\,\mathrm{d}\tau \\
			\leq & O(1) \int_0^{\frac{t-1}{2}}\left(\frac{1}{(t-\tau)^{\frac{3}{2}}}+e^{-\sigma_0^* (t-\tau)}\right) \frac{\delta^2}{\sqrt{1+\tau}} \,\mathrm{d}\tau+O(1) \int_{\frac{t-1}{2}}^{t-1}\left(\frac{1}{t-\tau}+e^{-\sigma_0^* (t-\tau)}\right) \frac{\delta^2}{\sqrt{1+\tau}\sqrt{\tau}} \,\mathrm{d}\tau\\
			&+O(1)\int_{t-1}^{t-2 \nu_0}\left(\frac{1}{t-\tau}+e^{-\sigma_0^* (t-\tau)}+(t-\tau)\right)\frac{\delta^2}{\sqrt{1+\tau} \sqrt{\tau}} \,\mathrm{d}\tau \\
			\leq & O(1) \left(\frac{\delta^2}{1+t}+\frac{\delta^2}{\sqrt{1+t}}\right)+O(1)\frac{\delta^2}{\sqrt{1+t}}\int_{\frac{t-1}{2}}^{t-1}\frac{1}{\sqrt{t-\tau}\sqrt{\tau}}\,\mathrm{d}\tau+O(1)\frac{\delta^2}{\sqrt{t}}\int_{t-1}^{t-2\nu_0}\frac{1}{t-\tau}\frac{1}{\sqrt{\tau}}\,\mathrm{d}\tau\\
			&+O(1)\frac{\delta^2}{\sqrt{t}}\int_{t-1}^{t-2\nu_0}\frac{1}{\sqrt{\tau}}\,\mathrm{d}\tau\\
			\leq &O(1)\frac{(1+|\log(\nu_0)|)\delta^2}{\sqrt{t}}.
		\end{align*}

    For $\partial_x\mathcal{R}_2^z$, we utilize the explicit representation for $\mathbb{G}_{44}$ in Eq.~\eqref{G44 rep}, together with comparison estimates for $H_x(x,t;y,\tau;D)-H_x(x,t;y,\tau;\frac{D}{v^2})$ from Lemma \ref{lemma: comparison2}, the estimates \eqref{G4kxx,tg1} and \eqref{G4kxx,tl1}, as well as the \emph{a priori} bounds in Eq.~\eqref{tau condition}. One thus deduces that
		\begin{align*}
			&\left|\partial_x\mathcal{R}_2^z(x, t)\right|  \\
			\leq & \left|\int_{t-2 \nu_0}^{t-\nu_0} \int_{\mathbb{R}} \partial_{xy} \mathbb{G}_{42}(x-y, t-\tau)\left[\frac{a(v-1)^2}{v}+\frac{a(\theta-1)(1-v)}{v}-\frac{au^2}{2c_v}+\frac{\mu u_y(v-1)}{v}\right.\right.\\
			&\left.\left.\qquad-\mathcal{X}(\frac{t-\tau}{\nu_0})\frac{qa}{c_v}z\right]\,\mathrm{d}y\,\mathrm{d}\tau\right| \\
			&+\left|\int_{t-2 \nu_0}^{t-\nu_0} \int_{\mathbb{R}} \partial_{xy}\mathbb{G}_{43}(x-y;t-\tau) \left[\frac{a(\theta-1)+a(1-v)}{v} u+\frac{\nu(v-1)}{v}\theta_y+\left(\frac{\nu}{c_v v}-\frac{\mu}{v}\right)uu_y\right.\right.\\
			&\left.\left.+\frac{qD(v^2-1)}{v^2}z_y\right]\,\mathrm{d}y\,\mathrm{d}\tau\right|+\left|\int_{t-2 \nu_0}^{t-\nu_0} \int_{\mathbb{R}} \partial_{xy}\mathbb{G}_{43}(x-y;t-\tau)\mathcal{X}(\frac{t-\tau}{\nu_0})\left(\frac{q\nu}{c_v}-qD\right)z_y\,\mathrm{d}y\,\mathrm{d}\tau\right|\\
			&+\left|\int_{t-2 \nu_0}^{t-\nu_0} \int_{\mathbb{R}}\partial_{xy} \mathbb{G}_{44}(x-y;t-\tau) \frac{D(v^2-1)}{v^2}z_y\,\mathrm{d}y\,\mathrm{d}\tau\right| \\
			&+\int_{t-2 \nu_0}^{t-\nu_0} \int_{\mathbb{R}} \frac{1}{\nu_0}\left|H_x(x,t;y,\tau;D)-H_x(x,t;y,\tau;\frac{D}{v^2}) \right|z(y,\tau)\,\mathrm{d}y\,\mathrm{d} \tau\\
			& +\int_{t-2 \nu_0}^{t-\nu_0} \int_{\mathbb{R}} \frac{1}{\nu_0}\left|\int_\tau^t \int_{\mathbb{R}} H_{\omega}(\omega,s;y, \tau;D)\frac{qa}{c_v}\partial_{\omega}\mathbb{G}_{42}(x-\omega,t-s) \,\mathrm{d}\omega \mathrm{d}s\right|z(y,\tau)\,\mathrm{d}y\,\mathrm{d} \tau \\
			& +\int_{t-2 \nu_0}^{t-\nu_0} \int_{\mathbb{R}} \frac{1}{\nu_0}\left|\int_\tau^t \int_{\mathbb{R}} H_{\omega}(\omega,s;y, \tau;D)\left(\frac{q\nu}{c_v}-qD\right)\partial_{\omega}^2\mathbb{G}_{43}(x-\omega,t-s)\,\mathrm{d}\omega\,\mathrm{d}s\right|z(y, \tau)\,\mathrm{d}y \,\mathrm{d}\tau\\
			\leq & O(1) \int_{t-2\nu_0}^{t-\nu_0} \int_{\mathbb{R}}\left(e^{-\sigma_0^*(t-\tau)-\sigma_0 |x-y|}+(t-\tau) e^{-\sigma_0|x-y|}\right)\left[\frac{\delta}{\sqrt{1+\tau} }\left(|v-1|+|\theta-1|+|u|+|u_y|\right.\right.\\
			&\left.\left.\qquad+|\theta_y|+|z_y|\right)+z\right]\,\mathrm{d}y \,\mathrm{d}\tau \\
			&+O(1) \int_{t-2\nu_0}^{t-\nu_0}\int_{\mathbb{R}} \left(\frac{e^{-\frac{(x-y)^2}{4D(t-\tau)}}}{(t-\tau)^{\frac{3}{2}}}+e^{-\sigma_0^*(t-\tau)-\sigma_0 |x-y|}+O(1)(t-\tau)e^{-\sigma_0|x-y|}\right)\frac{\delta}{\sqrt{1+\tau} }|z_y|\,\mathrm{d}y\,\mathrm{d}\tau\\
			& +O(1) \int_{t-2 \nu_0}^{t-\nu_0} \int_{\mathbb{R}} \frac{1}{\nu_0} \frac{e^{-\frac{(x-y)^2}{C_*(t-\tau)}}}{t-\tau}\left[|\log (t-\tau)| \sup _{t-2 \nu_0<s<t}\|v(\cdot, s)-1\|_{L^{\infty}}+\sup _{t-2 \nu_0<s<t}\|v(\cdot, s)-1\|_{BV}\right]\\
			&\qquad \frac{\delta}{\sqrt{1+\tau}} \,\mathrm{d}y\,\mathrm{d}\tau \\
			& +O(1) \int_{t-2 \nu_0}^{t-\nu_0} \int_{\mathbb{R}} \frac{1}{\nu_0} \frac{e^{-\frac{(x-y)^2}{C_*(t-\tau)}}}{t-\tau} \sqrt{t-\tau}\left[\sup _{t-2 \nu_0<s<t}\left\|\sqrt{s} u_x(\cdot, s)\right\|_{L^{\infty}}+\sup _{t-2 \nu_0<s<t}\|v(\cdot, s)-1\|_{L^1}\right]\\
			&\qquad \frac{\delta}{\sqrt{1+\tau}}\,\mathrm{d}y\,\mathrm{d}\tau \\
			& +O(1) \int_{t-2 \nu_0}^{t-\nu_0}\int_{\tau}^t \frac{1}{\nu_0}  \int_{\mathbb{R}}\int_{\mathbb{R}}\frac{e^{\frac{-(\omega-y)^2}{s-\tau}}}{s-\tau}\left(e^{-\sigma_0^*(t-s)-\sigma_0 |x-\omega|}+(t-s)e^{-\sigma_0|x-\omega|}\right)\,\mathrm{d}\omega\,\mathrm{d}y\frac{\delta}{\sqrt{1+\tau}}\,\mathrm{d}s \,\mathrm{d}\tau \\
			\leq &O(1)\int_{t-2\nu_0}^{t-\nu_0} e^{-\sigma_0^*(t-\tau)}\frac{\delta^2}{\sqrt{1+\tau}\sqrt{\tau}}\,\mathrm{d}\tau+O(1)\int_{t-2\nu_0}^{t-\nu_0}(t-\tau)\frac{\delta^2}{\sqrt{1+\tau}}\,\mathrm{d}\tau+O(1)\int_{t-2\nu_0}^{t-\nu_0} \frac{1}{t-\tau}\frac{\delta^2}{\sqrt{1+\tau}\sqrt{\tau}}\,\mathrm{d}\tau\\
			&+O(1)\int_{t-2\nu_0}^{t-\nu_0}\frac{1}{\nu_0}\frac{1}{\sqrt{t-\tau}}\left(|\log(t-\tau)|+1\right)\frac{\delta^2}{\sqrt{1+\tau}}\,\mathrm{d}\tau+O(1)\int_{t-2\nu_0}^{t-\nu_0}\frac{1}{\nu_0}\frac{\delta^2}{\sqrt{1+\tau}}\,\mathrm{d}\tau\\
			&+O(1)\int_{t-2\nu_0}^{t-\nu_0}\int_{\tau}^t\frac{1}{\nu_0}\frac{1}{\sqrt{s-\tau}}\left(e^{-\sigma_0^*(t-s)}+(t-s)\right)\frac{\delta}{\sqrt{1+\tau}}\,\mathrm{d}s\,\mathrm{d}\tau\\
			\leq &O(1)\left(\frac{\nu_0\delta^2}{\sqrt{t}}+\frac{\nu_0^2\delta^2}{\sqrt{t}}+\frac{\delta^2}{\sqrt{t}}\right)+O(1)\frac{|\log(\nu_0)|\delta^2}{\sqrt{\nu_0}\sqrt{t}}+O(1)\frac{\sqrt{\nu_0}\delta}{\sqrt{t}}.
		\end{align*}

        Finally, the proof of $\partial_x\mathcal{R}_3^z$ is similar to that of $\partial_x\mathcal{R}_1^z$. Let us write
        \begin{align*}
    &\left|\partial_x\mathcal{R}_3^{z}(x,t)\right| \leq I_{a} + I_b \\
    &:= \int_{t-\nu_0}^t \int_\R  \left|\partial_{xy}\mathbb{G}_{42}(x-y,t-\tau)\right|\left|\mathcal{M}_1(y,\tau)\right|\,\mathrm{d}y\,\mathrm{d}\tau+  \int_{t-2\nu_0}^{t-\nu_0} \int_\R \left|\partial_{xy}\mathbb{G}_{43}(x-y,t-\tau)\right|\left|\mathcal{M}_2(y,\tau)\right| 
    \,\mathrm{d} y\,\mathrm{d} \tau,
        \end{align*}
    where
    \begin{align*}
    &\mathcal{M}_1 := \frac{a(v-1)^2}{v} + a \frac{(\theta-1)(1-v)}{v} - \frac{au^2}{2c_v} + \frac{\mu u_y(v-1)}{v} - \frac{qaz}{c_v},\\
    &\mathcal{M}_2 := \frac{a(\theta-1)+a(1-v)}{v}u+ \frac{\nu (v-1)}{v}\theta_y + \left(\frac{\nu }{c_v v} - \frac{\mu}{v} \right)uu_y - \left( \frac{q\nu}{c_v}-qD \right)z_y.
\end{align*}

To control $I_a$, thanks to the hypothesis $\mathcal{G}(\tau)<\delta$, we have the \emph{a priori} bounds $\|v-1\|_{L^\infty}, \|u\|_{L^\infty}, \dots, \|z_y\|_{L^\infty} \leq \frac{\delta}{\sqrt{1+\tau}}$, where the spatial $L^\infty_x$-bounds are evaluated at time $\tau$. In this way, we obtain that 
\begin{align*}
&\left|\mathcal{M}_1(y,\tau)\right| + \left|\mathcal{M}_2(y,\tau)\right| \\
&\quad\leq O(1) \Big\{|v-1| + |\theta-1| + |u| + |u_y| + |z| + |u| + |z_y|+ |\theta_y|\Big\}  \leq O(1) \frac{\delta}{\sqrt{1+\tau}}.
\end{align*}
Thus, in view of the estimates in Eqs.~\eqref{G4kxx,tg1}, \eqref{G4kxx,tl1} for the second-order derivatives of the Green's matrix, we obtain that 
\begin{align*}
    I_a &\leq O(1)\int_{t-\nu_0}^t \int_{\R} \left\{
    e^{-\sigma_0^*|t-\tau|-\sigma_0|x-y|} + (t-\tau)e^{-\sigma_0|x-y|}\right\}\times \nonumber\\
   &\qquad\qquad\quad\qquad \times\left\{\left|\mathcal{M}_1(y,\tau)\right| + \left|\mathcal{M}_2(y,\tau)\right| \right\}\,\mathrm{d}y\,\mathrm{d}\tau \leq \frac{O(1)\,\delta}{\sqrt{t}},
\end{align*}
where the constant in $O(1)$ depends on $\nu_0$ (among other parameters).

The other integral over $\tau \in [t-2\nu_0, t-\nu_0]$ is similar:
		\begin{align*}
	I_b \leq&O(1)\int_{t-2\nu_0}^{t-\nu_0}\int_{\mathbb{R}}\left(e^{-\sigma_0^*(t-\tau)-\sigma_0|x-y|}+(t-\tau)e^{-\sigma_0|x-y|}\right)\frac{\delta}{\sqrt{1+\tau}}\left(|v-1|+|\theta-1|+|u|+|u_y|\right)\,\mathrm{d}y \,\mathrm{d}\tau\\
			&+O(1)\int_{t-2\nu_0}^{t-\nu_0}\int_{\mathbb{R}}\left(e^{-\sigma_0^*(t-\tau)-\sigma_0|x-y|}+(t-\tau)e^{-\sigma_0|x-y|}\right)\left(|z_y|+z\right)\,\mathrm{d}y \,\mathrm{d}\tau\\
			\leq &O(1)\int_{t-2\nu_0}^{t-\nu_0}\left(e^{-\sigma_0^*(t-\tau)}+(t-\tau)\right)\frac{\delta^2}{\sqrt{1+\tau}}\,\mathrm{d}\tau+\int_{t-2\nu_0}^{t-\nu_0}\left(e^{-\sigma_0^*(t-\tau)}+(t-\tau)\right)\frac{\delta}{\sqrt{\tau}}\,\mathrm{d}\tau\\
			\leq &O(1)\left(\frac{\nu_0\delta^2}{\sqrt{t}}+\frac{\nu_0^2\delta^2}{\sqrt{t}}+\frac{\nu_0\delta}{\sqrt{t}}+\frac{\nu_0^2\delta}{\sqrt{t}}\right).
		\end{align*}

        Combining the above estimates, for sufficiently small $\delta>0$ and $t\geq t_{\sharp}\geq 4\nu_0$ we have that
		\begin{equation}
			\left\|z_x(\cdot,t)\right\|_{L_x^{\infty}}\leq C(\nu_0)\frac{\delta^*}{\sqrt{t}}+ O(1)\left(\frac{|\log(\nu_0)|\delta^2}{\sqrt{\nu_0}\sqrt{t}}+\frac{\sqrt{\nu_0}\delta}{\sqrt{t}}\right).
		\end{equation}

        \smallskip
        \noindent\textbf{Estimate of $\left\|z_x(\cdot,t)\right\|_{L_x^{1}}$:}
		For $t\geq t_{\sharp}\geq 4\nu_0$, using the estimates in Eq.~\eqref{G4kx,tg1} and the smallness condition $\|v_0-1\|_{L_x^1}< \delta^*$ in Eq.~\eqref{tau condition}, we get
		\begin{align*}
			&\int_{\mathbb{R}}\left|\int_{\mathbb{R}}\partial_x\mathbb{G}_{41}(x-y,t)\left(v(y,0)-1\right)\,\mathrm{d}y\right|\,\mathrm{d}x \\ \leq & C(\nu_0)\int_{\mathbb{R}}\int_{\mathbb{R}}\sum_{j=1}^4\frac{e^{-\frac{\left(x-y+\beta_j t\right)^2}{4 C t}}}{t}\left|v(y,0)-1\right|\,\mathrm{d}y\,\mathrm{d}x+ C(\nu_0) \int_{\mathbb{R}}\int_{\mathbb{R}}e^{-\sigma_0^* t-\sigma_0|x-y|}\left|v(y,0)-1\right|\,\mathrm{d}y\,\mathrm{d}x\\
			\leq &C(\nu_0) \delta^*.
		\end{align*}
		Similarly, for the other $\mathbb{G}_{4j}$ terms we have that
		\begin{align*}
			&\int_{\mathbb{R}}\left|\int_{\mathbb{R}}\partial_x\mathbb{G}_{42}(x-y,t)u(y,0)\,\mathrm{d}y\right|\,\mathrm{d}x \leq C(\nu_0)\delta^*,\\
			&\int_{\mathbb{R}}\left|\int_{\mathbb{R}}\partial_x\mathbb{G}_{43}(x-y,t)\left(E(y,0)-c_v\right)\,\mathrm{d}y\right|\,\mathrm{d}x \leq C(\nu_0)\delta^*,\\
			&\int_{\mathbb{R}}\left|\int_{\mathbb{R}}\partial_x\mathbb{G}_{44}(x,t;y,0)z(y,0)\,\mathrm{d}y\right|\,\mathrm{d}x=\int_{\mathbb{R}}\left|\int_{\mathbb{R}}\mathbb{G}_{44}(x-y,t)z(y,0)\,\mathrm{d}y\right|\,\mathrm{d}x\leq C(\nu_0)\delta^*.
		\end{align*}
		For the integral remainder term, we split the time integral using the cutoff function $\mathcal{X}(\frac{t-\tau}{\nu_0})$. Applying the bounds  for $H_x$ in Lemma \ref{lemma: Liu}, estimates for $\partial_x\mathbb{G}_{44}$ in \eqref{G4kx,tg1}\eqref{G4kx,tl1}, and \emph{a priori} bound \eqref{tau condition}, we find that
		\begin{align*}
			&\int_{\mathbb{R}}\left|\int_0^t\int_{\mathbb{R}} \partial_x\mathbb{G}_{44}(x,t;y,\tau)K\phi(\theta)z(y,\tau) \,\mathrm{d}y\,\mathrm{d}\tau\right|\,\mathrm{d}x\\ =&\int_{\mathbb{R}}\left|\int_0^t\int_{\mathbb{R}}\left(\frac{t-\tau}{\nu_0}\right)H_x(x,t;y,\tau;\frac{D}{v^2})K\phi(\theta)z(y,\tau) \,\mathrm{d}y\,\mathrm{d}\tau\right|\,\mathrm{d}x\\
			&+\int_{\mathbb{R}}\left|\int_0^t\int_{\mathbb{R}}\left(1- \mathcal{X}\left(\frac{t-\tau}{\nu_0}\right)\right) \partial_x\mathbb{G}_{44}(x-y,t-\tau)K\phi(\theta)z(y,\tau) \,\mathrm{d}y\,\mathrm{d}\tau\right|\,\mathrm{d}x\\
			\leq&O(1)\int_{t-2\nu_0}^{t-\nu_0}\int_{\mathbb{R}}\int_{\mathbb{R}}\left|H_x(x,t;y,\tau;\frac{D}{v^2})\right|z(y,\tau) \,\mathrm{d}y\,\mathrm{d}x\,\mathrm{d}\tau\\
			&+O(1)\int_{t-\nu_0}^{t}\int_{\mathbb{R}}\int_{\mathbb{R}} \left|H_x(x,t;y,\tau;\frac{D}{v^2})\right|K\phi(\theta)z(y,\tau)\,\mathrm{d}y\,\mathrm{d}x\,\mathrm{d}\tau\\
			&+O(1)\int_{0}^{t-2\nu_0}\int_{\mathbb{R}}\int_{\mathbb{R}} \left|\partial_x\mathbb{G}_{44}(x-y,t-\tau)\right|K\phi(\theta)z(y,\tau)\,\mathrm{d}y \,\mathrm{d}x\,\mathrm{d}\tau\\
			&+O(1)\int_{t-2\nu_0}^{t-\nu_0}\int_{\mathbb{R}}\int_{\mathbb{R}} \left|\partial_x\mathbb{G}_{44}(x-y,t-\tau)\right|K\phi(\theta)z(y,\tau)\,\mathrm{d}y\,\mathrm{d}x \,\mathrm{d}\tau\\
			\leq &O(1)\int_{t-2\nu_0}^{t-\nu_0}\int_{\mathbb{R}}\int_{\mathbb{R}}\frac{e^{\frac{-(x-y)^2}{t-\tau}}}{t-\tau}z(y,\tau)\,\mathrm{d}y\,\mathrm{d}x\,\mathrm{d}\tau\\
			&+O(1)\int_{t-\nu_0}^{t}\int_{\mathbb{R}}\int_{\mathbb{R}} \frac{e^{\frac{-(x-y)^2}{t-\tau}}}{t-\tau}z(y,\tau) \,\mathrm{d}y\,\mathrm{d}x\,\mathrm{d}\tau\\
			&+O(1)\int_{0}^{t-2\nu_0}\int_{\mathbb{R}}\int_{\mathbb{R}} \left(\sum_{j=1}^4 \frac{ e^{-\frac{\left(x-y+\beta_j(t-\tau)\right)^2}{4 C (t-\tau)}}}{t-\tau}+ e^{-\sigma_0^* (t-\tau)-\sigma_0|x-y|}\right)K\phi(\theta)z(y,\tau)\,\mathrm{d}y \,\mathrm{d}x\,\mathrm{d}\tau\\
			&+O(1)\int_{t-2\nu_0}^{t-\nu_0}\int_{\mathbb{R}}\int_{\mathbb{R}} \left(\left|\partial_x\left(\frac{e^{-\frac{(x-y)^2}{4D(t-\tau)}}}{\sqrt{4\pi D(t-\tau)}}\right)\right|+ e^{-\sigma_0^*(t-\tau)-\sigma_0 |x-y|}+ (t-\tau)e^{-\sigma_0|x-y|}\right)\\
			&\qquad \qquad \cdot z(y,\tau)\,\mathrm{d}y\,\mathrm{d}x \,\mathrm{d}\tau\\
			\leq &O(1)\int_{t-2\nu_0}^{t} \frac{1}{\sqrt{t-\tau}}\delta \,\mathrm{d}\tau+O(1)\int_{0}^{t-2\nu_0}\left(\frac{1}{\sqrt{t-\tau}}+e^{-\sigma_0^* (t-\tau)}\right)M(\tau)\,\mathrm{d}\tau\\
			&+O(1)\int_{t-2\nu_0}^{t-\nu_0}\left(\frac{1}{\sqrt{t-\tau}}+e^{-\sigma_0^* (t-\tau)}+(t-\tau)\right)\delta\,\mathrm{d}\tau\\
			\leq &O(1){\nu_0}\delta+O(1)\int_{0}^{t-2\nu_0}\frac{1}{\sqrt{t-\tau}}\frac{\delta^*}{1+\tau}\,\mathrm{d}\tau+O(1)(\sqrt{\nu_0}+\nu_0+\nu_0^2)\delta\\
			\leq &O(1)(\sqrt{\nu_0}\delta+\nu_0\delta+\nu_0^2\delta)+C(\nu_0)\delta^*.
		\end{align*}

		We next bound the $L^1$-norms of $\partial_x\mathcal{R}_1^z$, $\partial_x\mathcal{R}_2^z$ and $\partial_x\mathcal{R}_3^z$. For $\partial_x\mathcal{R}_1^z$, we split the time integral as before, and note the bounds for $\partial_{xy}\mathbb{G}_{4k}$ from \eqref{G4kxx,tg1} \eqref{G4kxx,tl1} and the \emph{a priori} condition $\mathcal{G}(\tau)<\delta$ \eqref{tau condition}. This gives us 
		\begin{align*}
			&\int_{\mathbb{R}}\left|\partial_x\mathcal{R}_1^z\right|\,\mathrm{d}x\\
			\leq & \int_{\mathbb{R}}\left|\int_0^{t-1} \int_{\mathbb{R}} \partial_{xy} \mathbb{G}_{42}(x-y, t-\tau)\left[\frac{a(v-1)^2}{v}+\frac{a(\theta-1)(1-v)}{v}-\frac{a u^2}{2 c_v}+\frac{\mu u_y(v-1)}{v}\right]\,\mathrm{d}y\,\mathrm{d}\tau\right|\,\mathrm{d}x \\
			& +\int_{\mathbb{R}}\left|\int_0^{t-1} \int_{\mathbb{R}} \partial_{xy} \mathbb{G}_{43}(x-y, t-\tau)\left[\left(\frac{a(\theta-1)+a(1-v)}{v}\right) u+\frac{\nu \theta_y(v-1)}{v}+\left(\frac{\nu}{c_v}-\frac{\mu}{v}\right) u u_y \right.\right.\\
			&\left.\left.+\frac{qD(v^2-1)}{v^2}z_y\right]\,\mathrm{d}y\,\mathrm{d}\tau\right|\,\mathrm{d}x +\int_{\mathbb{R}}\left|\int_0^{t-1} \int_{\mathbb{R}} \partial_{xy} \mathbb{G}_{44}(x-y, t-\tau)\frac{D(v^2-1)}{v^2}z_y\,\mathrm{d}y\,\mathrm{d}\tau\right|\,\mathrm{d}x\\
			&+\int_{\mathbb{R}}\left|\int_{t-1}^{t-2 \nu_0} \int_{\mathbb{R}} \partial_{xy} \mathbb{G}_{42}(x-y, t-\tau)\left(\frac{a(v-1)^2}{v}+\frac{a(\theta-1)(1-v)}{v}-\frac{a u^2}{2 c_v}+\frac{\mu u_y(v-1)}{v}\right)\,\mathrm{d}y\,\mathrm{d}\tau\right|\,\mathrm{d}x \\
			& +\int_{\mathbb{R}}\left|\int_{t-1}^{t-2\nu_0} \int_{\mathbb{R}} \partial_{xy} \mathbb{G}_{43}(x-y, t-\tau)\left[\left(\frac{a(\theta-1)+a(1-v)}{v}\right) u+\frac{\nu \theta_y(v-1)}{v}+\left(\frac{\nu}{c_v}-\frac{\mu}{v}\right) u u_y \right.\right.\\
			&\left.\left.+\frac{qD(v^2-1)}{v^2}z_y\right]\,\mathrm{d}y\,\mathrm{d}\tau\right|\,\mathrm{d}x +\int_{\mathbb{R}}\left|\int_{t-1}^{t-2\nu_0} \int_{\mathbb{R}} \partial_{xy} \mathbb{G}_{44}(x-y, t-\tau)\frac{D(v^2-1)}{v^2}z_y\,\mathrm{d}y\,\mathrm{d}\tau\right|\,\mathrm{d}x\\
			\leq & O(1)\int_0^{t-1}\int_{\mathbb{R}}\int_{\mathbb{R}}\left(\sum_{j=1}^4 \frac{e^{-\frac{\left(x-y+\beta_j (t-\tau)\right)^2}{4 (t-\tau)}}}{(t-\tau)^{\frac{3}{2}}}+C(\nu_0) e^{-\sigma_0^*(t-\tau)-\sigma_0|x-y|}\right) \frac{\delta}{\sqrt{1+\tau}}\left(|v-1|+|\theta-1|\right.\\
			&\left.\qquad+|u|+|u_y|+|\theta_y|+|z_y|\right)\,\mathrm{d}y \,\mathrm{d}x \,\mathrm{d}\tau\\
			&+O(1) \int_{t-1}^{t-2 \nu_0}\int_{\mathbb{R}} \int_{\mathbb{R}}\left(e^{-\sigma_0^*(t-\tau)-\sigma_0 |x-y|}+(t-\tau) e^{-\sigma_0|x-y|}\right)\frac{\delta}{\sqrt{1+\tau} }\left(|v-1|+|\theta-1|+|u|+|u_y|\right.\\
			&\left.\qquad+|\theta_y|+|z_y|\right)\,\mathrm{d}y\,\mathrm{d}x \,\mathrm{d}\tau \\
			&+O(1) \int_{t-1}^{t-2 \nu_0}\int_{\mathbb{R}}\int_{\mathbb{R}} \left(\frac{e^{-\frac{(x-y)^2}{4D(t-\tau)}}}{(t-\tau)^{\frac{3}{2}}}+e^{-\sigma_0^*(t-\tau)-\sigma_0 |x-y|}+(t-\tau)e^{-\sigma_0|x-y|}\right)\frac{\delta}{\sqrt{1+\tau} }|z_y|\,\mathrm{d}y \,\mathrm{d}x \,\mathrm{d}\tau \\
			\leq & O(1) \int_0^{t-1}\left(\frac{1}{t-\tau}+e^{-\sigma_0^* (t-\tau)}\right) \frac{\delta^2}{\sqrt{1+\tau}} \,\mathrm{d}\tau+O(1)\int_{t-1}^{t-2 \nu_0}\left(\frac{1}{t-\tau}+e^{-\sigma_0^* (t-\tau)}+(t-\tau)\right)\frac{\delta^2}{\sqrt{1+\tau}} \,\mathrm{d}\tau \\
			\leq & O(1)\left(1+|\log(\nu_0)|\right)\delta^2.
		\end{align*} 
		For $\partial_x\mathcal{R}_2^z$, substituting the representation for $\mathbb{G}_{44}$ in Eq.~\eqref{G44 rep}, and combining the same estimates for $\partial_x\mathcal{R}_1^z$ with the comparison estimates for $H_x$ in Lemma~\ref{lemma: comparison2}, we deduce that
		\begin{align*}
			&\int_{\mathbb{R}}\left|\partial_x\mathcal{R}_2^z(x,t)\right|\,\mathrm{d}x\\
			\leq & \int_{t-2\nu_0}^{t-\nu_0}\int_{\mathbb{R}}\int_{\mathbb{R}} \left|\partial_{xy}\mathbb{G}_{42}(x-y, t-\tau)\right|\left|\frac{a(v-1)^2}{v}+\frac{a(\theta-1)(1-v)}{v}-\frac{au^2}{2c_v}+\frac{\mu u_y(v-1)}{v}-\mathcal{X}(\frac{t-\tau}{\nu_0})\frac{qa}{c_v}z\right|\\&\qquad \,\mathrm{d}y \,\mathrm{d}x \,\mathrm{d}\tau \\
			&+\int_{t-2 \nu_0}^{t-\nu_0}\int_{\mathbb{R}}\int_{\mathbb{R}} \left|\partial_{xy}\mathbb{G}_{43}(x-y;t-\tau)\right| \left|\frac{a(\theta-1)+a(1-v)}{v} u+\frac{\nu(v-1)}{v}\theta_y+\left(\frac{\nu}{c_v v}-\frac{\mu}{v}\right)uu_y\right.\\
			&\left.+\frac{qD(v^2-1)}{v^2}z_y\right|\,\mathrm{d}y \,\mathrm{d}x \,\mathrm{d}\tau+\int_{t-2 \nu_0}^{t-\nu_0}\int_{\mathbb{R}} \int_{\mathbb{R}} \left|\partial_{xy}\mathbb{G}_{43}(x-y;t-\tau)\right|\mathcal{X}(\frac{t-\tau}{\nu_0})\left|\left(\frac{q\nu}{c_v}-qD\right)z_y\right|\,\mathrm{d}y\,\mathrm{d}\tau\\
			&+\int_{t-2 \nu_0}^{t-\nu_0}\int_{\mathbb{R}}\int_{\mathbb{R}} \left|\partial_{xy}\mathbb{G}_{44}(x-y;t-\tau)\right| \left|\frac{D(v^2-1)}{v^2}z_y\right|\,\mathrm{d}y \,\mathrm{d}x \,\mathrm{d}\tau\\
			&+\int_{t-2 \nu_0}^{t-\nu_0}\int_{\mathbb{R}} \int_{\mathbb{R}} \frac{1}{\nu_0}\left|H_x(x,t;y,\tau;D)-H_x(x,t;y,\tau;\frac{D}{v^2}) \right|z(y,\tau)\,\mathrm{d}y \,\mathrm{d}x \,\mathrm{d} \tau\\
			& +\int_{t-2 \nu_0}^{t-\nu_0}\int_{\mathbb{R}} \int_{\mathbb{R}} \frac{1}{\nu_0}\left|\int_\tau^t \int_{\mathbb{R}} H_{\omega}(\omega,s;y, \tau;D)\frac{qa}{c_v}\partial_{\omega}\mathbb{G}_{42}(x-\omega,t-s) \,\mathrm{d}\omega \mathrm{d}s\right|z(y,\tau)\,\mathrm{d}y \,\mathrm{d}x \,\mathrm{d} \tau \\
			& +\int_{t-2 \nu_0}^{t-\nu_0} \int_{\mathbb{R}}\int_{\mathbb{R}} \frac{1}{\nu_0}\left|\int_\tau^t \int_{\mathbb{R}} H_{\omega}(\omega,s;y, \tau;D)\left(\frac{q\nu}{c_v}-qD\right)\partial_{\omega}^2\mathbb{G}_{43}(x-\omega,t-s)\,\mathrm{d}\omega\,\mathrm{d}s\right|z(y, \tau)\,\mathrm{d}y \,\mathrm{d}x \,\mathrm{d}\tau\\
			\leq & O(1) \int_{t-2\nu_0}^{t-\nu_0} \int_{\mathbb{R}} \int_{\mathbb{R}}\left(e^{-\sigma_0^*(t-\tau)-\sigma_0 |x-y|}+(t-\tau) e^{-\sigma_0|x-y|}\right)\left[\frac{\delta}{\sqrt{1+\tau} }\left(|v-1|+|\theta-1|+|u|+|u_y|\right.\right.\\
			&\left.\left.\qquad+|\theta_y|+|z_y|\right)+z+|z_y|\right](y,\tau)\,\mathrm{d}y \,\mathrm{d}x \,\mathrm{d}\tau \\
			&+O(1) \int_{t-2\nu_0}^{t-\nu_0}\int_{\mathbb{R}}\int_{\mathbb{R}} \left(\frac{e^{-\frac{(x-y)^2}{4D(t-\tau)}}}{(t-\tau)^{\frac{3}{2}}}+e^{-\sigma_0^*(t-\tau)-\sigma_0 |x-y|}+(t-\tau)e^{-\sigma_0|x-y|}\right)\frac{\delta}{\sqrt{1+\tau} }|z_y|(y,\tau)\,\mathrm{d}y \,\mathrm{d}x \,\mathrm{d}\tau\\
			& +O(1) \int_{t-2 \nu_0}^{t-\nu_0} \int_{\mathbb{R}}\int_{\mathbb{R}} \frac{1}{\nu_0} \frac{e^{-\frac{(x-y)^2}{C_*(t-\tau)}}}{t-\tau}\left[|\log (t-\tau)| \sup _{t-2 \nu_0<s<t}\|v(\cdot, s)-1\|_{L_x^{\infty}}+\sup _{t-2 \nu_0<s<t}\|v(\cdot, s)-1\|_{BV}\right]\\
			&\qquad z(y,\tau) \,\mathrm{d}y\,\mathrm{d}x\,\mathrm{d}\tau \\
			& +O(1) \int_{t-2 \nu_0}^{t-\nu_0} \int_{\mathbb{R}}\int_{\mathbb{R}} \frac{1}{\nu_0} \frac{e^{-\frac{(x-y)^2}{C_*(t-\tau)}}}{t-\tau} \sqrt{t-\tau}\left[\sup _{t-2 \nu_0<s<t}\left\|\sqrt{s} u_x(\cdot, s)\right\|_{L_x^{\infty}}+\sup _{t-2 \nu_0<s<t}\|v(\cdot, s)-1\|_{L_x^1}\right] \\
			&\qquad z(y,\tau) \,\mathrm{d}y\,\mathrm{d}x \,\mathrm{d}\tau \\
			& +O(1) \int_{t-2 \nu_0}^{t-\nu_0}\int_{\tau}^t \frac{1}{\nu_0}  \int_{\mathbb{R}}\int_{\mathbb{R}}\int_{\mathbb{R}}\frac{e^{\frac{-(\omega-y)^2}{s-\tau}}}{s-\tau}\left(e^{-\sigma_0^*(t-s)-\sigma_0 |x-\omega|}+(t-s)e^{-\sigma_0|x-\omega|}\right)z(y,\tau)\,\mathrm{d}\omega\,\mathrm{d}y\,\mathrm{d}x\,\mathrm{d}s \,\mathrm{d}\tau \\
			\leq &O(1)\int_{t-2\nu_0}^{t-\nu_0}\left(e^{-\sigma_0^*(t-\tau)}+(t-\tau)\right)\frac{\delta^2}{\sqrt{1+\tau}}\,\mathrm{d}\tau+O(1)\int_{t-2\nu_0}^{t-\nu_0}\left(e^{-\sigma_0^*(t-\tau)}+(t-\tau)\right)\delta \,\mathrm{d}\tau\\
			&+O(1)\int_{t-2\nu_0}^{t-\nu_0}\left(\frac{1}{t-\tau}+e^{-\sigma_0^*(t-\tau)}+(t-\tau)\right)\frac{\delta^2}{\sqrt{1+\tau}}\,\mathrm{d}\tau\\
			&+O(1)\int_{t-2\nu_0}^{t-\nu_0}\frac{1}{\nu_0}\left(\frac{1}{\sqrt{t-\tau}}\left[|\log(t-\tau)|+1+\sqrt{t-\tau}\right]\right)\delta^2 \,\mathrm{d}\tau\\
			&+O(1)\int_{t-2\nu_0}^{t-\nu_0}\int_{\tau}^{t}\frac{1}{\nu_0}\frac{1}{\sqrt{s-\tau}}\left(e^{-\sigma_0^*(t-s)}+(t-s)\right)\delta \,\mathrm{d}s\,\mathrm{d}\tau\\
			\leq &O(1)\left(\nu_0\delta^2+\nu_0^2\delta^2+\nu_0\delta+\nu_0^2\delta+\delta^2+\frac{|\log(\nu_0)|\delta^2}{\sqrt{\nu_0}}+\sqrt{\nu_0}\delta^2+\sqrt{\nu_0}\delta\right).
		\end{align*}

        Finally, for the $L^1_x$-norm of $\partial_x\mathcal{R}_3^z$, we bound that
               \begin{align*}
\int_\R\left|\partial_x\mathcal{R}_3^{z}(x,t)\right|\,\mathrm{d}x &\leq \int_{t-\nu_0}^t \int_\R\int_\R  \left|\partial_{xy}\mathbb{G}_{42}(x-y,t-\tau)\right|\left|\mathcal{M}_1(y,\tau)\right|\,\mathrm{d} y\,\mathrm{d} x\,\mathrm{d} \tau\nonumber\\
&\qquad +  \int_{t-2\nu_0}^{t-\nu_0} \int_\R \int_\R \left|\partial_{xy}\mathbb{G}_{43}(x-y,t-\tau)\right|\left|\mathcal{M}_2(y,\tau)\right| 
    \,\mathrm{d} y\,\mathrm{d} x\,\mathrm{d} \tau
        \end{align*}
    where, as before,
    \begin{align*}
    &\mathcal{M}_1 := \frac{a(v-1)^2}{v} + a \frac{(\theta-1)(1-v)}{v} - \frac{au^2}{2c_v} + \frac{\mu u_y(v-1)}{v} - \frac{qaz}{c_v},\\
    &\mathcal{M}_2 := \frac{a(\theta-1)+a(1-v)}{v}u+ \frac{\nu (v-1)}{v}\theta_y + \left(\frac{\nu }{c_v v} - \frac{\mu}{v} \right)uu_y - \left( \frac{q\nu}{c_v}-qD \right)z_y.
\end{align*}
To this end, we adapt the arguments for $\partial_x\mathcal{R}_1^z$. Recall that $\|v-1\|_{L^\infty}, \|u\|_{L^\infty}, \dots, \|z_y\|_{L^\infty} \leq \frac{\delta}{\sqrt{1+\tau}}$ (all evaluated at time $\tau$) due to the assumption $\mathcal{G}(\tau)<\delta$. We thus bound
\begin{align*}
&\left|\mathcal{M}_1(y,\tau)\right|\chi_{[t-\nu_0,t]}(\tau) +\left|\mathcal{M}_2(y,\tau)\right|\chi_{[t-2\nu_0,t-\nu_0]}(\tau) \\ &\qquad\leq O(1) \frac{\delta}{\sqrt{1+\tau}} \Big\{|v-1| + |\theta-1| + |u| + |u_y| + |\theta_y|\Big\}\chi_{[t-\nu_0,t]}(\tau) + O(1)\left\{z+|z_y|\right\}\chi_{[t-2\nu_0,t-\nu_0]}(\tau)\\
&\qquad \leq O(1) \left\{\frac{\delta^2}{\sqrt{1+\tau}} + \delta\right\}\chi_{[t-2\nu_0,t]}(\tau).
\end{align*} 
Also recall from Eqs.~\eqref{G4kxx,tg1}, \eqref{G4kxx,tl1} the pointwise bounds for $\partial_{xy}\mathbb{G}_{4j}(x-y,t-\tau)$ for $j \in \{2,3\}$. One thus obtains that 	\begin{align*}
\int_{\mathbb{R}}\left|\partial_x\mathcal{R}_3^{z}(x,t)\right|\,\mathrm{d}x &\leq O(1)\int_{t-2\nu_0}^{t}\left(e^{-\sigma_0^*(t-\tau)}+(t-\tau)\right)\left\{\frac{\delta^2}{\sqrt{1+\tau}}+\delta\right\}\,\mathrm{d}\tau \\
		&\leq O(1)\left(\nu_0\delta^2+\nu_0^2\delta^2+\nu_0\delta+\nu_0^2\delta\right).
		\end{align*}

		Therefore, for the sufficiently small $\delta$ and $t\geq t_{\sharp}\geq 4\nu_0$, we obtain the $L^{1}$-estimates for $z_x(x,t)$ as follows:
		\begin{equation}
			\left\|z_x(\cdot,t)\right\|_{L_x^{1}}\leq C(\nu_0)\delta^*+ O(1)\left(\frac{|\log(\nu_0)|\delta^2}{\sqrt{\nu_0}}+\sqrt{\nu_0}\delta\right).
		\end{equation}

		Finally, combining the $L^\infty \cap L^1$-estimates for $\theta_x$ (in absence of $z$) in \cite[Lemma 4.8]{WangHT2021}, the estimates for $z_x$ derived above, and the bounds for $u_x(x,t)$ in  Lemma~\ref{lemma: es u large t}, one arrives at the required estimates for $\theta_x$. This completes the proof.  
	\end{proof}

	The following Lemma is adapted from \cite[Lemma 4.9]{WangHT2021}; the proof is safely omitted here.
	\begin{lemma}
		\label{lemma: es v large t}
		Let $(v, u, \theta, z)$ be the local solution constructed in Theorem \ref{thm: local exi}, and assume that condition \eqref{tau condition} holds.
		Then, for $t>t_{\sharp}$, $v(x,t)$ satisfies the following estimates
		\begin{equation*}
			\left\{\begin{array}{l}
				\|v(\cdot, t)-1\|_{L_x^1} \leq C(\nu_0) \delta^*+O(1)\delta^2, \\
				\|\sqrt{1+t}\,(v(\cdot, t)-1)\|_{L_x^{\infty}} \leq C(\nu_0) \delta^*+O(1)\delta^2,\\
				\|v(\cdot,t)\|_{BV}\leq C(\nu_0) \delta^*+O(1)\frac{\delta^2}{\sqrt{\nu_0}}.
			\end{array}\right.
		\end{equation*}
	\end{lemma}

With the above preparations, now we may deduce the main theorem of this section. 
    
	\begin{theorem}[Global existence]
 There exists a universal constant $\delta^*>0$ such that the following holds. Suppose that the initial data $(v_0,u_0,\theta_0,z_0)$ satisfies		\begin{equation*}
			\left\|v_0-1\right\|_{L_x^1}+\left\|v_0\right\|_{B V}+\left\|u_0\right\|_{L_x^1}+\left\|u_0\right\|_{B V}+\left\|\theta_0-1\right\|_{L_x^1}+\left\|\theta_0\right\|_{B V}+\left\|z_0\right\|_{L_x^1}+\left\|z_0\right\|_{B V} \leq \delta^*.
		\end{equation*}
	Then the local solution to Eq.~\eqref{PDE,2} constructed in Theorems~\ref{thm: local exi} and \ref{thm: sta} extends globally in time.

    Moreover, there exists a positive constant $C_3$ such that the global solution satisfies the following large-time behaviour:
		\begin{equation}
			\begin{aligned}
				& \|\sqrt{t+1}(v(\cdot, t)-1)\|_{L_x^{\infty}}+\|\sqrt{t+1} u(\cdot, t)\|_{L_x^{\infty}}+\|\sqrt{t+1}(\theta(\cdot, t)-1)\|_{L_x^{\infty}}+\|\sqrt{t+1}z(\cdot, t)\|_{L_x^{\infty}} \\
				& \quad+\left\|\sqrt{t} u_x(\cdot, t)\right\|_{L_x^{\infty}}+\left\|\sqrt{t} \theta_x(\cdot, t)\right\|_{L_x^{\infty}}+\left\|\sqrt{t} z_x(\cdot, t)\right\|_{L_x^{\infty}} \leq C_3 \delta^{*} \quad \text{for} \quad t \in(0,+\infty).
			\end{aligned}
		\end{equation}
	\end{theorem}
	\begin{proof}
The assertion follows from a standard continuity argument.
    
		Let $C_{\sharp}$, $t_{\sharp}$ and $\delta$ be the parameters as in Theorem~\ref{thm: local exi}. In view of Theorems~\ref{thm: local exi} and \ref{thm: sta}, the unique weak solution $(v, u, \theta, z)$ exists on $[0,t_{\sharp})$. Under the smallness assumption 
		\begin{equation}
			\label{ini last}
			\left\|v_0-1\right\|_{L_x^1}+\left\|v_0\right\|_{B V}+\left\|u_0\right\|_{L_x^1}+\left\|u_0\right\|_{B V}+\left\|\theta_0-1\right\|_{L_x^1}+\left\|\theta_0\right\|_{B V}+\left\|z_0\right\|_{L_x^1}+\left\|z_0\right\|_{B V} \leq \delta^*,
		\end{equation}
		we may apply Lemma~\ref{lemma: T+t} to define a stopping time $T$ as in Eq.~\eqref{stop time}, such that
		\begin{equation*}
			T>t_{\sharp},\quad \mathcal{G}(T)\geq \delta,\quad \mathcal{G}(\tau)<\delta \quad \text{for all} \quad\tau<T,
		\end{equation*}
		provided that $\delta^*$ is sufficiently small. By the definition of $\mathcal{G}(\tau)$ and Lemma~\ref{lemma: T+t}, the lifespan of the solution $(v, u, \theta, z)$ is larger than $T$. Thus, to establish the global existence of the
		solution, it suffices to show $T=+\infty$ for sufficient small $\delta^*$.

        Suppose for  contradiction that
		\begin{equation}
			\label{contra}
			T<+\infty \quad \text{and} \quad \mathcal{G}(T)\geq \delta, \quad \text{for arbitrary positive }  \delta^*.
		\end{equation}
		Applying Lemmas \ref{lemma: es u large t}--\ref{lemma: es v large t}, we obtain the following estimate at time $T$:
		\begin{equation*}
			\mathcal{G}(T) \leq C(\nu_0)\delta^*+O(1) \frac{\left|\log \left(\nu_0\right)\right|}{\sqrt{\nu_0}} \delta^2+O(1) \sqrt{\nu_0} \delta+\left(C\left(\nu_0\right) \delta^*+O(1) \frac{\left|\log \left(\nu_0\right)\right|}{\sqrt{\nu_0}} \delta^2+O(1) \sqrt{\nu_0} \delta\right)^2,
		\end{equation*}
		where $\nu_0$ is a small positive constant with $\nu_0\leq \frac{t_{\sharp}}{4}$. Note that $C_{\sharp}$ and $t_{\sharp}$ remain uniformly bounded as $\delta^*$ and $\delta$ become small, and the O(1) coefficients are independent of $\nu_0$, $\delta$, $\delta^*$. Therefore, we can first choose $\nu_0$ to be sufficiently small such that
		\begin{equation*}
			O(1)\sqrt{\nu_0}\delta\leq \frac{\delta}{6}.
		\end{equation*}

		Fix this $\nu_0$ and then choose $\delta$ so small that
		\begin{equation*}
			O(1)\frac{\left|\log \left(\nu_0\right)\right|}{\sqrt{\nu_0}} \delta^2\leq \frac{\delta}{6},
		\end{equation*}
		and then choose $\delta^*$ so small that
		\begin{equation}
			\label{relation delta}
			C(\nu_0)\delta^*\leq \frac{\delta}{6}.
		\end{equation}
        Then we have
		\begin{equation*}
			\mathcal{G}(T)<\delta,
		\end{equation*}
		which contradicts the assumption~\eqref{contra}. Hence, $T=+\infty$.
		
		Therefore, for sufficiently small $\delta^*>0$  that verifies Eq.~\eqref{ini last}, there exists $\delta$ such that $\mathcal{G}(\tau)<\delta$ for all $\tau>0$. This implies the global existence, uniqueness and the large time behaviour of the weak solution. The positive constant $C_3$ is directly determined by Eq.~\eqref{relation delta}.

The proof is now complete.  
	\end{proof}
	
	\newpage
	\appendix
	\section{}\label{appendix, new}

In the appendix, we collect several lengthy yet direct computations omitted from the main text. We first tabulate the coefficients in the expansion of $\lambda_j$ at infinity as in \S\ref{sec: high freq}, Eq.~\eqref{appro lambda}. See Yu--Wang--Zhang \cite{WangHT2021}. 	\begin{align*}
		\beta_1^*= & \frac{v p_v}{\mu}, \\
		A_{1,1}= & -\frac{v^3\left(\nu \theta_e p_v^2+\mu p p_e p_v\right)}{\nu \mu^3 \theta_e}, \\
		A_{1,2}= & \frac{v^3\left(\mu^2 p^2 v^2 p_e^2 p_v+2 \nu^2 v^2 \theta_e^2 p_v^3+3 \nu \mu p v^2 \theta_e p_e p_v^2+\mu^2 p v^2 p_e p_v^2-\nu^2 \mu^2 \theta_e^2 p_v^2+\nu \mu^3(-p) \theta_e p_e p_v\right)}{\nu^2 \mu^5 \theta_e^2}, \\
		A_{1,3}= & -\frac{v^3 p_v\left(\mu^3 p^3 v^4 p_e^3+\mu^2 p^2 v^2 p_e^2\left(3 v^2 p_v\left(2 \nu \theta_e+\mu\right)-2 \nu \mu^2 \theta_e\right)+\nu^3 \theta_e^3 p_v\left(\mu^4+5 v^4 p_v^2-4 \mu^2 v^2 p_v\right)\right)}{\nu^3 \mu^7 \theta_e^3}\\
		& -\frac{v^3 p_v\left(\mu p p_e\left(\nu^2 \mu^4 \theta_e^2+v^4 p_v^2\left(10 \nu^2 \theta_e^2+4 \nu \mu \theta_e+\mu^2\right)-2 \nu \mu^2 v^2 \theta_e p_v\left(3 \nu \theta_e+\mu\right)\right)\right)}{\nu^3 \mu^7 \theta_e^3},\\
		\alpha_2^*= & \frac{\mu}{v},\\
		\beta_2^*= & \frac{v\left(\mu p p_e+\nu \theta_e p_v-\mu p_v\right)}{\mu\left(\mu-\nu \theta_e\right)}, \\
		A_{2,1}= & \frac{v^3\left(\mu^3 p^2 p_e^2-\mu p p_e p_v\left(\nu^2 \theta_e^2-3 \nu \mu \theta_e+2 \mu^2\right)+p_v^2\left(\mu-\nu \theta_e\right)^3\right)}{\mu^3\left(\mu-\nu \theta_e\right)^3}, \\
		A_{2,2}= & \frac{v^3\left(\mu p p_e+p_v\left(\nu \theta_e-\mu\right)\right)\left(2 \mu^4 p^2 v^2 p_e^2+\mu p p_e\left(\mu-\nu \theta_e\right)\left(\mu^3\left(\mu-\nu \theta_e\right)-v^2 p_v\left(\nu \theta_e-2 \mu\right)^2\right)\right)}{\mu^5\left(\mu-\nu \theta_e\right)^5} \\
		& +\frac{v^3\left(\mu p p_e+p_v\left(\nu \theta_e-\mu\right)\right)\left(p_v\left(2 v^2 p_v-\mu^2\right)\left(\mu-\nu \theta_e\right)^4\right)}{\mu^5\left(\mu-\nu \theta_e\right)^5}, \\
		A_{2,3}= & \frac{v^3\left(5 \mu^7 p^4 v^4 p_e^4+\mu^3 p^3 v^2 p_e^3\left(\mu-\nu \theta_e\right)\left(4 \mu^4\left(\mu-\nu \theta_e\right)+v^2 p_v\left(\nu^3 \theta_e^2-6 \nu^2 \mu \theta_e^2+15 \nu \mu^2 \theta_e-20 \mu^3\right)\right)\right)}{\mu^7\left(\mu-\nu \theta_e\right)^7} \\
		& +\frac{v^3 \mu^2 p^2 p_e^2\left(\mu-\nu \theta_e\right)^2\left(\mu^5\left(\mu-\nu \theta_e\right)^2+3 v^4 p_v^2\left(-2 \nu^3 \theta_e^3+9 \nu^2 \mu \theta_e^2-15 \nu \mu^2 \theta_e+10 \mu^3\right)\right)}{\mu^7\left(\mu-\nu \theta_e\right)^7} \\
		& -\frac{v^3 \mu^2 p^2 p_e^2\left(\mu-\nu \theta_e\right)^2 2 \mu^2 v^2 p_v\left(-\nu^3 \theta_e^3+5 \nu^2 \mu \theta_e^2-10 \nu \mu^2 \theta_e+6 \mu^3\right)}{\mu^7\left(\mu-\nu \theta_e\right)^7}\\
		& -\frac{v^3 \mu p p_e p_v\left(\mu-\nu \theta_e\right)^3\left(\mu^4\left(\mu-\nu \theta_e\right)^2\left(2 \mu-\nu \theta_e\right)+v^4 p_v^2\left(-10 \nu^3 \theta_e^3+36 \nu^2 \mu \theta_e^2-45 \nu \mu^2 \theta_e+20 \mu^3\right)\right)}{\mu^7\left(\mu-\nu \theta_e\right)^7} \\
		& -\frac{v^3 \mu p p_e p_v\left(\mu-\nu \theta_e\right)^3\left(2 \mu^2 v^2 p_v\left(3 \nu^3 \theta_e^2-11 \nu^2 \mu \theta_e^2+14 \nu \mu^2 \theta_e-6 \mu^3\right)\right)}{\mu^7\left(\mu-\nu \theta_e\right)^7} \\
		& +\frac{v^3 p_v^2\left(\mu^4+5 v^4 p_v^2-4 \mu^2 v^2 p_v\right)\left(\mu-\nu \theta_e\right)^7}{\mu^7\left(\mu-\nu \theta_e\right)^7}.\\
		\alpha_3^*= & \frac{\nu \theta_e}{v}, \\
		\beta_3^*= & \frac{p v p_e}{\nu \theta_e-\mu}, \\
		A_{3,1}= & \frac{p v^3 p_e\left(\nu p \theta_e p_e+p_v\left(\nu \theta_e-\mu\right)\right)}{\nu \theta_e\left(\nu \theta_e-\mu\right)^3}, \\
		A_{3,2}= & \frac{p v^3 p_e\left(2 \nu^2 p^2 v^2 \theta_e^2 p_e^2+p p_e\left(\mu-\nu \theta_e\right)\left(\nu^2 \theta_e^2\left(\mu-\nu \theta_e\right)+v^2 p_v\left(\mu-4 \nu \theta_e\right)\right)\right)}{\nu^2 \theta_e^2\left(\nu \theta_e-\mu\right)^5} \\
		& +\frac{p v^3 p_e\left(p_v\left(\mu-\nu \theta_e\right)^2\left(\nu \theta_e\left(\nu \theta_e-\mu\right)+v^2 p_v\right)\right)}{\nu^2 \theta_e^2\left(\nu \theta_e-\mu\right)^5},\\
		A_{3,3}= & \frac{p v^3 p_e\left(5 \nu^3 p^3 v^4 \theta_e^3 p_e^2-p^2 v^2 p_e^2\left(\mu-\nu \theta_e\right)\left(4 \nu^3 \theta_e^3\left(\nu \theta_e-\mu\right)+v^2 p_v\left(15 \nu^2 \theta_e^2-6 \nu \mu \theta_e+\mu^2\right)\right)\right)}{\nu^3 \theta_e^3\left(\nu \theta_e-\mu\right)^7} \\
		& +\frac{p v^3 p_e p p_e\left(\mu-\nu \theta_e\right)^2\left(\nu^3 \theta_e^3\left(\mu-\nu \theta_e\right)^2-3 v^4 p_v^2\left(\mu-3 \nu \theta_e\right)+2 \nu v^2 \theta_e p_v\left(4 \nu^2 \theta_e^2-5 \nu \mu \theta_e+\mu^2\right)\right)}{\nu^3 \theta_e^3\left(\nu \theta_e-\mu\right)^7} \\
		& +\frac{p v^3 p_e\left(-p_v\left(\mu-\nu \theta_e\right)^3\left(\nu \theta_e\left(\nu \theta_e-\mu\right)+v^2 p_v\right)^2\right)}{\nu^3 \theta_e^3\left(\nu \theta_e-\mu\right)^7}.
	\end{align*}

    Next, the matrices $M_j^{*,k}$ in Lemma~\ref{lemma: M star} are as follows:
	\begin{align*}
		& M_1^{*, 0}=\left(\begin{array}{cccc}
			1 & 0 & 0 & 0 \\
			0 & 0 & 0 & 0\\
			0 & 0 & 0 & 0\\
			0 & 0 & 0 & 0
		\end{array}\right), \quad M_1^{*, 1}=\left(\begin{array}{cccc}
			0 & \frac{v}{\mu} & 0 & 0\\
			-\frac{v p_v}{\mu} & 0 & 0 & 0\\
			-\frac{u v p_v}{\mu} & 0 & 0 & 0\\
			0 & 0 & 0 & 0
		\end{array}\right), \\
		&M_1^{*, 2}=\left(\begin{array}{cccc}
			-\frac{v^2 p_v}{\mu^2} & -\frac{u v^2 p_e}{\nu \mu \theta_e} & \frac{v^2 p_e}{\nu \mu \theta_e} & \frac{-q v^2 p_e}{\nu \mu \theta_e}\\
			0 & \frac{v^2 p_v}{\mu^2} & 0 & 0\\
			-\frac{p v^2 p_v}{\nu \mu \theta_e} & \frac{u v^2 p_v}{\mu^2} & 0 & 0\\
			0 & 0 & 0 & 0
		\end{array}\right), \\
		& M_1^{*, 3}=\left(\begin{array}{cccc}
			0 & -\frac{v^3\left(p \mu p_e+2 \nu p_v \theta_e\right)}{\nu \mu^3 \theta_e} & 0 & 0\\
			\frac{v^3 p_v\left(p \mu p_e+2 \nu p_v \theta_e\right)}{\nu \mu^3 \theta_e} & \frac{u v^3 p_e p_v}{\nu \mu^2 \theta_e} & -\frac{v^3 p_e p_v}{\nu \mu^2 \theta_e} & \frac{q v^3 p_e p_v}{\nu \mu^2 \theta_e}\\
			\frac{u v^3 p_v\left(p \mu p_e+2 \nu p_v \theta_e\right)}{\nu \mu^3 \theta_e} & -\frac{v^3\left(p p_v-u^2 p_e p_v\right)}{\nu \mu^2 \theta_e} & -\frac{u v^3 p_e p_v}{\nu \mu^2 \theta_e} &\frac{q u v^3 p_e p_v}{\nu \mu^2 \theta_e}\\
			0 & 0 & 0 & 0
		\end{array}\right), \\
		& M_1^{*, 4}=(\xi_1,\xi_2),\quad
		\xi_1=\left(\begin{array}{cc}
			\frac{v^4\left(p p_e p_v \mu^2+2 p \nu p_e p_v \theta_e \mu+3 \nu^2 p_v^2 \theta_e^2\right)}{\nu^2 \mu^4 \theta_e^2} & \frac{u v^4 p_e\left(p \mu p_e+\mu p_v+2 \nu p_v \theta_e\right)}{\nu^2 \mu^3 \theta^2}\\
			0 & \frac{-2 p \mu p_e p_v v^3-3 \nu p_v^2 \theta_e v^4}{\nu \mu^4 \theta_e}\\
			\frac{p v^4 p_v\left(p \mu p_e+\mu p_v+2 \nu p_v \theta_e\right)}{\nu^2 \mu^3 \theta_e^2} & \frac{v^4\left(p u p_e p_v \mu^2-2 p u \nu p_e p_v \theta_e \mu-3 u \nu^2 p_v^2 \theta_e^2\right)}{\nu^2 \mu^4 \theta_e^2}\\
			0 & 0
		\end{array}\right),\\
		&\xi_2=\left(\begin{array}{cc}
			-\frac{v^4 p_e\left(p \mu p_e+\mu p_v+2 \nu p_v \theta_e\right)}{\nu^2 \mu^3 \theta_e^2} & \frac{q v^4 p_e\left(p \mu p_e+\mu p_v+2 \nu p_v \theta_e\right)}{\nu^2 \mu^3 \theta_e^2}\\
			0 & 0\\
			-\frac{p v^4 p_e p_v}{\nu^2 \mu^2 \theta_e^2} & \frac{q p v^4 p_e p_v}{\nu^2 \mu^2 \theta_e^2}\\
			0 & 0
		\end{array}\right)\\
		& M_2^{*, 0}=\left(\begin{array}{cccc}
			0 & 0 & 0 & 0\\
			0 & 1 & 0 & 0\\
			0 & u & 0 & 0\\
			0 & 0 & 0 & 0
		\end{array}\right), \quad M_2^{*, 1}=\left(\begin{array}{cccc}
			0 & -\frac{v}{\mu} & 0 & 0\\
			\frac{v p_v}{\mu} & -\frac{u v p_e}{\mu-\nu \theta_e} & \frac{v p_e}{\mu-\nu \theta_e} & -\frac{qv p_e}{\mu-\nu \theta_e}\\
			\frac{u v p_v}{\mu} & \frac{v\left(p-u^2 p_e\right)}{\mu-\nu \theta_e} & \frac{u v p_e}{\mu-\nu \theta_e} & -\frac{qu v p_e}{\mu-\nu \theta_e}\\
			0 & 0 & 0 & 0
		\end{array}\right), \\
		& M_2^{*, 2}=\left(\begin{array}{cccc}
			\frac{v^2 p_v}{\mu^2} & -\frac{u v^2 p_e}{\mu^2-\nu \nu \theta_e} & \frac{v^2 p_e}{\mu^2-\nu \mu \theta_e}  & -\frac{qv^2 p_e}{\mu^2-\nu \mu \theta_e}\\
			0 & \frac{v^2\left(p \mu^2 p_e-p_v\left(\mu-\nu \theta_e\right)^2\right)}{\mu^2\left(\mu-\nu \theta_e\right)^2} & 0 & 0\\
			-\frac{p v^2 p_v}{\mu^2-\nu \mu \theta_e} & \frac{u v^2\left(2 p \mu^2 p_e-p_v\left(\mu-\nu \theta_e\right)^2\right)}{\mu^2\left(\mu-\nu \theta_e\right)^2} & -\frac{p v^2 p_e}{\left(\mu-\nu \theta_e\right)^2} & \frac{qp v^2 p_e}{\left(\mu-\nu \theta_e\right)^2}\\
			0 & 0 & 0 & 0
		\end{array}\right), \\
		& M_2^{*, 3}=(\zeta_1,\zeta_2),\quad 
		\zeta_1=\left(\begin{array}{cc}
			0 & \frac{v^3\left(2 p_v+\frac{p \mu p_e\left(\nu \theta_e-2 \mu\right)}{\left(\mu-\nu \theta_e\right)^2}\right)}{\mu^3}\\
			-\frac{v^3 p_v\left(2 p_v+\frac{p \mu p_e\left(\nu \theta_e-2 \mu\right)}{\left(\mu-\nu \theta_e\right)^2}\right)}{\mu^3} & -\frac{u v^3 p_e\left(2 p \mu^2 p_e-p_v\left(2 \mu^2-3 \nu \theta_e \mu+\nu^2 \theta_e^2\right)\right)}{\mu^2\left(\mu-\nu \theta_e\right)^3}\\
			-\frac{u v^3 p_v\left(2 p_v\left(\mu-\nu \theta_e\right)^2+p \mu p_e\left(\nu \theta_e-2 \mu\right)\right)}{\mu^3\left(\mu-\nu \theta_e\right)^2} & \frac{v^3\left(p-u^2 p_e\right)\left(2 p \mu^2 p_e-p_v\left(2 \mu^2-3 \nu \theta_e \mu+\nu^2 \theta_e^2\right)\right)}{\mu^2\left(\mu-\nu \theta_e\right)^3}\\
			0 & 0
		\end{array}\right),\\
		&\zeta_2=\left(\begin{array}{cc}
			0 & 0\\
			\frac{v^3 p_e\left(2 p \mu^2 p_e-p_v\left(2 \mu^2-3 \nu \theta_e \mu+\nu^2 \theta_e^2\right)\right)}{\mu^2\left(\mu-\nu \theta_e\right)^3} & -\frac{q v^3 p_e\left(2 p \mu^2 p_e-p_v\left(2 \mu^2-3 \nu \theta_e \mu+\nu^2 \theta_e^2\right)\right)}{\mu^2\left(\mu-\nu \theta_e\right)^3}\\
			\frac{u v^3 p_e\left(2 p \mu^2 p_e-p_v\left(2 \mu^2-3 \nu \theta_e \mu+\nu^2 \theta_e^2\right)\right)}{\mu^2\left(\mu-\nu \theta_e\right)^3} & -\frac{qu v^3 p_e\left(2 p \mu^2 p_e-p_v\left(2 \mu^2-3 \nu \theta_e \mu+\nu^2 \theta_e^2\right)\right)}{\mu^2\left(\mu-\nu \theta_e\right)^3}\\
			0 & 0
		\end{array}\right),\\
		& M_2^{*, 4}=\left(\eta_1, \eta_2, \eta_3\right), \quad \eta_1=\left(\begin{array}{c}
			\frac{v^4 p_v\left(\frac{p \mu p_e\left(3 \mu-2 \nu \theta_e\right)}{\left(\mu-\nu \theta_e\right)^2}-3 p_v\right)}{\mu^4} \\
			0 \\
			-\frac{p v^4 p_v\left(p \mu p_e\left(3 \mu-\nu \theta_e\right)+p_v\left(-3 \mu^2+5 \nu \theta_e \mu-2 \nu^2 \theta_e^2\right)\right)}{\mu^3\left(\mu-\nu \theta_e\right)^3}\\
			0
		\end{array}\right), \\
		& \eta_2=\left(\begin{array}{c}
			-\frac{u v^4 p_e\left(p \mu p_e\left(3 \mu-\nu \theta_e\right)+p_v\left(-3 \mu^2+5 \nu \theta_e \mu-2 \nu^2 \theta_e^2\right)\right)}{\mu^3\left(\mu-\nu \theta_e\right)^3} \\
			\frac{v^4\left(3 p^2 p_e^2 \mu^4+2 p p_e p_v\left(-3 \mu^3+6 \nu_e \mu^2-4 \nu^2 \theta_e^2 \mu+\nu^3 \theta_e^3\right) \mu+3 p_v^2\left(\mu-\nu \theta_e\right)^4\right)}{\mu^4\left(\mu-\nu \theta_e\right)^4}\\
			\frac{u v^4\left(6 p^2 p_e^2 \mu^4+p p_e p_v\left(-9 \mu^3+16 \nu \theta_e \mu^2-9 \nu^2 \theta_e^2 \mu+2 \nu^3 \theta_e^3\right) \mu+3 p_v^2\left(\mu-\nu \theta_e\right)^4\right)}{\mu^4\left(\mu-\nu \theta_e\right)^4} 
		\end{array}\right),\\
		&\eta_3=\left(\begin{array}{cc}
			\frac{v^4 p_e\left(p \mu p_e\left(3 \mu-\nu \theta_e\right)+p_v\left(-3 \mu^2+5 \nu \theta_e \mu-2 \nu^2 \theta_e^2\right)\right)}{\mu^3\left(\mu-\nu \theta_e\right)^3} & -\frac{qv^4 p_e\left(p \mu p_e\left(3 \mu-\nu \theta_e\right)+p_v\left(-3 \mu^2+5 \nu \theta_e \mu-2 \nu^2 \theta_e^2\right)\right)}{\mu^3\left(\mu-\nu \theta_e\right)^3} \\
			0 & 0\\
			\frac{p v^4 p_e\left(p_v\left(3 \mu^2-4 \nu \theta_e \mu+\nu^2 \theta_e^2\right)-3 p \mu^2 p_e\right)}{\mu^2\left(\mu-\nu \theta_e\right)^4} & -\frac{q p v^4 p_e\left(p_v\left(3 \mu^2-4 \nu \theta_e \mu+\nu^2 \theta_e^2\right)-3 p \mu^2 p_e\right)}{\mu^2\left(\mu-\nu \theta_e\right)^4}\\
			0 & 0
		\end{array}\right)
	\end{align*}

	The following are the matrices $M_j^0$ and $M_j^1$ in Lemma~\ref{lemma: regular M} and Theorem~\ref{thm: es for regular G}:
	\begin{align*}
		& M_1^0=\left(\begin{array}{cccc}
			\frac{p p_e}{p p_e-p_v} & -\frac{u p_e}{p p_e-p_v} & \frac{p_e}{p p_e-p_v} & 0\\
			0 & 0 & 0 & 0\\
			-\frac{p p_v}{p p_e-p_v} & \frac{u p_v}{p p_e-p_v} & -\frac{p_v}{p p_e-p_v} & 0\\
			0 & 0 & 0 & 0
		\end{array}\right), \\
		& M_2^0=\left(\begin{array}{cccc}
			-\frac{p_v}{2 p p_e-2 p_v} & \frac{u p_e-\sqrt{p p_e-p_v}}{2 p p_e-2 p_v} & -\frac{p_e}{2 p p_e-2 p_v} & 0\\
			\frac{p_v}{2 \sqrt{p p_e-p_v}} & \frac{1}{2}-\frac{u p_e}{2 \sqrt{p p_e-p_v}} & \frac{p_e}{2 \sqrt{p p_e-p_v}} & 0\\
			\frac{\left(p+u \sqrt{p p_e-p_v}\right) p_v}{2 p p_e-2 p_v} & \frac{-p_e \sqrt{p p_e-p_v} u^2-p_v u+p \sqrt{p p_e-p_v}}{2 p p_e-2 p_v} & \frac{p_e\left(p+u \sqrt{p p_e-p_v}\right)}{2 p p_e-2 p_v} & 0\\
			0 & 0 & 0 & 0
		\end{array}\right), \\
		& M_3^0=\left(\begin{array}{cccc}
			-\frac{p_v}{2 p p_e-2 p_v} & \frac{u p_e+\sqrt{p p_e-p_v}}{2 p p_e-2 p_v} & -\frac{p_e}{2 p p_e-2 p_v} & 0\\
			-\frac{p_v}{2 \sqrt{p p_e-p_v}} & \frac{1}{2}\left(\frac{u p_e}{\sqrt{p p_e-p_v}}+1\right) & -\frac{p_e}{2 \sqrt{p p_e-p_v}} & 0\\
			\frac{\left(p-u \sqrt{p p_e-p_v}\right) p_v}{2 p p_e-2 p_v} & -\frac{-p_e \sqrt{p p_e-p_v} u^2+p_v u+p \sqrt{p p_e-p_v}}{2 p p_e-2 p_v} & \frac{p_e\left(p-u \sqrt{p p_e-p_v}\right)}{2 p p_e-2 p_v} & 0\\
			0 & 0 & 0 & 0
		\end{array}\right) .
	\end{align*}
	\begin{align*}
		& M_1^1=\left(\begin{array}{cccc}
			0 & \frac{p \nu p_e \theta_e}{v\left(p_v-p p_e\right)^2} & 0 & 0\\
			-\frac{p \nu p_e p_v \theta_e}{v\left(p_v-p p_e\right)^2} & \frac{u \nu p_e p_v \theta_e}{v\left(p_v-p p_e\right)^2} & -\frac{\nu p_e p_v \theta_e}{v\left(p_v-p p_e\right)^2} & 0\\
			-\frac{p u \nu p_e p_v \theta_e}{v\left(p_v-p p_e\right)^2} & -\frac{\nu\left(p-u^2 p_e\right) p_v \theta_e}{v\left(p_v-p p_e\right)^2} & -\frac{u \nu p_e p_v \theta_e}{v\left(p_v-p p_e\right)^2} & 0\\
			0 & 0 & 0 & 0
		\end{array}\right), \\
		& M_2^1=\left(\xi_1, \xi_2, \xi_3\right), \\
		& \xi_1=\left(\begin{array}{c}
			\frac{p_v\left(\mu p_v-p p_e\left(\mu+3 \nu \theta_e\right)\right)}{4 v\left(p p_e-p_v\right)^{5/2}} \\
			\frac{p \nu p_e p_v \theta_e}{2 v\left(p_v-p p_e\right)^2} \\
			\frac{p p_v\left(p_e\left(p \mu+\nu\left(p+2 u \sqrt{p p_e-p_v}\right) \theta_e\right)-p_v\left(\mu-2 \nu \theta_e\right)\right)}{4 v\left(p p_e-p_v\right)^{5/2}}\\
			0
		\end{array}\right), \\
		& \xi_2=\left(\begin{array}{c}
			\frac{p_e\left(-2 p \nu \sqrt{p p_e-p_v} \theta_e-u p_v\left(\mu-2 \nu \theta_e\right)+p u p_e\left(\mu+\nu \theta_e\right)\right)}{4 v\left(p p_e-p_v\right)^{5/2}} \\
			-\frac{\left(p_e\left(p \sqrt{p p_e-p_v} \mu+\left(2 u \nu p_v-p \nu \sqrt{p p_e-p_v}\right) \theta_e\right)-\mu \sqrt{p p_e-p_v} p_v\right)}{4 v\left(p_v-p p_e\right)^2} \\
			-\frac{\left(2 p^2 u\left(\mu-\nu \theta_e\right) p_e^2+u p_v\left(\nu\left(5 p+2 u \sqrt{p p_e-p_v}\right) \theta_e-3 p \mu\right) p_e+p_v\left(u \mu p_v-2 p \nu \sqrt{p p_e-p_v} \theta_e\right)\right)}{4 v\left(p p_e-p_v\right)^{5/2}}\\
			0
		\end{array}\right), \\
		& \xi_3=\left(\begin{array}{cc}
			-\frac{p_e\left(p p_e\left(\mu+\nu \theta_e\right)-p_v\left(\mu-2 \nu \theta_e\right)\right)}{4 v\left(p p_e-p_v\right)^{5 / 2}} & 0 \\
			\frac{\nu p_e p_v\theta_e}{2 v\left(p_v-p p_e\right)^2} & 0\\
			\frac{p_e\left(p_e\left(\mu-\nu \theta_e\right) p^2+p_v\left(2 \nu\left(2 p+u \sqrt{p p_e-p_v}\right) \theta_e-p \mu\right)\right)}{4 v\left(p p_e-p_v\right)^{5/2}} & 0\\
			0 & 0
		\end{array}\right),\\
		& M_3^1=\left(\zeta_1, \zeta_2, \zeta_3\right), \\
		& \zeta_1=\left(\begin{array}{c}
			\frac{p_v\left(p p_e\left(\mu+3 \nu \theta_e\right)-\mu p_v\right)}{4 v\left(p p_e-p_v\right)^{5/2}} \\
			\frac{p \nu p_e p_v \theta_e}{2 v\left(p_v-p p_e\right)^2} \\
			\frac{p_v\left(p p_v\left(\mu-2 \nu \theta_e\right)-p p_e\left(p \mu+\nu\left(p-2u \sqrt{p p_e-p_v}\right) \theta_e\right)\right)}{4 v\left(p p_e-p_v\right)^{5/2}}\\
			0
		\end{array}\right), \\
		& \zeta_2=\left(\begin{array}{c}
			-\frac{p_e\left(2 p \nu \sqrt{p p_e-p_v} \theta_e-u p_v\left(\mu-2\nu \theta_e\right)+p u p_e\left(\mu+\nu \theta_e\right)\right)}{4 v\left(p p_e-p_v\right)^{5/2}} \\
			\frac{\left(p_e\left(p \mu \sqrt{p p_e-p_v}-\nu\left(\sqrt{p p_e-p_v} p+2 u p_v\right) \theta_e\right)-\mu \sqrt{p p_e-p_v} p_v\right)}{4 v\left(p_v-p p_e\right)^2} \\
			\frac{\left(2 p^2 u\left(\mu-\nu \theta_e\right) p_e^2+u p_v\left(\nu\left(5 p-2 u \sqrt{p p_e-p_v}\right) \theta_e-3 p \mu\right) p_e+p_v\left(u \mu p_v+2 p \nu \sqrt{p p_e-p_v} \theta_e\right)\right)}{4 v\left(p p_e-p_v\right)^{5/2}}\\
			0
		\end{array}\right), \\
		& \zeta_3=\left(\begin{array}{cc}
			\frac{p_e\left(p p_e\left(\mu+\nu \theta_e\right)-p_v\left(\mu-2 \nu \theta_e\right)\right)}{4 v\left(p p_e-p_v\right)^{5/2}} & 0\\
			\frac{\nu p_e p_v \theta_e}{2 v\left(p_v-p p_e\right)^2} & 0\\
			\frac{p_e\left(p_e\left(\nu \theta_e-\mu\right) p^2+p_v\left(p \mu+\left(2 u \nu \sqrt{p p_e-p_v}-4 p \nu\right) \theta_e\right)\right)}{4 v\left(p p_e-p_v\right)^{5/2}} & 0\\
			0 & 0
		\end{array}\right) .
	\end{align*}

	\medskip
	\noindent
	{\bf Acknowledgement}. 
The research of SL is supported by NSFC Projects 12201399, 12331008, and 12411530065, Young Elite Scientists Sponsorship Program by CAST 2023QNRC001, National Key Research $\&$ Development Programs 2023YFA1010900 and 2024YFA1014900, Shanghai Rising-Star Program 24QA2703600, Shanghai Qiguang Scholarship, and the Shanghai Frontier Research Institute for Modern Analysis. 

The research of HW is supported by National Key R$\&$D Program 2022YFA1007300 and NSFC 12371220.

The research of JY is partially supported by National Key Research $\&$ Development Programs 2023YFA1010900 and 2024YFA1014900.

\end{document}